\documentclass{article}
\usepackage{tikz-cd}
\usepackage[english]{babel}
\usepackage[utf8]{inputenc}
\usepackage[all,cmtip]{xy}
\usepackage{extarrows}
\usepackage{amsfonts}
\usepackage{amsmath}
\usepackage{mathtools}
\usepackage{amssymb}
\usepackage{amsthm}
\usepackage[hidelinks]{hyperref} %Remove hideinks if you want the weird red/green boxes
\usepackage{mathtools}
\usepackage[capitalize]{cleveref}
\usepackage[shortlabels]{enumitem}

\usepackage[top=1.5in, bottom=1in, left=1in, right=1in]{geometry}
\usepackage{hyperref}
\usepackage{setspace}

\newtheorem{thm}{Theorem}[section]
\newtheorem{lem}[thm]{Lemma}
\newtheorem*{clm}{Claim}
\newtheorem{cor}[thm]{Corollary}
\newtheorem{prop}[thm]{Proposition}

\newtheorem{example}[thm]{Example}

\theoremstyle{definition}
\newtheorem{defi}[thm]{Definition}

\theoremstyle{remark}
\newtheorem{rem}{Remark}

\newcommand{\N}
{\mathbb N}
\newcommand{\ov}{\overline}
\newcommand{\Z}{\mathbb Z}

\newcommand{\E}{\mathcal{E}}
\newcommand{\F}{\mathcal{F}}
\newcommand{\G}{\mathcal{G}}
\newcommand{\U}{\mathcal{U}}
\newcommand{\h}{\mathcal{H}}

\newcommand{\HA}{\mathcal{H}_A}
\newcommand{\K}{\mathbb{K}}
\newcommand{\com}{\mathcal{K}}
\newcommand{\s}{\mathcal{S}}
\newcommand{\cc}{\mathcal{C}}
\newcommand{\sr}{C_0(-1,1)}
\newcommand{\R}{\mathbb R}
\newcommand{\C}{\mathbb C}

\DeclarePairedDelimiter\bra{\langle}{\rvert}
\DeclarePairedDelimiter\set{\lbrace}{\rbrace}

\DeclarePairedDelimiter\ket{\lvert}{\rangle}
\DeclarePairedDelimiterX\braket[2]{\langle}{\rangle}{#1 \delimsize\vert #2}
\DeclarePairedDelimiterX\inr[2]{\langle}{\rangle}{#1 \delimsize, #2}
\DeclarePairedDelimiter{\norm}{\lVert}{\rVert}

\newcommand*\closure[1]{\overline{#1}}

\title{The spectral model of (Real) K-theory}

\author{Anupam Datta\thanks{Universit\"at Bonn \\anupamd@math.uni-bonn.de}}

\date{}
\begin{document}

\maketitle

\begin{abstract}
    We use homotopy theoretic ideas to study the $K$-theory of (graded, Real) $C^*$-algebras in detail. After laying the foundations, and deriving the formal properties, the comparison of the model with the Kasparov picture of $K$-theory has been made, and Bott periodicity has been proven using a Dirac-dual Dirac method.
\end{abstract}

%%%%%%%%%%%%%%%%%%%%%%%%%%%%%%%%%%%%%%%%%%%%%%%%%%%%%%%%%

\tableofcontents

\newpage
\pagenumbering{arabic}
\section{Introduction}
This work is devoted to understanding the ``spectral model of K-theory", originally due to Nigel Higson and Erik Guentner (\cite{NIG}). This model is quite general in its setup; it works with $\Z/2\Z$-graded $C^*$-algebras with a possible Real structure. The spectral K-groups are defined as the homotopy groups of a ``spectral K-theory space". As a consequence, we can use tools from homotopy theory to deduce results about these groups, which help to streamline a lot of arguments which were otherwise quite technical in the classical models of operator K-theory (see for example section 3.2). However, the fact that in the unital and ungraded case, spectral K-theory is naturally isomorphic to the well-known projective module picture of K-theory is not present in the literature with sufficient details. We establish this isomorphism with the necessary details. Also, the proofs presented in \cite{NIG} don't carry over to the Real case. One of the main results of this thesis is to show that the Real spectral model "classifies" Real K-theory, and can be used to deduce fundamental results about the same.

Here's a rough idea about the contents of the article:  In section 2, we give the necessary background and examples of the objects we study. Then, we cast spectral K-theory as a topologically enriched functor, and deduce some corollaries to this formalism. 
In section 3, we dicuss stability and long exact sequences in spectral K-theory. Sections 4,5 and 6 are framed to carefully develop sums and products in the spectral picture. We do most of the work at the space level, which has not been done systematically in the literature. After all this, we finally define the spectral K-theory groups in section 7, and deduce some statements about them from the work done previously. Then in section 8, we prove our main result, deducing that spectral K-theory is isomorphic to the projective module picture of K-theory for unital and ungraded $C^*$-algebras. In fact, we show that both these pictures are naturally isomorphic to a third picture, what we call the ``Fredholm picture of K-theory", and it stems primarily from ideas in KK-theory. Finally, in section 9, we develop Bott periodicity in the spectral model, and deduce the 2-fold and 8-fold periodicity of complex and Real K-theory from it.

\section{The first definitions, examples, and properties}
As mentioned in the introduction, we want to develop a $K$-theory for $\Z/2\Z$-graded and Real $C^*$-algebras. We define these terms in the first part of this section, and give enough examples and counterexamples to motivate them. We will assume familiarity with the theory of (complex, ungraded) $C^*$-algebras throughout this article.
\subsection{Graded and Real C*-algberas}
We begin our considerations by recalling the notion of a $\Z/2\Z$-grading on a $C^*$-algebra.
\begin{defi} \label{defi: Z/2Z graded}
    A $\Z/2\Z$-graded complex $C^*$-algebra $(A,\varepsilon)$ is a complex $C^*$-algebra $A$ equipped with a $*$-automorphism $\varepsilon$ such that $\varepsilon^2=1$. $\varepsilon$ is called the grading on $A$. The $+1$ and $-1$ eigenspaces of $A$ are called the even and odd elements of $A$ respectively.
\end{defi}
\begin{example} \label{ex: grading}
    \,
    \begin{enumerate}
        \item Every $C^*$-algebra can be ``trivally graded" by taking the grading operator to be the identity.
        \item Let $A$ be a $C^*$-algebra, and $u$ a self-adjoint unitary in $\mathcal{M}(A)$, the multiplier algebra of $A$. Then, $A$ can be   ``inner graded" by setting $\varepsilon(a)=uau ~\forall~ a \in A$.
        \item The $C^*$-algebra $C_0(\R)$ can be equipped with a grading given by $f(x) \mapsto f(-x)$. The even and odd parts under this grading precisely correspond to the set of even and odd functions respectively. This is not an inner grading.
        \item Let $\h$ be a $\Z/2\Z$-graded Hilbert space, that is, a Hilbert space equipped with a direct sum decomposition $\h=\h_0 \oplus \h_1$. Then, $\mathcal{B}(\h)$ and $\mathcal{K}(\h)$ become graded $C^*$-algebras with the diagonal-off diagonal grading. These gradings are actually inner; they are implemented by the operator $\begin{pmatrix}
            1 & 0 \\
            0 & -1
        \end{pmatrix}:\h_0 \oplus \h_1 \rightarrow \h_0 \oplus \h_1$.
        \item The maximal tensor product of two graded $C^*$-algebras (denoted by $- \hat\otimes -$) can be equipped with a grading; see the references below for the construction.
    \end{enumerate}
\end{example}

More examples and elaborations on gradings can be found in \cite{NIG}, \cite{BLA}, and \cite{RUF}. At this point, the reader might have the question\footnote{just like the author had, too!}: Why should we care about graded $C^*$-algebras if we are interested in only the ungraded ones? One of the reasons is, we have a natural identification between $\mathcal{B}(\h)$ and the set of odd, self-adjoint operators on $\mathcal{B}(\hat\h)$, where $\h$ is an arbitrary Hilbert space, and $\hat\h$ denotes the graded Hilbert space $\h \oplus \h$ (and similarly for the compact operators). This way, we can apply tools like the functional calculus on these self-adjoint operators, which would otherwise not be available to us in the ungraded world. See Section 8 to see some of this in action!

Next, we move onto Real $C^*$-algebras.

\begin{defi}
A Real $C^*$-algebra $(A,\tau)$ is a complex $C^*$-algebra $A$ equipped with a $\C$-anti-linear $*$-automorphism $\tau$ such that $\tau^2=1$. $\tau$ is called the Real structure on $A$. \footnote{Sometimes, the Real structure will simply be denoted by $a \mapsto \bar{a}$.} The real $C^*$-algebra associated to $\tau$ is $A_{\tau}:=\{ a \in A \mid a = \tau(a)\}$.
\end{defi}
\begin{example} \label{ex: Real}
    \,
    \begin{enumerate}
        \item The complex $C^*$-algebra $\C$, equipped with the complex conjugation map is a Real $C^*$-algebra. We will denote this Real $C^*$-algebra by $\mathbf{C}$ in this article. The real $C^*$-algebra associated to this is simply $\R$.
        \item Let $X$ be a Real compact Hausdorff space, that is, equipped with a homeomorphism $\tau_X$ such that $\tau_X^2=1$. Then, $C(X,\C)$ can be equipped with a Real structure given by $\tau(f)(x)=\overline{f(\tau_X(x))}$.
        \item More generally, if $(A,\tau_A)$ is a Real $C^*$-algebra and $(X,\tau_X)$ is a Real compact Hausdorff space, then  $C(X,A)$ can be equipped with a Real structure given by $\tau(f)(x)=\tau_A(f(\tau_X(x)))$. Note that $C(X,A)_\tau=\{f \in C(X,A) \mid \tau_A \circ f = f \circ \tau_X\}$. 
        \item Let $\h$ be a Real Hilbert space, that is, a complete complex inner product space equipped with a $\C$-antilinear involution $\overline{(.)}$ such that $\overline{\langle x,y \rangle}=\langle \overline{x},\overline{y} \rangle$. Then, $\mathcal{B}(\h)$ and $\mathcal{K}(\h)$ can be equipped with a Real structure given by $\tau(T)(h)=\overline{T(\overline{h})}$.
        \item More generally, if $\h$ is a Real Hilbert module over a Real $C^*$-algebra $(A,\tau_A)$, then $\mathcal{B}_A(\h)$ and $\mathcal{K}_A(\h)$ can be equipped with a Real structure given by $\tau(T)h=\overline{T(\overline{h})}$, where $\overline{(.)}$ is the $A$-antilinear involution on $\h$ satisfying $\tau(\langle x,y \rangle)=\langle \overline{x},\overline{y} \rangle$. 
        \item The maximal tensor product of two Real $C^*$-algebras can be equipped with a Real structure given by the tensor product of the individual Real structures. 
        \item Not every complex $C^*$-algebras can be equipped with a Real structure, see for example \cite{CON}.
    \end{enumerate}
\end{example}
More examples and elaborations on Real $C^*$-algebras can be found in \cite{SCH} and \cite{ROS}. One might wonder how do real and Real $C^*$-algebras compare? To that end, we have the following result:
\begin{thm}
    (See \cite{GOOD}) There is an equivalence between the category of Real $C^*$-algebras with Real structure preserving $*$-homomorphisms, and the category of real $C^*$-algebras and $*$-homomorphisms. The correspondence goes as follows\footnote{the process of assigning a Real $C^*$-algebra to a real one as given here is called complexification}:
    \[
    % https://q.uiver.app/#q=WzAsNixbMCwwLCJcXHtcXHRleHR7UmVhbCBDKi1hbGdlYnJhc31cXH0iXSxbMiwwLCJcXHtcXHRleHR7cmVhbCBDKi1hbGdlYnJhc31cXH0iXSxbMCwxLCIoQSxcXHRhdSkiXSxbMiwxLCJBX1xcdGF1Il0sWzIsMiwiQiJdLFswLDIsIihCIFxcb3RpbWVzX3tcXG1hdGhiYntSfX1cXG1hdGhiYntDfSwxXFxvdGltZXMgXFxvdmVybGluZXsoLil9KSJdLFsyLDMsIiIsMCx7InN0eWxlIjp7InRhaWwiOnsibmFtZSI6Im1hcHMgdG8ifX19XSxbNCw1LCIiLDAseyJzdHlsZSI6eyJ0YWlsIjp7Im5hbWUiOiJtYXBzIHRvIn19fV1d
\begin{tikzcd}
	{\{\text{Real C*-algebras}\}} && {\{\text{real C*-algebras}\}} \\
	{(A,\tau)} && {A_\tau} \\
	{(B \otimes_{\mathbb{R}}\mathbb{C},1\otimes \overline{(.)})} && B
	\arrow[maps to, from=2-1, to=2-3]
	\arrow[maps to, from=3-3, to=3-1]
\end{tikzcd}
    \]
\end{thm}
Thus, while working with Real $C^*$-algebras, one has the tools from the complex world available (such as spectral theory), and the results in the Real world can be transferred over to the real world. On top of that, working directly with real $C^*$-algebras is quite technical as the very definition has more conditions than in the complex world (see \cite{SCH}).

Finally, we consider the situation when our $C^*$-algebras have a grading as well as a Real structure. We want these two structures to interact well, as the following definition shows: 

\begin{defi}\label{defi: Z/2Z graded real}
     A $\Z/2\Z$-graded Real $C^*$-algebra $(A,\varepsilon,\tau)$ is a complex $C^*$-algebra $A$ equipped with a grading $\varepsilon$ and a Real structure $\tau$ which commute, that is, $\varepsilon \circ \tau = \tau \circ \varepsilon$.
\end{defi}
\begin{example} \label{ex: grading and Real}
\,
\begin{enumerate}
    \item Every Real $C^*$-algebra equipped with the trivial grading becomes a Real, $\Z/2\Z$-graded $C^*$-algebra.
    \item  If $(A,\tau_A)$ is a $\Z/2\Z$-graded Real $C^*$-algebra and $(X,\tau_X)$ is a Real compact Hausdorff space, then  $C(X,A)$ can be equipped with a Real structure given by $\tau(f)(x)=\tau_A(f(\tau_X(x)))$ and a grading given pointwise, thereby making it into a $\Z/2\Z$-graded Real $C^*$-algebra. 
    \item The graded maximal tensor product of two $\Z/2\Z$-graded Real $C^*$-algebras can be equipped with a Real structure given by the tensor product of the individual ones. The symbol $- \hat\otimes -$ will stand for the graded maximal tensor product throughout the article.
    \item $C_0(\R)$ becomes a $\Z/2\Z$-graded, Real $C^*$-algebra when equipped with the even-odd grading, and the real structure given by point-wise conjugation. We denote this Real $\Z/2\Z$-graded $C^*$-algebra by $\s$.
    \item Let $\h$ be a $\Z/2\Z$-graded, Real Hilbert space. Then, $\mathcal{B}(\h)$ and $\mathcal{K}(\h)$ become $\Z/2\Z$-graded, Real $C^*$-algebras with the grading and Real structure as discussed in \cref{ex: grading} and \cref{ex: Real}. 
\end{enumerate}
\end{example}
\begin{rem}
    The last two examples will be extremely crucial to us in this article.
\end{rem}
Having given all the relevant definitions, we will now focus only in the generality of $\Z/2\Z$-graded, Real $C^*$-algebras unless otherwise mentioned. We \textit{don't} always mention the adjective ``Real" (but sometimes we do when it's a very important definition / relevant for the proof), it assumed to be there by default. If the reader is only interested in the complex case, they may ignore the Real structure, and the theory goes through. Furthermore, every grading we consider in this article are over $\Z/2\Z$; so at times we drop it too.
\subsection{Hom sets as spaces}

We will be coming up with a homotopy theoretic model of operator $K$-theory, in the sense that the $K$-theory groups of a $\Z/2\Z$-graded (Real) $C^*$-algebra will be defined as the homotopy groups of its ``spectral $K$-theory space". We move towards building this space now, by first coming up with a topology on the space of grading preserving $*$-homomorphisms between two graded $C^*$-algebras $A$ and $B$.
\begin{defi}
    Let $A$ and $B$ be $\Z/2\Z$-graded (Real) $C^*$-algebras. We equip the set of graded (Real) $*$-homomorphisms from $A$ to $B$, denoted $\operatorname{Hom}(A,B)$, with the point-norm topology. That is, a net $\varphi_i \rightarrow \varphi$ in $\operatorname{Hom}(A,B) \iff \varphi_i(a) \rightarrow \varphi(a)$ in the norm topology in $B$ for all $a \in A$. $\operatorname{Hom}(A,B)$ is a pointed space, with $\underline{0}$, the zero homomorphism being the basepoint.
\end{defi}
\begin{rem}
    The point-norm topology is the ``correct" topology to put on the hom sets of $C^*$-algebras. Indeed, it can be shown quite easily that when $X$ and $Y$ are compact Hausdorff spaces, then the Gelfand-Naimark duality 
    \[
    \operatorname{Map}_{~\text{Top}}(X,Y) \rightarrow \operatorname{Hom}(C(Y),C(X))
    \] upgrades to a homeomorphism of topological spaces when $\operatorname{Map}_{~\text{Top}}(X,Y)$ is equipped with the compact open topology in this case, and $\operatorname{Hom}(C(Y),C(X))$ is equipped with the point-norm topology.\footnote{More justifications to the ``correctness" of the point-norm topology will be given in this subsection.}
\end{rem}
Our next goal would be to develop the notion of homotopy between graded $*$-homomorphisms. Before that, we need a few preparatory results concerning mappings into the $\operatorname{Hom}$ spaces from an arbitrary compact Hausdorff space. We will be using some properties of the compact open topology in the upcoming results, all of them can be found in \cite{HAT},\cite{DAVIS}. 

\begin{rem}
    We bring to the attention of the reader a technical point: At this juncture, we have \textit{not} k-ified the point-norm topology on hom sets of $C^*$-algebras. As a result, the hom spaces need not be compactly generated. However, no issues arise here, since we are only mapping from a compact Hausdorff space into them.
\end{rem}

\begin{prop} \label{prop:map into hom space}
    Let $X$ be a compact Hausdorff space, and $A$,$B$ be $\Z/2\Z$-graded $C^*$-algebras. Then, there is a natural homeomorphism of spaces 
    \[
    \operatorname{Map}(X,\operatorname{Hom}(A,B)) \cong \operatorname{Hom}(A,C(X,B)),
    \]
    where the mapping space is equipped with the compact-open topology.
    Furthermore, if $(X,x_0)$ is a pointed space, then the aforementioned homeomorphism restricts to a homeomorphism
    \[
    \operatorname{Map}_*(X,\operatorname{Hom}(A,B)) \cong \operatorname{Hom}(A,C_0(X,B)),
    \]
    where $C_0(X,B)$ is the ideal in $C(X,B)$ consisting of functions which vanish at $x_0$.
\end{prop}
\begin{proof}
     We construct two sided natural maps:
    \[
    % https://q.uiver.app/#q=WzAsNCxbMCwwLCJcXFBoaTpcXG9wZXJhdG9ybmFtZXtNYXB9KFgsXFxvcGVyYXRvcm5hbWV7SG9tfShBLEIpKSBcXHJpZ2h0YXJyb3cgXFxvcGVyYXRvcm5hbWV7SG9tfShBLEMoWCxCKSkiXSxbMCwxLCJmIFxcbG9uZ21hcHN0byBhIFxcbWFwc3RvIFxce3ggXFxtYXBzdG8gZih4KShhKVxcfSJdLFsxLDAsIlxcUHNpOlxcb3BlcmF0b3JuYW1le0hvbX0oQSxDKFgsQikpIFxccmlnaHRhcnJvdyBcXG9wZXJhdG9ybmFtZXtNYXB9KFgsXFxvcGVyYXRvcm5hbWV7SG9tfShBLEIpKSJdLFsxLDEsIlxcdmFycGhpIFxcbG9uZ21hcHN0byB4IFxcbWFwc3RvIFxce2EgXFxtYXBzdG8gXFx2YXJwaGkoYSkoeClcXH0iXV0=
\begin{tikzcd}
	{\Phi:\operatorname{Map}(X,\operatorname{Hom}(A,B)) \rightarrow \operatorname{Hom}(A,C(X,B))} & {\Psi:\operatorname{Hom}(A,C(X,B)) \rightarrow \operatorname{Map}(X,\operatorname{Hom}(A,B))} \\
	{f \longmapsto a \mapsto \{x \mapsto f(x)(a)\}} & {\varphi \longmapsto x \mapsto \{a \mapsto \varphi(a)(x)\}}
\end{tikzcd}
    \]
    Clearly, they are well-defined set maps, and are mutual inverses. It remains to show that they are continuous. 

    For the continuity of $\Phi$
    it suffices to show that for each $a \in A$, the map $\operatorname{ev}_a \circ \Phi: \operatorname{Map}(X,\operatorname{Hom}(A,B)) \rightarrow C(X,B)$ is continuous. The rest follows from the commutative diagram below, and since the left map therein, being composition with a single function $\operatorname{Hom}(A,B) \rightarrow A$, is continuous with respect to the compact-open topology:
    \[
    % https://q.uiver.app/#q=WzAsNCxbMCwwLCJcXG9wZXJhdG9ybmFtZXtNYXB9KFgsXFxvcGVyYXRvcm5hbWV7SG9tfShBLEIpKSAiXSxbMSwwLCJcXG9wZXJhdG9ybmFtZXtIb219KEEsQyhYLEIpKSJdLFsxLDEsIkMoWCxCKSJdLFswLDEsIlxcb3BlcmF0b3JuYW1le01hcH0oWCxCKSJdLFswLDEsIlxcUGhpIl0sWzEsMiwiXFxvcGVyYXRvcm5hbWV7ZXZ9X2EiXSxbMCwzLCItIFxcY2lyYyBcXG9wZXJhdG9ybmFtZXtldn1fYSIsMl0sWzMsMiwiPSJdXQ==
\begin{tikzcd}
	{\operatorname{Map}(X,\operatorname{Hom}(A,B)) } & {\operatorname{Hom}(A,C(X,B))} \\
	{\operatorname{Map}(X,B)} & {C(X,B)}
	\arrow["\Phi", from=1-1, to=1-2]
	\arrow["{\operatorname{ev}_a}", from=1-2, to=2-2]
	\arrow["{- \circ \operatorname{ev}_a}"', from=1-1, to=2-1]
	\arrow["{=}", from=2-1, to=2-2]
\end{tikzcd}
    \]
    For the continuity of $\Psi$, note that since $X$ is compact, it suffices to show that the adjoint map $\Psi^{\operatorname{ad}}:X \times \operatorname{Hom}(A,C(X,B)) \rightarrow \operatorname{Hom}(A,B)$, is continuous. It suffices to show that for each $a \in A$, the map $\operatorname{ev}_a \circ \Psi^{\operatorname{ad}}:X \times \operatorname{Hom}(A,C(X,B)) \rightarrow B$ is continuous. But this follows from the following commutative diagram:

\[
% https://q.uiver.app/#q=WzAsNCxbMCwwLCJYIFxcdGltZXMgXFxvcGVyYXRvcm5hbWV7SG9tfShBLEMoWCxCKSkgIl0sWzIsMCwiXFxvcGVyYXRvcm5hbWV7SG9tfShBLEIpIl0sWzIsMSwiQiJdLFswLDEsIlggXFx0aW1lcyBDKFgsQikiXSxbMCwzLCIxIFxcdGltZXNcXG9wZXJhdG9ybmFtZXtldn1fYSJdLFszLDIsIih4LGYpIFxcbWFwc3RvIGYoeCkiXSxbMCwxLCJcXFBzaV57XFxvcGVyYXRvcm5hbWV7YWR9fSJdLFsxLDIsIlxcb3BlcmF0b3JuYW1le2V2fV9hIl1d
\begin{tikzcd}
	{X \times \operatorname{Hom}(A,C(X,B)) } && {\operatorname{Hom}(A,B)} \\
	{X \times C(X,B)} && B
	\arrow["{1 \times\operatorname{ev}_a}", from=1-1, to=2-1]
	\arrow["{(x,f) \mapsto f(x)}", from=2-1, to=2-3]
	\arrow["{\Psi^{\operatorname{ad}}}", from=1-1, to=1-3]
	\arrow["{\operatorname{ev}_a}", from=1-3, to=2-3]
\end{tikzcd}
\]      
    It is direct to see that the maps restrict as asserted in the pointed case.
\end{proof}
\begin{prop} \label{prop: hom enriched over spaces}
    Let $A,B,C$ be $\Z/2\Z$-graded $C^*$-algebras. Then, the composition map
    \[
    M:\operatorname{Hom}(A,B) \times \operatorname{Hom}(B,C) \rightarrow \operatorname{Hom}(A,C)
    \]
    is continuous.
\end{prop}
\begin{proof}
    Let $\varphi_i \rightarrow \varphi$ and $\psi_i \rightarrow \psi$ be convergent nets in $\operatorname{Hom}(A,B)$ and $\operatorname{Hom}(B,C)$ respectively. We have to show that $\psi_i \circ \varphi_i \rightarrow \psi \circ \varphi$ in $\operatorname{Hom}(A,C)$. By the definition of the point-norm topology, it suffices to show that $\psi_i \circ \varphi_i(a) \rightarrow \psi \circ \varphi(a) ~\forall~ a \in A$. Now, for a given $a \in A$, we have 
    \begin{align*}
        &\lVert \psi(\varphi(a)) - \psi_i(\varphi_i(a))\rVert \\
        \leq& \lVert \psi(\varphi(a)) - \psi_i(\varphi(a))\rVert + \lVert \psi_i(\varphi(a)) - \psi_i(\varphi_i(a))\rVert \\
        \leq& \lVert \psi(\varphi(a)) - \psi_i(\varphi(a))\rVert + \lVert \varphi(a) - \varphi_i(a)\rVert ~\text{(as $*$-homomorphisms are contractive.)}.
    \end{align*}
    Both the summands can be arbitrarily small because of the convergences of $\psi_i$ and $\varphi_i$ and the conclusion follows.
\end{proof}
Now, we study the notion of homotopy between graded $*$-homomorphisms.
\begin{defi}
    Let $A$ and $B$ be $\Z/2\Z$-graded $C^*$-algebras. Two graded $*$-homomorphisms $\varphi,\psi: A \rightarrow B$ are said to be homotopic if there exists a path in $\operatorname{Hom}(A,B)$ between them.
\end{defi}
As an immediate consequence of \cref{prop:map into hom space} and \cref{prop: hom enriched over spaces}, we have the following two corollaries:
\begin{cor}\label{cor: homotopic homomorphism}
     Let $A$ and $B$ be $\Z/2\Z$-graded $C^*$-algebras. Two graded $*$-homomorphisms $\varphi_0,\varphi_1: A \rightarrow B$ are homotopic iff there exists a graded $*$-homomorphism $\Phi:A \rightarrow C([0,1],B)$ such that $\operatorname{ev}_i \circ \Phi=\varphi_i$ for $i=0,1$.
\end{cor}
\begin{rem}
    The maps $\operatorname{ev}_t \circ \Phi$ ($t \in [0,1]$) as above will often be denoted as $\Phi_t$.
\end{rem}
\begin{cor} \label{lem: 1st lemma}
    Let $A,B,C$ be graded $C^*$-algebras. Then, given any graded $*$-homomorphism $\varphi:A \rightarrow B$, the induced maps $\operatorname{Hom}(C,A) \rightarrow \operatorname{Hom}(C,B)$ and $\operatorname{Hom}(B,C) \rightarrow \operatorname{Hom}(A,C)$ are continuous. Moreover, homotopic graded $*$-homomorphisms yield based homotopic maps.
\end{cor}
\subsection{The Spectral K-theory space underlying a Hilbert space}
We are now ready to give the definition of the spectral $K$-theory space.
\begin{defi} \label{defi: spectral k theory space}
    Let $A$ be $\Z/2\Z$-graded (Real) $C^*$-algebra, and $\mathcal{H}$ be a graded (Real) Hilbert space. Define \textit{the spectral $K$-theory space of $A$ underlying $\h$} to be the space $\K(A,\mathcal{H}):=\operatorname{Hom}(\s,A \hat\otimes \mathcal{K}(\mathcal{H}))$.
\end{defi}
To get a feel of this space, we give examples of elements in it, in the case when $A=\mathbf{C}$ and $\h=\hat{l^2}$. The example needs some familiarity with unbounded operator theory, and especially the spectral theorem for self-adjoint operators; see \cite{USAO} for the relevant background. We denote the space $\K(\mathbf{C},\hat{l^2})$ by $\K(\mathbf{C})$ in what follows (see also \cref{defi: spectral K-theory space}).
\begin{lem}\label{lem: unbounded}
    Let $D$ an odd, self-adjoint Real operator on $\hat{l^2}$ with compact resolvents. Then, the functional calculus of $D$, $\Phi_D:\s \rightarrow \mathcal{K}(\h)$ is an element of $\K(\mathbf{C})$. Moreover, if $P_0$ denotes the projection onto the kernel of $D$, then $\Phi_D$ and the graded Real homomorphism $f \mapsto f(0)P_0$ lie in the same path-component of $\K(\mathbf{C})$.
\end{lem}
\begin{proof}
    That $\Phi_D$ is an element of $\K(\mathbf{C})$ is immediate since $D$ has compact resolvents. We therefore focus only on the second assertion.
    
    As $D$ has compact resolvents, $\sigma(D)$ is discrete, countable, and consists of eigenvalues of $D$. By the spectral theorem, we can write $D=\sum_{\lambda \in \sigma(D)} \lambda P_\lambda$, where $P_\lambda$ denotes the projection onto the eigenspace of $\lambda$. We then have \[
    f(D)=\sum_{\lambda \in \sigma(D)} f(\lambda) P_\lambda
    \] for all bounded measurable functions $f$ on $\R$. Consider the map
    \begin{align*}
        [0,1] &\rightarrow \K(\mathbf{C})\\
        t &\longmapsto \begin{cases}
            f &\mapsto f(t^{-1}D) ~\text{if}~ t>0 \\
            f &\mapsto f(0)P_0 ~\text{if}~ t=0
        \end{cases}
    \end{align*}
    It suffices to show that it is continuous. The continuity at all non-zero points are trivial. To see the continuity at $0$, we calculate, for an arbitrary but fixed $f \in \s$, 
    \begin{align*}
        \lVert f(s^{-1}D)-f(0)P_0 \rVert &= \lVert f(s^{-1}D)-f(0)\chi_0(D) \rVert \\
        &= r(f(s^{-1}D)-f(0)\chi_0(D)), ~\text{where $r$ denotes the spectral radius} \\
        &\leq \operatorname{sup}\{\lvert f(s^{-1}\lambda)-f(0)\chi_0(\lambda) \rvert \mid \lambda \in \sigma(D)\} ~\text{as $\sigma(g(D)) \subseteq \overline{g(\sigma(D))}$} \\\
        &= \operatorname{sup}\{\lvert f(s^{-1}\lambda)\rvert \mid \lambda \in \sigma(D), \lambda \neq 0\},
    \end{align*}\footnote{\text{or $0$, if $\sigma(D)=\{0\}$}.}
    which can be made arbitrarily small as $s \rightarrow 0$. The rest follows from the definition of the point-norm topology.
\end{proof}
We want to think about spectral $K$-theory space in a categorical way. To that end, we set up some topologically enriched  categories:
\begin{defi}
    Let $\mathfrak{C}^*$ denote the category of graded (Real) $C^*$-algebras with graded $*$-homomorphisms as morphisms. We equip the $\operatorname{Hom}$ sets with the k-ification of the point-norm topology, thereby making it into a topologically enriched category.
\end{defi}
\begin{rem}
    That $\mathfrak{C}^*$ is indeed topologically enriched follows from \cref{prop: hom enriched over spaces}.
\end{rem}
\begin{defi}
    Let $\mathfrak{H}$ denote the category of graded separable (Real) Hilbert spaces with even (Real) isometries as morphisms. We equip the $\operatorname{Hom}$ sets with the k-ification of the strong-$*$ topology, thereby making it into a topologically enriched category.
\end{defi}
\begin{rem}
    That $\mathfrak{H}$ is indeed topologically enriched follows from the basic functional analysis fact that composition of isometries is continuous with respect to the strong-$*$ topology.
\end{rem}
\begin{rem}
    All Hilbert spaces we consider in this article will be separable even if we don't explicitly mention it.
\end{rem}
\begin{rem}
    Throughout this article, we will be using the well known fact that the strong, strong-* and $\sigma$-strong-* (also called strict) topologies on the spaces of even isometries (and unitaries) coincide. See \cite{BLA} and \cite{JT} for the relevant definitions and properties.
\end{rem}
Note that, we have a functor 
\[
    \K(-,-): \mathfrak{C}^* \times \mathfrak{H} \rightarrow \mathfrak{Top}_*,
\] given as the composition of the functors
\[
    \mathfrak{C}^* \times \mathfrak{H} \xrightarrow{1 \times \com(-)} \mathfrak{C}^* \times \mathfrak{C}^* \xrightarrow{- \hat\otimes -} \mathfrak{C}^* \xrightarrow{\operatorname{Hom}(\s,-)} \mathfrak{Top}_*.
\]
\begin{rem}
    We oftentimes pass only a single argument to $\K(-,-)$. In this case, the other parameter is taken to be the identity by default.
\end{rem}
Our main result of this section is that $\K(-,-)$ is a topologically enriched functor, after k-ification:
\begin{thm} \label{thm: enriched cat}
    The functor 
    \[
        k\K(-,-): \mathfrak{C}^* \times \mathfrak{H} \rightarrow \mathfrak{Top}^{CG}_*
    \]
    is a topologically enriched functor.
\end{thm}
We build towards the proof of this theorem by proving some lemmas, which are of independent interest, and will be needed later too.

We begin with a general lemma on Banach spaces:
    \begin{lem} \label{lem: Ban spaces}
    Let $X,Y$ be Banach spaces, and $Z$ a topological space. Then, a map
    \begin{align*}
        G:Z &\rightarrow B_1(\mathcal{B}(X,Y))
    \end{align*} is continuous $\iff$ the map
    \begin{align*}
        G^{\operatorname{ad}}: Z \times X &\rightarrow Y \\
        (z,x) &\mapsto G(z)x
    \end{align*}
    is continuous, where $B_1(\mathcal{B}(X,Y))$ denotes the unit ball (in the norm topology) inside the space $\mathcal{B}(X,Y)$ of bounded operators between $X$ and $Y$ equipped with the point-norm topology.
\end{lem}
\begin{proof}
    $\implies$ We have to show that if $z_i \rightarrow z$ in $Z$ and $x_i \rightarrow x$ in $X$, then $G(z_i)x_i \rightarrow G(z)x$ in $Y$. We have 
    \begin{align*}
        &\lVert G(z)x - G(z_i)x_i \rVert \\
        \leq&\lVert G(z)x - G(z_i)x \rVert + \lVert G(z_i)x - G(z_i)x_i \rVert \\
        \leq&\lVert G(z)x - G(z_i)x \rVert + \lVert x - x_i \rVert
    \end{align*}
    where for the last inequality, we used that $G$ mapped into the space of contractive operators. By the continuity of $G$, the first summand can be made arbitrarily small, and this concludes the forward implication.

    $\impliedby$ By definition of the point-norm topology, this immediately follows by the continuity of $G^{\operatorname{ad}}$.
\end{proof}
\begin{lem}\label{lem: cont isometries}
    Let $\h_0$,$\h_1$ be graded Hilbert spaces. Then, the following map is continuous:
    \[
    \operatorname{Ad}(-):\operatorname{Iso}^{\operatorname{ev}}(\mathcal{H}_0,\mathcal{H}_1) \rightarrow \operatorname{Hom}(\mathcal{K}(\mathcal{H}_0), \mathcal{K}(\mathcal{H}_1))
    \]
\end{lem}
\begin{proof}
    In the light of \cref{lem: Ban spaces}, our claim will follow immediately once we show that the following map is continuous:
    \begin{align*}
        G:\operatorname{Iso}^{\operatorname{ev}}(\mathcal{H}_0,\mathcal{H}_1) \times  \mathcal{K}(\mathcal{H}_0) \rightarrow  \mathcal{K}(\mathcal{H}_1)\\
        (\iota,T) \mapsto \operatorname{Ad}(\iota)(T)
    \end{align*}
   Let $\iota_i \rightarrow \iota$ in $\operatorname{Iso}^{\operatorname{ev}}(\h_0,\h_1)$, and $T_i \rightarrow T$ in $\mathcal{K}(\h)$. Then, we have
    \begin{align*}
        & \lVert \iota T \iota^* - \iota_i T_i {\iota_i}^* \rVert \\
        \leq & \lVert \iota T \iota^* - \iota_i T {\iota_i}^* \rVert + \lVert \iota_i T {\iota_i}^* - \iota_i T_i {\iota_i}^* \rVert \\
        \leq & \lVert \iota T \iota^* - \iota_i T {\iota_i}^* \rVert + \lVert T-T_i \rVert \\
        \leq & \lVert \iota T \iota^* - \iota_i T {\iota}^* \rVert + \lVert \iota_i T {\iota}^* - \iota_i T {\iota_i}^* \rVert + \lVert T-T_i \rVert
    \end{align*}
    The rest follows as the strong-$*$ and the $\sigma$-strong-$*$ topologies on $\operatorname{Iso}^{\operatorname{ev}}(\h_0,\h_1)$ coincide.
\end{proof}
\begin{lem} \label{lem: cont of tensor product}
    Let $A,B,C,D$ be $\Z/2\Z$-graded $C^*$-algebras. Then, if $\varphi_i \rightarrow \varphi$ in $\operatorname{Hom}(A,B)$, and $\psi_i \rightarrow \psi $ in $\operatorname{Hom}(C,D)$, then $\varphi_i \hat\otimes \psi_i \rightarrow \varphi \hat\otimes \psi$ in $\operatorname{Hom}(A\hat\otimes C, B\hat\otimes D)$.
\end{lem}
\begin{proof}
    We first prove it for simple tensors: Given $a \hat\otimes b$ with $a \in A, b \in B$, we have
    \begin{align*}
        &\lVert \varphi(a) \hat\otimes \psi(b) - \varphi_i(a) \hat\otimes \psi_i(b) \rVert \\
         \leq &\lVert \varphi(a) \hat\otimes \psi(b) - \varphi_i(a) \hat\otimes \psi(b) \rVert + \lVert \varphi_i(a) \hat\otimes \psi(b) - \varphi_i(a) \hat\otimes \psi_i(b) \rVert \\
        =& \lVert (\varphi(a)-\varphi_i(a)) \hat\otimes \psi(b)) \rVert + \lVert \varphi_i(a) \hat\otimes (\psi(b)-\psi_i(b)) \rVert   
    \end{align*} which can be made arbitrarily small as $\lVert x \hat\otimes y \rVert \leq \lVert x \rVert \lVert y \rVert$ and $*$-homomorphisms are contractive.
    
    Now, for a general element $x \in A \hat\odot B \subseteq A \hat\otimes B$, write $x=\sum_{k=1}^{n}a_k \hat\otimes b_k$. Then, we have 
    \[
    (\varphi_i \hat\otimes \psi_i)(\sum_{k=1}^{n}a_k \hat\otimes b_k) = \sum_{k=1}^{n} (\varphi_i\hat\otimes \psi_i)(a_k \hat\otimes b_k) \rightarrow \sum_{k=1}^{n} (\varphi\hat\otimes \psi)(a_k \hat\otimes b_k)=(\varphi \hat\otimes \psi)(\sum_{k=1}^{n}a_k \hat\otimes b_k), 
    \] by the analysis done on elementary tensors before.

    Finally, for a general $x \in A\hat\otimes B$, we can approximate $x$ by an element of $y \in A \hat\odot B$. Then, 
    \begin{align*}
      & \lVert (\varphi \hat\otimes \psi)(x) - (\varphi_i \hat\otimes \psi_i)(x) \rVert \\
      \leq & \lVert (\varphi \hat\otimes \psi)(x) - (\varphi \hat\otimes \psi)(y) \rVert + \lVert (\varphi \hat\otimes \psi)(y) - (\varphi_i \hat\otimes \psi_i)(y) \rVert + \lVert (\varphi_i \hat\otimes \psi_i)(y) - (\varphi_i \hat\otimes \psi_i)(x) \rVert \\
      \leq & \lVert x-y \rVert + \lVert (\varphi \hat\otimes \psi)(y) - (\varphi_i \hat\otimes \psi_i)(y) \rVert + \lVert x-y \rVert,
    \end{align*}
    and the rest follows by what's shown previously.
\end{proof}
\begin{proof}[(Proof of \cref{thm: enriched cat})]
    The most non-trivial part of the assertion is that the map induced between the hom spaces is continuous, i.e., the following map is continuous:
    \[
    k\operatorname{Hom}(A,B) \times_{\operatorname{CG}} k\operatorname{Iso}^{\operatorname{ev}}(\h_0,\h_1) \rightarrow \operatorname{Map}(k\K(A,\h_0),k\K(B,\h_1)).
    \]
    We only show this here; note that the map can be written as a composite
    \[
        % https://q.uiver.app/#q=WzAsNSxbMCwwLCJrXFxvcGVyYXRvcm5hbWV7SG9tfShBLEIpIFxcdGltZXNfe1xcb3BlcmF0b3JuYW1le0NHfX0ga1xcb3BlcmF0b3JuYW1le0lzb31ee1xcb3BlcmF0b3JuYW1le2V2fX0oXFxtYXRoY2Fse0h9XzAsXFxtYXRoY2Fse0h9XzEpIl0sWzAsMSwiayhcXG9wZXJhdG9ybmFtZXtIb219KEEsQikgXFx0aW1lcyBcXG9wZXJhdG9ybmFtZXtJc299XntcXG9wZXJhdG9ybmFtZXtldn19KFxcbWF0aGNhbHtIfV8wLFxcbWF0aGNhbHtIfV8xKSkiXSxbMCwyLCJrKFxcb3BlcmF0b3JuYW1le0hvbX0oQSxCKSBcXHRpbWVzIFxcb3BlcmF0b3JuYW1le0hvbX0oXFxtYXRoY2Fse0t9KFxcbWF0aGNhbHtIfV8wKSxcXG1hdGhjYWx7S30oXFxtYXRoY2Fse0h9XzEpKSkiXSxbMiwyLCJrXFxvcGVyYXRvcm5hbWV7SG9tfShBIFxcaGF0XFxvdGltZXMgXFxtYXRoY2Fse0t9KFxcbWF0aGNhbHtIfV8wKSxCIFxcaGF0XFxvdGltZXNcXG1hdGhjYWx7S30oXFxtYXRoY2Fse0h9XzEpKSJdLFsyLDAsIlxcb3BlcmF0b3JuYW1le01hcH0oa1xcbWF0aGJie0t9KEEsXFxtYXRoY2Fse0h9XzApLGtcXG1hdGhiYntLfShCLFxcbWF0aGNhbHtIfV8xKSkiXSxbMCwxLCI9Il0sWzEsMiwiaygxIFxcdGltZXNcXG9wZXJhdG9ybmFtZXtBZH0oLSkpIl0sWzIsMywiaygtIFxcaGF0XFxvdGltZXMtKSJdLFszLDRdLFswLDRdXQ==
\begin{tikzcd}
	{k\operatorname{Hom}(A,B) \times_{\operatorname{CG}} k\operatorname{Iso}^{\operatorname{ev}}(\mathcal{H}_0,\mathcal{H}_1)} && {\operatorname{Map}(k\mathbb{K}(A,\mathcal{H}_0),k\mathbb{K}(B,\mathcal{H}_1))} \\
	{k(\operatorname{Hom}(A,B) \times \operatorname{Iso}^{\operatorname{ev}}(\mathcal{H}_0,\mathcal{H}_1))} \\
	{k(\operatorname{Hom}(A,B) \times \operatorname{Hom}(\mathcal{K}(\mathcal{H}_0),\mathcal{K}(\mathcal{H}_1)))} && {k\operatorname{Hom}(A \hat\otimes \mathcal{K}(\mathcal{H}_0),B \hat\otimes\mathcal{K}(\mathcal{H}_1))}
	\arrow["{=}", from=1-1, to=2-1]
	\arrow["{k(1 \times\operatorname{Ad}(-))}", from=2-1, to=3-1]
	\arrow["{k(- \hat\otimes-)}", from=3-1, to=3-3]
	\arrow[from=3-3, to=1-3]
	\arrow[from=1-1, to=1-3]
\end{tikzcd},
    \]
    where the right vertical map is the one adjoint to the map 
    \[
    k\operatorname{Hom}(A \hat\otimes \mathcal{K}(\mathcal{H}_0),B \hat\otimes\mathcal{K}(\mathcal{H}_1)) \times k\K(A,\h_0) \rightarrow k\K(B,\h_1).
    \]
    In light of the properties of the compact open topology and the k-ification funtor, we are now done: we have already seen that $(1 \times \operatorname{Ad}(-))$, $- \hat\otimes -$, and compositions of graded $*$-homomorphisms are continuous in \cref{lem: cont isometries}, \cref{lem: cont of tensor product} and \cref{prop: hom enriched over spaces} respectively.
\end{proof}
As a consequence of the results developed in this subsection, we have the following:
\begin{cor}\label{prop: baby functoriality}
   Let $\h_0,\h_1$ be graded Hilbert space, and $A,B$ be graded $C^*$-algebras. Then, 
   \begin{enumerate}
       \item Homotopic $*$-homomorphisms $A \rightarrow B$ induce based homotopic maps $\K(A,\h_0) \rightarrow \K(B,\h_0)$.
       \item Strongly path-connected even isometries $\h_0 \rightarrow \h_1$ induce based homotopic maps $ \K(A,\h_0) \rightarrow \K(A,\h_1)$.
        \item $0:A \rightarrow B$ induces the constant map $\K(A,\h_0) \rightarrow \K(B,\h_0)$.
   \end{enumerate}
\end{cor}
\section{Two fundamental properties of spectral K-theory}
In this subsection, we prove a couple of well known properties of $K$-theory from the classical ungraded picture, namely stability and existence of long exact sequences, in the spectral picture. These will also be needed later when we develop products in K-theory. Some of the ideas here were motivated from \cite{EBE}.

 Throughout this section, $H$ denotes an arbitrary (but fixed) graded (Real) Hilbert space. For the proofs, we need to sometimes work with a somewhat simpler model of the spectral $K$-theory space underlying $H$. This, we do in terms of Cayley transforms, as follows:

\begin{lem}\label{lem: cay transform}
    Let $A$ be a $\Z/2\Z$-graded $C^*$-algebra $(A,\varepsilon)$ (with a possible Real structure given by $a \mapsto \bar{a}$). Then there is a natural homeomorphism between $\operatorname{Hom}(\s,A)$ and \[\operatorname{Cay}(A):=\{ u \in \U(A^+) \mid u^*=\varepsilon (u) (=\bar{u}), u \mapsto 1 ~\text{under the canonical map}~ A^+ \rightarrow \C\}\]
\end{lem}
\begin{proof}
     By the universal property of adjunction of units, we get $\operatorname{Hom}(\s,A) \cong \operatorname{Hom^\prime}_{C^*}(\s^+, A^+)$, the space of all unital $*$-homomorphisms preserving the augmentation. Consider the map 
    
    \begin{align*}
         \alpha : \operatorname{Hom^\prime}_{C^*}(\s^+, A^+) & \rightarrow \operatorname{Cay}(A) \\
    \varphi &\mapsto \varphi(\frac{x-i}{x+i})
    \end{align*}
    We claim that this is the required homeomorphism. Indeed, note that $\s^+ \cong C(\R^+) \cong C(\mathbb{S}^1)$, with the later isomorphism being induced from the underlying homeomorphism of spaces $\R^+ \rightarrow \mathbb{S}^1$, given by $x \mapsto \frac{x-i}{x+i}$. The proposition follows once we invoke the fact that $C(\mathbb{S}^1)$ is the universal $C^*$-algebra on a single unitary (see \cite{SUN}), and take into consideration the induced grading and Real structure on $C(\mathbb{S}^1)$. 
\end{proof}
\subsection{Stability}
For this subsection, we assume that the even and odd parts of $H$ are countably infinite dimensional.
Our goal is to show the following:
\begin{thm}\label{thm: stability}
    Let $A$ be a graded $C^*$-algebra, and $\mathcal{H}$ a graded Hilbert space. Then, the map $\Phi:A \rightarrow A \hat\otimes \mathcal{K}(\mathcal{H})$, given by $\Phi(a)=a\hat\otimes e_{11}$, where $e_{11}$ is projection onto an even one-dimensional subspace of $\mathcal{H}$, induces a homotopy equivalence $\K(\Phi):\K(A,H) \rightarrow \K(A \hat\otimes \mathcal{K}(\mathcal{H}),H)$.
\end{thm}
\begin{proof}
    Our goal is to show that the induced map on the spectral $K$-theory space $\K(\Phi):\K(A,H) \rightarrow \K(A \hat\otimes \mathcal{K}(\mathcal{H}),H)$ is a homotopy equivalence.
    First, note that the following diagram commutes:
    \[% https://q.uiver.app/#q=WzAsNCxbMCwwLCJBIl0sWzEsMCwiQSBcXGhhdFxcb3RpbWVzXFxtYXRoY2Fse0t9KFxcbWF0aGNhbHtIfSkiXSxbMCwxLCJBIFxcaGF0XFxvdGltZXMgXFxtYXRoYmJ7Q30iXSxbMSwxLCJBIFxcaGF0XFxvdGltZXNcXG1hdGhjYWx7S30oXFxtYXRoY2Fse0h9KSJdLFswLDIsImEgXFxtYXBzdG8gYSBcXGhhdFxcb3RpbWVzMSIsMl0sWzAsMSwiXFxQaGkiXSxbMiwzLCJcXG9wZXJhdG9ybmFtZXtBZH0oXFxpb3RhKSIsMl0sWzEsMywiXFxvcGVyYXRvcm5hbWV7aWR9Il1d
\begin{tikzcd}
	A & {A \hat\otimes\mathcal{K}(\mathcal{H})} \\
	{A \hat\otimes \mathbb{C}} & {A \hat\otimes\mathcal{K}(\mathcal{H})}
	\arrow["{a \mapsto a \hat\otimes1}"', from=1-1, to=2-1]
	\arrow["\Phi", from=1-1, to=1-2]
	\arrow["\operatorname{id} \hat\otimes {\operatorname{Ad}(\iota)}"', from=2-1, to=2-2]
	\arrow["{\operatorname{id}}", from=1-2, to=2-2],
\end{tikzcd}\]
where $\iota:\C \rightarrow \mathcal{H}$ is the even isometry onto the range of $e_{11}$.

Now, note that there is a canonical homeomorphism $\K(A \hat\otimes \com(\h),H) \cong \K(A,\h \hat\otimes H)$, and under this identification, $\K(\Phi)$ corresponds to the map $\K(\operatorname{Ad}j): \K(A,H) \rightarrow \K(A,\h \hat\otimes H)$, where $j: H=\C \hat\otimes H \rightarrow \h \hat\otimes H$ is given by $\iota \hat\otimes 1$. But, $H$ and $\mathcal{H} \hat\otimes H$ are both graded Hilbert spaces whose even and odd grading degree parts are countably infinite dimensional, and hence there exists an even unitary isomorphism $u$ between them. Then, $\K(\operatorname{Ad}u):\K(A,H) \rightarrow \K(A,\h \hat\otimes H)$ is a homeomorphism. As $\operatorname{Iso}^{\operatorname{ev}}(H,\h \hat\otimes H)$ is path-connected in the strong-$*$ topology, our result follows.
\end{proof}

\subsection{Long exact sequence}
\begin{thm}\label{thm: les}
    Let $A,B$ be $\Z/2\Z$-graded (Real) $C^*$-algebras, and $\varphi:A \rightarrow B$ be a surjective homomorphism. Let $J:=\ker \varphi$. Then, $\K(A,H) \rightarrow \K(B,H)$ is a Serre fibration, with fiber $\K(J)$.

\end{thm}
The following proposition contains the technical heart of the proof:
\begin{prop}\label{prop: Serre fib}
Let $\sigma:A \rightarrow B$ be a surjective homomorphism of $\Z/2\Z$-graded (possibly Real) $C^*$-algebras. Then the induced map $\sigma_* : \operatorname{Hom}(\s,A) \rightarrow \operatorname{Hom}(\s,B)$ is a Serre fibration.
\end{prop}
\begin{proof}
    We work in the general case when $(A,\varepsilon,\overline{(.)})$ is $\Z/2\Z$-graded and Real. We have to show that for any compact CW complex $X$, the following diagram has a solution: \[ % https://q.uiver.app/#q=WzAsNCxbMCwwLCJYIl0sWzIsMCwiXFxvcGVyYXRvcm5hbWV7SG9tfShcXG1hdGhjYWx7U30sQSkiXSxbMCwxLCJYIFxcdGltZXNbMCwxXSJdLFsyLDEsIlxcb3BlcmF0b3JuYW1le0hvbX0oXFxtYXRoY2Fse1N9LEIpIl0sWzAsMV0sWzAsMl0sWzIsM10sWzEsM10sWzIsMSwiIiwxLHsic3R5bGUiOnsiYm9keSI6eyJuYW1lIjoiZGFzaGVkIn19fV1d
\begin{tikzcd}
	X && {\operatorname{Hom}(\mathcal{S},A)} \\
	{X \times[0,1]} && {\operatorname{Hom}(\mathcal{S},B)}
	\arrow[from=1-1, to=1-3]
	\arrow[from=1-1, to=2-1]
	\arrow[from=2-1, to=2-3]
	\arrow[from=1-3, to=2-3]
	\arrow[dashed, from=2-1, to=1-3]
\end{tikzcd} \]
    Now, in light of the homeomorphisms discussed in \cref{prop:map into hom space}, we conclude that a solution in the above diagram is equivalent to a solution in the following diagram:
    \[
    % https://q.uiver.app/#q=WzAsNCxbMCwwLCJcXHswXFx9Il0sWzIsMCwiXFxvcGVyYXRvcm5hbWV7SG9tfShcXG1hdGhjYWx7U30sQyhYLEEpKSJdLFswLDEsIlswLDFdIl0sWzIsMSwiXFxvcGVyYXRvcm5hbWV7SG9tfShcXG1hdGhjYWx7U30sQyhYLEIpKSJdLFswLDFdLFswLDJdLFsyLDNdLFsxLDNdLFsyLDEsIiIsMSx7InN0eWxlIjp7ImJvZHkiOnsibmFtZSI6ImRhc2hlZCJ9fX1dXQ==
\begin{tikzcd}
	{\{0\}} && {\operatorname{Hom}(\mathcal{S},C(X,A))} \\
	{[0,1]} && {\operatorname{Hom}(\mathcal{S},C(X,B))}
	\arrow[from=1-1, to=1-3]
	\arrow[from=1-1, to=2-1]
	\arrow[from=2-1, to=2-3]
	\arrow[from=1-3, to=2-3]
	\arrow[dashed, from=2-1, to=1-3]
\end{tikzcd}
    \]
Furthermore, the induced map $C(X,A) \rightarrow C(X,B)$ is also surjective as a consequence of the Bartle-Graves selection theorem \cite{BARTLE}. So, without loss of generality, it suffices to show that when $\sigma:A \rightarrow B$ is a surjective homomorphism of $\Z/2\Z$-graded Real $C^*$-algebras, then $\sigma_*:\operatorname{Hom}(\s,A) \rightarrow \operatorname{Hom}(\s,B)$ has the path lifting property.

    In light of \cref{lem: cay transform}, it suffices to show that the induced map $\operatorname{Cay}(A) \rightarrow \operatorname{Cay}(B)$ has the path-lifting property. 
For the remainder of the proof, we need some more notation. Set \[A_0:=\{a \in A \mid a^*=\varepsilon(a)=\bar{a}\}.\] Similarly, we define $B_0$. Note that $A_0$ and $B_0$ are real Banach spaces. A modification of the proof that a surjective homomorphism of $C^*$-algebras restricts to a surjective linear map between their self-adjoint elements yields that $\sigma:A_0 \rightarrow B_0$ is surjective\footnote{more precisely, the grading, Real structure and $*$-operations define an irreducible $\Z/2^3\Z$-action on the spaces involved, and taking its isotypical parts is an exact functor. \label{foot}}. Note that \[\operatorname{Cay}(A)=\{u \in \mathcal{U}(A^+) \mid u-1 \in A_0\},\operatorname{Cay}(B)=\{u \in \mathcal{U}(B^+) \mid u-1 \in B_0\}.\]Let $p:[0,1] \rightarrow \operatorname{Cay}(B)$ be a path with a specified starting point $p(0)=\sigma(u_0)$ for some $u_0 \in \operatorname{Cay}(A)$. By the Bartle-Graves selection theorem, there exists a continuous $s:B_0 \rightarrow A_0$ with $s(0)=0$. Thus, there exists $\delta > 0$ such that $||x||<\delta \implies ||s(x)||<\frac{1}{2}$.  By compactness of $[0,1]$, $~\exists~ N \in \N$ such that $|s-t| \leq \frac{1}{N}$ $\implies ||p(s)-p(t)||<\delta$. We will build our lift $\Tilde{p}:[0,1] \rightarrow \operatorname{Cay}(A)$ inductively. The lift at $0$ is given to be $u_0$, and suppose, by induction, that $\Tilde{p}$ has been constructed on $[0,\frac{i}{N}]$ for some $0 \leq i < N$. Set \[p'(t):=\Tilde{p}(\frac{i}{N})+s(p(t)-p(\frac{i}{N})) ~\forall~ \frac{i}{N}\leq t \leq \frac{i+1}{N}.\] From the well-known ``$1+x$" lemma on invertibles in a Banach algebra, and since $p'(\frac{i}{N})$ is a unitary, we conclude that $p'(t) \in \{a \in \operatorname{GL}(A^+) \mid a-1 \in A_0\}, ~\forall~ \frac{i}{N}<t \leq \frac{i+1}{N} $.

    We have a well known strong deformation retraction from $\operatorname{GL}(A^+)$ onto $\mathcal{U}(A^+)$ as follows:
    \begin{align*}
        H:&\operatorname{GL}(A^+)\times [0,1] \rightarrow \operatorname{GL}(A^+) \\
        &(a,t) \mapsto t.a\sqrt{a^*a}^{-1}+(1-t).a
    \end{align*}
    We first show that $H$ preserves the space $A_{00}:=\{a \in \operatorname{GL}(A^+) \mid a-1 \in A_0\}$. To that end, the following claim is made:

    \textbf{Claim:} $a \in A_{00} \implies a\sqrt{a^*a}^{-1} \in A_{00}$.
    \begin{proof}
        As $a-1\in A_{0}$, we have $a^*=\varepsilon(a)=\bar{a}$. A direct calculation (we need to use the fact that the adjoint, grading and Real structure commutes with the continuous functional calculus of the positive element $a^*a$) yields that \[(a\sqrt{a^*a}^{-1})^*=\sqrt{a^*a}^{-1}a^*,~\varepsilon(a\sqrt{a^*a}^{-1})=\overline{a\sqrt{a^*a}^{-1}}=a^*\sqrt{aa^*}^{-1}.\] We claim that for any invertible $b$ in a unital $C^*$-algebra, $\sqrt{b^*b}^{-1}b^*=b^*\sqrt{bb^*}^{-1}$. To see this, first note that $\sigma(bb^*)=\sigma(b^*b)$. This is because $\sigma(xy) \cup \{0\} = \sigma(yx) \cup \{0\}$ in general, and $b$ is given to be invertible. Thus, by the real Stone-Weierstrass theorem, there is a sequence $\{p_n\}$ of polynomials with real coefficients which converge uniformly to $\frac{1}{\sqrt{x}}$ on $\sigma(bb^*)=\sigma(b^*b)$. Thus, $b^*p_n(bb^*)$ and $p_n(b^*b)b^*$ converges in norm to $b^*\sqrt{bb^*}^{-1}$ and $\sqrt{b^*b}^{-1}b^*$ respectively. But, $b^*p_n(bb^*) = p_n(b^*b)b^* ~\forall~ n \in \N$, as $b^*(bb^*)^m$ and $(b^*b)^mb^* ~\forall~ m \in \N$.
        
        It also remains to note that $a\sqrt{a^*a}^{-1} \mapsto 1$ under the $*$-homomorphism $A^+ \rightarrow \C$. But this is clear.
    \end{proof}
    Thus, as the space $\{a \in A^+ \mid a-1 \in A_0\}$ is closed under convex combinations, we get a well-defined continuous map 
    \begin{align*}
        H:&A_{00}\times [0,1] \rightarrow A_{00} \\
        &(a,t) \mapsto t.a\sqrt{a^*a}^{-1}+(1-t).a
    \end{align*}
    Now, we can set 
    \[
    \Tilde{p}(t)=H_1(p'(t)) ~\forall~ t \in [\frac{i}{N},\frac{i+1}{N}].
    \]
    Thus, $\Tilde{p}$ is defined on $[0,\frac{i+1}{N}]$, and we now build our desired lift of $p$ by induction.
\end{proof}

\begin{proof}[Proof (of \cref{thm: les})]
        Forgetting the grading for a bit, we note that the short exact sequence 
        \[% https://q.uiver.app/#q=WzAsNSxbMCwwLCIwIl0sWzEsMCwiSiJdLFsyLDAsIkEiXSxbMywwLCJCIl0sWzQsMCwiMCJdLFswLDFdLFsxLDJdLFsyLDNdLFszLDRdXQ==
\begin{tikzcd}
	0 & J & A & B & 0
	\arrow[from=1-1, to=1-2]
	\arrow[from=1-2, to=1-3]
	\arrow[from=1-3, to=1-4]
	\arrow[from=1-4, to=1-5]
\end{tikzcd}\]
yields the short exact \begin{equation}\label{eqn:1}
    % https://q.uiver.app/#q=WzAsNSxbMCwwLCIwIl0sWzEsMCwiSiBcXG90aW1lc1xcbWF0aGNhbHtLfSJdLFsyLDAsIkFcXG90aW1lc1xcbWF0aGNhbHtLfSJdLFszLDAsIkJcXG90aW1lc1xcbWF0aGNhbHtLfSJdLFs0LDAsIjAiXSxbMCwxXSxbMSwyXSxbMiwzXSxbMyw0XV0=
\begin{tikzcd}
	0 & {J \otimes\mathcal{K}(H)} & {A\otimes\mathcal{K}(H)} & {B\otimes\mathcal{K}(H)} & 0
	\arrow[from=1-1, to=1-2]
	\arrow[from=1-2, to=1-3]
	\arrow[from=1-3, to=1-4]
	\arrow[from=1-4, to=1-5],
\end{tikzcd}
\end{equation} because of the exactness properties of the maximal tensor product in the ungraded world (see \cite{BO}, Proposition 3.7.1). Now, the underlying Banach spaces of graded and ungraded maximal tensor products are the same (Lemma E.2.18 of \cite{RUF}), and so we have the exact sequence
\[% https://q.uiver.app/#q=WzAsNSxbMCwwLCIwIl0sWzEsMCwiSiBcXGhhdFxcb3RpbWVzXFxtYXRoY2Fse0t9Il0sWzIsMCwiQVxcaGF0XFxvdGltZXNcXG1hdGhjYWx7S30iXSxbMywwLCJCXFxoYXRcXG90aW1lc1xcbWF0aGNhbHtLfSJdLFs0LDAsIjAiXSxbMCwxXSxbMSwyXSxbMiwzXSxbMyw0XV0=
\begin{tikzcd}
	0 & {J \hat\otimes\mathcal{K}(H)} & {A\hat\otimes\mathcal{K}(H)} & {B\hat\otimes\mathcal{K}(H)} & 0
	\arrow[from=1-1, to=1-2]
	\arrow[from=1-2, to=1-3]
	\arrow[from=1-3, to=1-4]
	\arrow[from=1-4, to=1-5]
\end{tikzcd}\]
In particular, we get $A \hat\otimes \mathcal{K} (H)\rightarrow B \hat\otimes \mathcal{K}(H)$ is a surjective homomorphism of graded $C^*$-algebras with kernel $\K(J,H)$. The proof follows.
\end{proof}
\section{Two important graded Real *-homomorphisms}
In this section, we discuss two graded $*$-homomorphisms involving $\s$. These will give birth to a plethora of sums and products between spectral $K$-theory spaces which we will see in the subsequent sections.
\subsection{The diagonal map}
This subsection will be rather small, as the construction is straight-forward. The homomorphism we are considering in this section is the diagonal $*$-homomorphism $\operatorname{dia}:\s \rightarrow \s \oplus \s$, given by $\operatorname{dia}(f)=(f,f) ~\forall~ f \in \s$. Recall that, the notation is a bit misleading since $\oplus$ is not a categorical direct sum, but in fact a categorical direct product. However, this is the notation used in $C^*$-algebra texts (see \cite{MUR}), and we stick to this notation.

We conclude this subsection by stating the following property of the diagonal $*$-homomorphism, the proof of which is obvious.
\begin{lem}\label{lem: dia emb}
    The following diagram commutes:
    \[
   % https://q.uiver.app/#q=WzAsNCxbMCwwLCJcXG1hdGhjYWx7U30iXSxbMiwwLCJcXG1hdGhjYWx7U30gXFxvcGx1cyBcXG1hdGhjYWx7U30iXSxbMCwyLCJcXG1hdGhjYWx7U30gXFxvcGx1cyBcXG1hdGhjYWx7U30iXSxbMiwyLCJcXG1hdGhjYWx7U30gXFxvcGx1cyBcXG1hdGhjYWx7U30gXFxvcGx1cyBcXG1hdGhjYWx7U30iXSxbMCwxLCJcXG9wZXJhdG9ybmFtZXtkaWF9Il0sWzAsMiwiXFxvcGVyYXRvcm5hbWV7ZGlhfSIsMl0sWzIsMywiXFxvcGVyYXRvcm5hbWV7ZGlhfSIsMl0sWzEsMywiMSBcXG9wbHVzXFxvcGVyYXRvcm5hbWV7ZGlhfSJdXQ==
\begin{tikzcd}
	{\mathcal{S}} && {\mathcal{S} \oplus \mathcal{S}} \\
	\\
	{\mathcal{S} \oplus \mathcal{S}} && {\mathcal{S} \oplus \mathcal{S} \oplus \mathcal{S}}
	\arrow["{\operatorname{dia}}", from=1-1, to=1-3]
	\arrow["{\operatorname{dia}}"', from=1-1, to=3-1]
	\arrow["{\operatorname{dia}}"', from=3-1, to=3-3]
	\arrow["{1 \oplus\operatorname{dia}}", from=1-3, to=3-3]
\end{tikzcd}
\] 
\end{lem}
\subsection{The codiagonal map}
We now come up with a codiagonal/comultiplication map $\s \rightarrow \s \hat\otimes \s$.  The construction of the comultiplication map is much more intricate than the diagonal map, so we develop it slowly.

First, we need to understand the $\Z/2\Z$-graded Real $C^*$-algebra $\s \hat\otimes \s$ better. We follow the treatment as in \cite{RZ}.
 \begin{lem}\label{lem: s x s}
     $\s \hat\otimes \s$ is isomorphic as a graded, Real $C^*$-algebra to the following:

     Consider the reflection automorphisms on $C_0(\R^2)$ given by $\alpha_x,\alpha_y$ as follows:
     \[
     % https://q.uiver.app/#q=WzAsNCxbMCwwLCJcXGFscGhhX3g6Q18wKFxcbWF0aGJie1J9XjIpIFxccmlnaHRhcnJvdyBDXzAoXFxtYXRoYmJ7Un1eMikgIl0sWzIsMCwiXFxhbHBoYV95OkNfMChcXG1hdGhiYntSfV4yKSBcXHJpZ2h0YXJyb3cgQ18wKFxcbWF0aGJie1J9XjIpICJdLFswLDEsImYgXFxtYXBzdG8gKCh4LHkpXFxtYXBzdG8gZigteCx5KSkiXSxbMiwxLCJmIFxcbWFwc3RvICgoeCx5KVxcbWFwc3RvIGYoeCwteSkpIl1d
\begin{tikzcd}
	{\alpha_x:C_0(\mathbb{R}^2) \rightarrow C_0(\mathbb{R}^2) } && {\alpha_y:C_0(\mathbb{R}^2) \rightarrow C_0(\mathbb{R}^2) } \\
	{f \mapsto ((x,y)\mapsto f(-x,y))} && {f \mapsto ((x,y)\mapsto f(x,-y))}
\end{tikzcd}
     \]
     Define 
     \[
     C_0^{(i,j)}(\R^2)=\{f \in C_0(\R^2) \mid \alpha_x(f)=(-1)^{i}f ~\text{and}~ \alpha_y(f)=(-1)^{j}f\}.
     \]
     The Banach space $C_0(\R^2)$ then becomes a $\Z/2\Z$-graded $C^*$-algebra, by taking the grading operator to be $\alpha_x \circ \alpha_y$, and product and $*$-operation  by setting
     \[
     f \widehat{\bullet}g:=(-1)^{jk}f.g,~~~f^*=(-1)^{ij}\overline{f} ~\text{for}~ f \in C_0^{(i,j)}(\R^2), g \in C_0^{(k,l)}(\R^2)
     \]
    and extended linearly. The Real structure is pointwise conjugation.
     
 \end{lem}
\begin{proof}
    That the prescribed operations indeed make $C_0(\R^2)$ into a $\Z/2\Z$-graded Real $C^*$-algebra is routine; for example, to see that $\alpha_x \circ \alpha_y$ is indeed a grading, one needs to note that $\alpha_x$ and $\alpha_y$ are commuting involutions on $C_0(\R^2)$.

    Note that, from \cite{RUF} Lemma E.2.18, we have that the map
    \begin{align*}
        \s \hat\otimes \s &\rightarrow C_0(\R^2) \\
        f \hat\otimes g &\mapsto ((x,y) \mapsto f(x)g(y))
    \end{align*}
    is a Banach space isomorphism. We claim that this map becomes a Real, graded $*$-homomorphism with the prescribed operations on $C_0(\R^2)$, which is again routine.
\end{proof}    
\begin{rem}
    We emphasize one aspect about the grading on ($C_0(\R^2)$,$\widehat{\bullet}$):  since $\alpha_x$ and $\alpha_y$ are commuting involutions, there is a direct sum decomposition
     \[
     C_0(\R^2)=C_0^{(0,0)}(\R^2) \oplus C_0^{(1,1)}(\R^2) \oplus C_0^{(1,0)}(\R^2) \oplus C_0^{(0,1)}(\R^2).
     \] 
     The grading $\alpha_x \circ \alpha_y$ on $C_0(\R^2)$ is given by considering the first two summands as the even part, and the last two summands as the odd part.
\end{rem}
\begin{prop}\label{prop:comultiplication}
    There is a unique graded $*$-homomorphism from $\s \rightarrow \s \hat\otimes \s$ such that $u \mapsto u \hat\otimes u$, $v \mapsto u \hat\otimes v + v \hat\otimes u$, where $u:=e^{-x^2}$, $v=xe^{-x^2}$.
\end{prop}
\begin{proof}
    The uniqueness of such a homomorphism, if it exists, is clear by the Stone-Weierstrass theorem, as the algebra generated by $u$ and $v$ is dense in $\s$. So, we just need to show existence.

    For the purpose of existence, we take the description of $\s \hat\otimes \s$ as in \cref{lem: s x s}.   Let $r : \mathbb{R}^{2} \to
\mathbb{R}$ be the Euclidean norm, that is, $r(x, y) = \sqrt{x^{2}
+y^{2}}$. Define the functions $\xi_x$ and $\xi_y$ given by $\xi_x(x,y)=\frac{x}{r(x,y)}$ and $\xi_y=\frac{y}{r(x,y)} ~\forall~ x,y \in \R^2 \setminus \{(0,0)\}$, and $\xi_x(0,0)=\xi_y(0,0)=\frac{1}{\sqrt{2}}$.
Note that $\xi_x$ and $\xi_y$ are globally bounded, and is continuous everywhere except at
$(0,0)$. Furthermore, the pointwise products $\xi_x \cdot f, \xi_y \cdot f  \in \mathcal{C}_0(\mathbb{R}^{2})$ for every $f \in
\mathcal{C}_0(\mathbb{R}^{2})$ vanishing at $(0,0)$. We define
linear maps,
\begin{align}
	\widetilde{\triangle}^{(0)} : \mathcal{S}^{(0)} 
	&\to \mathcal{C}_0^{(0,0)}(\mathbb{R}^{2}) \subset
	\mathcal{C}_0(\mathbb{R}^{2}),\notag\\
	f
	&\mapsto f \circ r,
\end{align}
\begin{align}
	\widetilde{\triangle}^{(1)} : \mathcal{S}^{(1)} 
	&\to \mathcal{C}_0^{(1,0)}(\mathbb{R}^{2}) \oplus
	\mathcal{C}_0^{(0,1)}(\mathbb{R}^{2})  \subset
	\mathcal{C}_0(\mathbb{R}^{2}),\notag\\
	f
	&\mapsto \xi_x \cdot (f \circ r) + \xi_y \cdot (f \circ r).
\end{align}
and let $\widetilde{\triangle} := \widetilde{\triangle}^{(0)} \oplus
\widetilde{\triangle}^{(1)} : \mathcal{S} \to
\mathcal{C}_0(\mathbb{R}^{2})$. 
\begin{clm}
    $\widetilde{\triangle}$ is a graded $*$-homomorphism $\s \rightarrow (\mathcal{C}_0(\mathbb{R}^{2}), \hat{\bullet})$, which satisfies
\begin{equation}
	\widetilde{\triangle}(u) = e^{-x^{2} - y^{2}}, \qquad
	\widetilde{\triangle}(v) = (x+y)e^{-x^{2} -y^{2}}.
	\label{eqn:lol}
\end{equation}
\end{clm}
\begin{proof}
    Additivity of $\widetilde{\triangle}$ is clear from definitition, and so is the fact that it is graded. Multiplicativity of $\widetilde{\triangle}$ can be checked only on even and odd functions in $\s$; here, we only show multiplicativity when $f,g \in \s^{(1)}$. Indeed, we have 
    \begin{align*}
        \widetilde{\triangle}(f)\hat{\bullet}\widetilde{\triangle}(g) &= \widetilde{\triangle}^{(1)}(f)\hat{\bullet}\widetilde{\triangle}^{(1)}(g)\\
        &=(\xi_x \cdot (f \circ r) + \xi_y \cdot (f \circ r)) \hat{\bullet} (\xi_x \cdot (g \circ r) + \xi_y \cdot (g \circ r)) \\
        &= \begin{cases}
            \xi_x^2.(fg \circ r) + \xi_x\xi_y.(fg \circ r) - \xi_y\xi_x.(fg \circ r) + \xi_y^2.(fg \circ r) & \text{on}~ \R^2 \setminus \{(0,0)\} \\
            0 &\text{at}~ (0,0). 
            \end{cases}\\
            &= \begin{cases}
            \xi_x^2.(fg \circ r)  + \xi_y^2.(fg \circ r) & \text{on}~ \R^2 \setminus \{(0,0)\} \\
            0 &\text{at}~ (0,0). \\
        \end{cases} \\
        &= 
            fg \circ r  \\
            &= \widetilde{\triangle}(fg)
    \end{align*}
    Similarly, it can be shown that $\widetilde{\triangle}$ preserves adjoints, and so its a graded $*$-homomorphism.
    \\
    The values of $\widetilde{\triangle}$ at $u$ and $v$ are clear from the definitions.
\end{proof}
Thus, we conclude that $\triangle:=\Phi^{-1} \circ \widetilde{\triangle}$ is a graded $*$-homomorphism $\s \rightarrow \s \hat\otimes \s$ such that $u \mapsto u \hat\otimes u$ , and $v \mapsto u \hat\otimes v + v \hat\otimes u$, as is clear by the definition of $\Phi$.
\end{proof}
We also have a graded $*$-homomorphism $\eta:\s \rightarrow \C$, given by $\eta(f)=f(0)$. 
\begin{cor}
    The $*$-homomorphisms $\eta$ and $\triangle$ equip $\s$ with a coalgebra structure; in the sense that the following diagrams commute
% https://q.uiver.app/#q=WzAsOCxbMCwwLCJcXG1hdGhjYWx7U30iXSxbMiwwLCJcXG1hdGhjYWx7U30gXFx3aWRlaGF0e1xcb3RpbWVzfSBcXG1hdGhjYWx7U30iXSxbMCwyLCJcXG1hdGhjYWx7U30gXFx3aWRlaGF0e1xcb3RpbWVzfSBcXG1hdGhjYWx7U30iXSxbMiwyLCJcXG1hdGhjYWx7U30gXFx3aWRlaGF0e1xcb3RpbWVzfSBcXG1hdGhjYWx7U30gXFx3aWRlaGF0e1xcb3RpbWVzfSBcXG1hdGhjYWx7U30iXSxbNiwwLCJcXG1hdGhjYWx7U30iXSxbNSwxLCJcXG1hdGhjYWx7U30iXSxbNywxLCJcXG1hdGhjYWx7U30iXSxbNiwyLCJcXG1hdGhjYWx7U30gXFx3aWRlaGF0e1xcb3RpbWVzfSBcXG1hdGhjYWx7U30iXSxbMCwxLCJcXERlbHRhIl0sWzAsMiwiXFxEZWx0YSIsMl0sWzIsMywiXFxEZWx0YSBcXHdpZGVoYXR7XFxvdGltZXN9IDEiLDJdLFsxLDMsIjEgXFx3aWRlaGF0e1xcb3RpbWVzfSBcXERlbHRhIl0sWzQsNywiXFxEZWx0YSJdLFs0LDUsIj0iLDJdLFs0LDYsIj0iXSxbNyw1LCJcXGV0YSBcXHdpZGVoYXR7XFxvdGltZXN9IDEiXSxbNyw2LCIxIFxcd2lkZWhhdHtcXG90aW1lc30gXFxldGEiLDJdXQ==
\[\begin{tikzcd}[ampersand replacement=\&]
	{\mathcal{S}} \&\& {\mathcal{S} \widehat{\otimes} \mathcal{S}} \&\&\&\& {\mathcal{S}} \\
	\&\&\&\&\& {\mathcal{S}} \&\& {\mathcal{S}} \\
	{\mathcal{S} \widehat{\otimes} \mathcal{S}} \&\& {\mathcal{S} \widehat{\otimes} \mathcal{S} \widehat{\otimes} \mathcal{S}} \&\&\&\& {\mathcal{S} \widehat{\otimes} \mathcal{S}}
	\arrow["\Delta", from=1-1, to=1-3]
	\arrow["\Delta"', from=1-1, to=3-1]
	\arrow["{\Delta \widehat{\otimes} 1}"', from=3-1, to=3-3]
	\arrow["{1 \widehat{\otimes} \Delta}", from=1-3, to=3-3]
	\arrow["\Delta", from=1-7, to=3-7]
	\arrow["{=}"', from=1-7, to=2-6]
	\arrow["{=}", from=1-7, to=2-8]
	\arrow["{\eta \widehat{\otimes} 1}", from=3-7, to=2-6]
	\arrow["{1 \widehat{\otimes} \eta}"', from=3-7, to=2-8]
\end{tikzcd}\]
\end{cor}
\begin{proof}
    This is immediate by considering the elements $u$ and $v$ in $\s$.
\end{proof}

\section{Addition and H-space structures}
In this section, we develop the notion of ``sum" of two elements in a spectral $K$-theory space of a graded $C^*$-algebra underlying (potentially different) graded Hilbert spaces.

Throughout this section, $\h,\h_0,\h_1,...$ will denote graded Hilbert spaces, and $A,B,C,...$ will denote $\Z/2\Z$-graded $C^*$-algebras.

\subsection{Sums and naturality}

\begin{defi}\label{prop: sum}
    We have a continuous map
    \[s:\K(A,\mathcal{H}_0) \times \K(A,\mathcal{H}_1) \rightarrow \K(A,\mathcal{H}_0 \oplus \mathcal{H}_1),\]
    which is natural with respect to $*$-homomorphisms $A \rightarrow A'$ and even isometries $\h_0 \rightarrow \h_0'$ and $\h_1 \rightarrow \h_1
    '$. On an element $(\varphi,\psi) \in \K(A,\mathcal{H}_0) \times \K(A,\mathcal{H}_1) $, it yields the composite \begin{equation}\label{eqn: diag embedding}
    % https://q.uiver.app/#q=WzAsNSxbMSwwLCJcXG1hdGhjYWx7U30gXFxvcGx1cyBcXG1hdGhjYWx7U30iXSxbMiwwLCJBXFxoYXRcXG90aW1lc1xcbWF0aGNhbHtLfShcXG1hdGhjYWx7SH1fMCkgXFxvcGx1cyBBXFxoYXRcXG90aW1lc1xcbWF0aGNhbHtLfShcXG1hdGhjYWx7SH1fMSkiXSxbNCwwLCJBXFxoYXRcXG90aW1lc1xcbWF0aGNhbHtLfShcXG1hdGhjYWx7SH1fMFxcb3BsdXNcXG1hdGhjYWx7SH1fMSkiXSxbMCwwLCJcXG1hdGhjYWx7U30iXSxbMywwLCJBXFxoYXRcXG90aW1lcyhcXG1hdGhjYWx7S30oXFxtYXRoY2Fse0h9XzApIFxcb3BsdXNcXG1hdGhjYWx7S30oXFxtYXRoY2Fse0h9XzEpKSJdLFswLDEsIlxcdmFycGhpXFxvcGx1c1xccHNpIl0sWzMsMCwiXFxvcGVyYXRvcm5hbWV7ZGlhfSJdLFsxLDQsIlxcY29uZyJdLFs0LDIsIjEgXFxoYXRcXG90aW1lcyBcXG9wZXJhdG9ybmFtZXtpbmN9Il1d
\begin{tikzcd}
	{\mathcal{S}} & {\mathcal{S} \oplus \mathcal{S}} & {A\hat\otimes\mathcal{K}(\mathcal{H}_0) \oplus A\hat\otimes\mathcal{K}(\mathcal{H}_1)} & {A\hat\otimes(\mathcal{K}(\mathcal{H}_0) \oplus\mathcal{K}(\mathcal{H}_1))} & {A\hat\otimes\mathcal{K}(\mathcal{H}_0\oplus\mathcal{H}_1)}
	\arrow["\varphi\oplus\psi", from=1-2, to=1-3]
	\arrow["{\operatorname{dia}}", from=1-1, to=1-2]
	\arrow["\cong", from=1-3, to=1-4]
	\arrow["{1 \hat\otimes \operatorname{inc}}", from=1-4, to=1-5]
\end{tikzcd}
\end{equation}
where the isomorphism in the middle is inverse to the natural iso $a \hat\otimes (b \oplus c) \mapsto a \hat\otimes b \oplus a \hat\otimes c$. We call the element $s(\varphi,\psi)$ as the sum of $\varphi$ and $\psi$. 
\end{defi}
    Note that the continuity of $s$ is immediate from the definition of the point-norm topology and \cref{lem: 1st lemma}. The naturality is immediate from \cref{eqn: diag embedding}.
\begin{rem}
    For our convenience, we will also call the map $s$ as ``add"/``sum" at some junctures.
\end{rem}
\begin{rem} \label{cor: hom preserve sum}
    Note that the naturality assertion in \cref{prop: sum} really means that the following diagram commutes strictly, where $\varphi: A \rightarrow A'$ is a graded $*$-homomorphism, and $\iota_0:\h_0 \rightarrow \h_0'$, $\iota_1:\h_1 \rightarrow \h_1'$ are even isometries.
    \[
    % https://q.uiver.app/#q=WzAsNCxbMCwwLCJcXG1hdGhiYntLfShBLFxcbWF0aGNhbHtIfV8wKSBcXHRpbWVzIFxcbWF0aGJie0t9KEEsXFxtYXRoY2Fse0h9XzEpIl0sWzIsMCwiXFxtYXRoYmJ7S30oQScsXFxtYXRoY2Fse0h9XzAnKSBcXHRpbWVzIFxcbWF0aGJie0t9KEEsXFxtYXRoY2Fse0h9XzEnKSJdLFswLDEsIlxcbWF0aGJie0t9KEEsXFxtYXRoY2Fse0h9XzAgXFxvcGx1cyBcXG1hdGhjYWx7SH1fMSkiXSxbMiwxLCJcXG1hdGhiYntLfShBJyxcXG1hdGhjYWx7SH1fMCcgXFxvcGx1cyBcXG1hdGhjYWx7SH1fMScpIl0sWzAsMiwiXFxvcGVyYXRvcm5hbWV7YWRkfSJdLFsxLDMsIlxcb3BlcmF0b3JuYW1le2FkZH0iXSxbMCwxLCJcXG1hdGhiYntLfShcXHZhcnBoaSxcXGlvdGFfMCkgXFx0aW1lcyBcXG1hdGhiYntLfShcXHZhcnBoaSxcXGlvdGFfMSkiXSxbMiwzLCJcXG1hdGhiYntLfShcXHZhcnBoaSxcXGlvdGFfMCBcXG9wbHVzIFxcaW90YV8xKSJdXQ==
\begin{tikzcd}
	{\mathbb{K}(A,\mathcal{H}_0) \times \mathbb{K}(A,\mathcal{H}_1)} && {\mathbb{K}(A',\mathcal{H}_0') \times \mathbb{K}(A,\mathcal{H}_1')} \\
	{\mathbb{K}(A,\mathcal{H}_0 \oplus \mathcal{H}_1)} && {\mathbb{K}(A',\mathcal{H}_0' \oplus \mathcal{H}_1')}
	\arrow["{\operatorname{add}}", from=1-1, to=2-1]
	\arrow["{\operatorname{add}}", from=1-3, to=2-3]
	\arrow["{\mathbb{K}(\varphi,\iota_0) \times \mathbb{K}(\varphi,\iota_1)}", from=1-1, to=1-3]
	\arrow["{\mathbb{K}(\varphi,\iota_0 \oplus \iota_1)}", from=2-1, to=2-3]
\end{tikzcd}
    \]
    We draw this diagram here for future reference.
\end{rem}
\begin{prop} \label{prop: sum comm and ass}
    Sums are associative and commutative, in the sense that the following diagrams commute strictly:
    \[
       % https://q.uiver.app/#q=WzAsOCxbMCwwLCJcXG1hdGhiYntLfShBLFxcbWF0aGNhbHtIfV8wKSBcXHRpbWVzIFxcbWF0aGJie0t9KEEsXFxtYXRoY2Fse0h9XzEpIl0sWzIsMCwiXFxtYXRoYmJ7S30oQSxcXG1hdGhjYWx7SH1fMSkgXFx0aW1lcyBcXG1hdGhiYntLfShBLFxcbWF0aGNhbHtIfV8wKSJdLFswLDEsIlxcbWF0aGJie0t9KEEsXFxtYXRoY2Fse0h9XzAgXFxvcGx1cyBcXG1hdGhjYWx7SH1fMSkiXSxbMiwxLCJcXG1hdGhiYntLfShBLFxcbWF0aGNhbHtIfV8xIFxcb3BsdXMgXFxtYXRoY2Fse0h9XzApIl0sWzAsMiwiXFxtYXRoYmJ7S30oQSxcXG1hdGhjYWx7SH1fMCkgXFx0aW1lcyBcXG1hdGhiYntLfShBLFxcbWF0aGNhbHtIfV8xKSBcXHRpbWVzIFxcbWF0aGJie0t9KEEsXFxtYXRoY2Fse0h9XzIpIl0sWzIsMiwiXFxtYXRoYmJ7S30oQSxcXG1hdGhjYWx7SH1fMCBcXG9wbHVzIFxcbWF0aGNhbHtIfV8xKSBcXHRpbWVzIFxcbWF0aGJie0t9KEEsXFxtYXRoY2Fse0h9XzIpIl0sWzAsMywiXFxtYXRoYmJ7S30oQSxcXG1hdGhjYWx7SH1fMCkgXFx0aW1lcyBcXG1hdGhiYntLfShBLFxcbWF0aGNhbHtIfV8xICBcXG9wbHVzIFxcbWF0aGNhbHtIfV8yKSJdLFsyLDMsIlxcbWF0aGJie0t9KEEsXFxtYXRoY2Fse0h9XzAgXFxvcGx1cyBcXG1hdGhjYWx7SH1fMSBcXG9wbHVzIFxcbWF0aGNhbHtIfV8yKSJdLFswLDEsIlxcb3BlcmF0b3JuYW1le2ZsaXB9Il0sWzAsMiwiXFxvcGVyYXRvcm5hbWV7YWRkfSIsMl0sWzIsMywiMSBcXGhhdFxcb3RpbWVzIFxcb3BlcmF0b3JuYW1le0FkfSgoeCx5KSBcXG1hcHN0byAoeSx4KSkiXSxbMSwzLCJcXG9wZXJhdG9ybmFtZXthZGR9Il0sWzQsNSwiXFxvcGVyYXRvcm5hbWV7YWRkfSBcXHRpbWVzMSJdLFs0LDYsIjEgXFx0aW1lcyBcXG9wZXJhdG9ybmFtZXthZGR9IiwyXSxbNiw3LCJcXG9wZXJhdG9ybmFtZXthZGR9Il0sWzUsNywiXFxvcGVyYXRvcm5hbWV7YWRkfSJdXQ==
\begin{tikzcd}
	{\mathbb{K}(A,\mathcal{H}_0) \times \mathbb{K}(A,\mathcal{H}_1)} && {\mathbb{K}(A,\mathcal{H}_1) \times \mathbb{K}(A,\mathcal{H}_0)} \\
	{\mathbb{K}(A,\mathcal{H}_0 \oplus \mathcal{H}_1)} && {\mathbb{K}(A,\mathcal{H}_1 \oplus \mathcal{H}_0)} \\
	{\mathbb{K}(A,\mathcal{H}_0) \times \mathbb{K}(A,\mathcal{H}_1) \times \mathbb{K}(A,\mathcal{H}_2)} && {\mathbb{K}(A,\mathcal{H}_0 \oplus \mathcal{H}_1) \times \mathbb{K}(A,\mathcal{H}_2)} \\
	{\mathbb{K}(A,\mathcal{H}_0) \times \mathbb{K}(A,\mathcal{H}_1  \oplus \mathcal{H}_2)} && {\mathbb{K}(A,\mathcal{H}_0 \oplus \mathcal{H}_1 \oplus \mathcal{H}_2)}
	\arrow["{\operatorname{flip}}", from=1-1, to=1-3]
	\arrow["{\operatorname{add}}"', from=1-1, to=2-1]
	\arrow["{1 \hat\otimes \operatorname{Ad}((x,y) \mapsto (y,x))}", from=2-1, to=2-3]
	\arrow["{\operatorname{add}}", from=1-3, to=2-3]
	\arrow["{\operatorname{add} \times1}", from=3-1, to=3-3]
	\arrow["{1 \times \operatorname{add}}"', from=3-1, to=4-1]
	\arrow["{\operatorname{add}}", from=4-1, to=4-3]
	\arrow["{\operatorname{add}}", from=3-3, to=4-3]
\end{tikzcd}
    \]
\end{prop}
\begin{proof}
    We only give the argument for commutativity; the argument for associativity is more straight-forward.

    Checking the strict commutativity of the space level diagram amounts to checking commutativity of the following diagram at the algebra level, where $\varphi \in \K(A,\h_0), \psi \in \K(A,\h_1)$:

    \[
    % https://q.uiver.app/#q=WzAsOSxbMCwwLCJcXG1hdGhjYWx7U30iXSxbMCwxLCJcXG1hdGhjYWx7U30gXFxvcGx1cyBcXG1hdGhjYWx7U30iXSxbMiwxLCJcXG1hdGhjYWx7U30gXFxvcGx1cyBcXG1hdGhjYWx7U30iXSxbMCwyLCJBIFxcaGF0XFxvdGltZXMgXFxtYXRoY2Fse0t9KFxcbWF0aGNhbHtIfV8wKSBcXG9wbHVzIEEgXFxoYXRcXG90aW1lc1xcbWF0aGNhbHtLfShcXG1hdGhjYWx7SH1fMSkiXSxbMCwzLCJBIFxcaGF0XFxvdGltZXMoXFxtYXRoY2Fse0t9KFxcbWF0aGNhbHtIfV8wKSBcXG9wbHVzIFxcbWF0aGNhbHtLfShcXG1hdGhjYWx7SH1fMSkpIl0sWzAsNCwiQSBcXGhhdFxcb3RpbWVzXFxtYXRoY2Fse0t9KFxcbWF0aGNhbHtIfV8wIFxcb3BsdXMgXFxtYXRoY2Fse0h9XzEpIl0sWzIsMiwiQSBcXGhhdFxcb3RpbWVzIFxcbWF0aGNhbHtLfShcXG1hdGhjYWx7SH1fMSkgXFxvcGx1cyBBIFxcaGF0XFxvdGltZXNcXG1hdGhjYWx7S30oXFxtYXRoY2Fse0h9XzApIl0sWzIsMywiQSBcXGhhdFxcb3RpbWVzKFxcbWF0aGNhbHtLfShcXG1hdGhjYWx7SH1fMSkgXFxvcGx1cyBcXG1hdGhjYWx7S30oXFxtYXRoY2Fse0h9XzApKSJdLFsyLDQsIkEgXFxoYXRcXG90aW1lc1xcbWF0aGNhbHtLfShcXG1hdGhjYWx7SH1fMSBcXG9wbHVzIFxcbWF0aGNhbHtIfV8wKSJdLFsxLDMsIlxcdmFycGhpIFxcb3BsdXMgXFxwc2kiXSxbMyw0LCJcXGNvbmciXSxbNCw1LCIxIFxcaGF0XFxvdGltZXMgXFxvcGVyYXRvcm5hbWV7aW5jfSJdLFsxLDIsIihmLGcpIFxccmlnaHRhcnJvdyhnLGYpIl0sWzIsNiwiXFxwc2kgXFxvcGx1cyBcXHZhcnBoaSIsMl0sWzYsNywiXFxjb25nIiwyXSxbNyw4LCIxIFxcaGF0XFxvdGltZXMgXFxvcGVyYXRvcm5hbWV7aW5jfSIsMl0sWzUsOCwiMSBcXGhhdFxcb3RpbWVzXFxvcGVyYXRvcm5hbWV7QWR9KFxcb3BlcmF0b3JuYW1le2ZsaXB9IFxcbWF0aGNhbHtIfV8wIH5cXCZ+IFxcbWF0aGNhbHtIfV8xKSJdLFswLDEsIlxcb3BlcmF0b3JuYW1le2RpYX0iXV0=
\begin{tikzcd}
	{\mathcal{S}} \\
	{\mathcal{S} \oplus \mathcal{S}} && {\mathcal{S} \oplus \mathcal{S}} \\
	{A \hat\otimes \mathcal{K}(\mathcal{H}_0) \oplus A \hat\otimes\mathcal{K}(\mathcal{H}_1)} && {A \hat\otimes \mathcal{K}(\mathcal{H}_1) \oplus A \hat\otimes\mathcal{K}(\mathcal{H}_0)} \\
	{A \hat\otimes(\mathcal{K}(\mathcal{H}_0) \oplus \mathcal{K}(\mathcal{H}_1))} && {A \hat\otimes(\mathcal{K}(\mathcal{H}_1) \oplus \mathcal{K}(\mathcal{H}_0))} \\
	{A \hat\otimes\mathcal{K}(\mathcal{H}_0 \oplus \mathcal{H}_1)} && {A \hat\otimes\mathcal{K}(\mathcal{H}_1 \oplus \mathcal{H}_0)}
	\arrow["{\varphi \oplus \psi}", from=2-1, to=3-1]
	\arrow["\cong", from=3-1, to=4-1]
	\arrow["{1 \hat\otimes \operatorname{inc}}", from=4-1, to=5-1]
	\arrow["{(f,g) \rightarrow(g,f)}", from=2-1, to=2-3]
	\arrow["{\psi \oplus \varphi}"', from=2-3, to=3-3]
	\arrow["\cong"', from=3-3, to=4-3]
	\arrow["{1 \hat\otimes \operatorname{inc}}"', from=4-3, to=5-3]
	\arrow["{1 \hat\otimes\operatorname{Ad}(\operatorname{flip} \mathcal{H}_0 ~\&~ \mathcal{H}_1)}", from=5-1, to=5-3]
	\arrow["{\operatorname{dia}}", from=1-1, to=2-1]
\end{tikzcd}
    \]
    It does commute, as can be checked by taking $(f,g) \in \s \oplus \s$ and checking its fate along the two itineraries.
\end{proof}
For our purposes later, we need to discuss functoriality of sums not just with respect to the $*$-homomorphisms as discussed in \cref{prop: sum}, but also for ``amplified" versions of them. First, we look at the following definition:
\begin{defi}\label{defi: amplified homo}
    \,
    \begin{enumerate}[(i)]
        \item Let $\kappa:\s \hat\otimes A \rightarrow B\hat\otimes \com(\h)$ be a graded $*$-homomorphism, and $\iota:\h_0 \rightarrow \h_0'$ an even isometry. Then, it induces a map $(\kappa,\iota)_*:\K(A,\h_0) \rightarrow \K(B,\h_0' \hat\otimes \h)$ by the following prescription:
        Given $\varphi_A \in \K(A,\h_0)$, the map $(\kappa,\iota)_*$ is given by the composite:
    \begin{equation}\label{eqn: prod 2}% 
% https://q.uiver.app/#q=WzAsNixbMCwwLCJcXG1hdGhjYWx7U30iXSxbMSwwLCJcXG1hdGhjYWx7U30gXFxoYXRcXG90aW1lcyBcXG1hdGhjYWx7U30iXSxbMiwwLCJcXG1hdGhjYWx7U30gXFxoYXRcXG90aW1lcyBBIFxcaGF0XFxvdGltZXMgXFxtYXRoY2Fse0t9KFxcbWF0aGNhbHtIfV8wKSJdLFszLDAsIkIgXFxoYXRcXG90aW1lcyBcXG1hdGhjYWx7S30oXFxtYXRoY2Fse0h9KVxcaGF0XFxvdGltZXMgXFxtYXRoY2Fse0t9KFxcbWF0aGNhbHtIfV8wJykiXSxbMSwxLCJCIFxcaGF0XFxvdGltZXMgXFxtYXRoY2Fse0t9KFxcbWF0aGNhbHtIfV8wJ1xcaGF0XFxvdGltZXMgXFxtYXRoY2Fse0h9KSJdLFszLDEsIkIgXFxoYXRcXG90aW1lcyBcXG1hdGhjYWx7S30oXFxtYXRoY2Fse0h9XzAnKSBcXGhhdFxcb3RpbWVzXFxtYXRoY2Fse0t9KFxcbWF0aGNhbHtIfSkiXSxbMCwxLCJcXHRyaWFuZ2xlIl0sWzEsMiwiMSBcXGhhdFxcb3RpbWVzIFxcdmFycGhpX0EiXSxbMiwzLCJcXGthcHBhIFxcaGF0XFxvdGltZXMgXFxvcGVyYXRvcm5hbWV7QWR9KFxcaW90YSkiXSxbMyw1LCJcXG9wZXJhdG9ybmFtZXtmbGlwfVxcbWF0aGNhbHtLfShcXG1hdGhjYWx7SH0pIH5cXCZ+IFxcbWF0aGNhbHtLfShcXG1hdGhjYWx7SH1fMCcpIl0sWzUsNCwiMSBcXGhhdFxcb3RpbWVzIFxcUHNpIl1d
\begin{tikzcd}
	{\mathcal{S}} & {\mathcal{S} \hat\otimes \mathcal{S}} & {\mathcal{S} \hat\otimes A \hat\otimes \mathcal{K}(\mathcal{H}_0)} & {B \hat\otimes \mathcal{K}(\mathcal{H})\hat\otimes \mathcal{K}(\mathcal{H}_0')} \\
	& {B \hat\otimes \mathcal{K}(\mathcal{H}_0'\hat\otimes \mathcal{H})} && {B \hat\otimes \mathcal{K}(\mathcal{H}_0') \hat\otimes\mathcal{K}(\mathcal{H})}
	\arrow["\triangle", from=1-1, to=1-2]
	\arrow["{1 \hat\otimes \varphi_A}", from=1-2, to=1-3]
	\arrow["{\kappa \hat\otimes \operatorname{Ad}(\iota)}", from=1-3, to=1-4]
	\arrow["{\operatorname{flip}\mathcal{K}(\mathcal{H}) ~\&~ \mathcal{K}(\mathcal{H}_0')}", from=1-4, to=2-4]
	\arrow["{1 \hat\otimes \Psi}", from=2-4, to=2-2]
\end{tikzcd}
    \end{equation}
    where $\Psi:\com(\h_0') \hat\otimes \com(\h) \rightarrow \com(\h_0' \hat\otimes \h)$ is the natural isomorphism.
    Note that $\varphi_A$ is clearly continuous from \cref{lem: 1st lemma} and \cref{lem: cont of tensor product}. Furthermore, if $\kappa \sim \kappa'$ and $\iota \sim \iota'$, then $(\kappa,\iota)_* \sim (\kappa',\iota')_*$ from the second part of \cref{lem: 1st lemma}. 
    \item Let $\kappa:A \hat\otimes \s \hat\otimes B \rightarrow C \hat\otimes \com(\h)$ be a graded $*$-homomorphism, and $\iota:\h_0 \rightarrow \h_0'$ an even isometry. Then, it induces a map $(\kappa,\iota)_*:\K(A \hat\otimes B,\h_0) \rightarrow \K(C, \h_0' \hat\otimes \h)$ by the following prescription: Given $\varphi \in \K(A \hat\otimes B,\h_0)$, the map $(\kappa,\iota)_*$ is given by the composite 
    \[
    % https://q.uiver.app/#q=WzAsNyxbMCwwLCJcXG1hdGhjYWx7U30iXSxbMSwwLCJcXG1hdGhjYWx7U30gXFxoYXRcXG90aW1lcyBcXG1hdGhjYWx7U30iXSxbMiwwLCJcXG1hdGhjYWx7U30gXFxoYXRcXG90aW1lcyBBIFxcaGF0XFxvdGltZXMgQiBcXGhhdFxcb3RpbWVzXFxtYXRoY2Fse0t9KFxcbWF0aGNhbHtIfV8wKSJdLFszLDAsIiBBIFxcaGF0XFxvdGltZXMgXFxtYXRoY2Fse1N9IFxcaGF0XFxvdGltZXMgQiBcXGhhdFxcb3RpbWVzXFxtYXRoY2Fse0t9KFxcbWF0aGNhbHtIfV8wKSJdLFszLDEsIkMgXFxoYXRcXG90aW1lcyBcXG1hdGhjYWx7S30oXFxtYXRoY2Fse0h9KVxcaGF0XFxvdGltZXNcXG1hdGhjYWx7S30oXFxtYXRoY2Fse0h9XzAnKSJdLFsxLDEsIkMgXFxoYXRcXG90aW1lcyBcXG1hdGhjYWx7S30oXFxtYXRoY2Fse0h9XzAnKVxcaGF0XFxvdGltZXNcXG1hdGhjYWx7S30oXFxtYXRoY2Fse0h9KSJdLFswLDEsIkMgXFxoYXRcXG90aW1lcyBcXG1hdGhjYWx7S30oXFxtYXRoY2Fse0h9XzAnXFxoYXRcXG90aW1lc1xcbWF0aGNhbHtIfSkiXSxbMCwxLCJcXHRyaWFuZ2xlIl0sWzEsMiwiMSBcXGhhdFxcb3RpbWVzIFxcdmFycGhpIl0sWzIsMywiXFxvcGVyYXRvcm5hbWV7ZmxpcH0gQSB+XFwmflxcbWF0aGNhbHtTfSJdLFszLDQsIlxca2FwcGFcXGhhdFxcb3RpbWVzIFxcb3BlcmF0b3JuYW1le0FkfShcXGlvdGEpIl0sWzQsNSwiXFxvcGVyYXRvcm5hbWV7ZmxpcH0gXFxtYXRoY2Fse0t9KFxcbWF0aGNhbHtIfV8wJykgflxcJn5cXG1hdGhjYWx7S30oXFxtYXRoY2Fse0h9KSJdLFs1LDYsIjEgXFxoYXRcXG90aW1lcyBcXFBzaSJdXQ==
\begin{tikzcd}
	{\mathcal{S}} & {\mathcal{S} \hat\otimes \mathcal{S}} & {\mathcal{S} \hat\otimes A \hat\otimes B \hat\otimes\mathcal{K}(\mathcal{H}_0)} & { A \hat\otimes \mathcal{S} \hat\otimes B \hat\otimes\mathcal{K}(\mathcal{H}_0)} \\
	{C \hat\otimes \mathcal{K}(\mathcal{H}_0'\hat\otimes\mathcal{H})} & {C \hat\otimes \mathcal{K}(\mathcal{H}_0')\hat\otimes\mathcal{K}(\mathcal{H})} && {C \hat\otimes \mathcal{K}(\mathcal{H})\hat\otimes\mathcal{K}(\mathcal{H}_0')}
	\arrow["\triangle", from=1-1, to=1-2]
	\arrow["{1 \hat\otimes \varphi}", from=1-2, to=1-3]
	\arrow["{\operatorname{flip} A ~\&~\mathcal{S}}", from=1-3, to=1-4]
	\arrow["{\kappa\hat\otimes \operatorname{Ad}(\iota)}", from=1-4, to=2-4]
	\arrow["{\operatorname{flip} \mathcal{K}(\mathcal{H}_0') ~\&~\mathcal{K}(\mathcal{H})}", from=2-4, to=2-2]
	\arrow["{1 \hat\otimes \Psi}", from=2-2, to=2-1]
\end{tikzcd}
    \]
    where $\Psi:\com(\h_0') \hat\otimes \com(\h) \rightarrow \com(\h_0' \hat\otimes \h)$ is the natural isomorphism.
    Note that $\varphi_A$ is clearly continuous from \cref{lem: 1st lemma} and \cref{lem: cont of tensor product}. Furthermore, if $\kappa \sim \kappa'$ and $\iota \sim \iota'$, then $(\kappa,\iota)_* \sim (\kappa',\iota')_*$ from the second part of \cref{lem: 1st lemma}. 
    \end{enumerate}
    \end{defi}
\begin{lem}\label{nat of amp morph}
    \,
    \begin{enumerate}[(i)]
        \item The maps discussed in \cref{defi: amplified homo} (i) are functorial in the following sense:
        if a strictly commutative diagram in $\mathfrak{C}^* \times \mathfrak{H}$ is given, as follows:
        \begin{equation*}\label{eqn: alg amp 1}
\begin{tikzcd}
	{(\mathcal{S} \hat\otimes A, \mathcal{H}_0)} && {(B \hat\otimes\mathcal{K}(\mathcal{H}),\mathcal{H}_1)} \\
	{(\mathcal{S} \hat\otimes A', \mathcal{H}_0')} && {(B' \hat\otimes\mathcal{K}(\mathcal{H}),\mathcal{H}_1')}
	\arrow["{(\kappa,\iota)}", from=1-1, to=1-3]
	\arrow["{A \rightarrow A'\\, \mathcal{H}_0 \rightarrow \mathcal{H}_0'}"', from=1-1, to=2-1]
	\arrow["{(\kappa',\iota')}", from=2-1, to=2-3]
	\arrow["{B \rightarrow B' \\, \mathcal{H}_1 \rightarrow \mathcal{H}_1'}", from=1-3, to=2-3]
\end{tikzcd}
        ,    
        \end{equation*}
        then the following diagram of spaces commute strictly:
        \begin{equation*}\label{eqn: space amp 1}
        % https://q.uiver.app/#q=WzAsNCxbMCwwLCJcXG1hdGhiYntLfShBLFxcbWF0aGNhbHtIfV8wKSJdLFswLDEsIlxcbWF0aGJie0t9KEEnLFxcbWF0aGNhbHtIfV8wJykiXSxbMiwwLCJcXG1hdGhiYntLfShCLFxcbWF0aGNhbHtIfV8xIFxcaGF0XFxvdGltZXMgXFxtYXRoY2Fse0h9KSkiXSxbMiwxLCJcXG1hdGhiYntLfShCJyxcXG1hdGhjYWx7SH1fMScgXFxoYXRcXG90aW1lcyBcXG1hdGhjYWx7SH0pKSJdLFswLDEsIlxce0EgXFxyaWdodGFycm93IEEnXFxcXCwgXFxtYXRoY2Fse0h9XzAgXFxyaWdodGFycm93IFxcbWF0aGNhbHtIfV8wJ1xcfV8qIiwyXSxbMCwyLCIoXFxrYXBwYSxcXGlvdGEpXyoiXSxbMSwzLCIoXFxrYXBwYScsXFxpb3RhJylfKiJdLFsyLDMsIlxce0IgXFxyaWdodGFycm93IEInIFxcXFwsIFxcbWF0aGNhbHtIfV8xIFxccmlnaHRhcnJvdyBcXG1hdGhjYWx7SH1fMSdcXH1fKiJdXQ==
\begin{tikzcd}
	{\mathbb{K}(A,\mathcal{H}_0)} && {\mathbb{K}(B,\mathcal{H}_1 \hat\otimes \mathcal{H}))} \\
	{\mathbb{K}(A',\mathcal{H}_0')} && {\mathbb{K}(B',\mathcal{H}_1' \hat\otimes \mathcal{H}))}
	\arrow["{\{A \rightarrow A'\\, \mathcal{H}_0 \rightarrow \mathcal{H}_0'\}_*}"', from=1-1, to=2-1]
	\arrow["{(\kappa,\iota)_*}", from=1-1, to=1-3]
	\arrow["{(\kappa',\iota')_*}", from=2-1, to=2-3]
	\arrow["{\{B \rightarrow B' \\, \mathcal{H}_1 \rightarrow \mathcal{H}_1'\}_*}", from=1-3, to=2-3]
\end{tikzcd}.
        \end{equation*}
        Moreover, if the first diagram commutes upto homotopy, then so does the second one.
        \item  The maps discussed in \cref{defi: amplified homo} (ii) are functorial in the following sense:
        if a commutative diagram in $\mathfrak{C}^* \times \mathfrak{H}$ is given, as follows:
        \begin{equation*}\label{eqn: alg amp 2}
\begin{tikzcd}
	{(A \hat\otimes \mathcal{S} \hat\otimes B, \mathcal{H}_0)} && {(C \hat\otimes\mathcal{K}(\mathcal{H}),\mathcal{H}_1)} \\
	{(A' \hat\otimes \mathcal{S} \hat\otimes B', \mathcal{H}_0')} && {(C' \hat\otimes\mathcal{K}(\mathcal{H}),\mathcal{H}_1')}
	\arrow["{(\kappa,\iota)}", from=1-1, to=1-3]
	\arrow["{A \rightarrow A'\\, B \rightarrow B',\mathcal{H}_0 \rightarrow \mathcal{H}_0'}"', from=1-1, to=2-1]
	\arrow["{(\kappa',\iota')}", from=2-1, to=2-3]
	\arrow["{C \rightarrow C' \\, \mathcal{H}_1 \rightarrow \mathcal{H}_1'}", from=1-3, to=2-3]
\end{tikzcd}
        \end{equation*}
then the following diagram of spaces commute strictly:
\begin{equation*}\label{eqn: space amp 2}
% https://q.uiver.app/#q=WzAsNCxbMCwwLCJcXG1hdGhiYntLfShBXFxoYXRcXG90aW1lcyBCLFxcbWF0aGNhbHtIfV8wKSJdLFswLDEsIlxcbWF0aGJie0t9KEEnIFxcaGF0XFxvdGltZXMgQicsXFxtYXRoY2Fse0h9XzAnKSJdLFsyLDAsIlxcbWF0aGJie0t9KEMsXFxtYXRoY2Fse0h9XzEgXFxoYXRcXG90aW1lcyBcXG1hdGhjYWx7SH0pKSJdLFsyLDEsIlxcbWF0aGJie0t9KEMnLFxcbWF0aGNhbHtIfV8xJyBcXGhhdFxcb3RpbWVzIFxcbWF0aGNhbHtIfSkpIl0sWzAsMSwiXFx7QSBcXHJpZ2h0YXJyb3cgQSdcXFxcLCBCIFxccmlnaHRhcnJvdyBCJyxcXG1hdGhjYWx7SH1fMCBcXHJpZ2h0YXJyb3cgXFxtYXRoY2Fse0h9XzAnXFx9XyoiLDJdLFswLDIsIihcXGthcHBhLFxcaW90YSlfKiJdLFsxLDMsIihcXGthcHBhJyxcXGlvdGEnKV8qIl0sWzIsMywiXFx7QyBcXHJpZ2h0YXJyb3cgQycgXFxcXCwgXFxtYXRoY2Fse0h9XzEgXFxyaWdodGFycm93IFxcbWF0aGNhbHtIfV8xJ1xcfV8qIl1d
\begin{tikzcd}
	{\mathbb{K}(A\hat\otimes B,\mathcal{H}_0)} && {\mathbb{K}(C,\mathcal{H}_1 \hat\otimes \mathcal{H}))} \\
	{\mathbb{K}(A' \hat\otimes B',\mathcal{H}_0')} && {\mathbb{K}(C',\mathcal{H}_1' \hat\otimes \mathcal{H}))}
	\arrow["{\{A \rightarrow A'\\, B \rightarrow B',\mathcal{H}_0 \rightarrow \mathcal{H}_0'\}_*}"', from=1-1, to=2-1]
	\arrow["{(\kappa,\iota)_*}", from=1-1, to=1-3]
	\arrow["{(\kappa',\iota')_*}", from=2-1, to=2-3]
	\arrow["{\{C \rightarrow C' \\, \mathcal{H}_1 \rightarrow \mathcal{H}_1'\}_*}", from=1-3, to=2-3]
\end{tikzcd}
\end{equation*}
 Moreover, if the first diagram commutes upto homotopy, then so does the second one.
    \end{enumerate}
\end{lem}
\begin{proof}
    We only discuss the second assertion; the first one is similar (and easier).

    Given $\Phi \in \K(A \hat\otimes B, \h_0)$, the two routes in \cref{eqn: space amp 2} yield the following diagram at the algebra level:
    \[
\begin{tikzcd}
	{\mathcal{S}} \\
	{\mathcal{S} \hat\otimes \mathcal{S}} \\
	{\mathcal{S} \hat\otimes A \hat\otimes B \hat\otimes \mathcal{K}(\mathcal{H}_0)} &&& {\mathcal{S} \hat\otimes A' \hat\otimes B' \hat\otimes \mathcal{K}(\mathcal{H}_0')} \\
	{A \hat\otimes \mathcal{S} \hat\otimes B \hat\otimes \mathcal{K}(\mathcal{H}_0)} &&& {A' \hat\otimes \mathcal{S} \hat\otimes  B' \hat\otimes \mathcal{K}(\mathcal{H}_0')} \\
	{C \hat\otimes \mathcal{K}(\mathcal{H}) \hat\otimes \mathcal{K}(\mathcal{H}_1)} &&& {C' \hat\otimes \mathcal{K}(\mathcal{H}) \hat\otimes \mathcal{K}(\mathcal{H}_1')} \\
	{C \hat\otimes \mathcal{K}(\mathcal{H}_1) \hat\otimes \mathcal{K}(\mathcal{H})} \\
	{C' \hat\otimes \mathcal{K}(\mathcal{H}_1') \hat\otimes \mathcal{K}(\mathcal{H})} &&& {C \hat\otimes \mathcal{K}(\mathcal{H}_1') \hat\otimes \mathcal{K}(\mathcal{H})}
	\arrow[from=1-1, to=2-1]
	\arrow["{1 \hat\otimes \Phi}", from=2-1, to=3-1]
	\arrow["{A \rightarrow A', B \rightarrow B',\mathcal{H_0} \rightarrow \mathcal{H_0'}}"', shift right, from=3-1, to=3-4]
	\arrow["{\operatorname{flip} A ~\&~ \mathcal{S}}"{pos=0.7}, from=3-1, to=4-1]
	\arrow["{\operatorname{flip} A' ~\&~ \mathcal{S}}", from=3-4, to=4-4]
	\arrow["{\kappa' \hat\otimes \operatorname{Ad}(\iota')}", from=4-4, to=5-4]
	\arrow["{\operatorname{flip} \mathcal{K}(\mathcal{H}_1') ~\&~ \mathcal{K}(\mathcal{H}) }", from=5-4, to=7-4]
	\arrow["{\kappa \hat\otimes \operatorname{Ad}(\iota)}", from=4-1, to=5-1]
	\arrow["{\operatorname{flip} \mathcal{K}(\mathcal{H}_0') ~\&~ \mathcal{K}(\mathcal{H}) }", from=5-1, to=6-1]
	\arrow["{C \rightarrow C', \mathcal{H}_1 \rightarrow \mathcal{H}_1'}", from=6-1, to=7-1]
	\arrow["{=}"', from=7-1, to=7-4]
	\arrow["{\{A \rightarrow A', B \rightarrow B',\mathcal{H_0} \rightarrow \mathcal{H_0'}\}_*}", color={rgb,255:red,214;green,92;blue,92}, from=4-1, to=4-4]
	\arrow["{\{C \rightarrow C', \mathcal{H}_1 \rightarrow \mathcal{H}_1'\}_*}", color={rgb,255:red,214;green,92;blue,92}, from=5-1, to=5-4]
\end{tikzcd}
    \]
    In the case of strict commutativity, we have to show that the outer part of the above diagram commutes strictly. We can see this by using the filler morphisms given in orange, and noting that the top and bottom rectangles commute strictly because of the functoriality of the flip isomorphisms, and the middle one does by our hypothesis.

    Similarly, in the case of homotopy commutativity, the top and bottom rectangles still commute strictly, but the middle one does only upto homotopy by our hypothesis.
\end{proof}
\begin{prop}\label{sum nat amp morph}
    \,
    \begin{enumerate}[(i)]
        \item  Sums are natural with respect to maps mentioned in \cref{defi: amplified homo} (i), in the sense that the following diagram commutes strictly, where $\varphi: \s \hat\otimes A \rightarrow B \hat\otimes \com(\h)$ is a graded $*$-homomorphism, and $\iota_0:\h_0 \rightarrow \h_0'$ , $\iota_1:\h_1 \rightarrow \h_1'$ are even isometries:
        \[
       % https://q.uiver.app/#q=WzAsNSxbMCwwLCJcXG1hdGhiYntLfShBLFxcbWF0aGNhbHtIfV8wKSBcXHRpbWVzIFxcbWF0aGJie0t9KEEsXFxtYXRoY2Fse0h9XzEpIl0sWzIsMCwiXFxtYXRoYmJ7S30oQSxcXG1hdGhjYWx7SH1fMCBcXG9wbHVzIFxcbWF0aGNhbHtIfV8xKSJdLFswLDIsIlxcbWF0aGJie0t9KEIsXFxtYXRoY2Fse0h9XzAnIFxcaGF0XFxvdGltZXMgXFxtYXRoY2Fse0h9KSBcXHRpbWVzIFxcbWF0aGJie0t9KEIsXFxtYXRoY2Fse0h9XzEnIFxcaGF0XFxvdGltZXMgXFxtYXRoY2Fse0h9KSAiXSxbMiwxLCJcXG1hdGhiYntLfShCLChcXG1hdGhjYWx7SH1fMCdcXG9wbHVzIFxcbWF0aGNhbHtIfV8xJylcXGhhdFxcb3RpbWVzIFxcbWF0aGNhbHtIfSkiXSxbMiwyLCJcXG1hdGhiYntLfShCLFxcbWF0aGNhbHtIfV8wJyBcXGhhdFxcb3RpbWVzIFxcbWF0aGNhbHtIfVxcb3BsdXMgXFxtYXRoY2Fse0h9XzEnIFxcaGF0XFxvdGltZXNcXG1hdGhjYWx7SH0pIl0sWzAsMSwiXFxvcGVyYXRvcm5hbWV7YWRkfSJdLFswLDIsIihcXHZhcnBoaSxcXGlvdGFfMClfKlxcdGltZXMoXFx2YXJwaGksXFxpb3RhXzEpXyoiLDJdLFsxLDMsIihcXHZhcnBoaSxcXGlvdGFfMCBcXG9wbHVzIFxcaW90YV8xKV8qIl0sWzMsNCwiXFxjb25nIl0sWzIsNCwiXFxvcGVyYXRvcm5hbWV7YWRkfSIsMl1d
\begin{tikzcd}
	{\mathbb{K}(A,\mathcal{H}_0) \times \mathbb{K}(A,\mathcal{H}_1)} && {\mathbb{K}(A,\mathcal{H}_0 \oplus \mathcal{H}_1)} \\
	&& {\mathbb{K}(B,(\mathcal{H}_0'\oplus \mathcal{H}_1')\hat\otimes \mathcal{H})} \\
	{\mathbb{K}(B,\mathcal{H}_0' \hat\otimes \mathcal{H}) \times \mathbb{K}(B,\mathcal{H}_1' \hat\otimes \mathcal{H}) } && {\mathbb{K}(B,\mathcal{H}_0' \hat\otimes \mathcal{H}\oplus \mathcal{H}_1' \hat\otimes\mathcal{H})}
	\arrow["{\operatorname{add}}", from=1-1, to=1-3]
	\arrow["{(\varphi,\iota_0)_*\times(\varphi,\iota_1)_*}"', from=1-1, to=3-1]
	\arrow["{(\varphi,\iota_0 \oplus \iota_1)_*}", from=1-3, to=2-3]
	\arrow["\cong", from=2-3, to=3-3]
	\arrow["{\operatorname{add}}"', from=3-1, to=3-3]
\end{tikzcd}
        \]
    \item Sums are natural with respect to maps mentioned in \cref{defi: amplified homo} (ii), in the sense that the following diagram commutes strictly, where $\kappa:A \hat\otimes \s \hat\otimes B \rightarrow C \hat\otimes \com(\h)$ is a graded $*$-homomorphism, and $\iota_0:\h_0 \rightarrow \h_0'$ , $\iota_1:\h_1 \rightarrow \h_1'$ are even isometries:
    \[
    % https://q.uiver.app/#q=WzAsNSxbMCwwLCJcXG1hdGhiYntLfShBIFxcaGF0XFxvdGltZXMgQixcXG1hdGhjYWx7SH1fMCkgXFx0aW1lcyBcXG1hdGhiYntLfShBIFxcaGF0XFxvdGltZXMgQixcXG1hdGhjYWx7SH1fMSkiXSxbMiwwLCJcXG1hdGhiYntLfShBIFxcaGF0XFxvdGltZXMgQixcXG1hdGhjYWx7SH1fMCBcXG9wbHVzIFxcbWF0aGNhbHtIfV8xKSJdLFsyLDEsIlxcbWF0aGJie0t9KEMsKFxcbWF0aGNhbHtIfV8wJyBcXG9wbHVzIFxcbWF0aGNhbHtIfV8xJykgXFxoYXRcXG90aW1lcyBcXG1hdGhjYWx7SH0pIl0sWzAsMiwiXFxtYXRoYmJ7S30oQyxcXG1hdGhjYWx7SH1fMCcgXFxoYXRcXG90aW1lcyBcXG1hdGhjYWx7SH0pIFxcdGltZXMgXFxtYXRoYmJ7S30oQyxcXG1hdGhjYWx7SH1fMScgXFxoYXRcXG90aW1lcyBcXG1hdGhjYWx7SH0pIl0sWzIsMiwiXFxtYXRoYmJ7S30oQyxcXG1hdGhjYWx7SH1fMCcgXFxoYXRcXG90aW1lcyBcXG1hdGhjYWx7SH0gXFxvcGx1cyBcXG1hdGhjYWx7SH1fMScgXFxoYXRcXG90aW1lcyBcXG1hdGhjYWx7SH0pIl0sWzAsMSwiXFxvcGVyYXRvcm5hbWV7c3VtfSJdLFsxLDIsIihcXGthcHBhLFxcaW90YV8wIFxcb3BsdXMgXFxpb3RhXzEpXyoiLDJdLFswLDMsIihcXGthcHBhLFxcaW90YV8wKV8qIFxcdGltZXMgKFxca2FwcGEsXFxpb3RhXzEpXyoiXSxbMiw0LCJcXGNvbmciLDJdLFszLDQsIlxcb3BlcmF0b3JuYW1le3N1bX0iXV0=
\begin{tikzcd}
	{\mathbb{K}(A \hat\otimes B,\mathcal{H}_0) \times \mathbb{K}(A \hat\otimes B,\mathcal{H}_1)} && {\mathbb{K}(A \hat\otimes B,\mathcal{H}_0 \oplus \mathcal{H}_1)} \\
	&& {\mathbb{K}(C,(\mathcal{H}_0' \oplus \mathcal{H}_1') \hat\otimes \mathcal{H})} \\
	{\mathbb{K}(C,\mathcal{H}_0' \hat\otimes \mathcal{H}) \times \mathbb{K}(C,\mathcal{H}_1' \hat\otimes \mathcal{H})} && {\mathbb{K}(C,\mathcal{H}_0' \hat\otimes \mathcal{H} \oplus \mathcal{H}_1' \hat\otimes \mathcal{H})}
	\arrow["{\operatorname{sum}}", from=1-1, to=1-3]
	\arrow["{(\kappa,\iota_0 \oplus \iota_1)_*}"', from=1-3, to=2-3]
	\arrow["{(\kappa,\iota_0)_* \times (\kappa,\iota_1)_*}", from=1-1, to=3-1]
	\arrow["\cong"', from=2-3, to=3-3]
	\arrow["{\operatorname{sum}}", from=3-1, to=3-3]
\end{tikzcd}
    \]
    \end{enumerate}
\end{prop}
\begin{proof}
        We only talk about the second diagram; the first one is similar (and easier):
        Given $\varphi_0 \in \K(A \hat\otimes B,\h_0), \varphi_1 \in \K(A \hat\otimes B,\h_1)$, we have to check the commutativity of the following diagram at the algebra level:
        \[
\begin{tikzcd}
	& {\mathcal{S}} \\
	{\mathcal{S} \hat\otimes \mathcal{S} \oplus\mathcal{S}\hat\otimes\mathcal{S}} && {\mathcal{S}\hat\otimes(\mathcal{S} \oplus \mathcal{S})} \\
	{\mathcal{S}\hat\otimes A \hat\otimes B \hat\otimes \mathcal{K}(\mathcal{H}_0) \oplus\mathcal{S}\hat\otimes A \hat\otimes B \hat\otimes \mathcal{K}(\mathcal{H}_1)} && {\mathcal{S}\hat\otimes (A \hat\otimes B \hat\otimes \mathcal{K}(\mathcal{H}_0) \oplus A \hat\otimes B \hat\otimes \mathcal{K}(\mathcal{H}_1))} \\
	{A \hat\otimes \mathcal{S}\hat\otimes B \hat\otimes \mathcal{K}(\mathcal{H}_0) \oplus A \hat\otimes\mathcal{S} \hat\otimes B \hat\otimes \mathcal{K}(\mathcal{H}_1)} && {\mathcal{S}\hat\otimes (A \hat\otimes B \hat\otimes \mathcal{K}(\mathcal{H}_0 \oplus \mathcal{H}_1))} \\
	{C \hat\otimes \mathcal{K}(\mathcal{H})\hat\otimes \mathcal{K}(\mathcal{H}_0') \oplus C \hat\otimes \mathcal{K}(\mathcal{H})\hat\otimes \mathcal{K}(\mathcal{H}_1')} && {A \hat\otimes \mathcal{S}\hat\otimes  B \hat\otimes \mathcal{K}(\mathcal{H}_0 \oplus \mathcal{H}_1)} \\
	{C \hat\otimes \mathcal{K}(\mathcal{H}_0')\hat\otimes \mathcal{K}(\mathcal{H}) \oplus C \hat\otimes \mathcal{K}(\mathcal{H}_1')\hat\otimes \mathcal{K}(\mathcal{H})} && {C \hat\otimes\mathcal{K}(\mathcal{H}) \hat\otimes\mathcal{K}(\mathcal{H}_0' \oplus \mathcal{H}_1')} \\
	&& {C  \hat\otimes\mathcal{K}(\mathcal{H}_0' \oplus \mathcal{H}_1')\hat\otimes\mathcal{K}(\mathcal{H})} \\
	{C \hat\otimes \mathcal{K}(\mathcal{H}_0' \hat\otimes \mathcal{H} \oplus \mathcal{H}_1'\hat\otimes\mathcal{H})} && {C \hat\otimes\mathcal{K}((\mathcal{H}_0' \oplus \mathcal{H}_1')\hat\otimes\mathcal{H})}
	\arrow[from=1-2, to=2-1]
	\arrow[from=1-2, to=2-3]
	\arrow["{(1 \hat\otimes \varphi_0) \oplus (1 \hat\otimes \varphi_1)}", from=2-1, to=3-1]
	\arrow["{(\operatorname{flip}A ~\&~ \mathcal{S}) \oplus (\operatorname{flip}A ~\&~ \mathcal{S})}", from=3-1, to=4-1]
	\arrow["{\kappa\hat\otimes\operatorname{Ad}(\iota_0) \oplus \kappa\hat\otimes\operatorname{Ad}(\iota_1)}", from=4-1, to=5-1]
	\arrow["{(\operatorname{flip} \mathcal{K}(\mathcal{H}) ~\&~\mathcal{K}(\mathcal{H}_0')) \oplus (\operatorname{flip} \mathcal{K}(\mathcal{H}) ~\&~\mathcal{K}(\mathcal{H}_1'))}", from=5-1, to=6-1]
	\arrow["{\operatorname{add}}", from=6-1, to=8-1]
	\arrow["{1\hat\otimes(\varphi_0 \oplus \varphi_1)}", from=2-3, to=3-3]
	\arrow["{1 \hat\otimes\operatorname{add}}", from=3-3, to=4-3]
	\arrow["{\operatorname{flip}A ~\&~ \mathcal{S}}", from=4-3, to=5-3]
	\arrow["{\kappa \hat\otimes \operatorname{Ad}(\iota_0 \oplus \iota_1)}", from=5-3, to=6-3]
	\arrow["\cong"', from=8-3, to=8-1]
	\arrow["{\operatorname{flip} \mathcal{K}(\mathcal{H}_0' \oplus \mathcal{H}_1') ~\&~\mathcal{K}(\mathcal{H})}", from=6-3, to=7-3]
	\arrow["{1 \hat\otimes \Psi}", from=7-3, to=8-3]
\end{tikzcd}
        \]
    It does commute, as can be seen by taking homogeneous elementary tensors.
    \end{proof}
\subsection{H-space structures}
In the remainder of this section, we restrict our attention to graded Hilbert spaces whose \textit{even and odd parts are countably infinite dimensional}. Under this condition, the spectral $K$-theory spaces underlying $\h$ can be given a H-space structure.

For the convenience of the reader, we include the definition of $H$-space here:
\begin{defi}
    A pointed topological space $(X,e)$ is said to be an H-space if there is a continuous map $\mu:X \times X \rightarrow X$ such that the two maps $x \mapsto \mu(x,e)$ and $x \mapsto \mu(e,x)$ are based homotopic to the identity map $X \rightarrow X$. 
\end{defi}
More details about H-spaces and their properties can be found in \cite{HAT}, \cite{DAVIS}. 

For the remainder of this section, fix a graded $C^*$-algebra $A$, and a graded Hilbert space $\h$ whose even and odd parts countably infinite dimensional. 

\begin{defi}\label{defi: H space}
    Let $u:\h \oplus \h \rightarrow \h$ be an even unitary \footnote{it exists as $\h$ is taken to be infinite dimensional.}.
    Define a map
    \[
    \mu_u:\K(A,\h) \times \K(A,\h) \rightarrow \K(A,\h)
    \] on $\K(A,\h)$ as the composite 
    \[
    \K(A,\h) \times \K(A,\h) \xrightarrow{s} \K(A,\h \oplus \h) \xrightarrow{1 \hat\otimes \operatorname{Ad}(u)} \K(A,\h)
    \]
    when $s$ is as in \cref{prop: sum}.
\end{defi}
Before proceeding further, we recall the ``Dixmier-Douady swindle" (see \cite{DIX} for a proof):
\begin{lem}\label{lem: dix}
    Let $\com_1$,$\com_2$ be (ungraded) Hilbert spaces with $\com_2$ countably infinite dimensional. Then, $\operatorname{Iso}(\com_1,\com_2)$ and $\mathcal{U}(\com_1,\com_2)$ (when $\com_1$ and $\com_2$ are unitarily isomorphic) are contractible in the strong $*$-topology. 
\end{lem}

In what follows, fix an even unitary $u:\h \oplus \h \rightarrow \h$, whereby getting a map $\mu_u:\K(A,\h) \times \K(A,\h) \rightarrow \K(A,\h)$ as in \cref{defi: H space}.
\begin{lem} \label{lem: H space verification}
    The maps $\mu_u(0,-):\K(A,\h) \rightarrow \K(A,\h)$ and $\mu_u(0,-):\K(A,\h) \rightarrow \K(A,\h)$ are based homotopic to $\operatorname{id}:\K(A,\h) \rightarrow \K(A,\h)$.
\end{lem}
\begin{proof}
    Note that from the proof of \cref{thm: enriched cat}, the following map is continuous:
    \begin{align*}
        \operatorname{Iso}^{ev}(\mathcal{H},\mathcal{H} \oplus \mathcal{H}) \times \K(A,\mathcal{H}) &\rightarrow \K(A,\mathcal{H} \oplus \mathcal{H}) \\
        (\iota,\varphi) &\mapsto (1 \hat\otimes \operatorname{Ad}(\iota))(\varphi)
    \end{align*}
    Set $v:=u^* \in \operatorname{Iso}^{ev}(\mathcal{H},\mathcal{H} \oplus \mathcal{H})$. As $\operatorname{Iso}^{ev}(\mathcal{H},\mathcal{H} \oplus \mathcal{H})$ is path connected in the strong-$*$ topology (\cref{lem: dix}),  there is a path \[p:[0,1] \rightarrow \operatorname{Iso}^{ev}(\mathcal{H},\mathcal{H} \oplus \mathcal{H})\] such that $p(0)=u^*$ and $p(1)=\begin{pmatrix}
    1 \\ 0
    \end{pmatrix}$.  \\
   
   Thus, $H:\mathbb{K}(A,\h) \times [0,1] \rightarrow \mathbb{K}(A,\h)$, given by $H(\varphi,t)=(1 \hat\otimes \operatorname{Ad}p_t^*)(s(\varphi,0))$ is a based homotopy between $\mu_u(-,0)$ and $\operatorname{id}$. The case of $\mu_u(0,-)$ is similar.
    
\end{proof}
We have now shown that:
\begin{thm}
    If $A$ is a $\Z/2\Z$-graded $C^*$-algebra and $\h$ is a graded Hilbert space whose even and odd parts are countable infinite dimensional. Then, for any $u \in \mathcal{U}^{\operatorname{ev}}(\h \oplus \h,\h)$, $(\K(A,\h),\mu_u,0)$ is an H-space.
\end{thm}
\begin{lem}
    $(\K(A,\h),\mu_u,0)$ is a homotopy associative $H$-space, that is, the following diagram commutes upto homotopy:
    \[
    % https://q.uiver.app/#q=WzAsNCxbMCwwLCJcXG1hdGhiYntLfShBLFxcbWF0aGNhbHtIfSkgXFx0aW1lcyBcXG1hdGhiYntLfShBLFxcbWF0aGNhbHtIfSkgXFx0aW1lcyBcXG1hdGhiYntLfShBLFxcbWF0aGNhbHtIfSkiXSxbMiwwLCJcXG1hdGhiYntLfShBLFxcbWF0aGNhbHtIfSkgXFx0aW1lcyBcXG1hdGhiYntLfShBLFxcbWF0aGNhbHtIfSkiXSxbMCwxLCJcXG1hdGhiYntLfShBLFxcbWF0aGNhbHtIfSkgXFx0aW1lcyBcXG1hdGhiYntLfShBLFxcbWF0aGNhbHtIfSkiXSxbMiwxLCJcXG1hdGhiYntLfShBLFxcbWF0aGNhbHtIfSkiXSxbMCwxLCIxIFxcdGltZXMgXFxtdV91Il0sWzAsMiwiXFxtdV91IFxcdGltZXMgMSIsMl0sWzIsMywiXFxtdV91IiwyXSxbMSwzLCJcXG11X3UiXV0=
\begin{tikzcd}
	{\mathbb{K}(A,\mathcal{H}) \times \mathbb{K}(A,\mathcal{H}) \times \mathbb{K}(A,\mathcal{H})} && {\mathbb{K}(A,\mathcal{H}) \times \mathbb{K}(A,\mathcal{H})} \\
	{\mathbb{K}(A,\mathcal{H}) \times \mathbb{K}(A,\mathcal{H})} && {\mathbb{K}(A,\mathcal{H})}
	\arrow["{1 \times \mu_u}", from=1-1, to=1-3]
	\arrow["{\mu_u \times 1}"', from=1-1, to=2-1]
	\arrow["{\mu_u}"', from=2-1, to=2-3]
	\arrow["{\mu_u}", from=1-3, to=2-3]
\end{tikzcd}
    \]
\end{lem}
\begin{proof}
    We expand the above diagram into the following:
    \[
    % https://q.uiver.app/#q=WzAsOSxbMCwwLCJcXG1hdGhiYntLfShBLFxcbWF0aGNhbHtIfSkgXFx0aW1lcyBcXG1hdGhiYntLfShBLFxcbWF0aGNhbHtIfSkgXFx0aW1lcyBcXG1hdGhiYntLfShBLFxcbWF0aGNhbHtIfSkiXSxbMywwLCJcXG1hdGhiYntLfShBLFxcbWF0aGNhbHtIfSkgXFx0aW1lcyBcXG1hdGhiYntLfShBLFxcbWF0aGNhbHtIfSkiXSxbMCwyLCJcXG1hdGhiYntLfShBLFxcbWF0aGNhbHtIfSkgXFx0aW1lcyBcXG1hdGhiYntLfShBLFxcbWF0aGNhbHtIfSkiXSxbMywyLCJcXG1hdGhiYntLfShBLFxcbWF0aGNhbHtIfSkiXSxbMSwwLCJcXG1hdGhiYntLfShBLFxcbWF0aGNhbHtIfSkgXFx0aW1lcyBcXG1hdGhiYntLfShBLFxcbWF0aGNhbHtIfSBcXG9wbHVzIFxcbWF0aGNhbHtIfSkiXSxbMCwxLCJcXG1hdGhiYntLfShBLFxcbWF0aGNhbHtIfSBcXG9wbHVzIFxcbWF0aGNhbHtIfSkgXFx0aW1lcyBcXG1hdGhiYntLfShBLFxcbWF0aGNhbHtIfSkiXSxbMSwxLCJcXG1hdGhiYntLfShBLFxcbWF0aGNhbHtIfSBcXG9wbHVzIFxcbWF0aGNhbHtIfSBcXG9wbHVzIFxcbWF0aGNhbHtIfSkiXSxbMywxLCJcXG1hdGhiYntLfShBLFxcbWF0aGNhbHtIfSBcXG9wbHVzIFxcbWF0aGNhbHtIfSkiXSxbMSwyLCJcXG1hdGhiYntLfShBLFxcbWF0aGNhbHtIfSBcXG9wbHVzIFxcbWF0aGNhbHtIfSkiXSxbMCw0LCIxIFxcdGltZXMgcyJdLFs0LDEsIjEgXFx0aW1lcyBcXG9wZXJhdG9ybmFtZXtBZH0odSkiXSxbMCw1LCJzIFxcdGltZXMgMSIsMl0sWzUsMiwiMVxcaGF0XFxvdGltZXNcXG9wZXJhdG9ybmFtZXtBZH0odSkgXFx0aW1lcyAxIiwyXSxbNSw2LCJzIiwyXSxbNCw2LCJzIiwyXSxbMSw3LCJzIl0sWzcsMywiMSBcXGhhdFxcb3RpbWVzIFxcb3BlcmF0b3JuYW1le0FkfSh1KSJdLFsyLDgsInMiLDJdLFs2LDgsIjEgXFxoYXRcXG90aW1lcyBcXG9wZXJhdG9ybmFtZXtBZH0odSBcXG9wbHVzIDEpIiwyXSxbOCwzLCIxIFxcaGF0XFxvdGltZXNcXG9wZXJhdG9ybmFtZXtBZH0odSkiLDJdLFs2LDcsIjEgXFxoYXRcXG90aW1lcyBcXG9wZXJhdG9ybmFtZXtBZH0oMSBcXG9wbHVzIHUpIl1d
\begin{tikzcd}
	{\mathbb{K}(A,\mathcal{H}) \times \mathbb{K}(A,\mathcal{H}) \times \mathbb{K}(A,\mathcal{H})} & {\mathbb{K}(A,\mathcal{H}) \times \mathbb{K}(A,\mathcal{H} \oplus \mathcal{H})} && {\mathbb{K}(A,\mathcal{H}) \times \mathbb{K}(A,\mathcal{H})} \\
	{\mathbb{K}(A,\mathcal{H} \oplus \mathcal{H}) \times \mathbb{K}(A,\mathcal{H})} & {\mathbb{K}(A,\mathcal{H} \oplus \mathcal{H} \oplus \mathcal{H})} && {\mathbb{K}(A,\mathcal{H} \oplus \mathcal{H})} \\
	{\mathbb{K}(A,\mathcal{H}) \times \mathbb{K}(A,\mathcal{H})} & {\mathbb{K}(A,\mathcal{H} \oplus \mathcal{H})} && {\mathbb{K}(A,\mathcal{H})}
	\arrow["{1 \times s}", from=1-1, to=1-2]
	\arrow["{1 \times \operatorname{Ad}(u)}", from=1-2, to=1-4]
	\arrow["{s \times 1}"', from=1-1, to=2-1]
	\arrow["{1\hat\otimes\operatorname{Ad}(u) \times 1}"', from=2-1, to=3-1]
	\arrow["s"', from=2-1, to=2-2]
	\arrow["s"', from=1-2, to=2-2]
	\arrow["s", from=1-4, to=2-4]
	\arrow["{1 \hat\otimes \operatorname{Ad}(u)}", from=2-4, to=3-4]
	\arrow["s"', from=3-1, to=3-2]
	\arrow["{1 \hat\otimes \operatorname{Ad}(u \oplus 1)}"', from=2-2, to=3-2]
	\arrow["{1 \hat\otimes\operatorname{Ad}(u)}"', from=3-2, to=3-4]
	\arrow["{1 \hat\otimes \operatorname{Ad}(1 \oplus u)}", from=2-2, to=2-4]
\end{tikzcd}
    \]
    We have to show that the outer rectangle commutes upto homotopy. To that end, we first observe that all the reactangles except  the bottom right one commutes strictly, as follows from \cref{cor: hom preserve sum} and \cref{prop: sum comm and ass}. For the homotopy commutativity of the bottom right rectangle, it suffices to show that the unitaries $u \circ (1 \oplus u), u \circ (u \oplus 1)$ are homotopic in the strong topology. But this follows from \cref{lem: dix}.
\end{proof}
\begin{lem}
    $(\K(A,\h),\mu_u,0)$ is a homotopy commutative $H$-space, that is, the following diagram commutes upto homotopy:
    \[
    % https://q.uiver.app/#q=WzAsMyxbMCwwLCJcXG1hdGhiYntLfShBLFxcbWF0aGNhbHtIfSkgXFx0aW1lcyBcXG1hdGhiYntLfShBLFxcbWF0aGNhbHtIfSkiXSxbMiwwLCJcXG1hdGhiYntLfShBLFxcbWF0aGNhbHtIfSkgXFx0aW1lcyBcXG1hdGhiYntLfShBLFxcbWF0aGNhbHtIfSkiXSxbMSwxLCIgXFxtYXRoYmJ7S30oQSxcXG1hdGhjYWx7SH0pIl0sWzAsMiwiXFxtdV91Il0sWzEsMiwiXFxtdV91IiwyXSxbMCwxLCJcXG9wZXJhdG9ybmFtZXtmbGlwfSJdXQ==
\begin{tikzcd}
	{\mathbb{K}(A,\mathcal{H}) \times \mathbb{K}(A,\mathcal{H})} && {\mathbb{K}(A,\mathcal{H}) \times \mathbb{K}(A,\mathcal{H})} \\
	& { \mathbb{K}(A,\mathcal{H})}
	\arrow["{\mu_u}", from=1-1, to=2-2]
	\arrow["{\mu_u}"', from=1-3, to=2-2]
	\arrow["{\operatorname{flip}}", from=1-1, to=1-3]
\end{tikzcd}
    \]
\end{lem}
\begin{proof}
   We expand the above diagram into the following:
   \[
   % https://q.uiver.app/#q=WzAsNixbMCwwLCJcXG1hdGhiYntLfShBLFxcbWF0aGNhbHtIfSkgXFx0aW1lcyBcXG1hdGhiYntLfShBLFxcbWF0aGNhbHtIfSkiXSxbMCwxLCJcXG1hdGhiYntLfShBLFxcbWF0aGNhbHtIfSBcXG9wbHVzIFxcbWF0aGNhbHtIfSkiXSxbMCwyLCJcXG1hdGhiYntLfShBLFxcbWF0aGNhbHtIfSkiXSxbMiwwLCJcXG1hdGhiYntLfShBLFxcbWF0aGNhbHtIfSkgXFx0aW1lcyBcXG1hdGhiYntLfShBLFxcbWF0aGNhbHtIfSkiXSxbMiwxLCJcXG1hdGhiYntLfShBLFxcbWF0aGNhbHtIfSBcXG9wbHVzIFxcbWF0aGNhbHtIfSkiXSxbMiwyLCJcXG1hdGhiYntLfShBLFxcbWF0aGNhbHtIfSkiXSxbMCwxLCJcXG9wZXJhdG9ybmFtZXthZGR9Il0sWzEsMiwiMSBcXGhhdFxcb3RpbWVzIFxcb3BlcmF0b3JuYW1le0FkfSh1KSJdLFswLDMsIlxcb3BlcmF0b3JuYW1le2ZsaXB9Il0sWzEsNCwiMSBcXGhhdFxcb3RpbWVzIFxcb3BlcmF0b3JuYW1le0FkfSgoeCx5KSBcXG1hcHN0byh5LHgpKSJdLFsyLDUsIj0iXSxbNCw1LCIxIFxcaGF0XFxvdGltZXMgXFxvcGVyYXRvcm5hbWV7QWR9KHUpIl0sWzMsNCwiXFxvcGVyYXRvcm5hbWV7YWRkfSJdXQ==
\begin{tikzcd}
	{\mathbb{K}(A,\mathcal{H}) \times \mathbb{K}(A,\mathcal{H})} && {\mathbb{K}(A,\mathcal{H}) \times \mathbb{K}(A,\mathcal{H})} \\
	{\mathbb{K}(A,\mathcal{H} \oplus \mathcal{H})} && {\mathbb{K}(A,\mathcal{H} \oplus \mathcal{H})} \\
	{\mathbb{K}(A,\mathcal{H})} && {\mathbb{K}(A,\mathcal{H})}
	\arrow["{\operatorname{add}}", from=1-1, to=2-1]
	\arrow["{1 \hat\otimes \operatorname{Ad}(u)}", from=2-1, to=3-1]
	\arrow["{\operatorname{flip}}", from=1-1, to=1-3]
	\arrow["{1 \hat\otimes \operatorname{Ad}((x,y) \mapsto(y,x))}", from=2-1, to=2-3]
	\arrow["{=}", from=3-1, to=3-3]
	\arrow["{1 \hat\otimes \operatorname{Ad}(u)}", from=2-3, to=3-3]
	\arrow["{\operatorname{add}}", from=1-3, to=2-3]
\end{tikzcd}
   \]
   Our goal is to show that the outer diagram commutes upto homotopy. To that end, we note that the top rectangle strictly commutes by \cref{prop: sum comm and ass}, and the bottom one commutes up to homotopy since the two unitaries $u\circ ((x,y) \mapsto (y,x))$ and $u$ are homotopic in the strong topology.
\end{proof}
Summarizing, we have:
\begin{cor}\label{H-monoid}
    $(\K(A,\h),\mu_u,0)$ is a commutative H-monoid.
\end{cor}
\begin{rem}
    It is in fact also true that $(\K(A,\h),\mu_u,0)$ is a commutative $H$-group. But we don't use this fact in this article.
\end{rem}
\subsection{Sums preserve H-space structures}
As the title of the subsection suggests, we now see how the sums described in the first part of this section go with the H-space structures. We summarize everything in the below proposition. Throughout there, we assume that H-space structures $\mu_{u},\mu_{u'},...$ as described in \cref{defi: H space} have been fixed on the spaces $\K(A,\h),\K(B,\h')...$.
\begin{prop}\label{sum preserve H-space}
    \,
    \begin{enumerate}[(i)]
        \item Let $\varphi:A \rightarrow B$ be a graded $*$-homomorphism, and $\iota:\h \rightarrow \h'$ an even isometry. Then the induced map $\K(\varphi,\iota): \K(A,\mathcal{H})\rightarrow \K(B,\mathcal{H}')$ preserves the $H$-space structure, in the sense that the following diagram commutes upto homotopy:
    \[
    % https://q.uiver.app/#q=WzAsNCxbMCwwLCJcXG1hdGhiYntLfShBLFxcbWF0aGNhbHtIfSkgXFx0aW1lcyBcXG1hdGhiYntLfShBLFxcbWF0aGNhbHtIfSkiXSxbMiwwLCJcXG1hdGhiYntLfShCLFxcbWF0aGNhbHtIfScpIFxcdGltZXMgXFxtYXRoYmJ7S30oQixcXG1hdGhjYWx7SH0nKSJdLFswLDIsIlxcbWF0aGJie0t9KEEsXFxtYXRoY2Fse0h9KSJdLFsyLDIsIlxcbWF0aGJie0t9KEIsXFxtYXRoY2Fse0h9JykiXSxbMCwxLCJcXG1hdGhiYntLfShcXHZhcnBoaSxcXGlvdGEpIFxcdGltZXMgXFxtYXRoYmJ7S30oXFx2YXJwaGksXFxpb3RhKSJdLFsyLDMsIlxcbWF0aGJie0t9KFxcdmFycGhpLFxcbWF0aGNhbHtIfSkiXSxbMCwyLCJcXG11X3UiXSxbMSwzLCJcXG11X3YiXV0=
\begin{tikzcd}
	{\mathbb{K}(A,\mathcal{H}) \times \mathbb{K}(A,\mathcal{H})} && {\mathbb{K}(B,\mathcal{H}') \times \mathbb{K}(B,\mathcal{H}')} \\
	\\
	{\mathbb{K}(A,\mathcal{H})} && {\mathbb{K}(B,\mathcal{H}')}
	\arrow["{\mathbb{K}(\varphi,\iota) \times \mathbb{K}(\varphi,\iota)}", from=1-1, to=1-3]
	\arrow["{\mathbb{K}(\varphi,\iota)}", from=3-1, to=3-3]
	\arrow["{\mu_u}", from=1-1, to=3-1]
	\arrow["{\mu_v}", from=1-3, to=3-3]
\end{tikzcd}.
    \]
    \item Let $\kappa: \s \hat\otimes A \rightarrow B \hat\otimes \com(\h)$ be a graded $*$-homomorphism, and $\iota:\h_0 \rightarrow \h_1$ an even isometry. Then, the induced map $(\kappa,\iota)_*:\K(A,\h_0) \rightarrow \K(B,\h_1 \hat\otimes \h)$ as in \cref{defi: amplified homo} (i) preserves the H-space structure, in the sense that the following diagram commutes upto homotopy:
    \[
    % https://q.uiver.app/#q=WzAsNCxbMCwwLCJcXG1hdGhiYntLfShBLFxcbWF0aGNhbHtIfV8wKSBcXHRpbWVzIFxcbWF0aGJie0t9KEEsXFxtYXRoY2Fse0h9XzApIl0sWzIsMCwiXFxtYXRoYmJ7S30oQixcXG1hdGhjYWx7SH1fMSBcXGhhdFxcb3RpbWVzIFxcbWF0aGNhbHtIfSkgXFx0aW1lcyBcXG1hdGhiYntLfShCLFxcbWF0aGNhbHtIfV8xIFxcaGF0XFxvdGltZXMgXFxtYXRoY2Fse0h9KSJdLFswLDEsIlxcbWF0aGJie0t9KEEsXFxtYXRoY2Fse0h9XzApIl0sWzIsMSwiXFxtYXRoYmJ7S30oQixcXG1hdGhjYWx7SH1fMSBcXGhhdFxcb3RpbWVzIFxcbWF0aGNhbHtIfSkiXSxbMCwxLCIoXFxrYXBwYSxcXGlvdGEpXypcXHRpbWVzKFxca2FwcGEsXFxpb3RhKV8qIl0sWzAsMiwiXFxtdV91IiwyXSxbMSwzLCJcXG11X3YiXSxbMiwzLCIoXFxrYXBwYSxcXGlvdGEpXyoiXV0=
\begin{tikzcd}
	{\mathbb{K}(A,\mathcal{H}_0) \times \mathbb{K}(A,\mathcal{H}_0)} && {\mathbb{K}(B,\mathcal{H}_1 \hat\otimes \mathcal{H}) \times \mathbb{K}(B,\mathcal{H}_1 \hat\otimes \mathcal{H})} \\
	{\mathbb{K}(A,\mathcal{H}_0)} && {\mathbb{K}(B,\mathcal{H}_1 \hat\otimes \mathcal{H})}
	\arrow["{(\kappa,\iota)_*\times(\kappa,\iota)_*}", from=1-1, to=1-3]
	\arrow["{\mu_u}"', from=1-1, to=2-1]
	\arrow["{\mu_v}", from=1-3, to=2-3]
	\arrow["{(\kappa,\iota)_*}", from=2-1, to=2-3]
\end{tikzcd}
    \]
    \item Let $\kappa: A \hat\otimes \s \hat\otimes B \rightarrow C\hat\otimes \com(\h)$ be a graded $*$-homomorphism, and $\iota:\h_0 \rightarrow \h_1$ an even isometry. Then, the induced map $(\kappa,\iota)_*:\K(A \hat\otimes B,\h_0) \rightarrow \K(C,\h_1 \hat\otimes \h)$ as in \cref{defi: amplified homo} (ii) preserves the H-space structure, in the sense that the following diagram commutes upto homotopy:
    \[
    % https://q.uiver.app/#q=WzAsNCxbMCwwLCJcXG1hdGhiYntLfShBIFxcaGF0XFxvdGltZXMgQixcXG1hdGhjYWx7SH1fMCkgXFx0aW1lcyBcXG1hdGhiYntLfShBIFxcaGF0XFxvdGltZXMgQixcXG1hdGhjYWx7SH1fMCkiXSxbMCwxLCJcXG1hdGhiYntLfShBIFxcaGF0XFxvdGltZXMgQixcXG1hdGhjYWx7SH1fMCkiXSxbMiwwLCJcXG1hdGhiYntLfShDLFxcbWF0aGNhbHtIfV8xIFxcaGF0XFxvdGltZXNcXG1hdGhjYWx7SH0pIFxcdGltZXMgXFxtYXRoYmJ7S30oQyxcXG1hdGhjYWx7SH1fMSBcXGhhdFxcb3RpbWVzIFxcbWF0aGNhbHtIfSkiXSxbMiwxLCJcXG1hdGhiYntLfShDLFxcbWF0aGNhbHtIfV8xIFxcaGF0XFxvdGltZXMgXFxtYXRoY2Fse0h9KSJdLFswLDEsIlxcbXVfdSJdLFswLDIsIihcXGthcHBhLFxcaW90YSlfKiBcXHRpbWVzIChcXGthcHBhLFxcaW90YSlfKiJdLFsxLDMsIihcXGthcHBhLFxcaW90YSlfKiJdLFsyLDMsIlxcbXVfdiJdXQ==
\begin{tikzcd}
	{\mathbb{K}(A \hat\otimes B,\mathcal{H}_0) \times \mathbb{K}(A \hat\otimes B,\mathcal{H}_0)} && {\mathbb{K}(C,\mathcal{H}_1 \hat\otimes\mathcal{H}) \times \mathbb{K}(C,\mathcal{H}_1 \hat\otimes \mathcal{H})} \\
	{\mathbb{K}(A \hat\otimes B,\mathcal{H}_0)} && {\mathbb{K}(C,\mathcal{H}_1 \hat\otimes \mathcal{H})}
	\arrow["{\mu_u}", from=1-1, to=2-1]
	\arrow["{(\kappa,\iota)_* \times (\kappa,\iota)_*}", from=1-1, to=1-3]
	\arrow["{(\kappa,\iota)_*}", from=2-1, to=2-3]
	\arrow["{\mu_v}", from=1-3, to=2-3]
\end{tikzcd}
    \]
    
    \end{enumerate}
    
\end{prop}
\begin{proof}
   We only discuss commutativity of $(iii)$; $(i)$ and $(ii)$ are similar (and easier).

   We have the following diagram:
   \[
   % https://q.uiver.app/#q=WzAsNyxbMCwwLCJcXG1hdGhiYntLfShBIFxcaGF0XFxvdGltZXMgQixcXG1hdGhjYWx7SH1fMCkgXFx0aW1lcyBcXG1hdGhiYntLfShBIFxcaGF0XFxvdGltZXMgQixcXG1hdGhjYWx7SH1fMCkiXSxbMCwyLCJcXG1hdGhiYntLfShBIFxcaGF0XFxvdGltZXMgQixcXG1hdGhjYWx7SH1fMCkiXSxbMiwwLCJcXG1hdGhiYntLfShDLFxcbWF0aGNhbHtIfV8xIFxcaGF0XFxvdGltZXNcXG1hdGhjYWx7SH0pIFxcdGltZXMgXFxtYXRoYmJ7S30oQyxcXG1hdGhjYWx7SH1fMSBcXGhhdFxcb3RpbWVzIFxcbWF0aGNhbHtIfSkiXSxbMiwyLCJcXG1hdGhiYntLfShDLFxcbWF0aGNhbHtIfV8xIFxcaGF0XFxvdGltZXMgXFxtYXRoY2Fse0h9KSJdLFswLDEsIlxcbWF0aGJie0t9KEEgXFxoYXRcXG90aW1lcyBCLFxcbWF0aGNhbHtIfV8wIFxcb3BsdXMgXFxtYXRoY2Fse0h9XzApIl0sWzEsMSwiXFxtYXRoYmJ7S30oQywoXFxtYXRoY2Fse0h9XzEgXFxvcGx1cyBcXG1hdGhjYWx7SH1fMSkgXFxoYXRcXG90aW1lcyBcXG1hdGhjYWx7SH0pKSJdLFsyLDEsIlxcbWF0aGJie0t9KEMsXFxtYXRoY2Fse0h9XzEgXFxoYXRcXG90aW1lcyBcXG1hdGhjYWx7SH0gXFxvcGx1cyBcXG1hdGhjYWx7SH1fMSBcXGhhdFxcb3RpbWVzIFxcbWF0aGNhbHtIfSkiXSxbMCwyLCIoXFxrYXBwYSxcXGlvdGEpXyogXFx0aW1lcyAoXFxrYXBwYSxcXGlvdGEpXyoiXSxbMSwzLCIoXFxrYXBwYSxcXGlvdGEpXyoiLDAseyJsYWJlbF9wb3NpdGlvbiI6MzB9XSxbNCw1LCIoXFxrYXBwYSxcXGlvdGFfMCBcXG9wbHVzIFxcaW90YV8wKV8qIl0sWzIsNiwiXFxvcGVyYXRvcm5hbWV7YWRkfSJdLFs1LDYsIlxcY29uZyJdLFs2LDMsIjEgXFxoYXRcXG90aW1lcyBcXG9wZXJhdG9ybmFtZXtBZH0odikiXSxbMCw0LCJcXG9wZXJhdG9ybmFtZXthZGR9Il0sWzQsMSwiMSBcXGhhdFxcb3RpbWVzIFxcb3BlcmF0b3JuYW1le0FkfSh1KSJdLFs1LDMsIjEgXFxoYXRcXG90aW1lcyBcXG9wZXJhdG9ybmFtZXtBZH0odyBcXGhhdFxcb3RpbWVzIDEpIiwyLHsibGFiZWxfcG9zaXRpb24iOjIwLCJzdHlsZSI6eyJib2R5Ijp7Im5hbWUiOiJkYXNoZWQifX19XV0=
\begin{tikzcd}
	{\mathbb{K}(A \hat\otimes B,\mathcal{H}_0) \times \mathbb{K}(A \hat\otimes B,\mathcal{H}_0)} && {\mathbb{K}(C,\mathcal{H}_1 \hat\otimes\mathcal{H}) \times \mathbb{K}(C,\mathcal{H}_1 \hat\otimes \mathcal{H})} \\
	{\mathbb{K}(A \hat\otimes B,\mathcal{H}_0 \oplus \mathcal{H}_0)} & {\mathbb{K}(C,(\mathcal{H}_1 \oplus \mathcal{H}_1) \hat\otimes \mathcal{H}))} & {\mathbb{K}(C,\mathcal{H}_1 \hat\otimes \mathcal{H} \oplus \mathcal{H}_1 \hat\otimes \mathcal{H})} \\
	{\mathbb{K}(A \hat\otimes B,\mathcal{H}_0)} && {\mathbb{K}(C,\mathcal{H}_1 \hat\otimes \mathcal{H})}
	\arrow["{(\kappa,\iota)_* \times (\kappa,\iota)_*}", from=1-1, to=1-3]
	\arrow["{(\kappa,\iota)_*}"{pos=0.3}, from=3-1, to=3-3]
	\arrow["{(\kappa,\iota_0 \oplus \iota_0)_*}", from=2-1, to=2-2]
	\arrow["{\operatorname{add}}", from=1-3, to=2-3]
	\arrow["\cong", from=2-2, to=2-3]
	\arrow["{1 \hat\otimes \operatorname{Ad}(v)}", from=2-3, to=3-3]
	\arrow["{\operatorname{add}}", from=1-1, to=2-1]
	\arrow["{1 \hat\otimes \operatorname{Ad}(u)}", from=2-1, to=3-1]
	\arrow["{1 \hat\otimes \operatorname{Ad}(w \hat\otimes 1)}"'{pos=0.2}, dashed, from=2-2, to=3-3]
\end{tikzcd}
   \]
   Our goal is to show that the outer diagram commutes up to homotopy. To that end, it suffices to show that for some (in fact any) isometry $w:\h_1 \oplus \h_1 \rightarrow \h_1$, each of the three smaller diagrams commute up to homotopy. The top rectangle commutes from \cref{sum nat amp morph} (ii). The trapezoid at the bottom left commutes upto homotopy by \cref{nat of amp morph} (ii) (here, we use the fact that $\operatorname{Iso}^{ev}(\h_0 \oplus \h_0,\h_1 \hat\otimes \h)$ is path-connected in the strong-$*$ topology). The triangle at the bottom right commutes up to homotopy too (here, we use the fact that $\operatorname{Iso}^{ev}((\h_1 \oplus \h_1)\hat\otimes \h,\h_1 \hat\otimes \h)$ is path-connected in the strong-$*$ topology).
\end{proof}
\section{Products in the spectral model}
One of the biggest advantages of working with the spectral model of $K$-theory is that it is well acquainted to handling sums and products. Having already studied sums, in this section we study products.

\subsection{Products and naturality}

We start off with the definition of products in the spectral picture:
\begin{defi} \label{prop: product}
    Given $\Z/2\Z$-graded $C^*$-algebras $A$ and $B$, we have a continuous map 
    \[
    p:\K(A,\h_0) \times \K(B,\h_1) \rightarrow \K(A \hat\otimes B,\h_0 \hat\otimes \h_1),
    \] which is natural with respect to $*$-homomorphisms $A \rightarrow A'$ and $B \rightarrow B'$,  and even isometries $\h_0 \rightarrow \h_0'$ and $\h_1 \rightarrow \h_1
    '$. On an element $(\varphi,\psi)\in  \K(A,\h_0) \times \K(B,\h_1)$, it yields the composite
    \begin{equation} \label{eqn:product}
            % https://q.uiver.app/#q=WzAsNSxbMCwwLCJcXG1hdGhjYWx7U30iXSxbMSwwLCJcXG1hdGhjYWx7U30gXFxoYXRcXG90aW1lc1xcbWF0aGNhbHtTfSJdLFsyLDAsIkFcXGhhdFxcb3RpbWVzIFxcbWF0aGNhbHtLfShcXG1hdGhjYWx7SH1fMClcXGhhdFxcb3RpbWVzIEIgXFxoYXRcXG90aW1lc1xcbWF0aGNhbHtLfShcXG1hdGhjYWx7SH1fMSkiXSxbMywwLCJBXFxoYXRcXG90aW1lcyBCIFxcaGF0XFxvdGltZXNcXG1hdGhjYWx7S30oXFxtYXRoY2Fse0h9XzApIFxcaGF0XFxvdGltZXNcXG1hdGhjYWx7S30oXFxtYXRoY2Fse0h9XzEpIl0sWzQsMCwiQVxcaGF0XFxvdGltZXMgQiBcXGhhdFxcb3RpbWVzXFxtYXRoY2Fse0t9KFxcbWF0aGNhbHtIfV8wXFxoYXRcXG90aW1lcyBcXG1hdGhjYWx7SH1fMSkiXSxbMCwxLCJcXHRyaWFuZ2xlIl0sWzEsMiwiXFx2YXJwaGkgXFxoYXRcXG90aW1lc1xccHNpIl0sWzIsMywiMVxcaGF0XFxvdGltZXNcXHRhdVxcaGF0XFxvdGltZXMxIl0sWzMsNCwiMVxcaGF0XFxvdGltZXMxXFxoYXRcXG90aW1lc1xcUHNpIl1d
\begin{tikzcd}
	{\mathcal{S}} & {\mathcal{S} \hat\otimes\mathcal{S}} & {A\hat\otimes \mathcal{K}(\mathcal{H}_0)\hat\otimes B \hat\otimes\mathcal{K}(\mathcal{H}_1)} & {A\hat\otimes B \hat\otimes\mathcal{K}(\mathcal{H}_0) \hat\otimes\mathcal{K}(\mathcal{H}_1)} & {A\hat\otimes B \hat\otimes\mathcal{K}(\mathcal{H}_0\hat\otimes \mathcal{H}_1)}
	\arrow["\triangle", from=1-1, to=1-2]
	\arrow["{\varphi \hat\otimes\psi}", from=1-2, to=1-3]
	\arrow["1\hat\otimes\tau\hat\otimes1", from=1-3, to=1-4]
	\arrow["1\hat\otimes1\hat\otimes\Psi", from=1-4, to=1-5]
\end{tikzcd}
    \end{equation}
    where $\tau:\com(\h_0) \hat\otimes B \rightarrow B \hat\otimes \com(\h_0)$ is the flip isomorphism, and $\Psi$ is the canonical isomorphism between $\mathcal{K}(\h_0) \hat\otimes \mathcal{K}(\h_1) $ and $ \mathcal{K}(\h_0\hat\otimes \h_1)$. We call the element $p(\varphi,\psi)$ as the product of $\varphi$ and $\psi$.
\end{defi}

    Note that the continuity of the map is immediate in light of \cref{lem: 1st lemma} and \cref{lem: cont of tensor product}. The naturality assertion is immediate from \cref{eqn:product}. 
\begin{rem}
   For our convenience, we will also call the map $p$ as ``prod" at some junctures.
\end{rem}
\begin{rem} \label{cor: hom preserve prod}
    Note that the naturality assertion in \cref{prop: product} really means that the following diagram commutes, where $\varphi_A: A \rightarrow A',\varphi_B: B \rightarrow B'$ are a graded $*$-homomorphism, and $\iota_0:\h_0 \rightarrow \h_0'$, $\iota_1:\h_1 \rightarrow \h_1'$ are even isometries.
    \[
   % https://q.uiver.app/#q=WzAsNCxbMCwwLCJcXG1hdGhiYntLfShBLFxcbWF0aGNhbHtIfV8wKSBcXHRpbWVzIFxcbWF0aGJie0t9KEIsXFxtYXRoY2Fse0h9XzEpIl0sWzIsMCwiXFxtYXRoYmJ7S30oQScsXFxtYXRoY2Fse0h9XzAnKSBcXHRpbWVzIFxcbWF0aGJie0t9KEInLFxcbWF0aGNhbHtIfV8xJykiXSxbMCwxLCJcXG1hdGhiYntLfShBIFxcaGF0XFxvdGltZXMgQixcXG1hdGhjYWx7SH1fMCBcXGhhdFxcb3RpbWVzIFxcbWF0aGNhbHtIfV8xKSJdLFsyLDEsIlxcbWF0aGJie0t9KEEnIFxcaGF0XFxvdGltZXMgQicsXFxtYXRoY2Fse0h9XzAnIFxcb3BsdXMgXFxtYXRoY2Fse0h9XzEnKSJdLFswLDIsIlxcb3BlcmF0b3JuYW1le3Byb2R9Il0sWzEsMywiXFxvcGVyYXRvcm5hbWV7cHJvZH0iXSxbMCwxLCJcXG1hdGhiYntLfShcXHZhcnBoaV9BLFxcaW90YV8wKSBcXHRpbWVzIFxcbWF0aGJie0t9KFxcdmFycGhpX0IsXFxpb3RhXzEpIl0sWzIsMywiXFxtYXRoYmJ7S30oXFx2YXJwaGlfQSBcXGhhdFxcb3RpbWVzIFxcdmFycGhpX0IsXFxpb3RhXzAgXFxoYXRcXG90aW1lcyBcXGlvdGFfMSkiXV0=
\begin{tikzcd}
	{\mathbb{K}(A,\mathcal{H}_0) \times \mathbb{K}(B,\mathcal{H}_1)} && {\mathbb{K}(A',\mathcal{H}_0') \times \mathbb{K}(B',\mathcal{H}_1')} \\
	{\mathbb{K}(A \hat\otimes B,\mathcal{H}_0 \hat\otimes \mathcal{H}_1)} && {\mathbb{K}(A' \hat\otimes B',\mathcal{H}_0' \oplus \mathcal{H}_1')}
	\arrow["{\operatorname{prod}}", from=1-1, to=2-1]
	\arrow["{\operatorname{prod}}", from=1-3, to=2-3]
	\arrow["{\mathbb{K}(\varphi_A,\iota_0) \times \mathbb{K}(\varphi_B,\iota_1)}", from=1-1, to=1-3]
	\arrow["{\mathbb{K}(\varphi_A \hat\otimes \varphi_B,\iota_0 \hat\otimes \iota_1)}", from=2-1, to=2-3]
\end{tikzcd}
    \]
    We draw this diagram here for future reference.
\end{rem}
\begin{prop}
    Products are commutative and associative, in the sense that the following diagrams commute strictly:
    \[
    % https://q.uiver.app/#q=WzAsOCxbMCwwLCJcXG1hdGhiYntLfShBLFxcbWF0aGNhbHtIfV8wKSBcXHRpbWVzIFxcbWF0aGJie0t9KEIsXFxtYXRoY2Fse0h9XzEpIl0sWzIsMCwiXFxtYXRoYmJ7S30oQixcXG1hdGhjYWx7SH1fMSkgXFx0aW1lcyBcXG1hdGhiYntLfShBLFxcbWF0aGNhbHtIfV8wKSJdLFswLDEsIlxcbWF0aGJie0t9KEEgXFxoYXRcXG90aW1lcyBCLCBcXG1hdGhjYWx7SH1fMCBcXGhhdFxcb3RpbWVzXFxtYXRoY2Fse0h9XzEpIl0sWzIsMSwiXFxtYXRoYmJ7S30oQiBcXGhhdFxcb3RpbWVzIEEsIFxcbWF0aGNhbHtIfV8xIFxcaGF0XFxvdGltZXMgXFxtYXRoY2Fse0h9XzApIl0sWzAsMiwiXFxtYXRoYmJ7S30oQSxcXG1hdGhjYWx7SH1fMCkgXFx0aW1lcyBcXG1hdGhiYntLfShCLFxcbWF0aGNhbHtIfV8xKSBcXHRpbWVzIFxcbWF0aGJie0t9KEMsXFxtYXRoY2Fse0h9XzIpIl0sWzAsMywiXFxtYXRoYmJ7S30oQSBcXGhhdFxcb3RpbWVzIEIsIFxcbWF0aGNhbHtIfV8wIFxcaGF0XFxvdGltZXNcXG1hdGhjYWx7SH1fMSkgXFx0aW1lcyBcXG1hdGhiYntLfShDLFxcbWF0aGNhbHtIfV8yKSJdLFsyLDIsIlxcbWF0aGJie0t9KEEsXFxtYXRoY2Fse0h9XzApIFxcdGltZXMgXFxtYXRoYmJ7S30oQiBcXGhhdFxcb3RpbWVzIEMsIFxcbWF0aGNhbHtIfV8xIFxcaGF0XFxvdGltZXNcXG1hdGhjYWx7SH1fMikgIl0sWzIsMywiXFxtYXRoYmJ7S30oQSBcXGhhdFxcb3RpbWVzIEIgXFxoYXRcXG90aW1lcyBDLCBcXG1hdGhjYWx7SH1fMCBcXGhhdFxcb3RpbWVzXFxtYXRoY2Fse0h9XzEgXFxoYXRcXG90aW1lcyBcXG1hdGhjYWx7SH1fMikiXSxbMCwxLCJcXG9wZXJhdG9ybmFtZXtmbGlwfSJdLFswLDIsIlxcb3BlcmF0b3JuYW1le3Byb2R9Il0sWzEsMywiXFxvcGVyYXRvcm5hbWV7cHJvZH0iXSxbMiwzLCIoXFx0YXVfe0EgXFxoYXRcXG90aW1lcyBCfSxcXHRhdV97XFxtYXRoY2Fse0h9XzAgXFxoYXRcXG90aW1lcyBcXG1hdGhjYWx7SH1fMX0pIiwyXSxbNCw1LCJcXG9wZXJhdG9ybmFtZXtwcm9kfSBcXHRpbWVzIDEiLDJdLFs0LDYsIjEgXFx0aW1lcyBcXG9wZXJhdG9ybmFtZXtwcm9kfSJdLFs1LDcsIlxcb3BlcmF0b3JuYW1le3Byb2R9IiwyXSxbNiw3LCJcXG9wZXJhdG9ybmFtZXtwcm9kfSJdXQ==
\begin{tikzcd}
	{\mathbb{K}(A,\mathcal{H}_0) \times \mathbb{K}(B,\mathcal{H}_1)} && {\mathbb{K}(B,\mathcal{H}_1) \times \mathbb{K}(A,\mathcal{H}_0)} \\
	{\mathbb{K}(A \hat\otimes B, \mathcal{H}_0 \hat\otimes\mathcal{H}_1)} && {\mathbb{K}(B \hat\otimes A, \mathcal{H}_1 \hat\otimes \mathcal{H}_0)} \\
	{\mathbb{K}(A,\mathcal{H}_0) \times \mathbb{K}(B,\mathcal{H}_1) \times \mathbb{K}(C,\mathcal{H}_2)} && {\mathbb{K}(A,\mathcal{H}_0) \times \mathbb{K}(B \hat\otimes C, \mathcal{H}_1 \hat\otimes\mathcal{H}_2) } \\
	{\mathbb{K}(A \hat\otimes B, \mathcal{H}_0 \hat\otimes\mathcal{H}_1) \times \mathbb{K}(C,\mathcal{H}_2)} && {\mathbb{K}(A \hat\otimes B \hat\otimes C, \mathcal{H}_0 \hat\otimes\mathcal{H}_1 \hat\otimes \mathcal{H}_2)}
	\arrow["{\operatorname{flip}}", from=1-1, to=1-3]
	\arrow["{\operatorname{prod}}", from=1-1, to=2-1]
	\arrow["{\operatorname{prod}}", from=1-3, to=2-3]
	\arrow["{(\tau_{A \hat\otimes B},\tau_{\mathcal{H}_0 \hat\otimes \mathcal{H}_1})}"', from=2-1, to=2-3]
	\arrow["{\operatorname{prod} \times 1}"', from=3-1, to=4-1]
	\arrow["{1 \times \operatorname{prod}}", from=3-1, to=3-3]
	\arrow["{\operatorname{prod}}"', from=4-1, to=4-3]
	\arrow["{\operatorname{prod}}", from=3-3, to=4-3]
\end{tikzcd}
    \]
\end{prop}
\begin{proof}
    We only discuss the commutativity case, the associativity case is more straight-forward.

    Checking the strict commutativity of the space level diagram amounts to checking commutativity of the following diagram at the algebra level, where $\varphi \in \mathbb{K}(A,\h_0), \psi \in \mathbb{K}(B,\h_1)$:
    \[
\begin{tikzcd}
	{\mathcal{S}} \\
	{\mathcal{S} \hat\otimes \mathcal{S}} && {\mathcal{S} \hat\otimes \mathcal{S}} \\
	{A \hat\otimes \mathcal{K}(\mathcal{H}_0) \hat\otimes B \hat\otimes \mathcal{K}(\mathcal{H}_1)} && {B \hat\otimes \mathcal{K}(\mathcal{H}_1) \hat\otimes A \hat\otimes \mathcal{K}(\mathcal{H}_0)} \\
	{A \hat\otimes B\hat\otimes \mathcal{K}(\mathcal{H}_0)  \hat\otimes \mathcal{K}(\mathcal{H}_1)} && {B \hat\otimes A\hat\otimes \mathcal{K}(\mathcal{H}_1)  \hat\otimes \mathcal{K}(\mathcal{H}_0)} \\
	{A \hat\otimes B \hat\otimes \mathcal{K}(\mathcal{H}_0 \hat\otimes \mathcal{H}_1)} && {B \hat\otimes A \hat\otimes \mathcal{K}(\mathcal{H}_1 \hat\otimes \mathcal{H}_0)}
	\arrow["{f \hat\otimes g \mapsto (-1)^{\partial f \partial g}g \hat\otimes f}", from=2-1, to=2-3]
	\arrow["\triangle", from=1-1, to=2-1]
	\arrow["{\varphi \hat\otimes \psi}", from=2-1, to=3-1]
	\arrow["{\operatorname{flip} B ~\&~ \mathcal{K}(\mathcal{H}_0)}", from=3-1, to=4-1]
	\arrow["\cong", from=4-1, to=5-1]
	\arrow["{\psi \hat\otimes \varphi}", from=2-3, to=3-3]
	\arrow["{\operatorname{flip} A ~\&~ \mathcal{K}(\mathcal{H}_1)}", from=3-3, to=4-3]
	\arrow["\cong", from=4-3, to=5-3]
	\arrow["{\tau_{A \hat\otimes B} \hat\otimes \operatorname{Ad}(\tau_{\mathcal{H}_0 \hat\otimes \mathcal{H}_1})}", from=5-1, to=5-3]
\end{tikzcd}
    \]
    It does commute, as can be checked by tracing the fate of the homogeneous element $f \hat\otimes g \in \s \hat\otimes \s$.
\end{proof}
We now show that products ``distribute" over sums. In fact, it suffices to show only one sided distributivity:
\begin{lem}\label{prop: sum distri over prod} 
    Products distribute over sums from the left, in the sense that the following diagram commutes strictly:
     \[
\begin{tikzcd}
	{\mathbb{K}(A,\mathcal{H}_0) \times \mathbb{K}(B,\mathcal{H}_1) \times \mathbb{K}(B,\mathcal{H}_2)} && {\mathbb{K}(A,\mathcal{H}_0) \times \mathbb{K}(A,\mathcal{H}_0) \times\mathbb{K}(B,\mathcal{H}_1) \times \mathbb{K}(B,\mathcal{H}_2)} \\
	{\mathbb{K}(A,\mathcal{H}_0) \times \mathbb{K}(B,\mathcal{H}_1 \oplus \mathcal{H}_2)} && {\mathbb{K}(A,\mathcal{H}_0) \times \mathbb{K}(B,\mathcal{H}_1) \times\mathbb{K}(A,\mathcal{H}_0) \times \mathbb{K}(B,\mathcal{H}_2)} \\
	&& {\mathbb{K}(A \hat\otimes B,\mathcal{H}_0 \hat\otimes \mathcal{H}_1) \times \mathbb{K}(A \hat\otimes B,\mathcal{H}_0 \hat\otimes \mathcal{H}_2)} \\
	{\mathbb{K}(A \hat\otimes B,\mathcal{H}_0 \hat\otimes (\mathcal{H}_1 \oplus  \mathcal{H}_2))} && {\mathbb{K}(A \hat\otimes B,\mathcal{H}_0 \hat\otimes \mathcal{H}_1 \oplus \mathcal{H}_0 \hat\otimes \mathcal{H}_2))}
	\arrow["{1\times \operatorname{add}}"', from=1-1, to=2-1]
	\arrow["{\operatorname{flip} \mathbb{K}(A,\mathcal{H}_0)~\&~\mathbb{K}(B,\mathcal{H}_1)}", color={rgb,255:red,214;green,92;blue,214}, from=1-3, to=2-3]
	\arrow["{\operatorname{dia}\times 1 \times 1}", from=1-1, to=1-3]
	\arrow["{\operatorname{prod}}"', from=2-1, to=4-1]
	\arrow["{ \operatorname{prod} \times \operatorname{prod}}", color={rgb,255:red,214;green,92;blue,214}, from=2-3, to=3-3]
	\arrow["{\operatorname{add}}", color={rgb,255:red,214;green,92;blue,214}, from=3-3, to=4-3]
	\arrow["\cong"', from=4-1, to=4-3]
\end{tikzcd}
    \]
     where the bottom isomorphism is the one induced by the canonical unitary isomorphism $\h_0 \hat\otimes (\h_1 \oplus \h_2) \rightarrow \h_0 \hat\otimes \h_1 \oplus \h_0 \hat\otimes \h_2$.
\end{lem}
\begin{rem}
    Purple arrows have been used on the right-hand side of this diagram to use that portion in a later diagram; it is not really needed for the proof below.
\end{rem}
\begin{proof}
    (We only discuss right distributivity of sum over products; left distributivity will be completely analogous.)
   Unravelling the above space level diagram down to the $C^*$-algebra level, we have:
   \[
\begin{tikzcd}
	& {\mathcal{S}} \\
	{\mathcal{S} \hat\otimes(\mathcal{S} \oplus \mathcal{S})} && {\mathcal{S} \hat\otimes \mathcal{S} \oplus \mathcal{S} \hat\otimes \mathcal{S} } \\
	{A \hat\otimes\mathcal{K}(\mathcal{H}_0) \hat\otimes(B \hat\otimes\mathcal{K}(\mathcal{H}_1) \oplus B \hat\otimes\mathcal{K}(\mathcal{H}_2))} && {A \hat\otimes\mathcal{K}(\mathcal{H}_0) \hat\otimes B \hat\otimes\mathcal{K}(\mathcal{H}_1) \oplus A \hat\otimes\mathcal{K}(\mathcal{H}_0) \hat\otimes B \hat\otimes\mathcal{K}(\mathcal{H}_2))} \\
	{A \hat\otimes\mathcal{K}(\mathcal{H}_0) \hat\otimes(B \hat\otimes\mathcal{K}(\mathcal{H}_1 \oplus \mathcal{H}_2))} && {A \hat\otimes B \hat\otimes\mathcal{K}(\mathcal{H}_0) \hat\otimes\mathcal{K}(\mathcal{H}_1) \oplus A \hat\otimes B \hat\otimes\mathcal{K}(\mathcal{H}_0) \hat\otimes\mathcal{K}(\mathcal{H}_2))} \\
	{A \hat\otimes B \hat\otimes\mathcal{K}(\mathcal{H}_0)\hat\otimes\mathcal{K}(\mathcal{H}_1 \oplus \mathcal{H}_2)} && {A \hat\otimes B \hat\otimes\mathcal{K}(\mathcal{H}_0 \hat\otimes\mathcal{H}_1) \oplus A \hat\otimes B \hat\otimes\mathcal{K}(\mathcal{H}_0 \hat\otimes\mathcal{H}_2))} \\
	{A \hat\otimes B \hat\otimes\mathcal{K}(\mathcal{H}_0\hat\otimes(\mathcal{H}_1 \oplus \mathcal{H}_2))} && {A \hat\otimes B \hat\otimes\mathcal{K}(\mathcal{H}_0 \hat\otimes\mathcal{H}_1 \oplus \mathcal{H}_0 \hat\otimes\mathcal{H}_2)}
	\arrow["{\alpha\hat\otimes \beta_1 \oplus \alpha \hat\otimes\beta_2}", from=2-3, to=3-3]
	\arrow["{\alpha \hat\otimes(\beta_1 \oplus \beta_2)}"', from=2-1, to=3-1]
	\arrow["{1 \hat\otimes \operatorname{add}}"', from=3-1, to=4-1]
	\arrow["{(\operatorname{flip} B ~\&~\mathcal{K}(\mathcal{H}_0)) \oplus (\operatorname{flip} B ~\&~\mathcal{K}(\mathcal{H}_1))}", from=3-3, to=4-3]
	\arrow["{\operatorname{flip} B ~\&~\mathcal{K}(\mathcal{H}_0)}"', from=4-1, to=5-1]
	\arrow["{1 \hat\otimes 1 \hat\otimes \Psi}"', from=5-1, to=6-1]
	\arrow["{1 \hat\otimes1 \hat\otimes \Psi \oplus1\hat\otimes1\hat\otimes \Psi}", from=4-3, to=5-3]
	\arrow["{\operatorname{add}}", from=5-3, to=6-3]
	\arrow["\cong", from=6-1, to=6-3]
	\arrow[from=1-2, to=2-1]
	\arrow["{(\triangle,\triangle)}", from=1-2, to=2-3]
\end{tikzcd}
   \]
   This diagram does commute, as can be seen by considering homogeneous elementary tensors. Note that the map $\s \rightarrow \s \hat\otimes (\s \oplus \s)$ is the one making the the diagram
   \[
   % https://q.uiver.app/#q=WzAsMyxbMSwwLCJcXG1hdGhjYWx7U30iXSxbMCwxLCJcXG1hdGhjYWx7U30gXFxoYXRcXG90aW1lcyhcXG1hdGhjYWx7U30gXFxvcGx1cyBcXG1hdGhjYWx7U30pIl0sWzIsMSwiXFxtYXRoY2Fse1N9IFxcaGF0XFxvdGltZXMgXFxtYXRoY2Fse1N9IFxcb3BsdXMgXFxtYXRoY2Fse1N9IFxcaGF0XFxvdGltZXMgXFxtYXRoY2Fse1N9ICJdLFswLDFdLFswLDIsIihcXHRyaWFuZ2xlLFxcdHJpYW5nbGUpIl0sWzIsMSwiXFxjb25nIiwyXV0=
\begin{tikzcd}
	& {\mathcal{S}} \\
	{\mathcal{S} \hat\otimes(\mathcal{S} \oplus \mathcal{S})} && {\mathcal{S} \hat\otimes \mathcal{S} \oplus \mathcal{S} \hat\otimes \mathcal{S} }
	\arrow[from=1-2, to=2-1]
	\arrow["{(\triangle,\triangle)}", from=1-2, to=2-3]
	\arrow["\cong"', from=2-3, to=2-1]
\end{tikzcd}
   \]
   commutative, where the bottom isomorphism is inverse to the natural map $f \hat\otimes (g \oplus h) \mapsto f \hat\otimes g \oplus f \hat\otimes h$.
\end{proof}
\begin{prop}
    Products distribute over sums.
\end{prop}
\begin{proof}
    Having already shown left distributivity, right distributivity now follows from the commutativity of products. (It can directly be shown just like right distributivity too, as well).
\end{proof}
\subsection{More naturality}
We now see how the maps in \cref{defi: amplified homo} go with products. To that end, we have the following:
\begin{prop}\label{amp morph nat prod}
     The $K$-theory product is natural with respect the maps described in \cref{defi: amplified homo} in the following sense: if $\alpha: \s \hat\otimes B \rightarrow B' \hat\otimes \com(\h)$ is a $*$-homomorphism, and $\varphi: A \rightarrow A'$ is a $*$-homomorphism, then the following diagram commutes strictly:
     \[
    % https://q.uiver.app/#q=WzAsNCxbMCwwLCJcXG1hdGhiYntLfShBLFxcbWF0aGNhbHtIfV8wKVxcdGltZXMgXFxtYXRoYmJ7S30oQixcXG1hdGhjYWx7SH1fMSkiXSxbMiwwLCJcXG1hdGhiYntLfShBIFxcaGF0XFxvdGltZXMgQixcXG1hdGhjYWx7SH1fMCBcXGhhdFxcb3RpbWVzXFxtYXRoY2Fse0h9XzEpIl0sWzAsMSwiXFxtYXRoYmJ7S30oQScsXFxtYXRoY2Fse0h9XzApXFx0aW1lcyBcXG1hdGhiYntLfShCJyxcXG1hdGhjYWx7SH1fMVxcaGF0XFxvdGltZXNcXG1hdGhjYWx7SH0pIl0sWzIsMSwiXFxtYXRoYmJ7S30oQScgXFxoYXRcXG90aW1lcyBCJyxcXG1hdGhjYWx7SH1fMCBcXGhhdFxcb3RpbWVzXFxtYXRoY2Fse0h9XzEgXFxoYXRcXG90aW1lc1xcbWF0aGNhbHtIfSkiXSxbMCwxLCJLLVxcdGV4dHt0aGVvcnkgcHJvZHVjdH0iXSxbMiwzLCJLLVxcdGV4dHt0aGVvcnkgcHJvZHVjdH0iXSxbMCwyLCJcXHZhcnBoaV8qIFxcdGltZXMgXFxhbHBoYV8qIiwyXSxbMSwzLCIoXFx2YXJwaGkgXFxoYXRcXG90aW1lcyBcXGFscGhhKV8qIl1d
\begin{tikzcd}
	{\mathbb{K}(A,\mathcal{H}_0)\times \mathbb{K}(B,\mathcal{H}_1)} && {\mathbb{K}(A \hat\otimes B,\mathcal{H}_0 \hat\otimes\mathcal{H}_1)} \\
	{\mathbb{K}(A',\mathcal{H}_0)\times \mathbb{K}(B',\mathcal{H}_1\hat\otimes\mathcal{H})} && {\mathbb{K}(A' \hat\otimes B',\mathcal{H}_0 \hat\otimes\mathcal{H}_1 \hat\otimes\mathcal{H})}
	\arrow["{K-\text{theory product}}", from=1-1, to=1-3]
	\arrow["{K-\text{theory product}}", from=2-1, to=2-3]
	\arrow["{\varphi_* \times \alpha_*}"', from=1-1, to=2-1]
	\arrow["{(\varphi \hat\otimes \alpha)_*}", from=1-3, to=2-3]
\end{tikzcd}
     \]
\end{prop}
\begin{rem}
Note that $\alpha$ is as in \cref{defi: amplified homo} (i), whereas $\varphi \hat\otimes \alpha$ is as in \cref{defi: amplified homo} (ii).
\end{rem}
\begin{proof}
    A direct diagram chase, let $\Phi_A \in \K(A,\h_0)$, $\Phi_B \in \K(B,\h_1)$. Then, at the algebra level, we have to show that the following diagram commutes:
    \[
\begin{tikzcd}
	& {\mathcal{S}} \\
	{A\hat\otimes \mathcal{K}(\mathcal{H}_0) \hat\otimes \mathcal{S}\hat\otimes B \hat\otimes \mathcal{K}(\mathcal{H}_1)} && {\mathcal{S} \hat\otimes A\hat\otimes \mathcal{K}(\mathcal{H}_0) \hat\otimes B \hat\otimes \mathcal{K}(\mathcal{H}_1)} \\
	&& {\mathcal{S} \hat\otimes A \hat\otimes B\hat\otimes \mathcal{K}(\mathcal{H}_0 \hat\otimes \mathcal{H}_1)} \\
	{A'  \hat\otimes \mathcal{K}(\mathcal{H}_0) \hat\otimes B' \hat\otimes \mathcal{K}(\mathcal{H})  \hat\otimes \mathcal{K}(\mathcal{H}_1)} && {A \hat\otimes \mathcal{S} \hat\otimes B\hat\otimes \mathcal{K}(\mathcal{H}_0  \hat\otimes\mathcal{H}_1)} \\
	{A'  \hat\otimes \mathcal{K}(\mathcal{H}_0) \hat\otimes B' \hat\otimes \mathcal{K}(\mathcal{H}_1)  \hat\otimes \mathcal{K}(\mathcal{H})} \\
	{A' \hat\otimes B' \hat\otimes \mathcal{K}(\mathcal{H}_0\hat\otimes \mathcal{H}_1)  \hat\otimes \mathcal{K}(\mathcal{H})} && {A' \hat\otimes B' \hat\otimes \mathcal{K}(\mathcal{H})\hat\otimes \mathcal{K}(\mathcal{H}_0  \hat\otimes \mathcal{H}_1)} \\
	&& {}
	\arrow["{\Phi_A \hat\otimes1\hat\otimes \Phi_B}"', from=1-2, to=2-1]
	\arrow["{1 \hat\otimes\Phi_A\hat\otimes\Phi_B}", from=1-2, to=2-3]
	\arrow["{\text{flip}~B~\&~\mathcal{K}(\mathcal{H}_0) ~\text{(step of product)}}", from=2-3, to=3-3]
	\arrow["{\text{flip}~A~\&~\mathcal{S}}", from=3-3, to=4-3]
	\arrow["{\varphi \hat\otimes 1 \hat\otimes\alpha \hat\otimes1}", from=2-1, to=4-1]
	\arrow["{\operatorname{flip} \mathcal{K}(\mathcal{H}) ~\&~\mathcal{K}(\mathcal{H}_1)}", from=4-1, to=5-1]
	\arrow["{\operatorname{flip} B' ~\&~\mathcal{K}(\mathcal{H}_0) (~\text{step of product})}", from=5-1, to=6-1]
	\arrow["{\varphi \hat\otimes\alpha \hat\otimes 1}", from=4-3, to=6-3]
	\arrow["{\operatorname{flip} \mathcal{K}(\mathcal{H}) ~\&~ \mathcal{K}(\mathcal{H}_0 \hat\otimes \mathcal{H}_1) }", from=6-3, to=6-1]
\end{tikzcd}
    \]
     It does commute, as can be directly check with elementary homogeneous tensors.
\end{proof}

\subsection{H-space structures and products}
We now see how products interact with the H-space structures on the spectral K-theory spaces.

In order to systematically do the treatment, we make the following definition:

\begin{defi}\label{defi: bi H-map}
    Let $(X,\mu_x),(Y,\mu_Y),(Z,\mu_Z)$ be $H$-spaces, and $\varphi: X \times Y \rightarrow Z$ a pointed continuous map. We say that $\varphi$ is a \textit{bi H-map} if the following diagram commutes upto pointed homotopy:
    \[
    % https://q.uiver.app/#q=WzAsMTIsWzAsMCwiWCBcXHRpbWVzIFkgXFx0aW1lcyBZIl0sWzIsMCwiWCBcXHRpbWVzIFggXFx0aW1lcyBZIFxcdGltZXMgWSJdLFswLDEsIlggXFx0aW1lcyBZIl0sWzIsMSwiWCBcXHRpbWVzIFkgXFx0aW1lcyBYIFxcdGltZXMgWSJdLFsyLDIsIlogXFx0aW1lcyBaIl0sWzAsMiwiWiJdLFs0LDAsIlggXFx0aW1lcyBYIFxcdGltZXMgWSJdLFs2LDAsIlggXFx0aW1lcyBYIFxcdGltZXMgWSBcXHRpbWVzIFkiXSxbNiwxLCJYIFxcdGltZXMgWSBcXHRpbWVzIFggXFx0aW1lcyBZIl0sWzYsMiwiWiBcXHRpbWVzIFoiXSxbNCwxLCJYIFxcdGltZXMgWSJdLFs0LDIsIloiXSxbMCwyLCIxXFx0aW1lcyBcXG11X1kiLDJdLFsxLDMsIlxcb3BlcmF0b3JuYW1le2ZsaXB9IFh+XFwmflkiXSxbMCwxLCJcXG9wZXJhdG9ybmFtZXtkaWF9XFx0aW1lcyAxIFxcdGltZXMgMSJdLFszLDQsIlxcdmFycGhpIFxcdGltZXMgXFx2YXJwaGkgIl0sWzQsNSwiXFxtdV9aIl0sWzIsNSwiXFx2YXJwaGkiLDJdLFs2LDcsIjEgXFx0aW1lcyAxIFxcdGltZXMgXFxvcGVyYXRvcm5hbWV7ZGlhfSJdLFs3LDgsIlxcb3BlcmF0b3JuYW1le2ZsaXB9IFh+XFwmflkiXSxbOCw5LCJcXHZhcnBoaSBcXHRpbWVzIFxcdmFycGhpICJdLFs2LDEwLCJcXG11X1ggXFx0aW1lcyAxIiwyXSxbMTAsMTEsIlxcdmFycGhpIiwyXSxbOSwxMSwiXFxtdV9aIl1d
\begin{tikzcd}
	{X \times Y \times Y} && {X \times X \times Y \times Y} && {X \times X \times Y} && {X \times X \times Y \times Y} \\
	{X \times Y} && {X \times Y \times X \times Y} && {X \times Y} && {X \times Y \times X \times Y} \\
	Z && {Z \times Z} && Z && {Z \times Z}
	\arrow["{1\times \mu_Y}"', from=1-1, to=2-1]
	\arrow["{\operatorname{flip} X~\&~Y}", from=1-3, to=2-3]
	\arrow["{\operatorname{dia}\times 1 \times 1}", from=1-1, to=1-3]
	\arrow["{\varphi \times \varphi }", from=2-3, to=3-3]
	\arrow["{\mu_Z}", from=3-3, to=3-1]
	\arrow["\varphi"', from=2-1, to=3-1]
	\arrow["{1 \times 1 \times \operatorname{dia}}", from=1-5, to=1-7]
	\arrow["{\operatorname{flip} X~\&~Y}", from=1-7, to=2-7]
	\arrow["{\varphi \times \varphi }", from=2-7, to=3-7]
	\arrow["{\mu_X \times 1}"', from=1-5, to=2-5]
	\arrow["\varphi"', from=2-5, to=3-5]
	\arrow["{\mu_Z}", from=3-7, to=3-5]
\end{tikzcd}
    \]
\end{defi}
\begin{lem}\label{bi H map}
     Let $(X,\mu_x),(Y,\mu_Y),(Z,\mu_Z)$ be commutative $H$-monoids, and $\varphi: X \times Y \rightarrow Z$ be a bi-H map. Then, $\varphi$ induces a bilinear monoid map 
     \[
     \pi_0(X) \times \pi_0(Y) \rightarrow \pi_0(Z)
     \]
\end{lem}
\begin{proof}
    By the functoriality of $\pi_0$, it is clear that there is a set map $\pi_0(X) \times \pi_0(Y) \rightarrow \pi_0(Z)$; the point is to check its bilinearity.
    
    We first fix $[x] \in \pi_0(X)$, and show that the map 
    \begin{align*}
        \varphi_{x_*}:\pi_0(Y) \rightarrow \pi_0(Z) \\
        [y] \mapsto [\varphi(x,y)]
    \end{align*}
    is linear. Indeed, on the one hand, we have \begin{align*}
       &[y_0] + [y_1] = [\mu_Y(y_0,y_1)] \\
       \implies& \varphi_{x_*}([y_0]+[y_1])=[\varphi(x,\mu_Y(y_0,y_1))] .
    \end{align*}
    On the other hand,
    \begin{align*}
        \varphi_{x_*}([y_0])+\varphi_{x_*}([y_1])=[\varphi(x,y_0)]+[\varphi(x,y_1)]=[\mu_Z(\varphi(x,y_0),\varphi(x,y_1))]
    \end{align*}
    The rest follows from the homotopy commutativity of the left diagram in \cref{defi: bi H-map}.

    Similarly, linearity upon fixing an element of $\pi_0(Y)$ can be shown by considering the right diagram in \cref{defi: bi H-map}.
\end{proof}
\begin{rem}
    We caution the reader that the term ``bi H-map" is not standard terminology. The author was unable to find a better word to describe the required properties.
\end{rem}
\begin{prop}\label{prod bilinear}
    Let $\h_0$,$\h_1$ be graded Hilbert spaces whose and even and odd parts are countably infinite dimensional. Fix H-space structures $\mu_{u_0},\mu_{u_1},\mu_u$ as prescribed in \cref{defi: H space} on $\K(A,\h_0),\K(B,\h_1)$ and $\K(A \hat\otimes B,\h_0 \hat\otimes \h_1)$. \footnote{equivalently, fix even unitaries $u_0:\h_0 \oplus \h_0 \rightarrow \h_0$,$u_1:\h_1 \oplus \h_1 \rightarrow \h_1$,$u:\h_0\hat\otimes\h_1 \oplus \h_0\hat\otimes\h_1 \rightarrow \h_0\hat\otimes\h_1$.}
    Then, the map described in \cref{prop: product} is a bi H-map.
\end{prop}
\begin{proof}
    Consider the following diagram, where the purple arrow is the composite of the purple arrows in \cref{prop: sum distri over prod} :
        \[
        % https://q.uiver.app/#q=WzAsNyxbMCwwLCJcXG1hdGhiYntLfShBLFxcbWF0aGNhbHtIfV8wKSBcXHRpbWVzIFxcbWF0aGJie0t9KEIsXFxtYXRoY2Fse0h9XzEpIFxcdGltZXMgXFxtYXRoYmJ7S30oQixcXG1hdGhjYWx7SH1fMSkiXSxbMiwwLCJcXG1hdGhiYntLfShBLFxcbWF0aGNhbHtIfV8wKV4yIFxcdGltZXNcXG1hdGhiYntLfShCLFxcbWF0aGNhbHtIfV8xKV4yIl0sWzAsMiwiXFxtYXRoYmJ7S30oQSxcXG1hdGhjYWx7SH1fMCkgXFx0aW1lcyBcXG1hdGhiYntLfShCLFxcbWF0aGNhbHtIfV8xXFxvcGx1cyBcXG1hdGhjYWx7SH1fMSkiXSxbMSwyLCJcXG1hdGhiYntLfShBIFxcaGF0XFxvdGltZXMgQixcXG1hdGhjYWx7SH1fMCBcXGhhdFxcb3RpbWVzIChcXG1hdGhjYWx7SH1fMVxcb3BsdXMgIFxcbWF0aGNhbHtIfV8xKSkiXSxbMiwzLCJcXG1hdGhiYntLfShBIFxcaGF0XFxvdGltZXMgQixcXG1hdGhjYWx7SH1fMCBcXGhhdFxcb3RpbWVzIFxcbWF0aGNhbHtIfV8xIFxcb3BsdXMgXFxtYXRoY2Fse0h9XzAgXFxoYXRcXG90aW1lcyBcXG1hdGhjYWx7SH1fMSkpIl0sWzAsMywiXFxtYXRoYmJ7S30oQSxcXG1hdGhjYWx7SH1fMCkgXFx0aW1lc1xcbWF0aGJie0t9KEIsXFxtYXRoY2Fse0h9XzEpIl0sWzEsMywiXFxtYXRoYmJ7S30oQSBcXGhhdFxcb3RpbWVzIEIsXFxtYXRoY2Fse0h9XzAgXFxoYXRcXG90aW1lcyBcXG1hdGhjYWx7SH1fMSkiXSxbMCwyLCIxXFx0aW1lcyBcXG9wZXJhdG9ybmFtZXthZGR9IiwyXSxbMiwzLCJcXG9wZXJhdG9ybmFtZXtwcm9kfSJdLFszLDQsIlxcY29uZyIsMl0sWzIsNSwiMSBcXHRpbWVzKFxcb3BlcmF0b3JuYW1le0FkfXVfMSlfKiIsMix7ImNvbG91ciI6WzAsNjAsNjBdfSxbMCw2MCw2MCwxXV0sWzMsNiwiXFxvcGVyYXRvcm5hbWV7QWR9KDFcXGhhdFxcb3RpbWVzIHVfMSlfKiIsMix7ImNvbG91ciI6WzAsNjAsNjBdfSxbMCw2MCw2MCwxXV0sWzUsNiwiXFxvcGVyYXRvcm5hbWV7cHJvZH0iLDIseyJjb2xvdXIiOlswLDYwLDYwXX0sWzAsNjAsNjAsMV1dLFs0LDYsIjFcXGhhdFxcb3RpbWVzKFxcb3BlcmF0b3JuYW1le0FkfXUpXyoiLDAseyJjb2xvdXIiOlsyNDAsNjAsNjBdfSxbMjQwLDYwLDYwLDFdXSxbMCwxLCJcXG9wZXJhdG9ybmFtZXtkaWF9XFx0aW1lcyAxIFxcdGltZXMgMSJdLFsxLDQsIiIsMCx7ImNvbG91ciI6WzMwMCw2MCw2MF19XV0=
\begin{tikzcd}
	{\mathbb{K}(A,\mathcal{H}_0) \times \mathbb{K}(B,\mathcal{H}_1) \times \mathbb{K}(B,\mathcal{H}_1)} && {\mathbb{K}(A,\mathcal{H}_0)^2 \times\mathbb{K}(B,\mathcal{H}_1)^2} \\
	\\
	{\mathbb{K}(A,\mathcal{H}_0) \times \mathbb{K}(B,\mathcal{H}_1\oplus \mathcal{H}_1)} & {\mathbb{K}(A \hat\otimes B,\mathcal{H}_0 \hat\otimes (\mathcal{H}_1\oplus  \mathcal{H}_1))} \\
	{\mathbb{K}(A,\mathcal{H}_0) \times\mathbb{K}(B,\mathcal{H}_1)} & {\mathbb{K}(A \hat\otimes B,\mathcal{H}_0 \hat\otimes \mathcal{H}_1)} & {\mathbb{K}(A \hat\otimes B,\mathcal{H}_0 \hat\otimes \mathcal{H}_1 \oplus \mathcal{H}_0 \hat\otimes \mathcal{H}_1))}
	\arrow["{1\times \operatorname{add}}"', from=1-1, to=3-1]
	\arrow["{\operatorname{prod}}", from=3-1, to=3-2]
	\arrow["\cong"', from=3-2, to=4-3]
	\arrow["{1 \times(\operatorname{Ad}u_1)_*}"', color={rgb,255:red,214;green,92;blue,92}, from=3-1, to=4-1]
	\arrow["{\operatorname{Ad}(1\hat\otimes u_1)_*}"', color={rgb,255:red,214;green,92;blue,92}, from=3-2, to=4-2]
	\arrow["{\operatorname{prod}}"', color={rgb,255:red,214;green,92;blue,92}, from=4-1, to=4-2]
	\arrow["{1\hat\otimes(\operatorname{Ad}u)_*}", color={rgb,255:red,92;green,92;blue,214}, from=4-3, to=4-2]
	\arrow["{\operatorname{dia}\times 1 \times 1}", from=1-1, to=1-3]
	\arrow[color={rgb,255:red,214;green,92;blue,214}, from=1-3, to=4-3]
\end{tikzcd}
        \]

    Our aim is to show that the outer diagram commutes upto homotopy.

    Note that the diagram spanned by the black arrows and the purple arrow commutes strictly by \cref{prop: sum distri over prod}. The lower left rectangle commutes strictly because of the naturality of products as mentioned in \cref{cor: hom preserve prod}. The blue-black-orangle triangle commutes upto homotopy as $\mathcal{U}^{\operatorname{ev}}(\h_0 \hat\otimes (\h_1 \oplus \h_1),\h_0 \hat\otimes \h_1 \oplus \h_0 \hat\otimes \h_1)$ is path connected in the strong-$*$ topology.
    So overall, the outer diagram commutes upto homotopy.
\end{proof}
\section{Spectral K-theory groups}
We now give the definition of $K$-groups in the spectral model. Having developed most of the machinery at the space level, most of the properties of these $K$-groups will be mere corollaries.

\begin{defi} \label{defi: spectral K-theory space}
    Given a $\Z/2\Z$-graded (Real) $C^*$-algebra $A$, we define \textit{the spectral K-theory space} (without mention of any given underlying Hilbert space) of $A$ to be $\K(A):=\K(A,\hat{l^2})$. 
\end{defi}
\begin{defi}\label{defi: spectral k-theory}
    Given a $\Z/2\Z$-graded (Real) $C^*$-algebra $A$, we define its spectral K-theory groups to be the homotopy groups of $\K(A)$ at the basepoint \underline{$0$}. That is, we define ${K_i}^{sp}(A):= \pi_i(\K(A),\underline{0})$.
\end{defi}

Before proceeding further, we note that for any graded Hilbert space $\h$ whose even and odd parts are countably infinite dimensional, there exists a preferred homotopy class of maps $\K(A,\h) \rightarrow \K(A)$ for each $\Z/2\Z$-graded $C^*$-algebra $A$, given by conjugation with any even unitary from $\h$ to $\hat{l^2}$. Any unlabelled map $\K(A,\h) \rightarrow \K(A)$ would be assumed to be (the homotopy class of) this. 

As an immediate corollary to our section on H-spaces, we have the following definition:
\begin{defi}
    By \textit{the} H-space structure on $\K(A)$, we mean the homotopy class of the following composite:
    \[
    \K(A) \times \K(A) \xrightarrow{s} \K(A,\hat{l^2} \oplus \hat{l^2}) \rightarrow \K(A)
    \]
\end{defi}
The following is an immediate corollary to \cref{H-monoid}.
\begin{cor}
     $K_0^{\operatorname{sp}}(A)$ is an abelian monoid. $K_1^{\operatorname{sp}}(A)$ is an abelian group, and so are the higher spectral $K$-groups by homotopy theory.
\end{cor}
\begin{rem}
    Note that as of now, we don't know yet that the spectral $K_0$-group is an abelian group. However, we phrase the subsequent results of this section treating them as abelian groups; it doesn't create any problem as abelian group homomorphisms are homomorphisms of the underlying abelian monoid, and constructions like tensor products are done on the underlying abelian monoid. We prove that the spectral $K_0$ groups are abelian groups in \cref{thm: sp=fr}.
\end{rem}
Basic homotopy theory immediately yields the following corollary to \cref{prop: baby functoriality}.
\begin{cor} \label{cor: baby functoriality}
    For each $i \geq 0$, $K_i^{\operatorname{sp}}(-)$ is a functor from the category of $\Z/2\Z$-graded $C^*$-algebras to the category of abelian groups which satisfies the following properties:
    \begin{enumerate}
        \item Homotopic $*$-homomorphisms $\varphi,\psi: A \rightarrow B$ induce the same  map $K_i^{\operatorname{sp}}(\varphi) = K_i^{\operatorname{sp}}(\psi): K_i^{\operatorname{sp}}(A) \rightarrow K_i^{\operatorname{sp}}(B)$.
        \item Any isometry $\iota\in \hat{l^2}$  induces an isomorphism $K_i^{\operatorname{sp}}(\iota):K_i^{\operatorname{sp}}(A) \rightarrow K_i^{\operatorname{sp}}(A)$.
        \item $0:A \rightarrow B$ induces the zero homomorphism $K_i^{\operatorname{sp}}(A) \rightarrow K_i^{\operatorname{sp}}(B)$.
    \end{enumerate} 
\end{cor}
We have the following two corollaries to the results established in section 3. The proofs follow from the homotopy invariance of homotopy groups, and from the long exact sequence of homotopy groups associated to a Serre fibration (see \cite{HAT} Theorem 4.41).
\newpage
\begin{cor} \label{stab and les group}
    \,
    \begin{enumerate}
        \item Let $A$ be a graded $C^*$-algebra, and $\mathcal{H}$ a graded Hilbert space. Then, the map $\Phi:A \rightarrow A \hat\otimes \mathcal{K}(\mathcal{H})$, given by $\Phi(a)=a\hat\otimes e_{11}$, where $e_{11}$ is projection onto an even one-dimensional subspace of $\mathcal{H}$, induces an isomorphism $K_i^{\operatorname{sp}}(\Phi):K_i^{\operatorname{sp}}(A) \rightarrow K_i^{\operatorname{sp}}(A \hat\otimes \mathcal{K}(\mathcal{H}))~\forall~ i \geq 0$.
        \item Let $A,B$ be graded $C^*$-algebras, and $\varphi:A \rightarrow B$ be a surjective graded $*$-homomorphism. Let $J:=\ker \varphi$. Then, we have a long exact sequence as follows:

\[
% https://q.uiver.app/#q=WzAsNixbMiwwLCJLX3tuKzF9XntzcH0oQikiXSxbMywwLCJLX25ee3NwfShKKSJdLFs0LDAsIktfbl57c3B9KEIpIl0sWzEsMCwiS197bisxfV57c3B9KEEpIl0sWzAsMCwiXFxjZG90cyJdLFs1LDAsIlxcY2RvdHMiXSxbMywwLCJcXHZhcnBoaV8qIl0sWzAsMV0sWzEsMiwiXFxvcGVyYXRvcm5hbWV7aW5jfV8qIl0sWzQsM10sWzIsNV1d
\begin{tikzcd}
	\cdots & {K_{n+1}^{sp}(A)} & {K_{n+1}^{sp}(B)} & {K_n^{sp}(J)} & {K_n^{sp}(B)} & \cdots
	\arrow["{\varphi_*}", from=1-2, to=1-3]
	\arrow[from=1-3, to=1-4]
	\arrow["{\operatorname{inc}_*}", from=1-4, to=1-5]
	\arrow[from=1-1, to=1-2]
	\arrow[from=1-5, to=1-6]
\end{tikzcd}\]
    The sequence terminates at $K_0^{\operatorname{sp}}(B)$.
    \end{enumerate}
\end{cor}

Now, we cast \cref{defi: amplified homo} at the level of K-groups:

\begin{defi}\label{amplified morphisms group}
    \,
    \begin{enumerate}
        \item Let $\kappa:\s \hat\otimes A \rightarrow B\hat\otimes \com(\h)$ be a graded $*$-homomorphism. Then, it induces a well-defined group homomorphism $\kappa_*:K_0^{\operatorname{sp}}(A) \rightarrow K_0^{\operatorname{sp}}(B)$, given by applying the $\pi_0$ functor to the  following composite:
        \[
        \K(A) \xrightarrow{(\kappa,1)_*} \K(B,\hat{l^2} \hat\otimes \h) \rightarrow \K(B).
        \]
        Note that $\kappa_*$ is indeed a monoid homomorphism in light of \cref{sum preserve H-space} (ii) and (i). Moreover, homotopic $*$-homomorphisms $\kappa,\kappa':\s \hat\otimes A \rightarrow B \hat\otimes \com(\h)$ yield the same map $\kappa_*=\kappa_*':K_0^{\operatorname{sp}}(A) \rightarrow K_0^{\operatorname{sp}}(B)$.
        \item Let $\kappa:A \hat\otimes \s \hat\otimes B \rightarrow C \hat\otimes \com(\h)$ be a graded $*$-homomorphism. Then, it induces a map $\kappa_*:K_0^{\operatorname{sp}}(A \hat\otimes B) \rightarrow K_0^{\operatorname{sp}}(C)$, given by applying the $\pi_0$ functor to the  following composite:
        \[
        \K(A \hat\otimes B) \xrightarrow{(\kappa,1)_*} \K(C,\hat{l^2} \hat\otimes \h) \rightarrow \K(C).
        \] 
        Note that $\kappa_*$ is indeed a monoid homomorphism in light of \cref{sum preserve H-space} (iii) and (i). Moreover, homotopic $*$-homomorphisms $\kappa,\kappa':A \hat\otimes \s \hat\otimes B \rightarrow C \hat\otimes \com(\h)$ yield the same map $\kappa_*=\kappa_*':K_0^{\operatorname{sp}}(A) \rightarrow K_0^{\operatorname{sp}}(B)$.
    \end{enumerate}
\end{defi}
\subsection{Products at the group level}
We can now form products at the ``group level". We will be brisk from now on, as all the hard work has been developed at the space level before.
\begin{defi}\label{Prod group level}
    Let $A$ and $B$ be $\Z/2\Z$-graded Real $C^*$-algebras. We have a well-defined homotopy class given by the composition 
    \[
    \K(A) \times \K(B) \xrightarrow{p} \K(A \hat\otimes B, \hat{l^2} \hat\otimes \hat{l^2}) \rightarrow \K(A \hat\otimes B),
    \]
    which is natural with respect to $*$-homomorphisms $A \rightarrow A'$ and $B \rightarrow B'$. It descends to a well-defined bilinear pairing (the bilinearity follows from \cref{prod bilinear},\cref{bi H map} and \cref{sum preserve H-space} (i)): 
    \[
    K_0^{\operatorname{sp}}(A) \times K_0^{\operatorname{sp}}(B) \rightarrow K_0^{\operatorname{sp}}(A \hat\otimes B)
    ,\]
    which we call as the K-theory product (at the group level).
\end{defi}
\begin{rem} \label{rem: prod}
        Note that the naturality assertion in \cref{Prod group level} is saying that the following diagram commutes, where $\varphi:A \rightarrow A'$ and $\psi: B \rightarrow B'$ are $*$-homomorphisms:
        \[
       % https://q.uiver.app/#q=WzAsNCxbMCwwLCJLXzBee1xcb3BlcmF0b3JuYW1le3NwfX0oQSlcXG90aW1lcyBLXzBee1xcb3BlcmF0b3JuYW1le3NwfX0oQikiXSxbMCwxLCJLXzBee1xcb3BlcmF0b3JuYW1le3NwfX0oQScpXFxvdGltZXMgS18wXntcXG9wZXJhdG9ybmFtZXtzcH19KEInKSJdLFsyLDAsIktfMF57XFxvcGVyYXRvcm5hbWV7c3B9fShBIFxcaGF0XFxvdGltZXMgQikiXSxbMiwxLCJLXzBee1xcb3BlcmF0b3JuYW1le3NwfX0oQScgXFxoYXRcXG90aW1lcyBCJykiXSxbMCwxLCJcXHZhcnBoaV8qXFxoYXRcXG90aW1lc1xccHNpXyoiLDJdLFswLDIsIkstXFx0ZXh0e3RoZW9yeSBwcm9kdWN0fSIsMl0sWzEsMywiSy1cXHRleHR7dGhlb3J5IHByb2R1Y3R9Il0sWzIsMywiKFxcdmFycGhpIFxcaGF0XFxvdGltZXMgXFxwc2kpXyoiXV0=
\begin{tikzcd}
	{K_0^{\operatorname{sp}}(A)\otimes K_0^{\operatorname{sp}}(B)} && {K_0^{\operatorname{sp}}(A \hat\otimes B)} \\
	{K_0^{\operatorname{sp}}(A')\otimes K_0^{\operatorname{sp}}(B')} && {K_0^{\operatorname{sp}}(A' \hat\otimes B')}
	\arrow["{\varphi_*\hat\otimes\psi_*}"', from=1-1, to=2-1]
	\arrow["{K-\text{theory product}}"', from=1-1, to=1-3]
	\arrow["{K-\text{theory product}}", from=2-1, to=2-3]
	\arrow["{(\varphi \hat\otimes \psi)_*}", from=1-3, to=2-3]
\end{tikzcd}
        \]
    That the diagram commutes is a consequence of \cref{cor: hom preserve prod} and \cref{Prod group level}. Moreover, all the maps in the diagram are group homomorphisms: K-theory product has already been seen to be one, and the justfication for the maps $\varphi_*,\psi_*$ and $(\varphi \hat\otimes \psi)_*$ comes from \cref{sum preserve H-space} (i).
    \end{rem}
    We also have the following naturality statement:
    \begin{cor} \label{cor: boy and products}
     The $K$-theory product is natural with respect the maps described in \cref{amplified morphisms group} in the following sense: if $\alpha: \s \hat\otimes B \rightarrow B' \hat\otimes \com(\h)$ is a $*$-homomorphism, and $\varphi: A \rightarrow A'$ is a $*$-homomorphism, then the following diagram commutes:
     \[
    % https://q.uiver.app/#q=WzAsNCxbMCwwLCJLXzBee1xcb3BlcmF0b3JuYW1le3NwfX0oQSlcXG90aW1lcyBLXzBee1xcb3BlcmF0b3JuYW1le3NwfX0oQikiXSxbMiwwLCJLX3swfV57XFxvcGVyYXRvcm5hbWV7c3B9fShBIFxcaGF0XFxvdGltZXMgQikiXSxbMCwxLCJLXzBee1xcb3BlcmF0b3JuYW1le3NwfX0oQScpXFxvdGltZXMgS18wXntcXG9wZXJhdG9ybmFtZXtzcH19KEInKSJdLFsyLDEsIktfezB9XntcXG9wZXJhdG9ybmFtZXtzcH19KEEnIFxcaGF0XFxvdGltZXMgQicpIl0sWzAsMSwiSy1cXHRleHR7dGhlb3J5IHByb2R1Y3R9Il0sWzIsMywiSy1cXHRleHR7dGhlb3J5IHByb2R1Y3R9Il0sWzAsMiwiXFx2YXJwaGlfKiBcXGhhdFxcb3RpbWVzIFxcYWxwaGFfKiIsMl0sWzEsMywiKFxcdmFycGhpIFxcaGF0XFxvdGltZXMgXFxhbHBoYSlfKiJdXQ==
\begin{tikzcd}
	{K_0^{\operatorname{sp}}(A)\otimes K_0^{\operatorname{sp}}(B)} && {K_{0}^{\operatorname{sp}}(A \hat\otimes B)} \\
	{K_0^{\operatorname{sp}}(A')\otimes K_0^{\operatorname{sp}}(B')} && {K_{0}^{\operatorname{sp}}(A' \hat\otimes B')}
	\arrow["{K-\text{theory product}}", from=1-1, to=1-3]
	\arrow["{K-\text{theory product}}", from=2-1, to=2-3]
	\arrow["{\varphi_* \hat\otimes \alpha_*}"', from=1-1, to=2-1]
	\arrow["{(\varphi \hat\otimes \alpha)_*}", from=1-3, to=2-3]
\end{tikzcd}
     \]
\end{cor}
\begin{proof}
    The diagram commutes as a consequence of \cref{amp morph nat prod} and \cref{Prod group level}, and since their a well defined homotopy class of isomorphisms between the Hilbert spaces appearing in \cref{amp morph nat prod} and $\hat{l^2}$. Moreover, all the maps in the diagram are group homomorphisms: K-theory product has already seen to be one, and the justfication for the maps $\varphi_*,\alpha_*$ and $(\varphi \hat\otimes \alpha)_*$ comes from \cref{sum preserve H-space} (i), (ii) and (iii).  
\end{proof}
We end this subsection by compiling some properties of K-theory product, which are a reflection of the properties of products already discussed at the space level:
\begin{prop} \label{prop: prod prop}
The $K$-theory product has the following properties:
\begin{enumerate}[(a)]
	\item It is associative, in the sense that if $x \in K_0^{\operatorname{sp}}(A)$, $y
		\in K_0^{\operatorname{sp}}(B)$, $z
		\in K_0^{\operatorname{sp}}(C)$, then $x \times (y \times z) = (x \times y) \times z$.
	\item It is commutative, in the sense that if $x \in K_0^{\operatorname{sp}}(A)$ and $y
		\in K_0^{\operatorname{sp}}(B)$, and if $\tau : A \hat{\otimes} B \to B
		\hat{\otimes} A$ is the transposition isomorphism, then
		$\tau_{*}(x \times y) = y \times x$.
	\item It is functorial, in the sense that if $\varphi : A \to  A'$ 
		and $\psi : B \to B'$ are graded $*$-homomorphisms then
		$(\varphi \hat{\otimes}  \psi)_{*}(x \times y) =
		\varphi_{*}(x) \times \psi_{*}(y)$.
\end{enumerate}
\end{prop}

\subsection{Asymptotic morphisms and more products at the group level}
In this section, we discuss the notion of an asymptotic morphism, and show that they too, just like $*$-homomorphisms, induce maps on spectral $K$-theory groups. However, unlike $*$-homomorphisms, there is no natural space level constructions inducing them.
\begin{defi}
	Let $A$ and $ B$ be graded $C^{*}$-algebras. An asymptotic morphism
from $A$ to $B$ is a family of functions $\varphi_t : A \to B$, $t \in
[1, \infty)$ satisfying the continuity condition that for all $a \in
A$
\[
t \longmapsto \varphi_t(a) : [1, \infty) \to B \text{ is bounded and
continuous}
\]
and the asymptotic conditions that for all $a, a_1, a_2 \in A$ and
$\lambda \in \mathbb{C}$ 
\[
\begin{rcases}
	\varphi_t(a_1a_2) -  \varphi_t(a_1)\varphi_t(a_2)\\
	\varphi_t(a_1+a_2)- \varphi_t(a_1) - \varphi_t(a_2)\\
	\varphi_t(\lambda a ) - \lambda \varphi_t(a)\\
	\varphi_t(a^{*}) - \varphi_t(a)^{*}
\end{rcases}
\to 0, \qquad \text{as }t \to \infty
\]
Since $A $ and $B$ are graded, we shall require that in addition
\[
\alpha(\varphi_t(a)) - \varphi_t(\alpha(a)) \to 0 \qquad \text{as }t
\to \infty,
\]
where $\alpha$ denotes the grading automorphism. We shall denote an
asymptotic morphism with  a dashed arrow, thus $\varphi : A
\dashrightarrow B$.
\end{defi}
Although the definition of an asymptotic morphism didn't necessarily ``involve" any $*$-homomorphisms, associated to every  $C^*$-algebra $B$, there is a $C^*$-algebra $\mathfrak{A}B$ such that the following result holds:
\begin{prop} \label{prop: asym}
    For $C^*$-algebras $A$ and $B$, there is a natural bijection between the set of equivalence classes of asymtotic morphisms $A \dashrightarrow B$ and $\operatorname{Hom}(A,\mathfrak{A}B)$. Furthermore, homotopy classes of asymtototic morphisms from $A$ to $B$ correspond to homotopy classes of $*$-homomorphisms from $A$ to $\mathfrak{A}B$. 
\end{prop}
\begin{defi}
Let $A$ be a $C^{*}$-algebra and let $T = [1, \infty).$ Then define
the asymptotic algebra of $A$, $\mathfrak{A}A$, by:
\[
\mathfrak{A}A := C_b(T, A) / C_0(T, A),
\]
where $C_b(T, A)$ denotes bounded continuous functions from $T$ to $A$
and $C_0(T, A)$ denotes those functions of $C_b(T, A)$ which vanish in
norm at infinity. Alternatively,
\[
	\mathfrak{A}A = \set{[f] \mid f : [1, \infty) \to A \mid f \text{
	is bounded and continuous}},
\]
and $[f]$ denotes the equivalence class of functions, where $f \sim g$
if 
\[
\lim_{t \to \infty} \norm{f(t) - g(t)} = 0.
\]
The algebra structure here is defined pointwise as expected.
\end{defi}
We don't prove \cref{prop: asym} here; a detailed proof can be found in Theorem 1.2.5 of \cite{SLB}. We just note what the natural bijection is going to be: given $\varphi: A \dashrightarrow B$, we map it to the composite
\[
A \xrightarrow{\Phi} C_b[1,\infty) \xrightarrow{\operatorname{q}} \mathfrak{A}B
,\] where $\Phi(a)(t):=\varphi_t(a)$, and $q$ is the quotient map.

We now show that asymptotic morphisms induce maps on the spectral $K$-theory groups. 
\begin{prop} \label{prop: man functoriality groups}
An asymptotic morphism $\varphi_t : A \to B$ determines a $K$-theory
map $\varphi_{*}: K_0^{\operatorname{sp}}(A) \to K_0^{\operatorname{sp}}(B)$, with the following properties:
\begin{enumerate}[(i)]
	\item The correspondence $\varphi \mapsto  \varphi_{*}$ is
		functorial with respect to composition with $*$-homomorphism
		$A_1 \to A$ and $B \to B_1$.
	\item The map $\varphi_{*}$ depends only on the homotopy class of
		$\varphi$.
	\item If each $\varphi_t$ is actually a $*$-homomorphism, then
		$\varphi_{*}:K_0^{\operatorname{sp}}(A) \to K_0^{\operatorname{sp}}(B)$ is the map induced by $\varphi_1$.
\end{enumerate}
\end{prop}
\begin{proof}
    \begin{enumerate}[(i)]
        \item We have a natural short exact sequence of graded $C^*$-algebras
    \begin{equation}\label{nat ses}
        0 \rightarrow C_0(T,B) \rightarrow C_b(T,B) \xrightarrow{q} \mathfrak{A}B \rightarrow 0,
    \end{equation}
    
    and $C_0(T,B)$ is contractible. By \cref{stab and les group} (ii), we have natural isomorphisms \[K_0^{\operatorname{sp}}(C_b(T,B)) \xrightarrow[\cong]{q_*} K_0^{\operatorname{sp}}(\mathfrak{A}B)\] By \cref{prop: asym}, we have a map $K_0^{\operatorname{sp}}(A) \xrightarrow{\tilde{\varphi}_*} K_0^{\operatorname{sp}}(\mathfrak{A}B)$. Our sought after natural map $K_0^{\operatorname{sp}}(A) \rightarrow K_0^{\operatorname{sp}}(B)$ is the one which makes the following diagram commute:
    \[
    % https://q.uiver.app/#q=WzAsNCxbMCwwLCJLXzBee1xcb3BlcmF0b3JuYW1le3NwfX0oQSkiXSxbMSwwLCJLXzBee1xcb3BlcmF0b3JuYW1le3NwfX0oQikiXSxbMCwxLCJLXzBee1xcb3BlcmF0b3JuYW1le3NwfX0oXFxtYXRoZnJha3tBfUIpIl0sWzEsMSwiS18wXntcXG9wZXJhdG9ybmFtZXtzcH19KENfYihULEIpKSJdLFswLDIsIlxcdGlsZGV7XFx2YXJwaGl9XyoiXSxbMywyLCJxXyoiLDJdLFszLDEsIntcXG9wZXJhdG9ybmFtZXtldn1fMX1fKiIsMl0sWzAsMSwiXFx2YXJwaGlfKiIsMCx7InN0eWxlIjp7ImJvZHkiOnsibmFtZSI6ImRhc2hlZCJ9fX1dXQ==
\begin{tikzcd}
	{K_0^{\operatorname{sp}}(A)} & {K_0^{\operatorname{sp}}(B)} \\
	{K_0^{\operatorname{sp}}(\mathfrak{A}B)} & {K_0^{\operatorname{sp}}(C_b(T,B))}
	\arrow["{\tilde{\varphi}_*}", from=1-1, to=2-1]
	\arrow["{q_*}^{-1}"', from=2-1, to=2-2]
	\arrow["{{\operatorname{ev}_1}_*}"', from=2-2, to=1-2]
	\arrow["{\varphi_*}", dashed, from=1-1, to=1-2]
\end{tikzcd}
    \]
    \item This follows from the second part of \cref{prop: asym} and the homotopy invariance mentioned in \cref{amplified morphisms group} (i).
    \item As all $\varphi_t$ are $*$-homomorphisms, the map $\tilde{\varphi_*}$ factors through \[
    K_0^{\operatorname{sp}}(A) \rightarrow K_0^{\operatorname{sp}}(C_b(T,B))\xrightarrow{q_*}K_0^{\operatorname{sp}}(\mathfrak{A}B),
    \] and the rest follows from the definition of the map.
    \end{enumerate}
\end{proof}
\begin{cor}
    The $K$-theory product is natural with respect the maps described in \cref{prop: man functoriality groups} in the following sense: if $\psi_t: B \dashrightarrow B'$ is an asymptotic morphism, and $\varphi: A \rightarrow A'$ is a $*$-homomorphism, then the following diagram commutes:
    \[
    % https://q.uiver.app/#q=WzAsNCxbMCwwLCJLXzBee1xcb3BlcmF0b3JuYW1le3NwfX0oQSlcXG90aW1lcyBLXzBee1xcb3BlcmF0b3JuYW1le3NwfX0oQikiXSxbMCwxLCJLXzBee1xcb3BlcmF0b3JuYW1le3NwfX0oQScpXFxvdGltZXMgS18wXntcXG9wZXJhdG9ybmFtZXtzcH19KEInKSJdLFsyLDAsIktfezB9XntcXG9wZXJhdG9ybmFtZXtzcH19KEEgXFxoYXRcXG90aW1lcyBCKSJdLFsyLDEsIktfezB9XntcXG9wZXJhdG9ybmFtZXtzcH19KEEnIFxcaGF0XFxvdGltZXMgQicpIl0sWzAsMSwiXFx2YXJwaGlfKlxcaGF0XFxvdGltZXNcXHBzaV8qIiwyXSxbMCwyLCJLLVxcdGV4dHt0aGVvcnkgcHJvZHVjdH0iLDJdLFsxLDMsIkstXFx0ZXh0e3RoZW9yeSBwcm9kdWN0fSJdLFsyLDMsIihcXHZhcnBoaSBcXGhhdFxcb3RpbWVzIFxccHNpKV8qIl1d
\begin{tikzcd}
	{K_0^{\operatorname{sp}}(A)\otimes K_0^{\operatorname{sp}}(B)} && {K_{0}^{\operatorname{sp}}(A \hat\otimes B)} \\
	{K_0^{\operatorname{sp}}(A')\otimes K_0^{\operatorname{sp}}(B')} && {K_{0}^{\operatorname{sp}}(A' \hat\otimes B')}
	\arrow["{\varphi_*\hat\otimes\psi_*}"', from=1-1, to=2-1]
	\arrow["{K-\text{theory product}}"', from=1-1, to=1-3]
	\arrow["{K-\text{theory product}}", from=2-1, to=2-3]
	\arrow["{(\varphi \hat\otimes \psi)_*}", from=1-3, to=2-3]
\end{tikzcd}
    \]
\end{cor}
\begin{proof}
    The above diagram can be factored as:
    \[
   % https://q.uiver.app/#q=WzAsOCxbMCwwLCJLXzBee1xcb3BlcmF0b3JuYW1le3NwfX0oQSlcXG90aW1lcyBLXzBee1xcb3BlcmF0b3JuYW1le3NwfX0oQikiXSxbMCwxLCJLXzBee1xcb3BlcmF0b3JuYW1le3NwfX0oQScpXFxvdGltZXMgS18wXntcXG9wZXJhdG9ybmFtZXtzcH19KFxcbWF0aGZyYWt7QX1CJykiXSxbMiwwLCJLX3swfV57XFxvcGVyYXRvcm5hbWV7c3B9fShBIFxcaGF0XFxvdGltZXMgQikiXSxbMiwxLCJLX3swfV57XFxvcGVyYXRvcm5hbWV7c3B9fShBJyBcXGhhdFxcb3RpbWVzIFxcbWF0aGZyYWt7QX1CJykiXSxbMCwzLCJLXzBee1xcb3BlcmF0b3JuYW1le3NwfX0oQScpIFxcb3RpbWVzIEtfMF57XFxvcGVyYXRvcm5hbWV7c3B9fShCJykiXSxbMiwzLCJLXzBee1xcb3BlcmF0b3JuYW1le3NwfX0oQSdcXGhhdFxcb3RpbWVzIEInKSJdLFswLDIsIktfMF57XFxvcGVyYXRvcm5hbWV7c3B9fShBJykgXFxvdGltZXMgS18wXntcXG9wZXJhdG9ybmFtZXtzcH19KENfYihULEInKSkiXSxbMiwyLCJLXzBee1xcb3BlcmF0b3JuYW1le3NwfX0oQScgXFxoYXRcXG90aW1lcyBDX2IoVCxCJykpIl0sWzAsMSwiXFx2YXJwaGlfKlxcaGF0XFxvdGltZXNcXHRpbGRle1xccHNpfV8qIiwyXSxbMCwyLCJLLVxcdGV4dHt0aGVvcnkgcHJvZHVjdH0iXSxbMSwzLCJLLVxcdGV4dHt0aGVvcnkgcHJvZHVjdH0iXSxbMiwzLCIoXFx2YXJwaGkgXFxoYXRcXG90aW1lcyBcXHRpbGRle1xccHNpfSlfKiJdLFs0LDUsIkstXFx0ZXh0e3RoZW9yeSBwcm9kdWN0fSJdLFsxLDYsIjEgXFxvdGltZXMgKHFfKileey0xfSIsMl0sWzYsNywiSy1cXHRleHR7dGhlb3J5IHByb2R1Y3R9Il0sWzMsNywiKDEgXFxoYXRcXG90aW1lcyBxXyopXnstMX0iXSxbNiw0LCIxIFxcb3RpbWVzIFxcb3BlcmF0b3JuYW1le2V2fV97MV8qfSJdLFs3LDUsIigxIFxcaGF0XFxvdGltZXNcXG9wZXJhdG9ybmFtZXtldl97MX0pfV8qIl1d
\begin{tikzcd}
	{K_0^{\operatorname{sp}}(A)\otimes K_0^{\operatorname{sp}}(B)} && {K_{0}^{\operatorname{sp}}(A \hat\otimes B)} \\
	{K_0^{\operatorname{sp}}(A')\otimes K_0^{\operatorname{sp}}(\mathfrak{A}B')} && {K_{0}^{\operatorname{sp}}(A' \hat\otimes \mathfrak{A}B')} \\
	{K_0^{\operatorname{sp}}(A') \otimes K_0^{\operatorname{sp}}(C_b(T,B'))} && {K_0^{\operatorname{sp}}(A' \hat\otimes C_b(T,B'))} \\
	{K_0^{\operatorname{sp}}(A') \otimes K_0^{\operatorname{sp}}(B')} && {K_0^{\operatorname{sp}}(A'\hat\otimes B')}
	\arrow["{\varphi_*\hat\otimes\tilde{\psi}_*}"', from=1-1, to=2-1]
	\arrow["{K-\text{theory product}}", from=1-1, to=1-3]
	\arrow["{K-\text{theory product}}", from=2-1, to=2-3]
	\arrow["{(\varphi \hat\otimes \tilde{\psi})_*}", from=1-3, to=2-3]
	\arrow["{K-\text{theory product}}", from=4-1, to=4-3]
	\arrow["{1 \otimes (q_*)^{-1}}"', from=2-1, to=3-1]
	\arrow["{K-\text{theory product}}", from=3-1, to=3-3]
	\arrow["{(1 \hat\otimes q)_*^{-1}}", from=2-3, to=3-3]
	\arrow["{1 \otimes \operatorname{ev}_{1_*}}", from=3-1, to=4-1]
	\arrow["{(1 \hat\otimes\operatorname{ev_{1})}_*}", from=3-3, to=4-3]
\end{tikzcd}
    \]
    and each of the rectangles commute because of \cref{rem: prod}; for the middle rectangle we also need the naturality of \cref{nat ses}.
\end{proof}
\begin{prop}\label{prop: legend functoriality}
A graded asymptotic morphism $\varphi_t:\mathcal{S} \hat\otimes A \dashrightarrow B
\hat\otimes \mathcal{K}(\h)$ determines a $K$-theory map $\varphi_{*}:K_0^{\operatorname{sp}}(A) \to
K_0^{\operatorname{sp}}(B)$, with the following properties:
\begin{enumerate}[(i)]
	\item The correspondence $\varphi \mapsto \varphi_{*}$ is
		functorial with respect to composition with $*$-homomorphisms
		$A_1 \to A$ and $B \to B_1$.
	\item The map $\varphi_{*}$ depends only on the homotopy class of
		$\varphi$.
	\item If each $\varphi_t$ is actually a $*$-homomorphism, then
		$\varphi_{*}: K_0^{\operatorname{sp}}(A) \to K_0^{\operatorname{sp}}(B)$ is the map induced by
		$\varphi_1$.
\end{enumerate}
\end{prop}
\begin{proof}
    \begin{enumerate}[(i)]
        \item Note that $\varphi: \s \hat\otimes A \dashrightarrow B \hat\otimes \com(\h)$ induces a $*$-homomorphism $\s \hat\otimes A \rightarrow \mathfrak{A}(B \hat\otimes \com(\h)) \cong \mathfrak{A}B \hat\otimes \com(\h)$. By \cref{amplified morphisms group}, this induces a map $K_0^{\operatorname{sp}}(A) \rightarrow K_0^{\operatorname{sp}}(\mathfrak{A}B)$, which can be pushed forward by $q_*^{-1}$ to obtain a map $K_0^{\operatorname{sp}}(A) \rightarrow K_0^{\operatorname{sp}}(B)$; the map which we denote by $\varphi_*$.
        \item This follows from the second part of \cref{prop: asym} and the homotopy invariance mentioned in \cref{amplified morphisms group} (ii).
        \item Similar to the proof of $(iii)$ of \cref{prop: man functoriality groups}.
    \end{enumerate}
\end{proof}

\begin{cor} \label{cor: legend and prod}
    The $K$-theory product is natural with respect the maps described in \cref{prop: legend functoriality} in the following sense: if $\alpha: \s \hat\otimes B \dashrightarrow B' \hat\otimes \com(\h)$ is an asymptotic morphism, and $\varphi: A \rightarrow A'$ is a $*$-homomorphism, then the following diagram commutes:
    \[
    % https://q.uiver.app/#q=WzAsNCxbMCwwLCJLXzBee1xcb3BlcmF0b3JuYW1le3NwfX0oQSlcXG90aW1lcyBLXzBee1xcb3BlcmF0b3JuYW1le3NwfX0oQikiXSxbMiwwLCJLX3swfV57XFxvcGVyYXRvcm5hbWV7c3B9fShBIFxcaGF0XFxvdGltZXMgQikiXSxbMCwxLCJLXzBee1xcb3BlcmF0b3JuYW1le3NwfX0oQScpXFxvdGltZXMgS18wXntcXG9wZXJhdG9ybmFtZXtzcH19KEInKSJdLFsyLDEsIktfezB9XntcXG9wZXJhdG9ybmFtZXtzcH19KEEnIFxcaGF0XFxvdGltZXMgQicpIl0sWzAsMSwiSy1cXHRleHR7dGhlb3J5IHByb2R1Y3R9Il0sWzIsMywiSy1cXHRleHR7dGhlb3J5IHByb2R1Y3R9Il0sWzAsMiwiXFx2YXJwaGlfKiBcXGhhdFxcb3RpbWVzIFxcYWxwaGFfKiIsMl0sWzEsMywiKFxcdmFycGhpIFxcaGF0XFxvdGltZXMgXFxhbHBoYSlfKiJdXQ==
\begin{tikzcd}
	{K_0^{\operatorname{sp}}(A)\otimes K_0^{\operatorname{sp}}(B)} && {K_{0}^{\operatorname{sp}}(A \hat\otimes B)} \\
	{K_0^{\operatorname{sp}}(A')\otimes K_0^{\operatorname{sp}}(B')} && {K_{0}^{\operatorname{sp}}(A' \hat\otimes B')}
	\arrow["{K-\text{theory product}}", from=1-1, to=1-3]
	\arrow["{K-\text{theory product}}", from=2-1, to=2-3]
	\arrow["{\varphi_* \hat\otimes \alpha_*}"', from=1-1, to=2-1]
	\arrow["{(\varphi \hat\otimes \alpha)_*}", from=1-3, to=2-3]
\end{tikzcd}
    \]
\end{cor}
\begin{proof}
    The diagram factors as follows:
    \[
    % https://q.uiver.app/#q=WzAsOCxbMCwwLCJLXzBee1xcb3BlcmF0b3JuYW1le3NwfX0oQSlcXG90aW1lcyBLXzBee1xcb3BlcmF0b3JuYW1le3NwfX0oQikiXSxbMiwwLCJLX3swfV57XFxvcGVyYXRvcm5hbWV7c3B9fShBIFxcaGF0XFxvdGltZXMgQikiXSxbMCwzLCJLXzBee1xcb3BlcmF0b3JuYW1le3NwfX0oQScpXFxvdGltZXMgS18wXntcXG9wZXJhdG9ybmFtZXtzcH19KEInKSJdLFsyLDMsIktfezB9XntcXG9wZXJhdG9ybmFtZXtzcH19KEEnIFxcaGF0XFxvdGltZXMgQicpIl0sWzAsMSwiS18wXntcXG9wZXJhdG9ybmFtZXtzcH19KEEnKVxcb3RpbWVzIEtfMF57XFxvcGVyYXRvcm5hbWV7c3B9fShcXG1hdGhmcmFre0F9QicpIl0sWzIsMSwiS197MH1ee1xcb3BlcmF0b3JuYW1le3NwfX0oQScgXFxoYXRcXG90aW1lcyBcXG1hdGhmcmFre0F9QicpIl0sWzAsMiwiS18wXntcXG9wZXJhdG9ybmFtZXtzcH19KEEnKSBcXG90aW1lcyBLXzBee1xcb3BlcmF0b3JuYW1le3NwfX0oQ19iKFQsQicpKSJdLFsyLDIsIktfMF57XFxvcGVyYXRvcm5hbWV7c3B9fShBJyBcXGhhdFxcb3RpbWVzIENfYihULEInKSkiXSxbMCwxLCJLLVxcdGV4dHt0aGVvcnkgcHJvZHVjdH0iXSxbMiwzLCJLLVxcdGV4dHt0aGVvcnkgcHJvZHVjdH0iXSxbMCw0LCJcXHZhcnBoaV8qIFxcb3RpbWVzIFxcdGlsZGV7XFxhbHBoYX1fKiIsMl0sWzEsNSwiKFxcdmFycGhpIFxcaGF0XFxvdGltZXMgXFx0aWxkZXtcXGFscGhhfSlfKiJdLFs0LDUsIkstXFx0ZXh0e3RoZW9yeSBwcm9kdWN0fSJdLFs2LDcsIkstXFx0ZXh0e3RoZW9yeSBwcm9kdWN0fSJdLFs0LDYsIjEgXFxvdGltZXMgKHFfKileey0xfSIsMl0sWzUsNywiKDEgXFxoYXRcXG90aW1lcyBxKV8qXnstMX0iXSxbNiwyLCIxIFxcb3RpbWVzIFxcb3BlcmF0b3JuYW1le2V2fV97MV8qfSIsMl0sWzcsMywiKDEgXFxoYXRcXG90aW1lc1xcb3BlcmF0b3JuYW1le2V2X3sxfX0pXyoiXV0=
\begin{tikzcd}
	{K_0^{\operatorname{sp}}(A)\otimes K_0^{\operatorname{sp}}(B)} && {K_{0}^{\operatorname{sp}}(A \hat\otimes B)} \\
	{K_0^{\operatorname{sp}}(A')\otimes K_0^{\operatorname{sp}}(\mathfrak{A}B')} && {K_{0}^{\operatorname{sp}}(A' \hat\otimes \mathfrak{A}B')} \\
	{K_0^{\operatorname{sp}}(A') \otimes K_0^{\operatorname{sp}}(C_b(T,B'))} && {K_0^{\operatorname{sp}}(A' \hat\otimes C_b(T,B'))} \\
	{K_0^{\operatorname{sp}}(A')\otimes K_0^{\operatorname{sp}}(B')} && {K_{0}^{\operatorname{sp}}(A' \hat\otimes B')}
	\arrow["{K-\text{theory product}}", from=1-1, to=1-3]
	\arrow["{K-\text{theory product}}", from=4-1, to=4-3]
	\arrow["{\varphi_* \otimes \tilde{\alpha}_*}"', from=1-1, to=2-1]
	\arrow["{(\varphi \hat\otimes \tilde{\alpha})_*}", from=1-3, to=2-3]
	\arrow["{K-\text{theory product}}", from=2-1, to=2-3]
	\arrow["{K-\text{theory product}}", from=3-1, to=3-3]
	\arrow["{1 \otimes (q_*)^{-1}}"', from=2-1, to=3-1]
	\arrow["{(1 \hat\otimes q)_*^{-1}}", from=2-3, to=3-3]
	\arrow["{1 \otimes \operatorname{ev}_{1_*}}"', from=3-1, to=4-1]
	\arrow["{(1 \hat\otimes\operatorname{ev_{1}})_*}", from=3-3, to=4-3]
\end{tikzcd}
    \]
    where the top rectangle commutes because of \cref{cor: boy and products}, and the other rectangles commute because of \cref{rem: prod}; for the middle rectangle, we also need the naturality of \cref{nat ses}.
\end{proof}
\subsection{A note on higher K-groups}
As the reader would have noticed, we developed products only at the $K_0$ level in the preceding subsection. For all our applications, this suffices. We still give a brief sketch of generalizing products to a pairing between the higher K-groups, as a formal consequence of results in homotopy theory. 

First, note that for any $\Z/2\Z$-graded $C^*$-algebra $A$, its higher spectral $K$-groups are naturally isomorphic to the $K_0^{\operatorname{sp}}$-group of a function algebra on it. More precisely:
\begin{prop}\label{higher K-group}
    For any $n\geq 1$, we have a natural isomorphism of abelian monoids:
    \[
         K_n^{\operatorname{sp}}(A) \rightarrow K_0^{\operatorname{sp}}(C_0(\R^n,A))
    \]
\end{prop}
\begin{proof}
    Our sought after isomorphism comes from the following series of isomorphisms:
    \begin{align*}
        K_n^{\operatorname{sp}}(A) &\cong \pi_0(\Omega^n\K(A)) \\
        &= \pi_0(\operatorname{Map}_*(\mathbb{S}^n,\K(A))) \\
        &\cong \pi_0(\K(C_0(\R^n,A))) \\
        &= K_0^{\operatorname{sp}}(C_0(\R^n,A)),
    \end{align*}
    where the middle isomorphism stems from the isomorphism in \cref{prop:map into hom space}. The fact that the isomorphism therein preserves the H-space structure (whence we have an isomorphism of monoids) is routine upon observing that $\K(A)$ is a H-space, whence the group structure on $\operatorname{Map}_*(\mathbb{S}^n,\K(A))$ is induced from the addition on $\K(A)$.
\end{proof}
Thus, every higher spectral $K$-group is a spectral $K_0$ group in disguise.

To discuss the pairings between higher $K$-groups, we first note the following:
\begin{lem}
    The product map in \cref{Prod group level} descends to a map
    \[
    p: \K(A) \wedge \K(B) \rightarrow \K(A \hat\otimes B).
    \]
\end{lem}
\begin{proof}
    This is obvious from observing that if one of product factors in \cref{prop: product} is the zero homomorphism, then so is the resultant.
\end{proof}
\begin{cor}
    Given any two graded $C^*$-algebras, there is a bilinear pairing $K_i^{\operatorname{sp}}(A) \otimes K_j^{\operatorname{sp}}(B) \rightarrow K_{i+j}^{\operatorname{sp}}(A\hat\otimes B)~\forall~ i,j \geq 0$ which is natural for $*$-homomorphisms $A \rightarrow A'$ and $B \rightarrow B'$.
\end{cor}
\begin{proof}[Proof Sketch]
 It is a result in homotopy theory that if $X,Y,Z$ are based spaces, then a continuous map $f: X \wedge Y \rightarrow Z$ induces a natural bilinear pairing $\pi_i(X) \otimes \pi_j(Y) \rightarrow \pi_{i+j}(Z) ~\forall~ i,j \geq 1$. Note that when $n=1$, by ``bilinear", we really mean being a group homomorphism separately. 

 This result immediately yields the desired pairing when $i,j \geq 1$. When $i=0$ or $j=0$, we use \cref{higher K-group} to reduce to a pairing between spectral $K_0$-groups, for which we have already established products.
\end{proof}
\section{Relating Spectral K-theory with usual K-theory}

In this section, we come up with a natural isomorphism between $K_0^{\operatorname{sp}}(A)$ and $K_0(A)$ (the strandard projective module picture) for any trivially graded, unital (Real) $C^*$-algebra $A$. To that end, we will compare both the the pictures with a third picture, which we call the Fredholm picture of $K$-theory. We begin by recalling the notions of bounded and compact operators on Hilbert modules.

\subsection{Preliminaries on Hilbert modules and technical results}

We assume the reader is familiar with the notion of a Hilbert module over a $C^*$-algebra (in the complex and ungraded case, \cite{WEG} is a good reference), and about operators on them. We have analogous notions for $\Z/2\Z$-graded (Real) Hilbert modules over a $\Z/2\Z$-graded (Real) $C^*$-algebra, see for example \cite{JT} and \cite{BLA}. 

Throughout this section, $A$ denotes a $\Z/2\Z$-graded (Real) $C^*$-algebra (unless otherwise mentioned), and $\E$ and $\F$ graded (Real) Hilbert modules over it.

For notational purposes, we recall the definitions of bounded, finite-rank and compact operators.

\begin{defi}\label{def:adjointable}
    A map $T: \E \rightarrow \F$ (apriori neither linear nor bounded) is said to be adjointable if there exists a map $T^*:\F \rightarrow \E$ satisfying $\langle x,Ty \rangle = \langle T^*x,y \rangle ~\forall~ x,y \in \E$. Such a map $T^*$ is then called the adjoint for $T$. By $\mathcal{B}(\E,\F)$, we denote the set of all adjointable maps between $\E$ and $\F$, and when $\E=\F$, we simply denote it by $\mathcal{B}(\E)$.
\end{defi}
It is well known (see \cite{WEG}) that the elements of $\mathcal{B}(\E,\F)$ are automatically $A$-linear and bounded, and $\mathcal{B}(\E)$ forms a $C^*$-algebra with respect to the operator norm. Furthermore, the grading on $\E$ equips $\mathcal{B}(\E)$ with the structure of a $\Z/2\Z$-graded $C^*$-algebra.
\begin{defi} \label{def:fin_rank}
    If $e\in \E, f \in \F$, define $\ket{f}\bra{e}:\E \rightarrow \F$, which maps $x$ to $f\inr{e}{x}$. They are called rank one operators between $\E$ and $\F$. They satisfy the following easy to verify facts:
    \begin{itemize}
        \item The operators $\ket{f}\bra{e}$ are adjointable, with $\ket{f}\bra{e}^*=\ket{e}\bra{f}$.
        \item If $T \in \operatorname{Hom}_A(\F,\G)$ (and in particular, if $T \in \mathcal{B}(\F,\G)$) , then $T\ket{f} \bra{e}=\ket{Tf} \bra{e}$. Similarly, if $T \in \mathbb{B}(\E,\F)$, then $\ket{g} \bra{e} T= \ket{g} \bra{T^*f}$.
        \item The subspace spanned by the rank one elements of $\mathcal{B}(\E)$ form an ideal, called the ideal of $A$-finite rank operators on $\E$, and denoted by $\mathcal{F}(\E)$.
        \item The closure of $\mathcal{F}(\E)$ in the norm topology is called the ideal of $A$-compact operators, and denoted by $\mathcal{K}(\E)$.
    \end{itemize}
\end{defi}
\begin{rem}
    In the Real case (when we are considering Real operators between Real Hilbert modules), there are potentially two definitions of the finite rank operators that can be taken: 
    \begin{itemize}
        \item A Real finite rank operator is one which is finite rank as prescribed above, and which is also a Real operator
        \item A Real finite rank operator is one which can be written in the form $\sum_{i=1}^{n}\ket{f_i}\bra{e_i}$, where $f_i,e_i$ are Real elements of the $\F$ and $\E$ respectively.
    \end{itemize}
    The two definitions are equivalent, as can be checked easily: that the latter definition implies the former is clear, and in the other direction, first note that $\ov{\ket{f}\bra{e}}=\ket{\ov{f}}\bra{\ov{e}}$, and that \[\frac{\ket{f}\bra{e}+\ket{\ov{f}}\bra{\ov{e}}}{2}=\ket{\frac{f+\ov{f}}{2}}\bra{\frac{e+\ov{e}}{2}}+\ket{\frac{f-\ov{f}}{2i}}\bra{\frac{e-\ov{e}}{2i}}.\]
    However, it is not true that every rank-one Real operator is of the form $\ket{f}\bra{e}$ for some Real $f \in \F, e \in \E$.
\end{rem}
Next, we recall some facts about the (internal and external) tensor products of Hilbert modules; see \cite{LANCE} for a thorough treatment.

First, recall that when $A$ is a $\Z/2\Z$-graded unital $C^*$-algebra, and $H$ a graded Hilbert module on it, then for every compact Hausdorff space $X$, the graded Hilbert $C(X) \hat\otimes A$-module $C(X) \hat\otimes H$ can be naturally identified with the graded Hilbert $C(X,A)$-module $C(X,H)$. We present a lemma characterizing bounded and compact operators on it:

\begin{prop}\label{lem: characterizing B and K}
    There are natural isomorphisms of graded $C^*$-algebras:
    \[
    \mathcal{B}_{C(X,A)}(C(X,H)) \cong \{F:X \rightarrow \mathcal{B}_A(H) \mid F ~\text{is strictly continuous} \},
    \]
    and 
    \[
    \mathcal{K}_{C(X,A)}(C(X,H)) \cong \{F:X \rightarrow \mathcal{K}_A(H) \mid F ~\text{is norm continuous} \},
    \]
    where both the function algebras are equipped with the supremum norm.
\end{prop}
\begin{proof}
    That the function algebras described above are $C^*$-algebras is routine. We first deal with bounded operators: there are natural maps in both directions:
    \begin{align*}
         \Phi_A: \mathcal{B}_{C(X,A)}(C(X,H)) &\longrightarrow \{F:X \rightarrow \mathcal{B}_A(H) \mid F ~\text{is strictly continuous} \} \\
         T &\longmapsto x \mapsto T_x, ~\text{where}~ T_x(h)=(T1_h)x.
    \end{align*}\footnote{$1_h$ stands for the constant function taking the value $h$}
    and 
    \begin{align*}
        \Psi_A: \{F:X \rightarrow \mathcal{B}_A(H) \mid F ~\text{is strictly continuous} \} &\longrightarrow \mathcal{B}_{C(X,A)}(C(X,H)) \\
        x \mapsto T_x &\longmapsto f \mapsto \{ x \mapsto T_x(f(x)) \}
    \end{align*}
    We only check now that both $\Phi_A$ and $\Psi_A$ are well-defined set maps; that they are mutually inverse $*$-homomorphisms is again routine.

    For the well-definedness of $\Phi_A$, we first note that the operators $T_x$ are indeed adjointable operators on $H$. Indeed, $(T_x)^*=(T^*)_x$, as can be easily seen by observing that $\langle T1_{h_1},1_{h_2} \rangle = \langle 1_{h_1},T^*(1_{h_2}) \rangle ~\forall~ h_1,h_2 \in H$. Next, we need to show that $\Phi_A(T)$ is strictly continuous for all $T \in \mathcal{B}_{C(X,A)}(C(X,H))$. Since $X$ is compact, in light of Lemma 2.3.6 of \cite{WEG}, it suffices to show that $\Phi_A(T)$ is continuous in the point-norm-$*$ topology. But this is clear: for all $h \in H$, we have \[x_i \rightarrow x \implies (T1_h)(x_i) \xrightarrow{\lVert . \rVert} (T1_h)(x) \implies T_{x_i}(h) \xrightarrow {\lVert . \rVert}T_x(h).\]

    For the well-definedness of $\Psi_A$, we first check that given $f \in C(X,H)$, the map $x \mapsto T_x(f(x))$ is also an element of $C(X,H)$. This can be seen easily as follows: given $x_i \rightarrow x$, we have 
    \begin{align*}
        &\lVert T_x(f(x))-T_{x_i}(f(x_i)) \rVert \\
        \leq & \lVert T_x(f(x))-T_{x_i}(f(x)) \rVert + \lVert T_{x_i}(f(x))-T_{x_i}(f(x_i)) \rVert 
    \end{align*}
   and both the summands can be made small enough for large enough $i$; the first one as $x \mapsto T_x$ is given to be strictly continuous, and the second one because $x \mapsto T_x$ is norm bounded courtesy the compactness of $X$. We can now conclude that $\Psi_A(x \mapsto T_x)$ is an adjointable operator on $C(X,H)$, since $\Psi_A(x \mapsto (T_x)^*)$ can be readily seen to be its adjoint.

   It remains now to show the statements involving compact operators. We first show that \[
   \Phi_A(\mathcal{K}_{C(X,A)}(C(X,H))) \subseteq \{F:X \rightarrow \mathcal{K}_A(H) \mid F ~\text{is norm continuous} \}.\] Indeed, the rank-one operator $\ket{f}\bra{g}$ maps to the norm continuous map \footnote{this requires a check, but is easy!} $x \mapsto \ket{f(x)}\bra{g(x)}$ under $\Phi_A$, and as $\mathcal{K}_{C(X,A)}(C(X,H))$ is the closed linear span of rank-one operators are the compact operators, the containment follows. To see that the containment is infact an equality, we first note that given $f \in C(X,\C)$ and $h,k \in H$, $\Phi_A(\ket{f_h}\bra{1_k})=x \mapsto f(x)\ket{h}\bra{k}$, where $f_h(x)=f(x)h$. As the collection $\{x \mapsto f(x)\ket{h}\bra{k} \mid f \in C(X,\C),h,k \in H\}$ is dense in $\{F:X \rightarrow \mathcal{K}_A(H) \mid F ~\text{is norm continuous} \}$, and $\com_A(\h)$ is a Banach space with respect to the norm topology, we are done by Lemma 6.4.16 of \cite{MUR}.
\end{proof}

 In the remainder of this subsection,  we include a series of technical results that will come very handy later. For the remainder of this subsection, $H$ denotes a graded Hilbert $A$-module, and $E \subseteq H$ is a graded Hilbert $A$-submodule of $H$. Let $\mathfrak{K}:={\{T \in \mathcal{K}(H) \mid \operatorname{Im}T \subseteq E, \operatorname{Im}T^* \subseteq E\}}$.

\begin{lem} \label{lem: restricting compacts}
    Let $T \in \mathfrak{K}$. Then, $T\mid_E: E \rightarrow E$ is an element of $\mathcal{K}(E)$.
\end{lem}
\begin{proof}
    Note that $\mathfrak{K}$ is a $C^*$-algebra. Let $\{S_i\}_{i \in I}$ be an approximate unit of it. Let $\varepsilon>0$ be given.
    Choose a finite rank operator $F$ in $\mathcal{K}(H)$ such that $\lVert T-F \rVert < \frac{\varepsilon}{2}$. Also, there exists some $i \in I$ such that $\lVert S_iFS_i - F \rVert < \frac{\varepsilon}{2}$. Thus, $\lVert T-S_iFS_i \rVert < \varepsilon$. But, it is direct from the elementary properties of finite rank operators as mentioned in \cref{def:fin_rank} that  $S_iFS_i$ is of the form $\sum_{k=1}^{n} \rvert r_k \rangle \langle s_k \rvert$, with $r_k,s_k \in E$. Thus, as $\lVert T\mid_E-S_iFS_i\mid_E \rVert \leq \lVert T-S_iFS_i \rVert < \varepsilon$, and $\varepsilon > 0$ is arbitrary, we conclude that $T\mid_E : E \rightarrow E$ is an element of $ \mathcal{K}(E)$. 
\end{proof}
\begin{rem}\label{rem: fin rank}
    From the proof of \cref{lem: restricting compacts}, it is clear that the $C^*$-algebra $\mathfrak{K}$ is the norm closure of the span of the set $\{\ket{x}\bra{y} \mid x,y \in E\}$ inside $\com(H)$.
\end{rem}
\begin{lem} \label{lem: extending co_pacts along isometries}
    Associated to the inclusion $j: E \hookrightarrow H$, there is a graded $*$-homomorphism $j_*:\mathcal{K}(E) \rightarrow \mathcal{K}(H)$. Moreover, $j_*=\operatorname{Ad}(j)$ when $E$ is complemented. 
\end{lem}
\begin{proof}
    We claim that the map $j_*:\mathcal{K}(E) \rightarrow \mathcal{K}(H)$, given by $j_*(T)=j(jT^*)^*$ is the one we are after. We first show:
    
    \textbf{Claim:}
        For any graded Hilbert $A$-module $F$, if $T \in \mathcal{K}(F,E)$, then $jT \in \mathcal{K}(F,H)$.
        
    \begin{proof}
        First, we deal the case when $T$ is of finite rank. Write $T=\sum_{k=1}^{n} \rvert r_k \rangle \langle s_k \lvert$, with $r_k \in E, s_k \in F$. Then, it is direct from the definition that $jT=\sum_{k=1}^{n} \rvert j(r_k) \rangle \langle s_k \lvert$. In the general case when $T$ is compact, write $T$ as a norm limit of finite rank operators, and the fact that $j$, being an isometry, preserves norms.
    \end{proof}
    Using the above claim twice, together with the fact that the adjoint of a compact operator is compact, yields that $j(jT^*)^* \in \mathcal{K}(H)$. It remains to show that it is a $*$-homomorphism.
    \begin{itemize}
        \item $\varphi$ is linear: $j(j(S+T)^*)^*=j(jS^*+jT^*)^*=j((jS^*)^*+(jT^*)^*)=j(jS^*)^*+j(jT^*)^*$; $j(j(\lambda T)^*)^*=\lambda j(jT^*)^*$.
        \item $\varphi$ preserves $*$:  We have to show that $(j(jS^*)^*)^*=j(jS)^*$, or, $j(jS^*)^*=(j(jS)^*)^*$. But this is direct to verify when $S$ is finite rank, and the general case follows by writing $S$ as a norm limit of finite ranks, and using that $j$ is an isometry.
        \item $\varphi$ preserves multiplication: We have to show that $j(j(ST)^*)^*=j(jS^*)^*j(jT^*)^*$. The left hand side directly simplifies to $jS(jT^*)^*$. It thus suffices to show that $S=(jS^*)^*j$. But this is again clear when $S$ is finite rank, and the general case follows by writing $S$ as a norm limit of finite rank operators, and using that $j$ is an isometry.
        \item $\varphi$ commutes with the grading: This is clear by considering the definition of the grading on the compact operators induced by the grading on the underlying Hilbert module.
    \end{itemize}
\end{proof}
\begin{thm}\label{thm: technical}
    There exists a unique $*$-isomorphism \[
    \Phi: \com(E) \rightarrow \mathfrak{K}
    \] such that $\Phi(\ket{x}\bra{y})=\ket{jx}\bra{jy}$. The inverse to $\Phi$ is given by \[-\mid_E:\mathfrak{K} \rightarrow \com(E).\]
\end{thm}
\begin{proof}
    The uniqueness of the isomorphism $\Phi$, if it exists, is clear since $\com(E)$ is the norm closure of the span of rank-one operators. We therefore focus on existence.
    
    We claim that the map $j_*$ as in \cref{lem: extending co_pacts along isometries} is the required isomorphism. Clearly, $j_*(\ket{x}\bra{y})=\ket{jx}\bra{jy}$. From this, it follows that the image of $j_*$ does indeed lie in $\mathfrak{K}$. 
    Both the composites $j_* \circ -\mid_E$ and $-\mid_E \circ j_*$ maps $\ket{x}\bra{y}$ to $\ket{x}\bra{y}$ (where $x,y \in E$), whence both are identity on a dense subset, and we are done. 
\end{proof}
We end this subsection with a couple of more lemmas using the notation developed in the above theorem.
\begin{lem} \label{lem: easy}
 Let $j_0,j_1:E \rightarrow H$ be two (not necessarily adjointable) even isometries, which are homotopic (in the norm topology). Then, ${j_0}_* \sim {j_1}_* : \mathcal{K}(E) \rightarrow \mathcal{K}(\h)$.
\end{lem}
\begin{proof}
    We claim that $t \mapsto j_{t_*}:\com(E) \rightarrow \com(H)$ is a homotopy of $*$-homomorphisms, where $t \mapsto j_t$ is the norm continuous path of isometries. We have to show that $t \mapsto j_{t_*}(T)$ is a continuous map $[0,1] \rightarrow \com(H)$ for each $T \in \com(E)$. We calculate:
    \begin{align*}
       \lVert j_{t}(j_{t}T^*)^* -  j_{s}j_{s}T^*)^* \rVert &=  \lVert j_{t}(j_{t}T^*)^* -  j_{t}(j_{s}T^*)^* +  j_{t}(j_{s}T^*)^* - j_{s}(j_{s}T^*)^* \rVert \\
       &\leq \lVert j_{t}(j_{t}T^*)^* -  j_{t}(j_{s}T^*)^* \rVert +  \lVert j_{t}(j_{s}T^*)^* - j_{s}(j_{s}T^*)^* \rVert \\
       &= \lVert (j_{t}T^*)^* - (j_{s}T^*)^* \rVert +  \lVert (j_{t}-j_{s})(j_{s}T^*)^*  \rVert \\
       &\leq \lVert j_t - j_s \rVert\lVert T^* \rVert + \lVert j_t - j_s \rVert \lVert (j_sT^*)^* \rVert ,
    \end{align*}
    and the continuity follows.
\end{proof}
Recall that $\hat{\HA}$ denotes the standard graded Hilbert $A$-module. It is going to be extremely crucial for us in the forthcoming subsections. Note that there's a natural identification between $A \hat\otimes \hat{l^2}$ and $\hat{\HA}$. Details about it can be found in \cite{BLA}, \cite{JT}.
\begin{lem} \label{lem: final}
    Let $E$ be a graded Hilbert submodule of $\hat{\HA}$. Let $j:E \rightarrow \hat{\HA}$ be the inclusion. Then the $*$-homomorphisms $j_*:\mathcal{K}(E) \rightarrow \mathcal{K}(\hat\HA)$ and $Ad(i):\mathcal{K}(E) \rightarrow \com(\hat\HA)$ obtained from an adjointable even isometry are homotopic.
\end{lem}
\begin{proof}
    Let $\iota_0,\iota_1:\hat\HA \rightarrow \hat\HA \oplus \hat\HA$ be the inclusions. Then, $\operatorname{Ad}(\iota_0) \sim \operatorname{Ad}(\iota_1): \mathcal{K}(\hat\HA) \rightarrow \mathcal{K}(\hat\HA \oplus \hat\HA)$ are homotopic homotopy equivalences, since both are strong-$*$ connected to an unitary isomorphism $\hat\HA \rightarrow \hat\HA \oplus \hat\HA$. So, it suffices to show that $\operatorname{Ad}(\iota_1) \circ j_* \sim \operatorname{Ad}(\iota_0) \circ \operatorname{Ad}(i)$. But, $\operatorname{Ad}(\iota_0) \circ \operatorname{Ad}(i) \sim (\iota_0 \circ i)_*$, and $\operatorname{Ad}(\iota_1) \circ j_*=(\iota_1 \circ j)_*$. Finally, note that $\iota_1 \circ j$,$ \iota_0 \circ i$ are homotopic (in the norm topology!) through the path $t \mapsto (\cos(t)i,\sin(t)j)$. The rest follows from \cref{lem: easy}.
\end{proof}
\subsection{Fredholm operators and their index}
We now discuss the most important class of operators for K-theory; namely the Fredholm operators. Keeping Atkinson's theorem from classical Hilbert space theory in mind, we have the following definition: (see also \cite{WEG}, \cite{GVF}):
\begin{defi}\label{defi: Fredholm}
	Let $F \in \mathcal{B}(\mathcal{E},\mathcal{F})$. We say that $F$ is an $A$-Fredholm operator if $1-FG \in \mathcal{K}(\mathcal{F})$ and $1-GF \in \mathcal{K}(\mathcal{E})$ for some $G \in \mathcal{B}(\mathcal{F},\mathcal{E})$. An equivalent definition when $\mathcal{E}=\mathcal{F}$ is to ask for invertibility in $\mathcal{B}(\mathcal{E})/\mathcal{K}(\mathcal{E})$. We write $F \in ~\text{Fred}_A(\mathcal{E},\mathcal{F})$, or simply $F \in ~\text{Fred}_A(\mathcal{E})$ when $\mathcal{E}=\mathcal{F}$. $G$ as in this definition is called a parametrix for $F$. If $1-FG$ and $1-GF$ are $A$-finite rank, we call $G$ a finite rank parametrix for $F$.
\end{defi}
\begin{rem}
    In the Real case, a Real Fredholm operator always has a Real parametrix. Indeed, if $G$ is any parametrix to $F$, $\frac{G+\bar{G}}{2}$ is a Real parametrix to $F$. Similarly, an odd, self-adjoint Fredholm operator has an odd, self-adjoint parametrix.
\end{rem}

\begin{rem}
	If $F \in ~\text{Fred}_A(\mathcal{E},\mathcal{F})$, then $F^\ast$ is also Fredholm. Indeed, take $G$ to be a parametrix for $F$. Then, as adjoint of an $A$-compact operator is $A$-compact, we get $1-G^\ast F^\ast \in \mathcal{K}(\mathcal{F})$ and $1-F^\ast G^\ast\in \mathcal{K}(\mathcal{E})$, as  desired. Similarly, it can be shown that the composition of two $A$-Fredholm operators is $A$-Fredholm.
\end{rem}
Note that we, in some sense, take Atkinson's theorem from classical Fredholm theory as a definition here. The next lemma shows that some equivalent formulations of Atkinson's theorem in the classical world also hold good here:
\begin{lem} \label{lem: fred <=> inv mod fin rank}
	The following are equivalent for $F \in \mathcal{B}(\mathcal{E},\mathcal{F})$:
	\begin{enumerate}[label=(\roman*)]
		\item F is $A$-Fredholm.
		\item $\exists~ G_1,G_2 \in \mathcal{B}(\mathcal{F},\mathcal{E})$ such that $1-G_1F \in \mathcal{K}(\mathcal{E})$, $1-FG_2 \in \mathcal{K}(\mathcal{F})$.
		\item $\exists~ H \in \mathcal{B}(\mathcal{F},\mathcal{E})$ such that $1-HF$ and $1-FH$ are $A$-finite rank operators.
		\item $\exists~ H_1,H_2 \in \mathcal{B}(\mathcal{F},\mathcal{E})$ such that $1-H_1F$, and $1-FH_2$ are $A$-finite rank operators.
	\end{enumerate}
\end{lem}
\begin{proof}\;

\begin{itemize}
    \item\;[$(i) \implies (ii)$] Obvious!

	\item\;[$(ii) \implies (i)$] Let $1-G_1F=K_1 \in \mathbb{K}(\mathcal{E})$, $1-FG_2=K_2 \in \mathbb{K}(\mathcal{F})$. Note that $G_2-K_1G_2=(1-K_1)G_2=(G_1F)G_2=G_1(FG_2)=G_1(1-K_2)=G_1-G_1K_2$. Thus, $G_2-G_1 = K_1G_2-G_1K_2 =: K_3$ is compact, which implies that $G_2F=G_1F+K_3F=1-K_1+K_3F$, whence $1-G_2F$ is compact. So we can take $G=G_2$ (or similarly,
	$G=G_1$).

	\item\;[$(iii) \iff (iv)$] Similar to the above equivalence; note that all we needed above was that compact operators on a Hilbert $C^\ast$-module forms an ideal. But so does the set of finite rank operators.

	\item\;[$(iii) \implies (i)$] Obvious, as finite rank operators are compact.

	\item\;[$(i) \implies (iv)$] We have $1-GF \in \mathcal{K}(\mathcal{E})$ implies that $GF$ is invertible in $\mathcal{B}(\mathcal{E})/\mathcal{K}(\mathcal{E})$. The rest follows from the easy result about Banach algebras (see \cite{GVF} for a proof): If $A$ is a unital Banach algebra, and $J$ is an ideal in it, then $a \in A$ is invertible modulo $\closure{J}$ if and only if it is invertible modulo $J$.
\end{itemize}
\end{proof}
We also note what \cref{lem: characterizing B and K} says about Fredholm operators on function algebras:
\begin{cor} \label{cor: Fred on function algebra}
    Fredholm operators on $C(X,H)$ correspond to precisely those strictly continuous functions $F:X \rightarrow \mathcal{B}_A(H)$ which takes values in Fredholm operators on $H$, and for which a strictly continuous family of pointwise-parametrices $G:X \rightarrow \mathcal{B}_A(H)$ exist such that $1-FG,1-GF: X \rightarrow \com_A(\h)$ is norm continuous. Moreover, if $F:X \rightarrow \mathcal{B}_A(H)$ is Fredholm, and if a parametrix $g$ to $F(x_0)$ for a fixed $x_0 \in X$ is given, then a parametrix $G:X \rightarrow \mathcal{B}_A(\h)$ to $F$ exsits such that $G(x_0)=g$.  
 \end{cor}
 \begin{proof}
     The first part of the lemma is immediate form \cref{lem: characterizing B and K}. We therefore only focus on the second part: from the proof of (ii) $\implies$ (i) in \cref{lem: fred <=> inv mod fin rank}, we have that any two parametrices to $F(x_0)$ differ by a compact operator. So, if $G'$ is any parametrix to $F$, then $G'(x_0)-g=:k \in \mathcal{K}_A(H)$. The operator $G:=G'-k:X \rightarrow \mathcal{B}_A(\h)$ is the desired parametrix to $F$.
 \end{proof}
 \begin{cor}\label{cor: path of Fredholm}
     Let $F:X \rightarrow \mathcal{B}_A(H)$ be a norm continuous family of (odd, self-adjoint, Real) Fredholm operators on $H$. Then,  $\Psi_A(F)$ is a (odd, self-adjoint, Real) Fredholm operator on $\mathcal{B}_{C(X,A)}(C(X,H))$. 
 \end{cor}
 \begin{proof}
     In light of \cref{cor: Fred on function algebra}, it suffices to show that there exists a norm continuous family $G : X \rightarrow \mathcal{B}_A(H)$ of pointwise-parametrices to $F$. Let $\mathcal{Q}_A(H):=\mathcal{B}_A(H)/\com_A(H)$. Choose a continuous cross-section $p:\mathcal{Q}_A(H) \rightarrow \mathcal{B}_A(H)$ of the quotient map $q$, which exists by the Bartle-Graves theorem. Then, our sought after $G$ can be seen as the following composite:
     \[
     % https://q.uiver.app/#q=WzAsNSxbMCwwLCJYIl0sWzEsMCwiXFxvcGVyYXRvcm5hbWV7RnJlZH1fQShIKSJdLFsxLDEsIlxcb3BlcmF0b3JuYW1le0dMfShcXG1hdGhjYWx7UX1fQShIKSkiXSxbMiwxLCJcXG9wZXJhdG9ybmFtZXtHTH0oXFxtYXRoY2Fse1F9X0EoSCkpIl0sWzIsMCwiXFxtYXRoY2Fse0J9X0EoSCkiXSxbMSwyLCJxIl0sWzIsMywiYSBcXG1hcHN0byBhXnstMX0iXSxbMyw0LCJyIl0sWzAsMSwiRiJdXQ==
\begin{tikzcd}
	X & {\operatorname{Fred}_A(H)} & {\mathcal{B}_A(H)} \\
	& {\operatorname{GL}(\mathcal{Q}_A(H))} & {\operatorname{GL}(\mathcal{Q}_A(H))}
	\arrow["q", from=1-2, to=2-2]
	\arrow["{a \mapsto a^{-1}}", from=2-2, to=2-3]
	\arrow["s", from=2-3, to=1-3]
	\arrow["F", from=1-1, to=1-2]
\end{tikzcd},
     \]
    Indeed, for any $x \in X$,
    \begin{align*}
        q(F(x)G(x))&=q(F(x))q(G(x)) \\
        &=q(F(x))q(F(x))^{-1} ~\text{since $s \circ q= \operatorname{id}_{\mathcal{Q}_A(H)}$}\\
        &=1.
    \end{align*}
    implying $G(x)$ is a parametrix to $F(x)$. The odd, self-adjoint, Real case follows as in \cref{foot}.
 \end{proof}
We now restrict our attention only to \textit{unital, trivially graded} $C^*$-algebras $A$. Some of the ideas below are adapted from \cite{EBERT}.

The classical theory of Fredholm operators on a Hilbert space worked very smoothly because the image of such an operator was always closed, and its kernel and cokernel were finite dimensional. However, it is not the case in general over Hilbert $C^*$-modules; see \cite{GVF}. However, if the image of an $A$- Fredholm operator is closed, then some well known results from Fredholm theory over Hilbert spaces do continue to hold. In what follows $\E,\F$ are Hilbert modules over $A$.

\begin{lem}\label{lem: taming}
    Let $F:\E \rightarrow \F$ be an $A$-Fredholm operator with closed range. Then, $\ker F$ is algebraically finitely generated and projective.
\end{lem}
\begin{proof}
    In light of Theorem 15.4.2 of \cite{WEG}, it suffices to show that the identity operator on $\ker F$ is compact. To that end, choose a $G \in \mathcal{B}(\F,\E)$ such that $1-GF=:R$ is an $A$-compact operator on $\E$. Then, note that 
    \[
    x \in \ker F \implies Rx=x.
    \]
    Also, as $\operatorname{im} F$ is given to be closed, Theorem 3.2 of \cite{LANCE} shows that $\ker F$ is complemented in $\E$, whence there exists a projection operator $P: \E \rightarrow \ker F$. Therefore, the identity on $\E$ factors as 
    \[
    % https://q.uiver.app/#q=WzAsNCxbMCwwLCJcXGtlciBGIl0sWzEsMCwiXFxrZXIgRiJdLFswLDEsIlxcbWF0aGNhbHtFfSJdLFsxLDEsIlxcbWF0aGNhbHtFfSJdLFswLDIsIlxcb3BlcmF0b3JuYW1le2luY30iXSxbMiwzLCJSIl0sWzMsMSwiUCIsMl0sWzAsMSwiXFxvcGVyYXRvcm5hbWV7aWR9Il1d
\begin{tikzcd}
	{\ker F} & {\ker F} \\
	{\mathcal{E}} & {\mathcal{E}}
	\arrow["{\operatorname{inc}}", from=1-1, to=2-1]
	\arrow["R", from=2-1, to=2-2]
	\arrow["P"', from=2-2, to=1-2]
	\arrow["{\operatorname{id}}", from=1-1, to=1-2]
\end{tikzcd},
    \]
    and the conclusion follows since $R$ is compact.
\end{proof}
\begin{lem}\label{surj implies ker iso}
    Given any surjective $A$-Fredholm operator $F$, there exists a $\delta>0$ such that if $G$ is a surjective Fredholm operator with $\lVert F-G \rVert < \delta$, then $\ker F \cong \ker G$. 
\end{lem}
\begin{proof}
    As $F$ is surjective, $FF^*$ is invertible (see the proof of Theorem 3.2 in \cite{LANCE}), and we can write $\ker F= \ker F_1$, where $F_1=(FF^*)^{-\frac{1}{2}}F$. Similarly, $\ker G= \ker G_1$, where $G_1=(GG^*)^{-\frac{1}{2}}G$. Then $F_1{F_1}^*={G_1}{G_1}^*=1$. So, ${G_1}^*{G_1}$ and ${F_1}^*{F_1}$ are the projections onto the kernels of $G$ and $F$ respectively. We can also choose a $\delta>0$ such that $\lVert F_1{F_1}^* - G_1{G_1}^* \rVert < 1$. It now follows from basic facts about projections in a unital $C^*$-algebra (see \cite{ROR} Section 2.2) that $F_1{F_1}^*$ and $G_1{G_1}^*$ are unitarily equivalent, whence have isomorphic kernels.
\end{proof}
In light of the above discussion, we define index of a $A$-Fredholm operator in the following way:

\begin{defi}\label{defi: index defi}
Let $F:\E \rightarrow \F$ be an $A$-Fredholm operator. Choose $n \in \mathbb{N}$, and an $A$-linear map $g:A^n \rightarrow \F$ such that $F+g:\E \oplus A^n \rightarrow \F$ is surjective. We call the pair $(n,g)$ a taming of $F$, and define the index of $F$, denoted $\operatorname{ind} F$, as $\operatorname{ind} F=[\ker(F+g)]-[A^n]$.   
\end{defi}
\begin{lem}\label{lem: ind well-defined}
    The above definition is well-defined, that is, a taming always exist, and the index is independent of the chosen taming.
\end{lem}
\begin{proof}
    First, we note that $F+g:\E \oplus A^n \rightarrow \F$ is Fredholm for any $A$-linear map $g:A^n \rightarrow \F$, since $g$ is $A$-finite rank. Thus, in light of \cref{lem: taming}, $\ker(F+g)$ is finitely-generated and projective (and hence a valid class in $K_0(A)$) as soon as $F+g$ is surjective. 
    
    As in \cref{lem: fred <=> inv mod fin rank}, choose an $A$-Fredholm operator $G: \F \rightarrow \E$ such that $R:=1-FG \in \mathbb{F}(\F)$. Write $R=\sum_{i=1}^{n}\ket{r_i}\bra{s_i}$. Define $L_r=\sum_{i=1}^{n} \ket{r_i}\bra{e_i}:A^n \rightarrow \F$. Then, by construction, $F+g$ is surjective, and we conclude that $(n,g)$ is a taming of $F$.

    It remains to show that the index is independent of the taming chosen. To that end, suppose $(n,g)$ and $(m,h)$ are two tamings. Then, consider the family of $A$-Fredholm operator parametrized over $[0,1]^2$ given by \begin{align*}
        F(s,t): \E \oplus A^n \oplus A^m \rightarrow \F \\
        (e,x,y) \mapsto e+sg(x)+th(y)
    \end{align*}
    By the definition of taming, $F(s,t)$ is surjective whenever $(s,t) \neq (0,0)$. Now, the path $t \mapsto F(t,1-t)$ is clearly a norm continuous path of surjective Fredholm operators. \cref{surj implies ker iso} now shows  (by covering the path with small enough balls and using compactness of $[0,1]$) that $\ker F(0,1) \cong \ker F(1,0)$. But, $\ker F(0,1) \cong \ker (F+h) \oplus A^n$, and $\ker F(1,0) \cong \ker(F+g) \oplus A^m$, and the rest follows by the group structure on $K_0(A)$.
\end{proof}

\subsection{The Fredholm picture of K-theory}
In this subsection, we restrict our attention only to unital $\Z/2\Z$-graded $C^*$-algebras $A$. Although this isn't a complete necessity, and perhaps something weaker like $\sigma$-unitality could suffice, the applications we have in mind later need unitality assumptions anyway.
\begin{defi}
    A Fredholm family/cycle over $A$ is a triple $(\E,\iota,F)$, where $(\E,\iota)$ is a countably generated $\Z/2\Z$-graded Hilbert $A$ module \footnote{often, the grading $\iota$ on $\E$ will be suppresed in the notation}, and $F$ is an odd, self-adjoint Fredholm operator on it. 
\end{defi}
\begin{defi}
    The following are some definitions circling Fredholm families:
    \begin{itemize}
        \item A Fredholm family $F$ is called acyclic if $F$ is invertible.
        \item If $(\E_0,\iota_0,F_0)$ and $(\E_1,\iota_1,F_1)$ are Fredholm families, then $(\E_0,\iota_0,F_0) \oplus (\E_1,\iota_1,F_1):=(\E_0 \oplus E_1, \iota_0 \oplus \iota_1,F_0 \oplus F_1
        )$.
        \item Two Fredholm cycles $(\E_0,\iota_0,F_0)$ and $(\E_1,\iota_1,F_1)$ are said to be isomorphic if there exists an even unitary $U: \E_0 \rightarrow \E_1$ such that $UF_0U^*=F_1$. We write $(\E_0,\iota_0,F_0) \cong_{U} (\E_1,\iota_1,F_1)$ in this case. Although the unitary $U$ is a crucial part of the data, we will often suppress it in the notation.
        \item Two Fredholm cycles $(\E_0,\iota_0,F_0)$ and $(\E_1,\iota_1,F_1)$ are said to be concordant (denoted $(\E_0,\iota_0,F_0) \sim (\E_1,\iota_1,F_1)$) if there exists $(\E,\iota,F)$, a Fredholm family over $A \hat\otimes C[0,1]$ such that $(\E \hat\otimes_{\operatorname{ev}_i}A, \operatorname{ev}_i(\iota) \operatorname{ev}_i(F)) \cong (\E_i,\iota_i,F_i)$ for $i=0,1$, where $\operatorname{ev}_t:A[0,1] \rightarrow A$ is the graded homomorphism given by $\operatorname{ev}_t(f)=f(t)$.
    \end{itemize}
\end{defi}
The notion of concordance is meant to be a substitute for ``paths of Fredholm operators", but is much more flexible. Indeed, we have the following result:
\begin{lem} \label{lem: homotopy implies concor}
    Let $\E$ be a countably generated Hilbert module, and $F,G:\E \rightarrow \E$ odd, self-adjoint Fredholm operators on it. If $F$ and $G$ are connected by a norm continuous path of odd, self-adjoint Fredholm operators, then $(\E,F) \sim (\E,G)$.
\end{lem}
\begin{proof}
    This follows immediately from \cref{lem: characterizing B and K} and \cref{cor: path of Fredholm}; the path between $F$ and $G$ works as a concordance.
\end{proof}

The next lemma shows the flexibility of the concordance relation.
\begin{lem} \label{cor: inv conc}
    Let $(\E,F)$ be a Fredholm cycle over $A$. If $F$ is invertible, then $(\E,F) \sim (0,0)$. 
\end{lem}
\begin{proof}
     Consider the following concordance $(\E \hat\otimes C_0[0,1),F \hat\otimes 1)$, where $C_0[0,1)$ is the ideal of functions on $C[0,1]$ which vanishes at $1$. This is a concordance between $(\E,F)$ and $(0,0)$, since $F \hat\otimes 1$ is Fredholm courtesy being invertible.
\end{proof}
We need another lemma in a similar spirit, whose proof is a little more technical. However, both the lemma and its proof will be used at critical junctures later. Before proceeding, we remind the reader of the famous Kasparov stabilization theorem, which asserts that every countably generated graded (Real) Hilbert $A$-module is isomorphic to a complemented submodule of $\hat{\HA}$. Moreover, the isomorphism can be constructed so that the complement of the image is again isomorphic is $\hat{\HA}$. This extremely important result is going to be our backbone for many of the results. See the seminal paper by G.G. Kasparov \cite{KAS} for proofs.
\begin{lem} \label{lem: add nondeg}
     Let $(\E,F)$ be a Fredholm cycle over $A$. Then, $(\E \oplus\hat{\HA},F \oplus U) \sim (\E,F)$, where $U:\hat{\HA} \rightarrow \hat{\HA}$ is an invertible operator.
\end{lem}
\begin{proof}
    By Kasparov stabilization, there exists an unitary isomorphism $u:\E \oplus \hat{\HA} \rightarrow \hat{\HA}$. Let $G:\hat{\HA} \rightarrow \hat{\HA}$ be such that the following diagram commutes:
 \[
 % https://q.uiver.app/#q=WzAsNCxbMCwwLCJFXzBcXG9wbHVzIFxcaGF0e1xcbWF0aGNhbHtIfV9BfSJdLFsyLDAsIkVfMFxcb3BsdXMgXFxoYXR7XFxtYXRoY2Fse0h9X0F9Il0sWzAsMSwiXFxoYXR7XFxtYXRoY2Fse0h9X0F9Il0sWzIsMSwiXFxoYXR7XFxtYXRoY2Fse0h9X0F9Il0sWzAsMiwidSJdLFswLDEsIkZfMCBcXG9wbHVzIDEiXSxbMSwzLCJ1Il0sWzIsMywiRyJdXQ==
\begin{tikzcd}
	{\mathcal{E}\oplus \hat{\mathcal{H}_A}} && {\mathcal{E}\oplus \hat{\mathcal{H}_A}} \\
	{\hat{\mathcal{H}_A}} && {\hat{\mathcal{H}_A}}
	\arrow["u", from=1-1, to=2-1]
	\arrow["{F \oplus \mathcal{I}}", from=1-1, to=1-3]
	\arrow["u", from=1-3, to=2-3]
	\arrow["G", from=2-1, to=2-3]
\end{tikzcd}
 \]
 On $C([0,1],\hat{\HA})$, let $\Tilde{G}$ denote the operator on it given by $\Tilde{G}(f)(t):=G(f(t)) ~\forall~ t \in [0,1]$. Note that under isomorphisms described in \cref{lem: characterizing B and K}, $\Tilde{G}$ corresponds to the constant function $G$ on $X$.
 
 Now, note that the following restricted diagram commutes too:
 \[
 % https://q.uiver.app/#q=WzAsNCxbMCwwLCJFXzAiXSxbMiwwLCJFXzAiXSxbMCwxLCJ1KEVfMCkiXSxbMiwxLCJ1KEVfMCkiXSxbMCwyLCJ1Il0sWzAsMSwiRl8wICJdLFsyLDMsIkdcXG1pZF97dShFXzApfSJdLFsxLDMsInUiXV0=
\begin{tikzcd}
	{\mathcal{E}} && {\mathcal{E}} \\
	{u(\mathcal{E})} && {u(\mathcal{E})}
	\arrow["u", from=1-1, to=2-1]
	\arrow["{F}", from=1-1, to=1-3]
	\arrow["{G\mid_{u(\mathcal{E})}}", from=2-1, to=2-3]
	\arrow["u", from=1-3, to=2-3]
\end{tikzcd}
\]
Setting $\E':=u(\E)$, it suffices to exhibit a concordance between $(\hat{\HA},G)$ and $(\E',G \mid_{\E'})$. This can be built as follows: Let 
\[
    \h=\{e:[0,1] \rightarrow \HA ~\text{continuous}~ \mid e(0) \in \E' \},
\]
and define consider $\Tilde{G}\mid_{\h}:\h \rightarrow \h$.

    Then, as $\E'$ is complementable in $\hat{\HA}$, we conclude that $\Tilde{G}\mid_{\h}$ is adjointable with $\Tilde{G}\mid_{\h}^*=\tilde{G}^*\mid_{\h}$. Furthermore, considering a parametrix $F_0'$ to $F_0$, we get $G':=(uF_0'u^*) \oplus U^{-1}$ is a parametrix to $G$, whence the operator $\Tilde{G'}(f)(t):=G'(f(t))$ is a parametrix to $\Tilde{G}$. It can also be seen that $\Tilde{G}'\mid_{\h}:\h \rightarrow \h$ is an adjointable operator. We finally observe that $1\mid_{\h}-\Tilde{G}'\mid_{\h}\Tilde{G}\mid_{\h}$ and $1\mid_{\h}-\Tilde{G}\mid_{\h}\Tilde{G}'\mid_{\h}$ are elements of $\com_{A[0,1]}(\h)$. But note that $\operatorname{im}(1-GG') \subseteq \E$ (here's where we crucially use the invertibility of $U$), and the rest follows from \cref{lem: restricting compacts}.

    Thus, we conclude that $(\h,\Tilde{G}\mid_{\h})$ is a Fredholm $A[0,1]$ cycle. It is also clear that $(\h\hat\otimes_{\operatorname{ev}_0}A,\operatorname{ev}_0(\Tilde{G}\mid_{\h})) \cong (\E',G\mid_{\E'})$, and $(\h\hat\otimes_{\operatorname{ev}_1}A,\operatorname{ev}_1(\Tilde{G}\mid_{\h})) \cong (\hat{\HA},G)$.
\end{proof}
\begin{lem}
    Concordance is an equivalence relation.
\end{lem}
\begin{proof}
    Throughout this proof, we use the notation of \cref{lem: characterizing B and K}.
    \begin{itemize}
        \item Reflexivity: To see that $(\E,F) \sim (\E,F)$, just consider the concordance $(\E \hat\otimes C[0,1],F \hat\otimes 1)$.
        \item Symmetry: Suppose $(\E_0,F_0) \sim (\E_1,F_1)$, with $(\E,F)$ being a concordance. Then, $(\E\hat\otimes_{\varphi}A[0,1],F\hat\otimes_{\varphi}\operatorname{id})$ is a concordance between $(\E_1,F_1)$ and $({\E_0},F_0)$, where $\varphi: A[0,1] \rightarrow A[0,1]$ is the order reversing isomorphism.         \item Transitivity: Suppose $(\E_0,F_0) \sim (\E_1,F_1)$, and $(\E_1,F_1) \sim (\E_2,F_2)$. We first consider the case where the concordances are of the form $(\hat\h_{A[0,1]},F')$ and $(\hat\h_{A[1,2]},F'')$ respectively. This will contain the bulk of the proof; the general case will use some of the ideas used here, and won't be spelled out in detail.
        
        From \cref{lem: characterizing B and K}, we can view $F'$ and $F''$ as strictly continuous maps $p:[0,1] \rightarrow \mathcal{B}_A(\hat\HA)$ and $q:[1,2] \rightarrow \mathcal{B}_A(\hat\HA) $ respectively. Now, from the definition of concordance, there exists even unitaries $u,v$ making the following diagrams commute:
        \[
      % https://q.uiver.app/#q=WzAsOCxbMCwwLCJcXG1hdGhjYWx7RX1fMSJdLFsxLDAsIlxcbWF0aGNhbHtFfV8xIl0sWzAsMSwiXFxoYXR7XFxtYXRoY2Fse0h9X0F9Il0sWzEsMSwiXFxoYXR7XFxtYXRoY2Fse0h9X0F9Il0sWzMsMCwiXFxtYXRoY2Fse0V9XzEiXSxbNCwwLCJcXG1hdGhjYWx7RX1fMSJdLFszLDEsIlxcaGF0e1xcbWF0aGNhbHtIfV9BfSJdLFs0LDEsIlxcaGF0e1xcbWF0aGNhbHtIfV9BfSJdLFswLDEsIkZfMSJdLFsyLDMsInAoMSkiXSxbMCwyLCJ1IiwyXSxbMSwzLCJ1Il0sWzQsNSwiRl8xIl0sWzQsNiwidiIsMl0sWzUsNywidiJdLFs2LDcsInEoMSkiXV0=
\begin{tikzcd}
	{\mathcal{E}_1} & {\mathcal{E}_1} && {\mathcal{E}_1} & {\mathcal{E}_1} \\
	{\hat{\mathcal{H}_A}} & {\hat{\mathcal{H}_A}} && {\hat{\mathcal{H}_A}} & {\hat{\mathcal{H}_A}}
	\arrow["{F_1}", from=1-1, to=1-2]
	\arrow["{p(1)}", from=2-1, to=2-2]
	\arrow["u"', from=1-1, to=2-1]
	\arrow["u", from=1-2, to=2-2]
	\arrow["{F_1}", from=1-4, to=1-5]
	\arrow["v"', from=1-4, to=2-4]
	\arrow["v", from=1-5, to=2-5]
	\arrow["{q(1)}", from=2-4, to=2-5]
\end{tikzcd}
        \]
    Combining the two diagrams, we obtain \[
   % https://q.uiver.app/#q=WzAsNCxbMCwwLCJcXGhhdHtcXG1hdGhjYWx7SH1fQX0iXSxbMSwwLCJcXGhhdHtcXG1hdGhjYWx7SH1fQX0iXSxbMCwxLCJcXGhhdHtcXG1hdGhjYWx7SH1fQX0iXSxbMSwxLCJcXGhhdHtcXG1hdGhjYWx7SH1fQX0iXSxbMCwxLCJwKDEpIl0sWzAsMiwidnVeKiIsMl0sWzIsMywicSgxKSIsMl0sWzEsMywidnVeKiJdXQ==
\begin{tikzcd}
	{\hat{\mathcal{H}_A}} & {\hat{\mathcal{H}_A}} \\
	{\hat{\mathcal{H}_A}} & {\hat{\mathcal{H}_A}}
	\arrow["{p(1)}", from=1-1, to=1-2]
	\arrow["{vu^*}"', from=1-1, to=2-1]
	\arrow["{q(1)}"', from=2-1, to=2-2]
	\arrow["{vu^*}", from=1-2, to=2-2]
\end{tikzcd}
    \]
     Now, we consider the map \begin{align*}
         \alpha:[0,2] &\rightarrow \mathcal{B}_A(\hat{\HA}) \\
         t &\mapsto \begin{cases}
             vu^*p(t)uv^* ~\text{if}~ 0 \leq t \leq 1 \\
             q(t) ~\text{if}~ 1 \leq t \leq 2
         \end{cases}
     \end{align*}  
     The well definedness of $\alpha$ as a function is clear.
    We claim that $t$ corresponds to a Fredholm operator on $C([0,2],\hat{\HA})$ under the isomorphisms of \cref{lem: characterizing B and K}. The strict continuity of $\alpha$ follows from the strict continuity of $p$ and $q$. As $p$ and $q$ corresponded to Fredholm operators, there exists parametrices $p':[0,1]\rightarrow \mathcal{B}_A(\hat{\HA})$ and $q':[1,2] \rightarrow \mathcal{B}_A(\hat{\HA})$ respectively. Moreover, we can demand that $q'(1)=vu^*p'(1)uv^*$ from the second part of \cref{cor: Fred on function algebra}. It follows that the following map is a parametrix to $\alpha$
    \begin{align*}
         \beta:[0,2] &\rightarrow \mathcal{B}_A(\hat{\HA}) \\
         t &\mapsto \begin{cases}
             vu^*p'(t)uv^* ~\text{if}~ 0 \leq t \leq 1 \\
             q'(t) ~\text{if}~ 1 \leq t \leq 2
         \end{cases}
     \end{align*}   
     Thus, $(\hat{\h_{A[0,2]}},\Psi_A(\alpha))$ is a concordance between $(\E_0,F_0)$ and $(\E_2,F_2)$.
           
        Now, we come back to the general case. Let $(\E',F')$ and $(\E'',F'')$ be concordances between $(\E_0,F_0)$ and $(\E_1,F_1)$, and between $(\E_1,F_1)$ and $(\E_2,F_2)$ respectively.
        Then, $(\E' \oplus \hat{\h_{A[0,1]}},F' \oplus 1)$ is a concordance between $(\E_0 \oplus \hat{\HA},F_0 \oplus 1)$ and $(\E_1 \oplus \hat{\HA},F_1 \oplus 1)$, and $(\E'' \oplus \hat{\h_{A[2,3]}},F'' \oplus 1)$ is a concordance between $(\E_1 \oplus \hat{\HA},F_1 \oplus 1)$ and $(\E_2 \oplus \hat{\HA},F_2 \oplus 1)$.

 We now invoke Kasparov stabilization to conclude that there exist Fredholm $A[0,1]$- cycles $(\hat{\h_{A[0,1]}},G_1)$ and $(\hat{\h_{A[1,2]}}, G_2)$ such that the former is a concordance between $(\E_0 \oplus \hat{\HA},F_0 \oplus 1)$ and $(\E_1 \oplus \hat{\HA},F_1 \oplus 1)$, and the latter is a concordance between $(\E_1 \oplus \hat{\HA},F_1 \oplus 1)$ and $(\E_2 \oplus \hat{\HA},F_2 \oplus 1)$. The first part of the proof now shows that $(\E_0 \oplus \hat{\HA},F_0 \oplus 1)$ and $(\E_2 \oplus \hat{\HA},F_2 \oplus 1)$ are concordant. 

 We still have to come up with a concordance between $(\E_0,F_0)$ and $(\E_1,F_1)$. We first note the concordance between $(\E_0,F_0)$ and $(\E_0 \oplus \hat{\HA},F_0 \oplus 1)$ given by \cref{lem: add nondeg}. Similarly, we can also cook up a concordance between  $(\E_1 \oplus \hat{\HA},F_1 \oplus 1)$ and $(\E_1,F_1)$. Just like the first part of the proof, we can concatenate the paths corresponding to the concordances so as to build a concordance between $(\E_0,F_0)$ and $(\E_2,F_2)$, with the module being of the form $\{e: [0,3] \rightarrow \hat{\HA} ~\text{continuous}~ \mid e(0) \in \E_0', e(3) \in \E_2' \}$, where $\E_0'$ and $\E_2'$ are complementable copies of $\E_0$ and $\E_2$ respectively inside $\hat{\HA}$, as was considered in the proof of \cref{lem: add nondeg}.
    \end{itemize}
\end{proof}

\iffalse

\fi

\begin{defi}
    Denote by $K_0^{\operatorname{Fr}}(A)$ the abelian monoid of concordance classes of Fredholm families over $A$. 
\end{defi}
\begin{lem}
    $K_0^{\operatorname{Fr}}(A)$ is a group.
\end{lem}
\begin{proof}
    The only thing we need to show is the existence of inverses. We claim that $(\E,-\iota,-F)$ is the inverse to $(\E,\iota,F)$. To that end, we have to show that $(\E \oplus \E, \iota \oplus -\iota, F \oplus -F) \sim (0,0)$. We define a norm continuous function $p:[0,1]\rightarrow \mathcal{B}(\E \oplus \E)$ as follows:
    \begin{align*}
        t \mapsto \begin{pmatrix}
            \cos(\frac{\pi}{2}t)F & \sin(\frac{\pi}{2}t) \\
            \sin(\frac{\pi}{2}t) & -\cos(\frac{\pi}{2}t)F
        \end{pmatrix}
    \end{align*}
    It is direct to see that $p(t)$ is an odd, self-adjoint operator on $\E \oplus \E$. We claim that $p(t)$ is invertible for $t>0$. Indeed, we have
    \[
    p(t)^*p(t)=\begin{pmatrix}
        \cos^2(\frac{\pi}{2}t)F^2+\sin^2(\frac{\pi}{2})t & 0 \\
        0 & \cos^2(\frac{\pi}{2}t)F^2+\sin^2(\frac{\pi}{2})t
    \end{pmatrix} 
    \geq \sin^2(\frac{\pi}{2}t).\begin{pmatrix}
        1 & 0 \\
        0 & 1
    \end{pmatrix},
    \]
    since $\begin{pmatrix}
         \cos^2(\frac{\pi}{2}t)F^2 & 0 \\
         0 &  \cos^2(\frac{\pi}{2}t)F^2
    \end{pmatrix}$ is positive. The result now follows since a positive element $a$ in a unital $C^*$-algebra is invertible if and only if $a \geq \varepsilon.1$ for some $\varepsilon>0$ (here, $a=p(t)p(t)^*$, and $\varepsilon=\sin^2(\frac{\pi}{2}t)>0$). Thus, $p(t)$ is invertible too.

    Thus, $p$ is a norm continuous path of Fredholm operators on $\E \oplus \E$, with $p(1)=\begin{pmatrix}
        0 & 1 \\
        1 & 0
    \end{pmatrix}$ being invertible. The conclusion now follows from \cref{cor: inv conc} and \cref{lem: homotopy implies concor}.
\end{proof}
\begin{lem}
    $K_0^{\operatorname{Fr}}$ is a functor from the category of unital, graded $C^*$-algebras to abelian groups.
\end{lem}
\begin{proof}
    This is a direct consequence of the functoriality of the internal tensor product construction: if $\varphi: A \rightarrow B$ is a graded $*$-homomorphism and $H$ a graded Hilbert $A$-module, then the induced graded unital $*$-homomorphism
    $
    \mathcal{B}_A(H) \xrightarrow{- \otimes_{\varphi} \operatorname{id}} \mathcal{B}_B(H \otimes_{\varphi} B)$ restricts to a graded  $*$-homomorphism
    $ \mathcal{K}_A(H) \xrightarrow{- \otimes_{\varphi} \operatorname{id}} \mathcal{K}_B(H)$ (Lemma 1.2.8 of \cite{JT}). This shows Fredholm cycles get mapped to Fredholm cycles. Concordant Fredholm cycles in $K_0^{\operatorname{Fr}}(A)$ get mapped to concordant Fredholm cycles in $K_0^{\operatorname{Fr}}(B)$ since the following diagram commutes
    \[
    % https://q.uiver.app/#q=WzAsNCxbMCwwLCJBWzAsMV0iXSxbMSwwLCJBIl0sWzAsMSwiQlswLDFdIl0sWzEsMSwiQiJdLFswLDEsIlxcb3BlcmF0b3JuYW1le2V2fV90Il0sWzIsMywiXFxvcGVyYXRvcm5hbWV7ZXZ9X3QiXSxbMSwzLCJcXHZhcnBoaSJdLFswLDIsIlxcdmFycGhpIFxcY2lyYy0iLDJdXQ==
\begin{tikzcd}
	{A[0,1]} & A \\
	{B[0,1]} & B
	\arrow["{\operatorname{ev}_t}", from=1-1, to=1-2]
	\arrow["{\operatorname{ev}_t}", from=2-1, to=2-2]
	\arrow["\varphi", from=1-2, to=2-2]
	\arrow["{\varphi \circ-}"', from=1-1, to=2-1]
\end{tikzcd}.
    \] The group structure is also preserved since the internal tensor product construction respects direct sums.
\end{proof}

\subsection{Revisiting classical K-theory}
In this subsection, we show that when $A$ is a unital, ungraded $C^*$-algebra, then $K_0^{\operatorname{Fr}}(A)$ is naturally isomorphic to $K_0(A)$, the well-studied projective module picture of $K$-theory (see \cite{WEG}, \cite{ROR} for the classical treatment). Although not completely necessary, in this subsection, we consider an ungraded picture of Fredholm cycles for the sake of convenience. We cast the definition in this setup.

In this subsection, $A$ denotes a unital and trivally graded $C^*$-algebra.
\begin{defi} \label{def: ungraded fred 1}
    A Fredholm family/cycle over $A$ is a triple $(\E,\F,F)$, where $\E,\F$ are countably generated Hilbert $A$-modules, and $F$ is an Fredholm operator from $\E$ to $\F$. 
\end{defi}
\begin{defi} \label{def: ungraded fred 2}
    The following are some relevant definitions circling Fredholm families over $A$:
    \begin{itemize}
        \item A Fredholm family $F$ is called acyclic if $F$ is invertible.
        \item If $(\E_0,\F_0,F_0)$ and $(\E_1,\F_1,F_1)$ are Fredholm families, then $(\E_0,\F_0,F_0) \oplus (\E_1,\F_1,F_1):=(\E_0 \oplus \E_1, \F_0 \oplus \F_1,F_0 \oplus F_1
        )$.
        \item Two Fredholm cycles $(\E_0,\F_0,F_0)$ and $(\E_1,\F_1,F_1)$ are said to be isomorphic if there exists unitaries $U: \E_1 \rightarrow \E_0, V:\F_0 \rightarrow F_1$ such that $VF_0U=F_1$. We write $(\E_0,F_0) \cong_{U,V} (\E_1,F_1)$ in this case. Although the unitaries $U,V$ are a crucial bit of the data, we will oftentimes be suppressing them and simply write $(\E_0,F_0) \cong (\E_1,F_1)$.
        \item Two Fredholm cycles $(\E_0,\F_0,F_0)$ and $(\E_1,\F_1,F_1)$ are said to be concordant (denoted by $(\E_0,\F_0,F_0) \sim (\E_1,\F_1,F_1)$) if there exists $(\E,\F,F)$, a Fredholm family over $A \otimes C[0,1]$, and there exists unitaries $u_0,v_0,u_1,v_1$ such that $(\E \otimes_{\operatorname{ev}_i}A, \F \otimes_{\operatorname{ev}_i}A, \operatorname{ev}_i(F)) \cong_{u_i,v_i} (\E_i,\F_i,F_i)$ for $i=0,1$, where $\operatorname{ev}_t:A[0,1] \rightarrow A$ is the homomorphism given by $\operatorname{ev}_t(f)=f(t)$. Note that in the notation we  suppress the data about the unitary isomorphisms, although they are not to be ignored.
    \end{itemize}
\end{defi}
\begin{lem}
    When $A$ is trivially graded, there is an obvious identification between the ungraded and graded versions of $K_0^{\operatorname{Fr}}(A)$.
\end{lem}
\begin{proof}
    Associating the ungraded Fredholm cycle $(\E_0,\F_0,F_0)$ to the graded one given by $(\E_0 \oplus \F_0, \begin{pmatrix}
        0 & F^*\\
        F & 0
    \end{pmatrix})$ is the required identification.
\end{proof}
The main theorem about this model is the following:

\begin{thm} \label{thm: K_0 iso}
   The natural map $\Phi: K_0(A) \rightarrow K_0^{\operatorname{Fr}}(A)$ is an isomorphism. In fact, the two functors $K_0^{\operatorname{Fr}}(-)$ and $K_0(-)$ are naturally isomorphic when restricted to unital ungraded $C^*$ algebras.
\end{thm}

Recall that the natural map $\Phi:K_0(A) \rightarrow K_0^{\operatorname{Fr}}(A)$ is given by $\Phi([\E]-[\F])=(\E,\F,0)$. The fact that this is a well-defined group homomorphism follows from the universal property of Grothendieck completion of an abelian monoid.

Before proving the theorem, we come up with an inverse map to $\Phi$. Infact, this inverse map will be a generalization of the notion of index of a Fredholm operator as defined in the last section:

\begin{defi}
    Let $(\E,\F,F) \in K_0^{\operatorname{Fr}}(A)$. We define the index of this Fredholm family as the index of F. 
\end{defi}

\textit{Proof that this is well defined:} We have to show that if $(\E_0,\F_0,F_0) \sim (\E_1,\F_1,F_1)$, then $\operatorname{ind} F_0=\operatorname{ind} F_1$. To this end, if the Fredholm $A[0,1]$-family $(\E,\F,F)$ realizes the concordance, then we claim that $\operatorname{ind} F_0=\operatorname{ind} F_1=(\operatorname{ind}F) \otimes_{\operatorname{ev}_0} A = (\operatorname{ind}F) \otimes_{\operatorname{ev}1} A$. To see the the equality between the last two quantities, note that $\operatorname{ev}_0=\operatorname{ev}_1 \circ \alpha$, where $\alpha:C([0,1],A) \rightarrow C([0,1],A)$ is the order reversing isomorphism, and use that the pushout construction is functorial. We next show that $\operatorname{ind} F_0 = (\operatorname{ind} F) \otimes_{\operatorname{ev}_0} A$.

Choose a taming $(n,g)$ of ${F}$, and set $\Tilde{F}:=F+g$. First of all, if $(\operatorname{ev}_0 (\E),\operatorname{ev}_0 (\F),\operatorname{ev}_0(F)) \cong_{u,v} (\E_0,\F_0,F)$, then the commutativity of the following diagram shows that $(n,v \circ \operatorname{ev}_0(g))$ is a taming of $F_0$:
\[
% https://q.uiver.app/#q=WzAsNCxbMCwwLCJcXG9wZXJhdG9ybmFtZXtldn1fMChcXG1hdGhjYWx7RX0pIFxcb3BsdXMgQV5uIl0sWzIsMCwiXFxvcGVyYXRvcm5hbWV7ZXZ9XzAoXFxtYXRoY2Fse0Z9KSJdLFswLDEsIlxcbWF0aGNhbHtFfV8wIFxcb3BsdXMgQV5uIl0sWzIsMSwiXFxtYXRoY2Fse0Z9XzAiXSxbMCwyLCJ1IFxcb3BsdXMgMSJdLFsxLDMsInYiXSxbMCwxLCJcXG9wZXJhdG9ybmFtZXtldn1fMChGKSBcXG9wbHVzIFxcb3BlcmF0b3JuYW1le2V2fV8wKGcpIiwwLHsic3R5bGUiOnsiaGVhZCI6eyJuYW1lIjoiZXBpIn19fV0sWzIsMywiRl8wIFxcb3BsdXMgdiBcXGNpcmNcXG9wZXJhdG9ybmFtZXtldn1fMChnKSIsMCx7InN0eWxlIjp7ImhlYWQiOnsibmFtZSI6ImVwaSJ9fX1dXQ==
\begin{tikzcd}
	{\operatorname{ev}_0(\mathcal{E}) \oplus A^n} && {\operatorname{ev}_0(\mathcal{F})} \\
	{\mathcal{E}_0 \oplus A^n} && {\mathcal{F}_0}
	\arrow["{u \oplus 1}", from=1-1, to=2-1]
	\arrow["v", from=1-3, to=2-3]
	\arrow["{\operatorname{ev}_0(F) \oplus \operatorname{ev}_0(g)}", two heads, from=1-1, to=1-3]
	\arrow["{F_0 \oplus v \circ\operatorname{ev}_0(g)}", two heads, from=2-1, to=2-3]
\end{tikzcd}.
\] Denote by $g_0$, the operator $v \circ \operatorname{ev}_0(g)$. Thus, to show that $\operatorname{ind} F_0 = (\operatorname{ind} F) \otimes_{\operatorname{ev}_0} A$, it suffices to show that $\ker(F_0+g_0) \cong \ker(\Tilde{F}) \otimes A$.

Now, we have the short exact sequence of $A[0,1]$-modules \[% https://q.uiver.app/#q=WzAsNSxbMCwwLCIwIl0sWzEsMCwiXFxrZXIoXFx0aWxkZXtGfSkiXSxbMiwwLCJcXG1hdGhjYWx7RX1cXG9wbHVzIEFebiJdLFs0LDAsIjAiXSxbMywwLCJcXG1hdGhjYWx7Rn0iXSxbMCwxXSxbMSwyXSxbMiw0LCJcXHRpbGRle0Z9Il0sWzQsM11d
\begin{tikzcd}
	0 & {\ker(\tilde{F})} & {\mathcal{E}\oplus A[0,1]^n} & {\mathcal{F}} & 0
	\arrow[from=1-1, to=1-2]
	\arrow[from=1-2, to=1-3]
	\arrow["{\tilde{F}}", from=1-3, to=1-4]
	\arrow[from=1-4, to=1-5]
\end{tikzcd}.\] Basic Hilbert module theory yields that, after taking adjoints, we have the short exact sequence \begin{equation}\label{eqn: ses}
    % https://q.uiver.app/#q=WzAsNSxbMCwwLCIwIl0sWzEsMCwiXFxtYXRoY2Fse0Z9Il0sWzIsMCwiXFxtYXRoY2Fse0V9XFxvcGx1cyBBXm4iXSxbMywwLCJ7XFxvcGVyYXRvcm5hbWV7aW19XFx0aWxkZXtGfV4qfV5cXHBlcnAiXSxbNCwwLCIwIl0sWzAsMV0sWzEsMiwiXFx0aWxkZXtGfV4qIl0sWzIsM10sWzMsNF1d
\begin{tikzcd}
	0 & {\mathcal{F}} & {\mathcal{E}\oplus A[0,1]^n} & {{\operatorname{im}\tilde{F}^*}^\perp} & 0
	\arrow[from=1-1, to=1-2]
	\arrow["{\tilde{F}^*}", from=1-2, to=1-3]
	\arrow[from=1-3, to=1-4]
	\arrow[from=1-4, to=1-5]
\end{tikzcd}.
\end{equation} Of course, here we crucially needed that $\tilde{F}$ is surjective, whence its adjoint is injective. Moreover, surjectivity of $\Tilde{F}$ implies that $\operatorname{im}\Tilde{F}^*$ is a closed submodule of $\E$ and $\ker \tilde{F} = \operatorname{im} {\Tilde{F}{}^{*}}^{\perp}$ (see Theorem 3.2 of \cite{LANCE}). Furthermore, as $\ker \Tilde{F}$ is given to be projective, \cref{eqn: ses} is split exact. Whence, upon applying the functor $- \otimes_{\operatorname{ev}_0} A$, we still have the split exact sequence
\begin{equation*}\label{eqn: sess}
    % https://q.uiver.app/#q=WzAsNSxbMCwwLCIwIl0sWzEsMCwiXFxtYXRoY2Fse0Z9IFxcb3RpbWVzX3tcXG9wZXJhdG9ybmFtZXtldn1fMH0gQSJdLFsyLDAsIihcXG1hdGhjYWx7RX1cXG9wbHVzIEFebikgXFxvdGltZXNfe1xcb3BlcmF0b3JuYW1le2V2fV8wfSBBIl0sWzMsMCwiXFxrZXIgXFx0aWxkZXtGfSBcXG90aW1lc197XFxvcGVyYXRvcm5hbWV7ZXZ9XzB9ICBBIl0sWzQsMCwiMCJdLFswLDFdLFsxLDIsIlxcdGlsZGV7Rn1eKiBcXG90aW1lc197XFxvcGVyYXRvcm5hbWV7ZXZ9XzB9IEEiXSxbMiwzXSxbMyw0XV0=
\noindent\begin{tikzcd}
	0 & {\mathcal{F} \otimes_{\operatorname{ev}_0} A} & {(\mathcal{E}\oplus A[0,1]^n) \otimes_{\operatorname{ev}_0} A} & {\ker \tilde{F} \otimes_{\operatorname{ev}_0}  A} & 0
	\arrow[from=1-1, to=1-2]
	\arrow["{\tilde{F}^* \otimes_{\operatorname{ev}_0} A}", from=1-2, to=1-3]
	\arrow[from=1-3, to=1-4]
	\arrow[from=1-4, to=1-5]
\end{tikzcd}.
\end{equation*}
The above sequence is equivalent to \[% https://q.uiver.app/#q=WzAsNSxbMCwwLCIwIl0sWzEsMCwiXFxtYXRoY2Fse0Z9XzAiXSxbMiwwLCIoXFxtYXRoY2Fse0V9XzBcXG9wbHVzIEFebl8wKSJdLFszLDAsIlxca2VyIFxcdGlsZGV7Rn0gXFxvdGltZXNfe1xcb3BlcmF0b3JuYW1le2V2fV8wfSAgQSJdLFs0LDAsIjAiXSxbMCwxXSxbMSwyLCJGXzBeKiJdLFsyLDNdLFszLDRdXQ==
\begin{tikzcd}
	0 & {\mathcal{F}_0} & {\mathcal{E}_0\oplus A^n} & {\ker \tilde{F} \otimes_{\operatorname{ev}_0}  A} & 0.
	\arrow[from=1-1, to=1-2]
	\arrow["{F_0^* + g_0^*}", from=1-2, to=1-3]
	\arrow[from=1-3, to=1-4]
	\arrow[from=1-4, to=1-5]
\end{tikzcd}\] We also have the short exact sequence \[% https://q.uiver.app/#q=WzAsNSxbMCwwLCIwIl0sWzEsMCwiXFxtYXRoY2Fse0Z9XzAiXSxbMiwwLCJcXG1hdGhjYWx7RX1fMFxcb3BsdXMgQV5uIl0sWzMsMCwiXFxvcGVyYXRvcm5hbWV7aW19eyhGXzAgKyBnXzApXip9XlxccGVycCJdLFs0LDAsIjAiXSxbMCwxXSxbMSwyLCIoRl8wK2dfMCleKiJdLFsyLDNdLFszLDRdXQ==
\begin{tikzcd}
	0 & {\mathcal{F}_0} & {\mathcal{E}_0\oplus A^n} & {\operatorname{im}{(F_0 + g_0)^*}^\perp} & 0
	\arrow[from=1-1, to=1-2]
	\arrow["{(F_0+g_0)^*}", from=1-2, to=1-3]
	\arrow[from=1-3, to=1-4]
	\arrow[from=1-4, to=1-5]
\end{tikzcd}\] since $(n,g_0)$ is a taming of $F_0$.
As $(F_0^*+g_0^*)=(F_0+g_0)^*$, and $\operatorname{im}{(F_0 + g_0)^*}^\perp = \ker (F_0+g_0)$, we conclude that $\ker (F_0+g_0) \cong \ker \tilde{F} \otimes_{\operatorname{ev}_0} A$, as required.
\begin{lem}
    $\operatorname{ind}: K_0^{\operatorname{fr}}(A) \rightarrow K_0(A)$ is a group homomorphism.
\end{lem}
\begin{proof}
    Let $(\E_0,\F_0,\F_0)$ and $(\E_1,\F_1,F_1)$ be two elements of $K_0^{\operatorname{Fr}}(A)$. Let $(n_0,g_0)$ and $(n_1,g_1)$ be tamings of them respectively. Then, it immediately follows from basic properties of projective modules that $(n_0+n_1,g_0+g_1)$ is a taming of $(\E_0 \oplus \E_1, \F_0 \oplus \F_1, F_0 \oplus F_1)$. The rest follows from \cref{lem: ind well-defined}.
\end{proof}
\begin{proof}[Proof (of \cref{thm: K_0 iso})]
    We show that $\operatorname{ind} \circ \Phi=\operatorname{id}$, and $\operatorname{ind}$ is injective.
    
    For showing that $\operatorname{ind} \circ \Phi = \operatorname{id}$, it suffices to observe that $\operatorname{ind} \circ \Phi (\E,0,0)=(\E,0,0)$ for any finitely generated projective $A$-module $\E$. This is clear from the well-definedness of the index map.
    
    To show injectivity of $\operatorname{ind}:K_0^{\operatorname{Fr}}(A) \rightarrow K_0(A)$, let $(\E,\F,F)$ be a Fredholm cycle with $\operatorname{ind} F=0$. Choose a taming $(n,g)$ of $F$. By basic Hilbert module theory, we have the decomposition $\E \oplus A^n= \ker (F+g) \oplus \ker (F+g)^\perp$, and the index of $F$ being $0$ yields that $\ker(F+g)$ and $A^n$ are stably isomorphic. Furthermore, $F+g:\ker(F+g)^\perp \rightarrow \F$ is an isomorphism. Thus, there exists an $m \in \N$, and an isomorphism of the form $(F+g) \oplus U:\ker(F+g)^\perp \oplus (\ker (F+g) \oplus A^m) \rightarrow \F \oplus (A^n \oplus A^m)$. This shows that $F+g:\E \oplus A^n \oplus A^m \rightarrow \F \oplus A^n \oplus A^m$ is a compact perturbation of an invertible, whence, so is $F \oplus 1_{n+m}:\E \oplus A^{n+m} \rightarrow \F \oplus A^{n+m}$. This shows that, $(\E \oplus A^{n+m},\F \oplus A^{n+m},F\oplus 1_{n+m}) \sim (0,0,0)$. But since $(A^{n+m},A^{n+m},1_{n+m})=(0,0,0)$, we conclude that $(\E,\F,F) \sim (0,0,0)$, as desired. 
\end{proof}
\subsection{The main theorem}
We want to come up with a natural isomorphism between $K_0^{\operatorname{sp}}(A)$ and $K_0^{\operatorname{Fr}}(A)$ for unital, $\Z/2\Z$-graded (Real) $C^*$-algebras $A$. For technical reasons, we come up with the following substitutes of $\K(A)$ and $K_0^{\operatorname{Fr}}(A)$:
\begin{itemize}
    \item We replace $\K(A)$ by the space of all graded (Real) $*$-homomorphisms from $C_0(-1,1) \rightarrow \com(\hat\HA)$, where the domain is still equipped with the even-odd grading, and the Real structure is given by pointwise conjugation. We denote this new space by $\tilde{\K}(A)$, and its homotopy groups as ${\Tilde{K}_{n}}^{\operatorname{sp}}(A)$ for $n \geq 0$. $\tilde{\K}(A)$ is homeomorphic to $\K(A)$ in light of the odd, increasing homeomorphism from $\R$ to $(-1,1)$ given by $x \mapsto \frac{x}{\sqrt{x^2+1}}$, and inverse given by $y \mapsto \frac{y}{\sqrt{1-y^2}}$.
    \item We replace $K_0^{\operatorname{Fr}}(A)$ with its ``essentially unitary counterpart", wherein instead of considering odd, self-adjoint Fredholm operators, we only consider odd, self-adjoint operators $F$ such that $\lVert F \rVert \leq 1$, and $F^2-1$ is compact. We can thereby see that their concordance classes form a group, which we denote by $K_0^{\operatorname{eu}}(A)$.
\end{itemize}
That $K_0^{\operatorname{eu}}(A)$ is naturally isomorphic to $K_0^{\operatorname{Fr}}(A)$ requires some work, which we do now.
\begin{prop}
    The natural map $K_0^{\operatorname{eu}}(A) \rightarrow K_0^{\operatorname{Fr}}(A)$ is a group isomorphism.
\end{prop}
\begin{proof}
    Since every essentially unitary operator is Fredholm, the well definedness of the map is clear. We only need to check that it's bijective.

    We first deal with surjectivity. Let $(\E,F)$ be a Fredholm cycle. Then, we first claim that $F$ is homotopic through a path of odd, self-adjoint Fredholm operators to an operator $G$ such that $\lVert G \rVert \leq 1$, and $G^2-1$ is compact. Indeed, as $F$ is self-adjoint and Fredholm, choose some $c>0$ such that $(-c,c) \cap \sigma_{\operatorname{ess}}(F) = \emptyset$. Consider the operator $h_c(F)$ in $\mathcal{B}(\E)$, where
    \begin{align*}
        h_c: \R &\rightarrow \R \\
        x &\mapsto \begin{cases}
            c ~\text{if}~ x \geq c \\
            x ~\text{if}~ -c \leq x \leq c \\
            -c ~\text{if}~ x \leq -c \\
        \end{cases}
    \end{align*}
    We have $h_c(F)^2-c^2$ is a compact operator, and $\lVert h_c(F) \rVert \leq c$. Thus, taking $G=\frac{h_c(F)}{c}$ works. To see that $G$ is homotopic to $F$, note that $F=h_{\lVert F \rVert}(F)$, and the functions $h_{\lVert F \rVert}$ and $h_{c}$ are homotopic in $C_b(\R)$, whence so are $F$ and $h_c(F)$, and the latter is trivially homotopic to $\frac{h_c(F)}{c}$. The surjectivity assertion now follows from (a graded analogue of) \cref{lem: homotopy implies concor}. 
    
    Now, we move towards injectivity. Suppose $(\E_0,F_0) \sim^{\operatorname{Fr}} (\E_1,F_1) $. Let $(\E,F)$ be a concordance between them. Analogous to the argument for the sujectivity part, we choose some $c>0$ such that $(-c,c) \cap \sigma_{\operatorname{ess}}(F) = \emptyset$. Then, we have $\frac{h_c(F)}{c}$ and $F$ are homotopic, with $(\E,\frac{h_c(F)}{c})$ being an ``essentially unitary" cycle over $A[0,1]$. We claim that $(\E \hat\otimes_{\operatorname{ev}_j} A, \operatorname{ev}_j(\frac{h_c(F)}{c})) \sim^{\operatorname{eu}} (\E \hat\otimes_{\operatorname{ev}_j} A,\operatorname{ev}_j(F))$ for $j=0,1$. We in fact show that $\operatorname{ev}_j(F)$ and $\operatorname{ev}_j(\frac{h_c(F)}{c})$ are homotopic through a path of odd, self-adjoint essentially unitary operators having norm less than or equal to one. First, note that $c \leq 1$, since $\operatorname{ev}_j(F)^2-1$ is compact. Note that, by properties of the internal tensor product, $\operatorname{ev}_j(\frac{h_c(F)}{c})=\frac{h_c(\operatorname{ev}_j(F))}{c}$. As $\lVert \operatorname{ev}_j(F) \rVert \leq 1$, we have $\operatorname{ev}_j(F)=h_1(\operatorname{ev}_j(F))$. Then the path $t \mapsto \frac{h_{(c-1)t+1}(\operatorname{ev}_j(F))}{(c-1)t+1}$ is the one we are after. Injectivity now follows from the transitivity of the ``essentially unitary" concordance relation.

\end{proof}

 Now, it suffices to come up with a natural map $\Tilde{K_0}^{\operatorname{sp}}(A) \rightarrow K_0^{\operatorname{eu}}(A)$ for every $\Z/2\Z$ graded (Real) $C^*$-algebra.

Let $\varphi \in \K(A)$. Recall that have the identification $\com(\hat\HA) \cong A \hat\otimes \com(\hat{l^2})$ (see \cite{BLA}). Set $\h_\varphi:=\overline{\varphi(\s)(\com(\hat\HA))}$.
\begin{lem} \label{lem: define phi nd}
    $\varphi$ restricts to a non-degenerate graded $*$-homomorphism $\varphi^{\operatorname{nd}}:C_0(-1,1) \rightarrow \com(\h_\varphi)$.
\end{lem}
\begin{proof}
    The definition of $\varphi^{\operatorname{nd}}$ is simply by restriction; the point is to observe that it is still an $A$-compact operator. But this is clear from a graded analogue of \cref{lem: restricting compacts}.
    
    The non-degeneracy of $\varphi^{\operatorname{nd}}$ is clear: it suffices to show that every element of the form $\varphi(f)x, f \in \sr, x \in \hat\HA$ can be expressed as $\varphi^{\operatorname{nd}}(f')(x'), f' \in \sr, x' \in \h_\varphi$. Now, it is easy to see that every element of $\sr$ can be written as product of two elements of $\sr$. Writing $f=g'g''$, we can set $f'=g'$, and $x'=\varphi(g'')(x)$
\end{proof}
The reason we care about non-degeneracy at this juncture is due to the following:
\begin{prop} \label{prop: extend homomorphism}
    Let $A,B,C$ be $\Z/2\Z$-graded $C^*$-algebras, such that $A$ is an ideal in $B$. Let $E$ be a $\Z/2\Z$-graded Hilbert $C$-module. Suppose that $\varphi:A \rightarrow \mathcal{B}(E)$ is a nondegenerate graded $*$-homomorphism. Then, $\varphi$ extends uniquely to a graded $*$-homomorphism $\tilde{\varphi}:B \rightarrow \mathcal{B}(E)$. Moreover, when $B$ is unital, the extension $\Tilde{\varphi}$ is unital.
\end{prop}
\begin{proof}
    Essentially \cite{LANCE} Proposition 2.1. That $\Tilde{\varphi}$ is graded, and the unitality assertion can be seen from its proof and the uniqueness of $\Tilde{\varphi}$.
\end{proof}
Thus, we have a commutative diagram as follows:
% https://q.uiver.app/#q=WzAsMyxbMCwwLCJcXG1hdGhjYWx7U30iXSxbMiwwLCJcXG1hdGhiYntCfShcXGhhdHtcXG1hdGhjYWx7SH1fQX0pIl0sWzAsMSwiXFxtYXRoY2Fse019KFxcbWF0aGNhbHtTfSkiXSxbMCwxLCJcXHZhcnBoaV57XFxvcGVyYXRvcm5hbWV7bmR9fSJdLFswLDIsIlxcaW90YSIsMl0sWzIsMSwiXFx0aWxkZXtcXHZhcnBoaV57XFxvcGVyYXRvcm5hbWV7bmR9fX0iLDIseyJzdHlsZSI6eyJib2R5Ijp7Im5hbWUiOiJkYXNoZWQifX19XV0=
\[\begin{tikzcd}
	{C_0(-1,1)} && {\mathbb{B}(\mathcal{H}_\varphi}) \\
	{C[-1,+1]}
	\arrow["{\varphi^{\operatorname{nd}}}", from=1-1, to=1-3]
	\arrow["\operatorname{inc}"', from=1-1, to=2-1]
	\arrow["{\tilde{\varphi}}"', dashed, from=2-1, to=1-3]
\end{tikzcd}\]
Moreover, $\tilde{\varphi}$ is unital from \cref{prop: extend homomorphism}. 
\begin{cor} \label{cor: cor to technical thm}
    Let $\iota_*:\mathcal{K}(\mathcal{H}_\varphi) \rightarrow \mathcal{K}(\hat\HA)$ be the map as described in \cref{lem: extending co_pacts along isometries} associated to the (not necessarily adjointable) isometry $\iota: \mathcal{H}_\varphi \hookrightarrow \hat\HA$. Then, $\iota_* \circ \varphi^{\operatorname{nd}}=\varphi:\s \rightarrow \mathcal{K}(\hat\HA)$. 
\end{cor}
\begin{proof}
    It is a direct consequence of \cref{thm: technical}.
\end{proof}
\begin{lem} \label{lem: phi tilde x}
    $\Tilde{\varphi}(x)$ is an odd, self-adjoint operator such that $\lVert \Tilde{\varphi}(x) \rVert \leq 1$, and $\Tilde{\varphi}(x)^2-1$ is compact. Moreover, $\varphi^{\operatorname{nd}}$ is the functional calculus of $\Tilde{\varphi}(x)$.
\end{lem}
\begin{proof}
    The first part of the lemma follows immediately since the function $x \mapsto x$ on $C[-1,1]$ has all those properties.

    To see that $\varphi^{\operatorname{nd}}=\Phi_{\Tilde{\varphi}(x)}:\sr \rightarrow \com(\h_\varphi)$, it suffices to check it on all polynomials in $C_0(-1,1)$ which vanish at $+1$ and $-1$ (as its dense in $C_0(-1,1)$ by Stone-Weierstrass). Now, for any polynomial $p$, $\Phi_{\Tilde{\varphi}(x)}(p)=p(\Tilde{\varphi}(x))=\Tilde{\varphi}(p)$. When $p \in C_0(-1,1)$, $\Tilde{\varphi}(p)=\varphi^{\operatorname{nd}}(p)$, and the claim follows.
\end{proof}
\begin{cor}
    [Definition of our natural map]
    The assignment 
    $\varphi \mapsto ({\h_\varphi},\Tilde{\varphi}(x))$ is a natural set map from $\Tilde{\K}(A) \rightarrow K_0^{\operatorname{eu}}(A)$, which descends to a well-defined natural monoid homomorphism 
    \begin{align*}
         \F_A: &\Tilde{K}_0^{\operatorname{sp}}(A) \rightarrow K_0^{\operatorname{eu}}(A) \\
        &[\varphi] \mapsto ({\h_\varphi},\Tilde{\varphi}(x))
    \end{align*}
\end{cor}
\begin{proof}
    In light of the preceding discussions, all we now only need to show the following:
    \begin{enumerate}
        \item  Homotopic elements of $\Tilde{K}_0^{\operatorname{sp}}(A)$ are concordant. This can be seen as follows: if $\varphi \sim \psi$, there exists a graded $*$-homomorphism $\Phi:\sr \rightarrow \operatorname{Map}([0,1],\com(\hat\HA)) \cong \com(\hat\h_{A[0,1]})$ such that $\Phi_0=\varphi$, $\Phi_1=\psi$. We first claim that the following diagram commutes ($j=0,1$):
        \[
            % https://q.uiver.app/#q=WzAsNSxbMCwwLCJDXzAoLTEsMSkiXSxbMSwwLCJcXG1hdGhjYWx7S30oXFxoYXR7XFxtYXRoY2Fse0h9X3tBWzAsMV19fSkiXSxbMSwxLCJcXG1hdGhjYWx7S30oXFxoYXR7XFxtYXRoY2Fse0h9X0F9KSJdLFsyLDAsIlxcbWF0aGNhbHtLfShcXG1hdGhjYWx7SH1fXFxQaGkpIl0sWzIsMSwiXFxtYXRoY2Fse0t9KFxcbWF0aGNhbHtIfV97XFx2YXJwaGlfan0pIl0sWzAsMSwiXFxQaGkiXSxbMSwyLCJcXG9wZXJhdG9ybmFtZXtldn1faigtKSJdLFswLDIsIlxcdmFycGhpX2oiLDJdLFsxLDMsIi1cXG1pZF97XFxtYXRoY2Fse0h9X1xcUGhpfSJdLFsyLDQsIi1cXG1pZF97XFxtYXRoY2Fse0h9X3tcXHZhcnBoaV9qfX0iXSxbMyw0LCJcXG9wZXJhdG9ybmFtZXtldn1faigtKSJdXQ==
\begin{tikzcd}
	{C_0(-1,1)} & {\mathcal{K}(\hat{\mathcal{H}_{A[0,1]}})} & {\mathcal{K}(\mathcal{H}_\Phi)} \\
	& {\mathcal{K}(\hat{\mathcal{H}_A})} & {\mathcal{K}(\mathcal{H}_{\varphi_j})}
	\arrow["\Phi", from=1-1, to=1-2]
	\arrow["{\operatorname{ev}_j(-)}", from=1-2, to=2-2]
	\arrow["{\varphi_j}"', from=1-1, to=2-2]
	\arrow["{-\mid_{\mathcal{H}_\Phi}}", from=1-2, to=1-3]
	\arrow["{-\mid_{\mathcal{H}_{\varphi_j}}}", from=2-2, to=2-3]
	\arrow["{\operatorname{ev}_j(-)}", from=1-3, to=2-3]
\end{tikzcd}
        \]
    The commutativity of the left triangle is by definition. For the square, in light of the functoriality of $\operatorname{ev}_j(-)$, it suffices to note that $\operatorname{ev}_j(\h_\Phi)=\h_{\varphi_j}$. But this follows since every algebraic generator $\Phi(f)x$ ($f \in \sr, x \in \hat\h_{A[0,1]}$) of $\h_\Phi$ is mapped to $\varphi_j(f)(\operatorname{ev}_j(x))$, and every algebraic generator $\varphi_j(f)y$ ($f \in \sr, y \in \hat\h_{A}$) of $\mathcal{H}_{\varphi_j}$ is of the form $\varphi_j(f)(\operatorname{ev}_j(y'))$ for some $y' \in \hat\h_{A[0,1]}$.
    This shows that $\operatorname{ev}_j(\h_\Phi)$ is a dense subset of $\h_{\varphi_j}$. The conclusion follows from the definition of the pushout construction.

    Thus, we have $\operatorname{ev}_j(-) \circ \Phi^{\operatorname{nd}}=\varphi^{\operatorname{nd}}$. The rest follows from the universality of the extension in \cref{prop: extend homomorphism}.
    \item $\F_A$ is a monoid homomorphism. This follows directly from the definitions of the addition on both sides, and as unitarily equivalent Fredholm cycles are concordant.
    \end{enumerate}

\end{proof}
We next proceed to show that $\mathcal{F}_A$ is surjective and injective. First, we deal with surjectivity. 

We need a preparatory lemma:

\begin{lem} \label{lem: for surjectivity}
    Let $(E,F) \in K_0^{\operatorname{eu}}(A)$ be an essentially unitary cycle. Then, it can be represented as $(E',G)$, where $E' \subseteq E$, and $\overline{\Phi_G(C_0(-1,1))E'}=E'$, where $\Phi_G$ denotes the functional calculus of $G$.
\end{lem}
\begin{proof}
     First, note that ${\Phi_{F}(C_0(-1,1))} \subseteq \com_A(E)$. Set $E':=\overline{\Phi_{F}(C_0(-1,1))E}$. Then, $F(E') \subseteq E'$, as $f \in C_0(-1,1) \implies xf \in C_0(-1,1)$. Thus, $G:=F\mid_{E'}$ is an element of $\mathcal{B}(E')$. Moreover, as $F^2-1 \in \com(E)$ and $\operatorname{Im}(F^2-1) \subseteq E'$, \cref{thm: technical} yields that $(F^2-1)\mid_{E'} \in \com(E')$, whence $G$ is an odd, self-adjoint essentially unitary operator on $E'$. 
    
     We now show that $\overline{\Phi_G(C_0(-1,1))E'}=E'$, i.e., the functional calculus of $G$ is non-degenerate. To that end, it suffices to observe that $\Phi_G=\Phi_{F}^{\operatorname{nd}}:C_0(-1,1) \rightarrow \mathcal{K}(E')$. To see this, note that, $\Phi_{F}(p)=p(F)$ for any polynomial $p$, and in particular, for polynomials which vanish at $+1$ and $-1$. For any such polynomial $p$, on the one hand, we have $\Phi_{F}^{\operatorname{nd}}(p)=p(F)\mid_{E'}=p(F\mid_{E'})$ by \cref{lem: define phi nd}. On the other hand, for all such polynomials $p$,$~\Phi_G(p)=p(G)=p(F\mid_{E'})$. Thus, $\Phi_G=\Phi_{F}$ on the dense subspace of $C_0(-1,1)$ comprising of polynomials on $[-1,1]$ which vanish at the end-points, and we are done.
    
    We finally claim that $(E,F) \sim (E',G)$. We build a concordance as follows: Set $\h=\{e:[0,1] \rightarrow E ~\text{continuous}~ \mid e(0) \in E' \}$, and view it as a graded Hilbert submodule of the graded Hilbert $A[0,1]$-module $\operatorname{Map}([0,1],E)$. Define 
    \begin{align*}
        \Tilde{F}&: \h \rightarrow \h \\
        (\Tilde{F}e)(t)&:=F(e(t))
    \end{align*}
    Then $\Tilde{F}$ is self-adjoint, hence an element of $\mathcal{B}(\h)$. Moreover, $\Tilde{F}^2-1$ is compact by (a graded analogue of) \cref{lem: characterizing B and K} and \cref{thm: technical}, and so $[\h,\Tilde{F}]$ is a concordance. It is now clear that $(\h \hat\otimes _{\operatorname{ev}_0}A,\operatorname{ev}_0(\Tilde{F})) \cong (E,F)$, $(\h \hat\otimes _{\operatorname{ev}_1}A,\operatorname{ev}_1(\Tilde{F})) \cong (E',G)$.
\end{proof}
\begin{rem} \label{rem: imp}
    Note that, from the above lemma, since $\Phi_G=\Phi_{F}^{\operatorname{nd}}$, we have $F\mid_{E'}=G=\Phi_G(x)=\Tilde{\Phi_{F}}(x)$
\end{rem}
\begin{cor}
    $\mathcal{F}_A$ is surjective.
\end{cor}
\begin{proof}
    Let $(E,F) \in K_0^{\operatorname{eu}}(A)$ be an essentially unitary cycle. By Kasparov stabilization, we can take $E=\hat\HA$ without loss of generality. Then, the functional calculus of $F$ is a graded $*$-homomorphism $\Phi_F:C_0(-1,1) \rightarrow \com(\hat\HA)$. Then, $\F_A(\Phi_F)=(\h_{\Phi_F},\Tilde{\Phi_F}(x))$, and \cref{lem: for surjectivity} and \cref{rem: imp} show that $(\hat\HA,F) \sim (\h_{\Phi_F},\Tilde{\Phi_F}(x))$. This shows surjectivity of $\F_A$.
\end{proof}
We now focus on the injectivity of $\mathcal{F}_A$. We now proceed to show injectivity. 

\begin{prop}
    $\F_A$ is injective.
\end{prop}
\begin{proof}
    Let $\varphi_0$ and $\varphi_1$ be elements of $\Tilde{\K}(A)$ such that $\F_A(\varphi_0)=\F_A(\varphi_1)$. This means $[\mathcal{H}_{\varphi_0},\tilde{\varphi_0}(x)] \sim [\mathcal{H}_{\varphi_1},\tilde{\varphi_1}(x)]$. Suppose $[\mathcal{E},F]$ is the concordance. The functional calculus of $F$ is map $\Phi:\sr \rightarrow \mathcal{K}_{A[0,1]}(\E)$. By Kasparov stabilization, we choose an adjointable isometry $\iota:E \rightarrow \hat\HA$.
    
    We now make a series of claims:
    \begin{clm}
        Applying the functor $- \otimes_{ev_j}A$ ($j=0,1$) to the sequence % https://q.uiver.app/#q=WzAsMyxbMCwwLCJcXG1hdGhjYWx7U30iXSxbMSwwLCJcXG1hdGhjYWx7S31fe0FbMCwxXX0oXFxtYXRoY2Fse0V9KSJdLFsyLDAsIlxcbWF0aGNhbHtLfV97QVswLDFdfShcXGhhdHtcXG1hdGhjYWx7SH1fQX0pIl0sWzAsMSwiXFxQaGkiXSxbMSwyLCJBZChcXGlvdGEpIl1d
\[\begin{tikzcd}
	{C_0(-1,1)} & {\mathcal{K}_{A[0,1]}(\mathcal{E})} & {\mathcal{K}_{A[0,1]}(\hat{\mathcal{H}_{A[0,1]}})}
	\arrow["\Phi", from=1-1, to=1-2]
	\arrow["{Ad(\iota)}", from=1-2, to=1-3]
\end{tikzcd}\] we get % https://q.uiver.app/#q=WzAsMyxbMCwwLCJcXG1hdGhjYWx7U30iXSxbMSwwLCJcXG1hdGhjYWx7S31fe0F9KFxcb3BlcmF0b3JuYW1le2V2fV9qXFxtYXRoY2Fse0V9KSJdLFsyLDAsIlxcbWF0aGNhbHtLfV97QX0oXFxoYXR7XFxtYXRoY2Fse0h9X0F9KSJdLFswLDEsIlxcb3BlcmF0b3JuYW1le2V2fV9qKFxcUGhpKSJdLFsxLDIsIkFkKFxcb3BlcmF0b3JuYW1le2V2fV9qKFxcaW90YSkpIl1d
\begin{tikzcd}
	{C_0(-1,1)} & {\mathcal{K}_{A}(\operatorname{ev}_j\mathcal{E})} & {\mathcal{K}_{A}(\hat{\mathcal{H}_A})}
	\arrow["{\operatorname{ev}_j(\Phi)}", from=1-1, to=1-2]
	\arrow["{Ad(\operatorname{ev}_j(\iota))}", from=1-2, to=1-3]
\end{tikzcd}
    \end{clm}
\begin{proof}
    Lemma 1.2.8 of \cite{JT} shows that the map % https://q.uiver.app/#q=WzAsMixbMCwwLCJcXG1hdGhjYWx7U30iXSxbMSwwLCJcXG1hdGhjYWx7S31fQShcXG9wZXJhdG9ybmFtZXtldn1falxcbWF0aGNhbHtFfSkiXSxbMCwxLCJcXG9wZXJhdG9ybmFtZXtldn1faihcXFBoaSkiXV0=
\begin{tikzcd}
	{C_0(-1,1)} & {\mathcal{K}_A(\operatorname{ev}_j\mathcal{E})}
	\arrow["{\operatorname{ev}_j(\Phi)}", from=1-1, to=1-2]
\end{tikzcd} is well-defined (that is, we indeed land in the compact operators). The rest follows from the functoriality of the pushout construction.
\end{proof}
\begin{clm}
    The $*$-homomorphisms $\operatorname{Ad(ev}_0(\iota)) \circ \operatorname{ev}_0(\Phi)$ and $\operatorname{Ad(ev}_1(\iota)) \circ \operatorname{ev}_1(\Phi)$ are homotopic.
\end{clm}
\begin{proof}
    We have that $\operatorname{Ad}(\iota) \circ \Phi: \s \rightarrow \mathcal{K}_{A[0,1]}(\hat{\mathcal{H}_{A[0,1]}}) \cong \operatorname{Map}([0,1],\mathcal{K}_A(\hat\HA))$ is a homotopy between $(\operatorname{Ad}(\iota) \circ \Phi)\mid_0$ and $(\operatorname{Ad}(\iota) \circ \Phi)\mid_1$. The rest follows since under the canonical isomorphism $\mathcal{K}(\hat{\mathcal{H}_{A[0,1]}}) \cong \operatorname{Map}([0,1],\mathcal{K}_A(\hat\HA))$, the functor $-\otimes_{\operatorname{ev}_j}A$ corresponds to the functor evaluation at time $j$.
\end{proof}
\begin{clm}
    In the setup of the last claim, $\operatorname{ev}_j(\Phi)$ is the functional calculus of $\operatorname{ev}_j(F)$.
\end{clm}
\begin{proof}
    We have the following commutative diagram, where c.f.c. stands for the continuous functional calculus:
\[% https://q.uiver.app/#q=WzAsNSxbMSwwLCJDKFxcc2lnbWEoRikpIl0sWzMsMCwiXFxtYXRoY2Fse0J9KFxcbWF0aGNhbHtFfSkiXSxbMSwxLCJDKFxcc2lnbWEoXFxvcGVyYXRvcm5hbWV7ZXZ9X2ooRikpKSJdLFszLDEsIlxcbWF0aGNhbHtCfShcXG9wZXJhdG9ybmFtZXtldn1faihcXG1hdGhjYWx7RX0pKSJdLFswLDAsIkNfMCgtMSwxKSJdLFswLDIsIlxcb3BlcmF0b3JuYW1le3Jlc30iLDJdLFswLDEsIlxcb3BlcmF0b3JuYW1le2MuZi5jfm9mfSBGIl0sWzIsMywiXFxvcGVyYXRvcm5hbWV7Yy5mLmN+b2Z+ZXZ9X2ooRikiXSxbMSwzLCJcXG9wZXJhdG9ybmFtZXtldn1fe2oqfSJdLFs0LDAsIi1cXG1pZF97XFxzaWdtYShGKX0iLDJdXQ==
\begin{tikzcd}
	{C_0(-1,1)} & {C(\sigma(F))} && {\mathcal{B}(\mathcal{E})} \\
	& {C(\sigma(\operatorname{ev}_j(F)))} && {\mathcal{B}(\operatorname{ev}_j(\mathcal{E}))}
	\arrow["{\operatorname{res}}"', from=1-2, to=2-2]
	\arrow["{\operatorname{c.f.c~of} F}", from=1-2, to=1-4]
	\arrow["{\operatorname{c.f.c~of~ev}_j(F)}", from=2-2, to=2-4]
	\arrow["{\operatorname{ev}_{j*}}", from=1-4, to=2-4]
	\arrow["{-\mid_{\sigma(F)}}"', from=1-1, to=1-2]
\end{tikzcd}\]
Indeed, the top row composite is $\Phi$, and the square commutes because the continuous functional calculus is uniquely defined upon knowing it for the constant function $1$, and the identity function.
\end{proof}
\begin{clm}
    The following diagram commutes:
    \[
    % https://q.uiver.app/#q=WzAsMyxbMCwwLCJDXzAoLTEsMSkiXSxbMSwwLCJcXG1hdGhjYWx7S31fe0F9KFxcb3BlcmF0b3JuYW1le2V2fV9qXFxtYXRoY2Fse0V9KSJdLFsxLDEsIlxcbWF0aGNhbHtLfV9BKFxcbWF0aGNhbHtIfV97XFx2YXJwaGlfan0pIl0sWzAsMSwiXFxvcGVyYXRvcm5hbWV7ZXZ9X2ooXFxQaGkpIl0sWzAsMiwiXFx2YXJwaGlfal57XFxvcGVyYXRvcm5hbWV7bmR9fSIsMl0sWzIsMSwiXFxvcGVyYXRvcm5hbWV7QWR9KHVfaikiLDJdXQ==
\begin{tikzcd}
	{C_0(-1,1)} & {\mathcal{K}_{A}(\operatorname{ev}_j\mathcal{E})} \\
	& {\mathcal{K}_A(\mathcal{H}_{\varphi_j})}
	\arrow["{\operatorname{ev}_j(\Phi)}", from=1-1, to=1-2]
	\arrow["{\varphi_j^{\operatorname{nd}}}"', from=1-1, to=2-2]
	\arrow["{\operatorname{Ad}(u_j)}"', from=2-2, to=1-2]
\end{tikzcd}
    \]

\end{clm}
\begin{proof}
    Obvious, in light of the above claim and since $\varphi_j^{\operatorname{nd}}$ is the functional calculus of $\Tilde{\varphi_j}(x)$ by \cref{lem: phi tilde x}.
\end{proof}
Summarizing, we now know the following:
\[
    \text{there exists adjointable even isometries $\kappa_j:\h_{\varphi_j} \rightarrow \hat\HA$ such that $\operatorname{Ad}(\kappa_0) \circ \varphi_0^{\operatorname{nd}} \sim \operatorname{Ad}(\kappa_1) \circ \varphi_1^{\operatorname{nd}} $}
\]
Now, \cref{cor: cor to technical thm} shows that ${\operatorname{inc}_j}_* \circ \varphi_j^{\operatorname{nd}} \sim \varphi_j$, and we are now done in light of \cref{lem: final}.
\end{proof}
Summarizing the work done in this section, we now have:
\begin{thm}\label{thm: sp=fr}
    Let $A$ be a $\Z/2\Z$-graded (Real) $C^*$-algebra. Then, there is a natural isomorphism between $K_0^{\operatorname{sp}}(A)$ and $K_0^{\operatorname{Fr}}(A)$. In particular, $K_0^{\operatorname{sp}}(A)$ is an abelian group.
\end{thm}
\begin{thm} \label{thm: nat iso ungraded}
     Let $A$ be a unital, trivially graded (Real) $C^*$-algebra. Then, there is a natural isomorphism between $K_0^{\operatorname{sp}}(A)$ and $K_0(A)$.
\end{thm}
\begin{cor}\label{cor: K-theory ungraded}
    When $A$ is unital and ungraded, under the isomorphism of \cref{thm: nat iso ungraded}, the graded $*$-homomorphism 
    \begin{align*}
        \theta: \s &\rightarrow \com(\hat{l^2}) \\
        f &\mapsto \begin{pmatrix}
            f(0)P_0 & 0 \\
            0 & f(0)P_1
        \end{pmatrix},
    \end{align*}
    where $P_0,P_1$ are compact projections, gets mapped to the class $[P_0]-[P_1] \in K_0(A)$.
\end{cor}
\begin{proof}
    This is just a matter of checking the fate of $\theta$ under the isomorphisms \[K_0^{\operatorname{sp}}(A) \rightarrow \Tilde{K_0}^{\operatorname{sp}}(A) \rightarrow K_0^{\operatorname{eu}}(A) \rightarrow K_0^{\operatorname{Fr}}(A) \rightarrow K_0(A)\]
    While viewing $\theta$ as a $*$-homomorphism $C_0(-1,1) \rightarrow \com(\hat{l^2})$, we note that  $\mathcal{H}_\theta=(P_0 \oplus P_1)\hat{l^2}$. The uniqueness assertion in \cref{prop: extend homomorphism}) immediately shows that 
    \begin{align*}
        \Tilde{\theta}: C[-1,1] &\rightarrow \com(\h_\theta) \\
        f &\mapsto \begin{pmatrix}
            f(0)P_0 & 0 \\
            0 & f(0)P_1
        \end{pmatrix},
    \end{align*}
    Thus, $\Tilde{\theta}(x)=0$, and so the class in $K_0^{\operatorname{Fr}}(A)$ corresponding to $\theta$ is $(\h_\theta,0)$. The rest follows from the definition of the index of a Fredholm cycle.
\end{proof}
Thus, any projection onto an even one-dimensional subspace yields an element of $K_0^{\operatorname{sp}}(\mathbf{C})$, which we denote by $1$ in light of \cref{thm: nat iso ungraded}. We need a technical lemma concerning this for the next section:
\begin{lem} \label{prop: 1 x = x}
With $1 \in K_0^{\operatorname{sp}}(\mathbf{C})$ as above, if
$A$ is any graded $C^{*}$-algebra and if $x \in K_0^{\operatorname{sp}}(A)$, then under the
isomorphism $\mathbf{C} \hat{\otimes} A \cong A$ the class $1 \times x$
corresponds to $x$.
\end{lem}
\begin{proof}
    Let $\theta: w \mapsto \langle w,x \rangle x$ be our rank-one projection, where $x$ is a homogeneous element in $\hat{l^2}$. Note that we have an isometry 
    \begin{align*}
        \iota:\hat{l^2} &\rightarrow \hat{l^2} \hat\otimes \hat{l^2}\\
        e_i &\mapsto x \hat\otimes e_i
    \end{align*}
    Then, the induced map $\iota_*:\com(\hat{l^2}) \rightarrow \com(\hat{l^2}) \hat\otimes \com(\hat{l^2})$ is given by $\iota_*(T)=\theta \hat\otimes T$ for homogeneous operators $T$. Now, $\iota$ is homotopic to an even unitary $u:\hat{l^2} \rightarrow \hat{l^2} \hat\otimes \hat{l^2}$, so that $u^*$ is an identification between $\hat{l^2} \hat\otimes \hat{l^2}$ and $\hat{l^2}$. From \cref{Prod group level}, we are now done by checking that for the generators $u$ and $v$ in $\s$, the composite \[% https://q.uiver.app/#q=WzAsNSxbMCwwLCJcXG1hdGhjYWx7U30iXSxbMSwwLCJcXG1hdGhjYWx7U30gXFxoYXRcXG90aW1lc1xcbWF0aGNhbHtTfSJdLFsyLDAsIlxcbWF0aGNhbHtLfShcXGhhdHtsXjJ9KSBcXGhhdFxcb3RpbWVzIEEgXFxoYXRcXG90aW1lcyBcXG1hdGhjYWx7S30oXFxoYXR7bF4yfSkgIl0sWzQsMCwiIEEgXFxoYXRcXG90aW1lcyBcXG1hdGhjYWx7S30oXFxoYXR7bF4yfSkgXFxoYXRcXG90aW1lc1xcbWF0aGNhbHtLfShcXGhhdHtsXjJ9KSAiXSxbNSwwLCJBIFxcaGF0XFxvdGltZXMgXFxtYXRoY2Fse0t9KFxcaGF0e2xeMn0pIl0sWzAsMSwiXFx0cmlhbmdsZSJdLFsxLDIsIjEgXFxoYXRcXG90aW1lcyB4Il0sWzMsNCwiMSBcXGhhdFxcb3RpbWVzIFxcaW90YV8qXnstMX0iXSxbMiwzLCJcXG9wZXJhdG9ybmFtZXtmbGlwfSBBIH5cXCZ+IFxcbWF0aGNhbHtLfShcXGhhdHtsXjJ9KSJdXQ==
\begin{tikzcd}
	{\mathcal{S}} & {\mathcal{S} \hat\otimes\mathcal{S}} & {\mathcal{K}(\hat{l^2}) \hat\otimes A \hat\otimes \mathcal{K}(\hat{l^2}) } && { A \hat\otimes \mathcal{K}(\hat{l^2}) \hat\otimes\mathcal{K}(\hat{l^2}) } & {A \hat\otimes \mathcal{K}(\hat{l^2})}
	\arrow["\triangle", from=1-1, to=1-2]
	\arrow["{1 \hat\otimes x}", from=1-2, to=1-3]
	\arrow["{1 \hat\otimes \iota^{*}_{*}}", from=1-5, to=1-6]
	\arrow["{\operatorname{flip} A ~\&~ \mathcal{K}(\hat{l^2})}", from=1-3, to=1-5]
\end{tikzcd}
    \] identifies the class of $1 \times x$ with $x$.
\end{proof}
\section{Bott Periodicity}
In this section, we state and prove the Bott periodicty theorem in the spectral model. As we have already deduced natural isomorphisms between the spectral and projection pictures of K-theory in the unital, ungraded case, we will make no distinctions between these models henceforth. Most results here are inspired from \cite{NIG} and \cite{SLB}.

We will be making extensive use of Clifford algebras to that end. This is because, in both the complex and Real cases, the well known periodicity stems out of the periodicity of the Clifford algebras.
\subsection{Clifford algebras}
In this section, we discuss Clifford algebras. As usual, we develop everything in the Real world. We will consider Clifford algebras in a variety of generalities, as follows:

\begin{defi} \label{defi: baby clifford}
Let $E$ be an $n$-dimensional $\R$-vector space and $Q$ a quadratic
form on $E$. The Clifford algebra $\operatorname{Cliff}(E, Q)$ is the
solution of the universal problem that asks for an $\R$-algebra $C$
and a linear map $j : E \to C$ with the following properties: If $A$
is an associative $\R$-algebra with unit $1_A$ and if $f : E \to A$ is
linear with $f(v)^{2}  = Q(v)1_A$ for all $v \in E$, then there is a
unique algebra homomorphism $\tilde{f} : C \to A$ with $\tilde{f}
\circ j = f$.
\end{defi}
\begin{rem}
The map $j : E \to \operatorname{Cliff}(E, Q)$ will be denoted by
$i_E$. The special case $Q = 0$ yields the exterior algebra
$\Lambda^{*}E$. For general $Q$ we obtain $\operatorname{Cliff}(E, Q)$
as the quotient of the tensor algebra
\[
\mathcal{T}(E) = \R \oplus \bigoplus_{k \ge 0} \underbrace{(E \otimes
\ldots \otimes E)}_{k \text{ times}}
\]
by the ideal which is generated by elements of the form $v \otimes v -
Q(v)1_{T(E)}$, $v \in E$.
\end{rem}
\begin{defi} \label{defi: cl on euclidean vs}
    Let $V$ be a finite dimensional Euclidean vector space, that is, a finite dimensional real vector space together with a positive-definite inner product. We define \[
    \operatorname{Cliff}(V):=\operatorname{Cliff(V,{\lVert . \rVert^2}_V})
    \]
\end{defi}
Now, we consider the situation where $V$ also carries a non-trivial involution.
\begin{defi}
    A finite dimensional pseudo-Euclidean vector space $(V,\kappa)$ is a finite dimensional Euclidean vector space equipped with an self-adjoint involution $\kappa$.
\end{defi}
Some elementary algebra shows that the following definition makes sense:
\begin{defi}
    Let $(V,\kappa)$ be a finite dimensional pseudo-Euclidean vector space. Then, we can equip $\operatorname{Cliff}(V)$ with an involution; by extending $\kappa$ multiplicatively on $\operatorname{Cliff}(V)$. We denote $\operatorname{Cliff}(V)$ together with this extended involution by $\operatorname{Cliff}(V,\kappa)$. 
\end{defi}
\begin{prop}\label{prop: cliff c* alg}
    Let $(V,\kappa)$ be a finite dimensional pseudo-Euclidean vector space. Then, $\operatorname{Cliff}(V,\kappa)$ can be canonically given the structure of a $\Z/2\Z$-graded real $C^*$-algebra.
\end{prop}
\begin{proof}
    The result is well known, we don't give all the details here; the details can be found in \cite{SCH} and \cite{SPIN}. We just mention the grading on $\operatorname{Cliff}(V,\kappa)$, and mention a graded injective algebra homomomorphism $c:\operatorname{Cliff}(V,\kappa) \rightarrow \operatorname{End}(\Lambda^*V)$, the real $C^*$-algebra of bounded operators on the graded real Hilbert space $\Lambda^*V$. The norm and adjoint operation on $\operatorname{Cliff}(V,\kappa)$ is then defined to be such that this algebra homomorphism is an isomorphism of $C^*$-algebras.
    
    We first discuss the grading on $\operatorname{Cliff}(V)$. Note that, by the universal property of Clifford algebras, the map $v \mapsto -v$ on $V$ extends to an involution preserving automorphism $\alpha$ of $(\operatorname{Cliff}(V,\kappa))$ such that $\alpha^2=1$. We declare $\alpha$ to be the grading operator on $\operatorname{Cliff}(V,\kappa)$.

    We next give the definition of 
    \[
    c: \operatorname{Cliff}(V,\kappa) \rightarrow \operatorname{End}(\Lambda^*V)
    \]
    For each $v \in V$,
	there is a homomorphism
	\[
	\operatorname{ext}_v : \Lambda^{*}V \to \Lambda^{*+1}(V), \qquad
	\omega \mapsto v \wedge \omega
	\]
	with $\operatorname{ext}_v^{2} =0$.

	If $L \in V^{*}$, there also is a unique homomorphism
	$\operatorname{ins}_L : \Lambda^{*}V \to \Lambda^{*-1}(V)$ such
	that
	\[
	\operatorname{ins}_L = L : V = \Lambda^{1}(V) \to \Lambda^{0}(V) =
	\mathbb{K}
	\]
	and
	\[
	\operatorname{ins}_L(\omega \wedge \eta) =
	\operatorname{ins}_L(\omega) \wedge \eta + (-1)^{k} \omega \wedge
	\operatorname{ins}_L(\eta) \qquad \forall \omega \in
	\Lambda^{k}(V), \eta \in \Lambda^{l}(V)
	\]
	We have $\operatorname{ins}_L^{2}  = 0$ and
	\[
	\operatorname{ext}_v \operatorname{ins}_L + \operatorname{ins}_L
	\operatorname{ext}_v = L(v).
	\]

	Now, as $V$ is an euclidean vector space, then  $\Lambda^{*}V$ 
	has an inner product and $\operatorname{ext}^{*}_v =
	\operatorname{ins}_{\langle v,- \rangle}$.

	Now define a map
	\[
	\gamma : V \oplus V \to \operatorname{End}(\Lambda^{*}E), \qquad
	(v, w) \mapsto \operatorname{ext}_{v+w} +
	\operatorname{ins}_{\langle v-w, -\rangle}.
	\]
	Then a direct calculation shows that 
	\[
	\gamma(v, w)^{2} = \lVert v \rVert^2 - \lVert w \rVert^2.
	\]
	By the universal property of the Clifford algebras, we can continue to an algebra homomorphism
	\[
	\gamma : \operatorname{Cliff}(V \oplus V, \lVert . \rVert^2 \oplus (-\lVert . \rVert^2)) \to
	\operatorname{End}(\Lambda^{*}V)
	\]
	to a graded algebra homomorphism (in fact, an isomorphism). Since $i : (V, \lVert . \rVert) \rightarrow (V \oplus V, \lVert . \rVert
	\oplus (-\lVert . \rVert)), v \mapsto (v, 0)$ is an isometry, there also is an
	algebra homomorphism $c := \gamma \circ \operatorname{Cliff}(i) :
	\operatorname{Cliff}(V, \lVert . \rVert) \to \operatorname{End}(\Lambda^{*}V)$, which we mention is the one we are after. 
    
\end{proof}
We remark that the grading on Clifford algebras actually makes its study easier; as the following lemma shows (see Theorem 1.2.3 of \cite{SCH}):
\begin{lem}\label{lem: tensor product Cliff}
    Let $(V_1,\kappa_1)$ and $(V_2,\kappa_2)$ be finite dimensional pseudo-Euclidean vector spaces. Then the map
    \begin{align*}
      f: V_1 \oplus V_2 &\rightarrow \operatorname{Cliff}(V_1,\kappa_1) \hat\otimes \operatorname{Cliff}(V_2,\kappa_2)  \\
      v_1 \oplus v_2 &\mapsto j_{V_1}(v_1) \hat\otimes 1 + 1 \hat\otimes j_{V_2}(v_2)
    \end{align*}
    induces an isomorphism
    \[
    \operatorname{Cliff}(V_1 \oplus V_2, \kappa_1 \oplus\kappa_2) \rightarrow \operatorname{Cliff}(V_1,\kappa_1)\hat\otimes \operatorname{Cliff}(V_2,\kappa_2)
    \]
    
\end{lem}
We now turn our attention to the complexified Clifford algebras. 
\begin{cor}\label{cor: complex cliff alg}
Let $(V,\kappa)$ be a finite dimensional pseudo-Euclidean vector space. Then, $\operatorname{Cliff}^{\C}(V,\kappa):=\operatorname{Cliff}(V,\kappa) \hat\otimes_{\R} \mathbf{C}$ is a $\Z/2\Z$-graded Real $C^*$-algebra with Real structure given by $\kappa \hat\otimes \bar{(.)}$.
\end{cor}
For the remainder of this section, we will make the notation a bit more concise by not mentioning the involution $\kappa$.
\begin{defi}
    Let $V$ be a finite dimensional pseudo-Euclidean vector space. Denote by $\cc(V)$ the $\Z/2\Z$-graded, Real $C^*$-algebra $C_0(V,\operatorname{Cliff}^{\C}(V))$, with pointwise grading, and Real structure by combining the Real structures on $V$ and $\operatorname{Cliff}^{\C}(V)$ as in \cref{ex: Real}.
\end{defi}
\begin{lem}
    Let $V$, $W$ be involutive finite dimensional Euclidean vector spaces. The map $f_1 \hat\otimes f_2 \mapsto f$, where $f(v+w)=f_1(v)f_2(w)$ determines an isomorphism 
    \[
        \operatorname{Cliff}^{\C}(V) \hat\otimes \operatorname{Cliff}^{\C}(W) \rightarrow \operatorname{Cliff}^{\C}(V \oplus W)
    \]
    of graded, Real $C^*$-algebras. 
\end{lem}
\begin{proof}
    This follows easily from (the complexified version of) \cref{lem: tensor product Cliff} and the isomorphism $C_0(V) \hat\otimes C_0(W) \cong C_0(V \oplus W)$.
\end{proof}
\begin{defi}
    Let $V$ be a pseudo-Euclidean vector space. Denote by $P:V \rightarrow \operatorname{Cliff}^{\C}(V)$ the composite $V \rightarrow \operatorname{Cliff}(V) \rightarrow \operatorname{Cliff}^{\C}(V)$, where the first map is given by $v \mapsto v$, and the second map is $- \hat\otimes 1$.
\end{defi}
\begin{defi}\label{defi: Bott map}
    The Bott element $b \in K_0^{\operatorname{sp}}(\cc(V))$ is the $K$-theory class of the $*$-homomorphism $\beta: \s \rightarrow \cc(V)$ defined by $\beta: f \rightarrow f(P)$. $\beta$ is called the Bott map.
\end{defi}
The following lemma is needed to show that the Bott map indeed is a class in spectral K-theory.
\begin{lem}
    The Bott map commutes with the grading and the Real structure. 
\end{lem}
\begin{proof}
    We first note that the set of all elements $f \in \s$ for which $\beta(\varepsilon_{\s}(f))=\varepsilon_{\cc(V)}(\beta(f))$ is a $C^*$-subalgebra of $\s$. Thus, by the Stone-Weierstrass theorem, it suffices to check the condition on the functions $e^{-x^2}$ and $xe^{-x^2}$. But for these, it is obvious in light of their power series expansions. Thus, $\beta$ commutes with the grading.

    To see that $\beta$ commutes with the Real structure, by another Stone-Weierstrass argument, we boil down to investigating it only on the function $f:=\frac{1}{x+i}$. We have to show that $\overline{f}v=\kappa(f(\kappa{v})) ~\forall~ v \in V$ (here, we think about extending $\kappa$ $\R$-linearly). The left hand side evaluates to $\frac{1}{v \otimes 1-1 \otimes i}$. The right-hand side evaluates to $\kappa(\frac{1}{\kappa v \otimes 1 + 1 \otimes i})=\frac{1}{v \otimes 1-1 \otimes i}$, as desired.
    
\end{proof} 
\subsection{Bott Periodicity - the statement and the formal part}
Having given the relevant background on Clifford algebras, we are now in a position to state the Bott periodicity theorem.

\begin{thm}\label{thm: BP}
    For every graded $C^*$-algebra $A$, and every finite-dimensional pseudo Euclidean vector space $V$, the Bott map
    \[
        \beta: K_0(A) \rightarrow K_0(A \hat\otimes \cc(V)),
    \]
    defined by $\beta(x)=x \times b$, is an isomorphism of abelian groups.
\end{thm}

In the remainder of this subsection, we discuss a trick, originally due to Michael Atiyah, which substantially reduces the workload needed to prove \cref{thm: BP}. 

\begin{defi}\label{defi:rotation}
Let us say that a graded $C^{*}$-algebra $B$ has the rotation property
if the automorphism $b_1 \widehat{\otimes} b_2 \mapsto
(-1)^{\partial b_1\partial b_2}b_2 \widehat{\otimes} b_1$ which
interchanges the two factors in the tensor product $B
\widehat{\otimes} B$ is homotopic to a tensor product $*$-homomorphism
$1 \otimes \iota : B \widehat{\otimes} B \to B \widehat{\otimes} B$.
\end{defi}
\begin{lem} \label{lem: C(V) sat rot trick}
    Let $(V,\kappa)$ be a finite dimensional pseudo-Euclidean vector space. The $C^*$-algebra $\cc(V)$ has the rotation property.
\end{lem}
\begin{proof}
    Under the isomorphism 
    \[
    \cc(V) \hat\otimes \cc(V) \cong \cc(V \oplus V),
    \]
    the flip isomorphism corresponds to the $*$-automorphism of $\cc(V \oplus V)$ associated to the map $\tau_{**}$ which exchanges the two copies of $V$ in the direct sum $V \oplus V$. We claim that $\tau$ is homotopic to the map $\xi:V \oplus V \rightarrow V \oplus V, \xi(v_1,v_2) = (v_1,-v_2)$, through a path of isometric isomorphisms of $V \oplus V$ which preserve the involution $\kappa$. This follows by observing that the standard homotopy $\begin{pmatrix}
        \sin(\frac{\pi}{2}t) & \cos(\frac{\pi}{2}t) \\
        \cos(\frac{\pi}{2}t) & -\sin(\frac{\pi}{2}t)
    \end{pmatrix}$ clearly commutes with the involution $\begin{pmatrix}
        \kappa & 0 \\
        0 & \kappa
    \end{pmatrix}$.
\end{proof}
The reason we care about the rotation property is because of the following theorem:
\begin{thm} \label{thm: rotation trick}
Let $B$ be a graded $C^{*}$-algebra with the rotation property, and $\mathcal{H}$ some graded Real Hilbert space whose even and odd parts are both infinite dimensional.
Suppose there exists a class $b \in K_0(B)$ and an asymptotic morphism
\[
\alpha : \mathcal{S} \widehat{\otimes}B \to \mathcal{K}(\mathcal{H})
\]
with the property that the induced $K$-theory homomorphism $\alpha_{*}
: K_0(B) \to K_0(\mathbf{C})$ maps $b$ to $1$. Then for every
$C^{*}$-algebra $A$ the maps,
\[
\alpha_{*} : K_0(A \widehat{\otimes}B) \to K_0(A) \qquad \text{and}\qquad
\beta_{*} : K_0(A) \to K_0(A \widehat{\otimes} B)
\]
induced by $\alpha$ and by multiplication by the $K$-theory class $b$
are inverse to one another. 
\end{thm}
\begin{rem}
    Note that in light of \cref{prop: 1 x = x}, and since $K_0(\mathbf{C}) \cong \Z$ from classical K-theory, the assumption on $\alpha$ is readily seen to be equivalent to saying that the $\alpha\beta:K_0(\mathbf{C}) \rightarrow K_0(\mathbf{C})$ is the identity.
\end{rem}
\begin{proof}
    From \cref{cor: legend and prod}, the following diagram commutes:
% https://q.uiver.app/#q=WzAsNCxbMCwwLCJLKEMpIFxcb3RpbWVzIEsoQSBcXGhhdHtcXG90aW1lc30gQikiXSxbMiwwLCJLKEMgXFxoYXR7XFxvdGltZXN9ICBBIFxcaGF0e1xcb3RpbWVzfSBCKSJdLFswLDEsIksoQykgXFxvdGltZXMgSyhBKSJdLFsyLDEsIksoQyBcXGhhdHtcXG90aW1lc30gQSkiXSxbMCwxLCJLIFxcdGV4dHstdGhlb3J5IHByb2R1Y3R9Il0sWzAsMiwiMSBcXG90aW1lcyBcXGFscGhhX3sqfSIsMl0sWzIsMywiSyBcXHRleHR7LXRoZW9yeSBwcm9kdWN0fSIsMl0sWzEsMywiXFxhbHBoYV97Kn0iXV0=
\[\begin{tikzcd}[ampersand replacement=\&]
	{K(C) \otimes K(A \hat{\otimes} B)} \&\&\& {K(C \hat{\otimes}  A \hat{\otimes} B)} \\
	{K(C) \otimes K(A)} \&\&\& {K(C \hat{\otimes} A)}
	\arrow["{K \text{-theory product}}", from=1-1, to=1-4]
	\arrow["{1 \otimes \alpha_{*}}"', from=1-1, to=2-1]
	\arrow["{K \text{-theory product}}"', from=2-1, to=2-4]
	\arrow["{\alpha_{*}}", from=1-4, to=2-4]
\end{tikzcd}\]
    We express this property by saying that $\alpha$ is multiplicative, i.e. $\alpha_*(x \times y)= x \times \alpha_*(y)$. It follows directly from this multiplicativity property of $\alpha$ that $\alpha$ is a left inverse to $\beta$:
    \[
    \alpha_*(\beta_*(x))=\alpha_*(x \times b)=x \times \alpha_*(b)= x \times 1 = x.
    \]
    It remains to show that $\alpha$ is also a right-inverse to $\beta$. To that end, it suffices to show that $\alpha$ has a left inverse. To that end, we consider the isomorphisms $\sigma: A \hat\otimes B \rightarrow B \hat\otimes A$ given by flipping, and $\tau: B \hat\otimes A \hat\otimes B \rightarrow B \hat\otimes A \hat\otimes B $ given by flipping the first and last factors in the tensor products. Note that the following diagram commutes:
    \[
    % https://q.uiver.app/#q=WzAsNixbMCwxLCJLXzAoQSBcXGhhdFxcb3RpbWVzIEIpIFxcb3RpbWVzIEtfMChCKSJdLFswLDIsIktfMChCIFxcaGF0XFxvdGltZXMgQSkgXFxvdGltZXMgS18wKEIpIl0sWzIsMiwiS18wKEIgXFxoYXRcXG90aW1lcyBBIFxcaGF0XFxvdGltZXMgQikiXSxbMiwxLCJLXzAoQSBcXGhhdFxcb3RpbWVzIEIgXFxoYXRcXG90aW1lcyBCKSJdLFswLDAsIktfMChCKSBcXG90aW1lcyBLXzAoQSBcXGhhdFxcb3RpbWVzIEIpIl0sWzIsMCwiS18wKEIgXFxoYXRcXG90aW1lcyBBIFxcaGF0XFxvdGltZXMgQikiXSxbMCwxLCJcXHNpZ21hXyogXFxvdGltZXMxIl0sWzEsMiwiSy1cXHRleHR7dGhlb3J5IHByb2R1Y3R9IiwyXSxbMCwzLCJLLVxcdGV4dHt0aGVvcnkgcHJvZHVjdH0iLDJdLFszLDIsIlxcdGF1X3sxMl8qfSIsMl0sWzQsNSwiSy1cXHRleHR7dGhlb3J5IHByb2R1Y3R9Il0sWzUsMywiKFxcdGV4dHtmbGlwfSB+QSBcXGhhdFxcb3RpbWVzIEIgflxcdGV4dHthbmR9fkIpXyoiLDJdLFs0LDAsInggXFxvdGltZXMgeSBcXG1hcHN0byB5IFxcb3RpbWVzIHgiLDJdXQ==
\begin{tikzcd}
	{K_0(B) \otimes K_0(A \hat\otimes B)} && {K_0(B \hat\otimes A \hat\otimes B)} \\
	{K_0(A \hat\otimes B) \otimes K_0(B)} && {K_0(A \hat\otimes B \hat\otimes B)} \\
	{K_0(B \hat\otimes A) \otimes K_0(B)} && {K_0(B \hat\otimes A \hat\otimes B)}
	\arrow["{\sigma_* \otimes1}", from=2-1, to=3-1]
	\arrow["{K-\text{theory product}}"', from=3-1, to=3-3]
	\arrow["{K-\text{theory product}}"', from=2-1, to=2-3]
	\arrow["{\tau_{12_*}}"', from=2-3, to=3-3]
	\arrow["{K-\text{theory product}}", from=1-1, to=1-3]
	\arrow["{(\text{flip} ~A \hat\otimes B ~\text{and}~B)_*}"', from=1-3, to=2-3]
	\arrow["{x \otimes y \mapsto y \otimes x}"', from=1-1, to=2-1]
\end{tikzcd}
    \]
    Indeed, the top rectangle commutes because of \cref{prop: prod prop} $(ii)$, and the bottom one commutes because of \cref{rem: prod}. This yields that,
    \[
    \sigma_*(y) \times z = \tau_*(z \times y), ~\forall~ y \in K_0(A \hat\otimes B), z \in K_0(B).
    \]

    As $B$ has the rotation property, $\tau$ is homotopic to $\iota \hat\otimes 1 \hat\otimes 1$, where $\iota$ is as in \cref{defi:rotation}.Thus, setting $z=b$ above, we obtain:
    \[
    \sigma_*(y) \times b = \tau_*(b \times y)=(\iota \hat\otimes 1 \hat\otimes 1)_*(b \times y)= \iota_*(b) \times y.
    \]
    Applying $\alpha_*$, we deduce that:
\[
\sigma_{*}(y) = \alpha_{*}(\sigma_{*}(y) \times b) =
\alpha_{*}(\iota_{*}(b) \times y) = \iota_{*}(b) \times \alpha_{*}(y)
\]
    where for the first and last equalities, we used the multiplicativity of $\alpha$. Applying the flip isomorphism $\sigma$ to the above equation, we obtain
    \[
        y=\alpha_*(y) \times \iota_*(b).
    \]
    Thus, multiplication by $\iota_*(b)$ is a left-inverse to $\alpha$, as desired.
\end{proof}
\subsection{Inverse to the Bott map}

To prove \cref{thm: BP}, it suffices to prove the following theorem in light of \cref{thm: rotation trick} and \cref{lem: C(V) sat rot trick}:

\begin{thm} \label{thm: hg sl Main thm}
    Given any pseudo Euclidean vector space $V$, there exists a Real graded Hilbert space $\mathcal{H}(V)$, and asymtotic morphism $\alpha: \s \hat\otimes \cc(V) \dashrightarrow \mathcal{K}(\mathcal{H}(V))$ for which the induced homomorphism $\alpha:K_0(\cc(V)) \rightarrow K_0(\mathbf{C})$ maps the Bott element $b \in K_0(\cc(V))$ to $1 \in K_0(\mathbf{C})$.
 \end{thm}  
 We first come up with the Hilbert space $\mathcal{H}(V)$.
  \begin{defi}
Let $V$ be a finite-dimensional Euclidean vector space. Let us provide the
finite-dimensional linear space underlying the algebra
$\operatorname{Cliff}^{\C}(V)$ with the Hilbert space structure for which the
monomials $e_{i_1} \cdots e_{i_p}$ (associated to an orthonormal basis of
$V$) are orthonormal. The Hilbert space structure so obtained is
independent of the choice of $e_1, \ldots , e_n$\footnote{strictly speaking, we look at the elements $e_i \otimes 1$, but we will refer to $x \otimes 1$ in $\operatorname{Cliff}^{\C}(V)$ as simply $x$ from now on.}. Denote by
$\mathcal{H}(V)$ the infinite-dimensional complex Hilbert space of
square-integrable $\operatorname{Cliff}^{\C}(V)$-valued functions on $V$, Thus:
\[
\mathcal{H}(V) = L^{2} (V, \operatorname{Cliff}^{\C}(V)).
\]
The Hilbert space $\mathcal{H}(V)$ is a graded Hilbert space, with grading
inherited from $\operatorname{Cliff}^{\C}(V)$, and a Real structure combining the ones on $V$ and $\operatorname{Cliff}^{\C}(V)$.
 \end{defi}
 To prove this theorem, we need to develop some more machinery:

 \begin{defi}
Let $V$ be a finite-dimensional Euclidean vector space and let $e, f \in
V$. Define linear operators on the finite-dimensional graded Hilbert space
underlying $\operatorname{Cliff}^{\C}(V)$ by the formulas
\begin{align*}
	e(x)  &= e \cdot x\\
	\hat{f}(x) &= (-1)^{\partial x}x \cdot f.
\end{align*}
\end{defi}
\begin{lem}
    The operator $e:\operatorname{Cliff}^{\C}(V) \rightarrow \operatorname{Cliff}^{\C}(V)$ is self-adjoint, while the operator $\hat f:\operatorname{Cliff}^{\C}(V) \rightarrow \operatorname{Cliff}^{\C}(V)$ is skew adjoint. Furthermore, $e$ and $\hat f$ are Real operators.
\end{lem}
\begin{proof}
    For the proof, it helps to have the point of view on Clifford algebras as used in the proof of \cref{prop: cliff c* alg}.

    Indeed, we have the graded $*$-algebra homomorphism:
    \begin{align*}
        c:\operatorname{Cliff}^{\C}(V) &\rightarrow \operatorname{End}(\Lambda^*V) \\
        v &\mapsto \operatorname{ext}_v + \operatorname{ext}_{v^*}.
    \end{align*}
    Then, $\gamma:\operatorname{Cliff}^{\C}(V) \rightarrow \Lambda^*V$, $\gamma(v)=c(v).1$ is a linear isomorphism. We define the positive normed functional
    \[
    T: \operatorname{Cliff}^{\C}(V) \rightarrow \C, T(x)=\frac{1}{2^n}(\operatorname{Tr}(c(x))).
    \]
        Under these identifications, it is easy to see that the scalar product on $\operatorname{Cliff}^{\C}(V)$ is given by $\langle x,y \rangle = T(x^*y)$.
    As $\operatorname{End}(\Lambda^*V)$ is inner graded, the properties of trace show that $T(x)=T(\alpha(x))$.
    Now, we compute
    \[
    \langle ex,y \rangle = T((ex)^*y) = T(x^*e^*y) \xlongequal{e=e^*}T(x^*ey)=\langle x,ey \rangle.
    \]
    and 
    \begin{align*}
        \langle \hat fx,y \rangle &= T((\alpha(x).f)^*y)\\&=T(f^*\alpha(x)^*y) \\
    &=T(\alpha(x)^*yf) ~(\text{since}~ T(abc)=T(bca), f=f^*)\\
    &= T(\alpha(\alpha(x)^*)\alpha(y)\alpha(f)) ~(\text{since}~ T(x)=T(\alpha(x)))\\
    &=T(x^*\alpha(y)\alpha(f)) \\
    &=-T(x^*\alpha(y)f) ~(\text{since}~ \alpha(f)=-f) \\
    &= -\langle x,\hat fy \rangle
    \end{align*}
    The Reality of $e$ and $\hat f$ follows immediately from the definition.
\end{proof}
 \begin{defi}
Let $V$ be a finite-dimensional Euclidean vector space. Denote by
$\mathfrak{s}(V)$ the dense subspace of $\mathcal{H}(V)$ comprised of
Schwartz-class $\operatorname{Cliff}^{\C}(V)$-valued functions:
\[
\mathfrak{s}(V,\operatorname{Cliff}^{\C}(V)) = \text{Schwartz-class
}\operatorname{Cliff}^{\C}(V)\text{-valued functions.}
\]
The Dirac operator of $V$ is the unbounded operator $D$ on
$\mathcal{H}(V)$, with domain $\mathfrak{s}(V,\operatorname{Cliff}^{\C}(V))$, defined by
\[
	(Df)(v) = \sum_{1}^{n} \hat{e}_i\left( \frac{\partial f}{\partial x_i}
	(v)\right) , 
\]
where $e_1,\ldots , e_n$ is an orthonormal basis of $V$ and $x_1, \ldots ,
x_n$ are the corresponding coordinates on $V$. 

\end{defi}
\begin{defi}
    Let $V$ be a finite dimensional pseudo Euclidean vector space. The Clifford operator is the Real unbounded operator on $\h(V)$, with domain $\mathfrak{s}(V,\operatorname{Cliff}^{\C}(V))$, which is given by the formula
    \[
    (Cf)(v)=vf(v)
    \]
\end{defi}
\begin{lem}
    The Dirac operator is Real.
\end{lem}
\begin{proof}
    We can factor the Dirac operator as 
    \[
    \mathfrak{s}(V,\operatorname{Cliff}^{\C}(V)) \xrightarrow{\nabla}  \mathfrak{s}(V,V \hat\otimes \operatorname{Cliff}^{\C}(V)) \xrightarrow{C} 
     \mathfrak{s}(V,\operatorname{Cliff}^{\C}(V)) \]
     where $\nabla(f)=\sum_{j=1}^n \hat{e_j} \hat\otimes \partial_j f$. It now suffices to show that both $C$ and $\nabla$ are Real operators.

     Reality of $C$: It suffices to observe that for an element $x \mapsto v \hat\otimes f(x) \in \mathfrak{s}(V,\operatorname{Cliff}^{\C}(V))$, where $v \in V$ is fixed, the following diagram commutes:
     \[
     % https://q.uiver.app/#q=WzAsNCxbMCwwLCJ4IFxcbWFwc3RvIHYgXFxoYXRcXG90aW1lcyBmKHgpIl0sWzIsMCwieCBcXG1hcHN0byB2Zih4KSJdLFsyLDEsInggXFxtYXBzdG8gXFxvdmVybGluZXtcXGthcHBhKHYpXFxrYXBwYSBmKFxca2FwcGEoeCkpfSJdLFswLDEsInggXFxtYXBzdG8gXFxvdmVybGluZXtcXGthcHBhKHYpIFxcaGF0XFxvdGltZXNcXGthcHBhIGYoXFxrYXBwYSB4KX0iXSxbMCwxLCJDIiwwLHsic3R5bGUiOnsidGFpbCI6eyJuYW1lIjoibWFwcyB0byJ9fX1dLFsxLDIsIlxcdGV4dHtSZWFsIGZvcm19Il0sWzAsMywiXFx0ZXh0e1JlYWwgZm9ybX0iLDJdLFszLDIsIkMiXV0=
\begin{tikzcd}
	{x \mapsto v \hat\otimes f(x)} && {x \mapsto vf(x)} \\
	{x \mapsto \overline{\kappa(v) \hat\otimes\kappa f(\kappa x)}} && {x \mapsto \overline{\kappa(v)\kappa f(\kappa(x))}}
	\arrow["C", maps to, from=1-1, to=1-3]
	\arrow["{\text{Real form}}", from=1-3, to=2-3]
	\arrow["{\text{Real form}}"', from=1-1, to=2-1]
	\arrow["C", from=2-1, to=2-3]
\end{tikzcd}
     \]
     Reality of $\nabla$: It's another diagram chase, but a bit more subtle:
     \[
     % https://q.uiver.app/#q=WzAsNixbMCwwLCJmIl0sWzIsMCwiXFxzdW1fe2o9MX1ee259IGVfaiBcXGhhdFxcb3RpbWVzIFxccGFydGlhbF9qZiJdLFsyLDEsInggXFxtYXBzdG8gXFxzdW1fe2o9MX1ee259IFxcb3ZlcmxpbmV7XFxrYXBwYSBlX2ogXFxoYXRcXG90aW1lcyBcXGthcHBhKERmKFxca2FwcGEgeCkuZV9qKX0iXSxbMCwxLCJ4IFxcbWFwc3RvIFxcb3ZlcmxpbmV7XFxrYXBwYSBmKFxca2FwcGEgeCl9Il0sWzIsMywieCBcXG1hcHN0b1xcc3VtX3tqPTF9XntufVxcb3ZlcmxpbmV7ZV9qIFxcaGF0XFxvdGltZXNcXGZyYWN7ZH17ZHR9XFxtaWRfe3Q9MH0oXFxrYXBwYSBmIChcXGthcHBhICh4ICsgdGVfaikpKX0iXSxbMiwyLCJ4IFxcbWFwc3RvIFxcc3VtX3tqPTF9XntufWVfaiBcXGhhdFxcb3RpbWVzIFxca2FwcGEgKEQoZikoXFxrYXBwYSB4KS4oXFxrYXBwYSBlX2opKSJdLFsxLDIsIlxcdGV4dHtSZWFsIHN0cnVjdHVyZX0iLDAseyJzdHlsZSI6eyJ0YWlsIjp7Im5hbWUiOiJtYXBzIHRvIn19fV0sWzAsMSwiXFxuYWJsYSIsMCx7InN0eWxlIjp7InRhaWwiOnsibmFtZSI6Im1hcHMgdG8ifX19XSxbMCwzLCJcXHRleHR7UmVhbCBzdHJ1Y3R1cmV9IiwwLHsic3R5bGUiOnsidGFpbCI6eyJuYW1lIjoibWFwcyB0byJ9fX1dLFszLDQsIlxcbmFibGEiLDJdLFs0LDUsIj0gKFxcdGV4dHtzaW5jZX1+XFxrYXBwYX5cXHRleHR7aXMgbGluZWFyOyBhbmQgY2hhaW4gcnVsZX0pIiwyXSxbNSwyLCJcXHRleHR7YXN9fiBcXGthcHBhKGVfaik9XFxwbSAxIiwyXV0=
\begin{tikzcd}
	f && {\sum_{j=1}^{n} \hat{e_j} \hat\otimes \partial_jf} \\
	{x \mapsto \overline{\kappa f(\kappa x)}} && {x \mapsto \sum_{j=1}^{n} \overline{\kappa \hat{e_j} \hat\otimes \kappa(Df(\kappa x).e_j)}} \\
	&& {x \mapsto \sum_{j=1}^{n}\overline{\hat{e_j} \hat\otimes \kappa (D(f)(\kappa x).(\kappa e_j))}} \\
	&& {x \mapsto\sum_{j=1}^{n}\overline{\hat{e_j} \hat\otimes\frac{d}{dt}\mid_{t=0}(\kappa f (\kappa (x + te_j)))}}
	\arrow["{\text{Real structure}}", maps to, from=1-3, to=2-3]
	\arrow["\nabla", maps to, from=1-1, to=1-3]
	\arrow["{\text{Real structure}}", maps to, from=1-1, to=2-1]
	\arrow["\nabla"', from=2-1, to=4-3]
	\arrow["{= (\text{since}~\kappa~\text{is linear; and chain rule})}"', from=4-3, to=3-3]
	\arrow["{=(\text{as}~ \kappa(e_j)=\pm 1)}"', from=3-3, to=2-3]
\end{tikzcd}
     \]
\end{proof}

Since the individual $\hat{e}_i$ are skew-adjoint and since they commute
with the partial derivatives we see that $D$ is formally self-adjoint on
$\mathfrak{s}(V)$. The following lemma is well known in the theory of differential operators:
 \begin{lem} \label{lem: HG Lem 1.8}
 Let $V$ be a finite-dimensional pseudo-Euclidean vector space. The Dirac operator on $V$ is closed and essentially self-adjoint. If $f \in \s$, if $h \in \cc(V)$ and if $M_h$ is the operator of pointwise multiplication by $h$ on the Hilbert space $\mathcal{H}(V)$, then the product $f(D)M_h$ is a compact operator on $\mathcal{H}(V)$.
 \end{lem}
 \begin{proof}
      $D$ is an essentially self-adjoint operator follows from well known theory of the Dirac operator (\cite{RUF} Proposition 8.1.8). Thus, it makes sense to talk about the operator $f(D)$, since there is a functional calculus for self-adjoint operators on a Hilbert space. (\cite{RUF} Theorem D.1.7).

      We next show that $f(D)M_h$ is a compact operator. First, take $f(x)=(x\pm i)^{-1}$, and $h \in C_c^{\infty}(V,\operatorname{Cliff(V)})$. Some well known characterization of essential self-adjointness yields that $(D \pm i):\mathfrak{s}(V) \rightarrow \mathcal{H}(V)$ has dense range. Thus, Rellich lemma and Garding's inequality (see \cite{RUF}) shows that $M_h(D\pm i)^{-1}$ is a compact operator on $\mathcal{H}(V)$. Thus, by taking adjoints, we get $(D\pm i)^{-1}M_h$ is compact. We can now see that $(D\pm i)^{-1}M_h$ is compact for all $h \in \cc(V)$, since $C_c^{\infty}(V,\operatorname{Cliff(V)})$ is dense in $\cc(V)$. We are now done by a direct application of the Stone-Weierstrass theorem, as the set of $f \in \s$ for which $f(D)M_h$ is compact, for all $h$, is an ideal in $\s$.
 \end{proof}
 \begin{rem} \label{rem: for HG Prop 1.5}
     From the proof of \cref{lem: HG Lem 1.8}, it also follows that $f(t^{-1}D)M_{h_t}$ is compact where $t \in [1,\infty)$, and $h_t(v):=h(t^{-1}v).$
 \end{rem}
 \begin{lem} \label{lem: HG Lem 1.9}
     In the notation of \cref{lem: HG Lem 1.8}, for every $f \in \s$ and $h \in \cc(V)$, we have 
     \[\lim_{t \rightarrow \infty} \lVert [f(t^{-1}D),M_{h_t}] \rVert=0,\]
     where $[.,.]$ denotes the graded commutator.
 \end{lem}
 \begin{proof}
     As in the proof of \cref{lem: HG Lem 1.8}, a Stone-Weierstrass argument reduces to proving the lemma for the functions $x \mapsto \frac{1}{x^2+1}$ and $x \mapsto \frac{x}{x^2+1}$, and $h$ is smooth and compactly supported. We show the computation for $f=\frac{1}{x^2+1}$.
\begin{align*}
	\left[ f(t^{-1}D), M_{h_t} \right] 
	&= \left[ (t^{-2}D^2+I)^{-1}, M_{h_{t}} \right]\\
	&= (t^{-2}D^2+I)^{-1}M_{h_{t}} \pm M_{h_{t}}(t^{-2}D^2+I)^{-1}\\
	&=
	(t^{-2}D^2+I)^{-1}M_{h_{t}}(t^{-2}D^2+I)(t^{-2}D^2+I)^{-1}\\
	& \qquad \qquad \pm (t^{-2}D^2+I)^{-1}(t^{-2}D^2+I)M_{h_{t}}(t^{-2}D^2
	+I)^{-1}\\
	&= (t^{-2}D^2+I)^{-1}\left[ M_{h_{t}}, (t^{-2}D^2+I) \right]
	(t^{-2}D+I)^{-1}\\
 &= (t^{-2}D^2+I)^{-1}\left[ M_{h_{t}}, t^{-2}D^2 \right]
	(t^{-2}D+I)^{-1}\\
	&= t^{-2}(t^{-2}D+I)^{-1} \left[ M_{h_{t}}, D^2 \right]
	(t^{-2}D^2+I)^{-1}.
\end{align*}

     which has norm bounded by $t^{-2}\lVert [M_{h_t},D^2]\rVert$, and the conclusion follows from here.
 \end{proof}
 \begin{prop}
     Upto equivalence, there is a unique asymptotic morphism \[\alpha_t: \s \hat\otimes \cc(V) \rightarrow \mathcal{K}(\mathcal{H}(V))\] for which, on elementary tensors,
     \[\alpha_t(f \hat\otimes h)=f(t^{-1}D)M_{h_t}.\]
 \end{prop}
 \begin{proof}
     For $t \in [1,\infty)$, we can define a linear map $\alpha_t:\s \hat\odot \cc(V) \rightarrow \mathcal{B}(\mathcal{H}(V))$ by the formula \[\alpha_t(f \hat\otimes h)=f(t^{-1}D)M_{h_t}.\] \cref{lem: HG Lem 1.9} shows that $\alpha_t$ defines a homomorphism from $\s \hat\odot \cc(V)$ into $\mathfrak{A}(\mathcal{B}(\mathcal{H}(V)))$. By the universal property of the tensor product, this extends to a $*$-homomorphism defined on $\s \hat\otimes \cc(V)$. \cref{rem: for HG Prop 1.5} shows that $f(t^{-1}D)M_{h_t}$ is compact. So, our $*$-homomorphism actually maps $\s \hat\otimes \cc(V)$ into the subalgebra $\mathfrak{A}(\mathcal{K}(\mathcal{H}(V)))$, as required.
 \end{proof}
 \subsection{The Harmonic Oscillator and the calculations}
 The $\alpha$ constructed in the previous subsection is the asymptotic morphism sought after in \cref{thm: hg sl Main thm}. The only part remaining is thus to show that $\alpha_*(b)=1$. We proceed to show this now.

 \begin{lem}
     Let $V$ be a finite dimensional pseudo Euclidean vector space. If $\cc(V)$ is represented on the Hilbert space $\mathcal{H}(V)$ by pointwise multiplication operators then the composition 
     \[% https://q.uiver.app/#q=WzAsMyxbMCwwLCJcXG1hdGhjYWx7U30iXSxbMSwwLCJcXG1hdGhjYWx7Q30oVikiXSxbMiwwLCJcXG1hdGhjYWx7Qn0oXFwsXFxtYXRoY2Fse0h9KFYpKSJdLFswLDEsIlxcYmV0YSJdLFsxLDIsIk0iXV0=
\begin{tikzcd}
	{\mathcal{S}} & {\mathcal{C}(V)} & {\mathcal{B}(\,\mathcal{H}(V))}
	\arrow["\beta", from=1-1, to=1-2]
	\arrow["M", from=1-2, to=1-3]
\end{tikzcd}\] maps $f \in \s$ to $f(C) \in \mathcal{B}(\mathcal{H}(V))$.
 \end{lem}
 \begin{proof}
     We do a direct calculation: for $g\in \h(V), v \in V$, we have 
     \begin{align*}
         &(M\beta(f))(g)(v) \\
    =&M_{\beta(f)}(g)(v) \\
         =& \beta(f)(v).g(v) \\
         =& f(P)(v).g(v) \\
         =& f(v).g(v)
     \end{align*}
     which can be checked to be equal to $f(C)(g)(v)$ as well.
 \end{proof}
 We shall compute the composition $\alpha_*(b)$ by analyzing the following operator:
 \begin{defi}
     Let $V$ be a finite-dimensional Euclidean vector space. Define an unbounded operator $B$ on $\mathcal{H}(V)$, with domain $\mathfrak{s}(V)$, by the formula 
\[
	(Bf)(v) = \sum_{i=1}^{n}x_ie_i(f(v)) +
	\sum_{i=1}^{n}\widehat{e}_i\left( \frac{\partial f}{\partial x_i}
	(v) \right) .  
\]
     Thus, $B=C+D$, where $C$ is the Clifford operator and $D$ is the Dirac operator.
 \end{defi}
\begin{rem}
     Note that $B=C+D$, where $C$ is the Clifford operator and $D$ is the Dirac operator. Furthermore, its straightforward to see that $B$ maps the Schwartz space $\mathfrak{s}(V)$ to itself, whence the operator $H=B^2$ is defined on $\mathfrak{s}(V)$.
\end{rem}
\begin{example}
Suppose $V = \mathbb{R}$. Then
\[
B = 
\begin{pmatrix}
	0 & x - d / dx\\
	x + d / dx & 0
\end{pmatrix},
\]
upon identifying $\mathcal{H}(V)$ with $L^{2}(\mathbb{R}) \oplus L^{2}
(\mathbb{R})$.

Observe that the operator $B$ maps the Schwartz space
$\mathfrak{s}(V,\operatorname{Cliff}^{\C}(V))$ into itself. So the operator $H = B^{2} $ is defined
on $\mathfrak{s}(V)$.
\end{example}
\begin{defi}
    Let $V$ be a finite-dimensional Euclidean vector space. Define the number operator on $\operatorname{Cliff}(V)$ by \[N=\sum_{i=1}^{n} \hat{e_i}e_i,\] where $e_i$'s form an orthonormal basis of $V$.
\end{defi}
\begin{lem}\label{lem: N=2p-n}
    $N$ maps the monomial $e_{i_1}...e_{i_p}$ (where $i_1 < i_2 <...<i_p$) to $(2p-n)e_{i_1}...e_{i_p}$.
\end{lem}
\begin{proof}
    We calculate
    \begin{align*}
        N(e_{i_1}...e_{i_p})&=\sum_{i=1}^{n}\hat{e_i}e_i(e_{i_1}...e_{i_p}) \\
        &=\sum_{i=1}^{n}\hat{e_i}(e_ie_{i_1}...e_{i_p}) \\
        &=\sum_{i=1}^{n}(-1)^{p+1}(e_ie_{i_1}...e_{i_p})e_i \\
        &=(2p-n)e_{i_1}...e_{i_p},
    \end{align*}
    as can be readily seen from the Clifford relations.
\end{proof}
The following lemma will be useful in calculations:
\begin{lem}  \label{lem: B2=C2+D2+N}
    $B^2=C^2+D^2+N$, and $N$ commutes with $C^2$ and $D^2$. 
\end{lem}
\begin{proof}
    The proof is a purely algebraic calculation with partial derivatives, and is done in detail in Proposition 3.2.13 of \cite{SLB}, so we skip the proof here. That $N$ commutes with $C^2$ and $D^2$ is also obvious from the definitions.
\end{proof}
\begin{prop} \label{HG Prop 1.16}
    Let $V$ be a finite dimensional Euclidean vector space of dimension $n$. Then,
    \begin{enumerate}
        \item the eigenvalues of $B^2$ are $2m ~(m=0,1,...)$, and each eigenvalue occurs with finite multiplicity.
        \item $B^2$ has an orthonormal basis of eigenfunctions within $\mathfrak{s}(V)$.
        \item the eigenvalue $0$ occurs precisely once, and the corresponding eigenfunction is $e^{-\frac{1}{2}\lVert v \rVert^2}.1$
    \end{enumerate}
\end{prop}
\begin{proof}
    We first consider the case when $V=\mathbb{R}$. In this case, a direct computation shows that 
    \[
    B^2=\begin{pmatrix}
        x^2-\frac{d^2}{dx^2}-1 & 0 \\
        0 & x^2-\frac{d^2}{dx^2}+1
    \end{pmatrix}.
    \]
    Thus, it suffices to show the following:
    
    \begin{clm}
        There is an orthonormal basis of eigenvectors for the operator 
    \[
    H=x^2-\frac{d^2}{dx^2}.
    \] defined on the Schwarz subspace of $L^2(\R)$, whose eigenvalues are all odd positive integers, and each eigenvalue has multiplicity one.
    \end{clm} 
    \begin{proof}
        Set $L=x+\frac{d}{dx}$, $R=x-\frac{d}{dx}$. Also, set $f_1(x)=e^{-\frac{1}{2}x^2}$. A direct calculation shows that 
        \[
        (RL)(f)(x)=x^2f(x)-f''(x)-f(x), (LR)(f)(x)=x^2f(x)-f''(x)+f(x)
        ,\]
        implying that $H=RL+I=LR-I$. Also, its direct to check that $Lf_1=0$, whence $H(f_1)=f_1$, implying $f_1$ is in the $1$-eigenspace of $B^2$. A direct calculation shows that $HR=RH+2R$, and so by induction, $HR^n=R^nH+2nR^n$. Thus, by setting $f_{n+1}:=R^nf_1$, we get $H(f_{n+1})=(2n+1)f_{n+1}$. Also, note that $f_{n+1}$ is a non-zero polynomial of degree $n$ times $f_1$. Thus, $\operatorname{span}_{n \in \N}\{f_n\}=\operatorname{span}_{n \in \N}\{x^ne^{\frac{-x^2}{2}}\}$, and the latter is known to be dense in $L^2(\R)$ (\cite{ROE}). Also, the functions $f_n$ are orthogonal, being eigenfunctions corresponding to distinct eigenvalues of a symmetric operator. So, we conclude that $\{\frac{f_n}{{\lVert f_n \rVert}_2}\}_{n \in \N}$ is an orthonormal basis of eigenfunctions of $B^2$, the eigenvalues of $H$ are exactly the odd positive integers, each having multiplicity one. 
    \end{proof}
    The general case follows by noting that $B^2$ is the tensor product of the restrictions of $B^2$ onto the span of each of the orthonormal basis vectors of $V$.

\iffalse

    The general case from \cref{lem: N=2p-n} and \cref{lem: B2=C2+D2+N}, since 
    \[
        B^2=C^2+D^2+N=\sum_{i=1}^{n} x_i^2 + \sum_{i=1}^{n} -\frac{\partial^2}{\partial x_i^2}+(2p-n) ~\text{on}~ \h_p(V)
    \]
    where $\mathcal{H}_p(V)$ denotes the subspace of $\mathcal{H}(V)$ comprised of functions $V \rightarrow \operatorname{Cliff}(V)$ whose values are combinations of the degree $p$ monomials $e_{i_1},...,e_{i_p}$. \textcolor{red}{An eigenbasis for $B^2$ can be found now by separation of variables.}
    \fi
\end{proof}
\begin{cor}
    Let $V$ be a finite dimensional Euclidean vector space. Then,
    \begin{enumerate}
        \item B is essentially self-adjoint
        \item The eigenvalues of $B$ are $\pm\sqrt{2m}~(m=0,1,...)$, and the eigenspace associated to each is finite dimensional.
        \item $B$ has an orthonormal basis of eigenfunctions within $\mathfrak{s}(V)$, and has compact resolvent.
        \item The kernel of $B$ is one-dimensional and is generated by the function $e^{-\frac{1}{2}\lVert v \rVert^2}$
    \end{enumerate}
\end{cor}
\begin{proof}
    \begin{enumerate}
        \item Note that $B$ is formally self-adjoint operator, by being the sum of the (essentially) self-adjoint operators $C$ and $D$. That $B$ is essentially self adjoint now follows from Theorem 2.2 of \cite{ST}.

        \item Note that, if $v$ is an eigenfunction of $B^2$ with eigenvalue $\lambda$, then \[B^2(Bv)=B(B^2v)=B(\lambda v)= \lambda Bv,\]
        implying $B$ leaves the eigenspaces corresponding to every eigenvalue of $B^2$ invariant. By \cref{HG Prop 1.16}, these eigenspaces are finite dimensional, whence $B$, viewed as an operator on the eigenspace corresponding to $2m$, has an eigenvalue by elementary linear algebra which has to be a square root of $2m$ ($m \geq 0$). As $B$ is odd, the eigenvalues are symmetric about the origin, and the conclusion follows. 
        \item That $B$ has an orthonormal basis consisting of eigenfunction now follows from the corresponding fact about $B^2$. That $B$ has compact resolvents also follows as the absolute value of the eigenvalues of $B$ tend to infinity.
        \item This follows immediately from the corresponding result about $B^2$, and the fact that $B$ leaves the eigenspaces corresponding to every eigenvalue of $B^2$ invariant.
    \end{enumerate}
\end{proof}
Our next goal is to show:
\begin{thm} \label{thm: HG Thm 1.17}
	The composition of $\alpha_t$ and $\beta$ is given by:
% https://q.uiver.app/#q=WzAsNCxbMCwwLCJcdFxcbWF0aGNhbHtTfSJdLFsxLDAsIlx0XFxtYXRoY2Fse1N9IFxcaGF0e1xcb3RpbWVzfSBcdFxcbWF0aGNhbHtTfSJdLFsyLDAsIlx0XFxtYXRoY2Fse1N9IFxcaGF0e1xcb3RpbWVzfSBDXzAoViwgXFxvcGVyYXRvcm5hbWV7Q2xpZmZ9KFYpKSJdLFszLDAsIlx0XFxtYXRoY2Fse0t9KFxcbWF0aGNhbHtIfShWKSkiXSxbMCwxLCJcdFxcRGVsdGEiXSxbMSwyLCJcdFxcb3BlcmF0b3JuYW1le2lkfSBcXGhhdHtcXG90aW1lc30gXFxiZXRhIl0sWzIsMywiXHRcXGFscGhhX3QiLDAseyJzdHlsZSI6eyJib2R5Ijp7Im5hbWUiOiJkb3R0ZWQifX19XV0=
\[\begin{tikzcd}
	{	\mathcal{S}} & {	\mathcal{S} \hat{\otimes} 	\mathcal{S}} & {	\mathcal{S} \hat{\otimes} C_0(V, \operatorname{Cliff}(V))} & {	\mathcal{K}(\mathcal{H}(V))}
	\arrow["{	\Delta}", from=1-1, to=1-2]
	\arrow["{	\operatorname{id} \hat{\otimes} \beta}", from=1-2, to=1-3]
	\arrow["{	\alpha_t}", dotted, from=1-3, to=1-4]
\end{tikzcd}\]is asymptotically equivalent to the asymptotic morphism $\gamma_t
	: \mathcal{S} \dashrightarrow \mathcal{K}(H)$ defined by
	\[
	\gamma_t(f) = f(t^{-1}B),
	\]
	for all $t \ge 1$.
\end{thm}

For the proof, we need some knowledge of some well known results from unbounded operator theory:
\begin{prop}\label{prop: mehler}[Mehler's formula]
	For the Clifford operator $C$ and the Dirac operator $D$ defined
	above we have the following identities for $s > 0$ 
	\[
	e^{-s(C^{2} +D^{2})} =
	e^{-\frac{1}{2}s_1C^{2}}e^{-s_2D^{2}}e^{-\frac{1}{2}s_1C^{2}}
	\quad \text{and}\quad
	e^{-s(C^{2} +D^{2})} =
	e^{-\frac{1}{2}s_1D^{2}}e^{-s_2C^{2}}e^{-\frac{1}{2}s_1D^{2}}
	\]
	with
	\[
	s_1 = \frac{\cosh (2s) - 1}{\sinh (2s)} \qquad \text{and}\qquad 
	s_2 = \frac{\sinh(2s)}{2}
	\]
\end{prop}
\begin{proof}
    See \cite{RUF} Theorem D.3.6 and Corollary D.3.7 for example.
\end{proof}

We also note the asymptotic conditions:

\begin{lem} \label{kkk}
	For an unbounded self-adjoint operator $X$ we have the following:
	\[
	\lim_{t \to \infty} \lVert {e^{-\frac{1}{2}\tau_1X^{2} } -
	e^{-\frac{1}{2}t^{-2}X^{2} }} \rVert = 0
	\]
	\[
	\lim_{t \to \infty} \lVert {e^{-\frac{1}{2}\tau_2X^{2} } -
	e^{-\frac{1}{2}t^{-2}X^{2} }} \rVert = 0
	\]
	\[
	\lim_{t \to \infty} \lVert{t^{-1}Xe^{-\frac{1}{2}\tau_1X^{2} } -
	t^{-1}Xe^{-\frac{1}{2}t^{-2}X^{2} }} \rVert= 0
	\]
	\[
	\lim_{t \to \infty} \lVert {t^{-1}Xe^{-\frac{1}{2}\tau_2X^{2} } -
	t^{-1}Xe^{-\frac{1}{2}t^{-2}X^{2} }} \rVert = 0
	\]
	where
	\[
	\tau_1 = \frac{\cosh(2t^{-2}) - 1}{\sinh(2t^{-2})} \qquad
	\text{and}\qquad  \tau_2 = \frac{\sinh(2t^{-2})}{2}
	\]
\end{lem}
\begin{proof}
    In light of the spectral theorem, it suffices to prove these when we replace $X$ by a real variable $x$, and the operator norm is replaced by the supremum norm on $C_0(\R)$. Then, all the limits become an easy exercise in calculus.
\end{proof}
With the relevant background mentioned, we are ready to proceed. We need another lemma, whose proof is similar to \cref{lem: HG Lem 1.9}.
\begin{lem} \label{lem: HG lem 1.12}
	Let $f, g \in C_0(\mathbb{R})$, then
	\[
	\lim_{t \to \infty} \lVert {[f(t^{-1}C), g(t^{-1}D)]} \rVert = 0
	\]
\end{lem}

\begin{lem} \label{kkkkkkkk}
	For operators $C$ and $D$ defined before, we have
	\begin{align*}
	    e^{-t^{-2}B^{2} } &\sim_{\text{asy}} e^{-t^{-2}C^{2} }e^{-t^{-2}D^{2} } \\
 t^{-1}Be^{-t^{-2}B^{2} } &\sim _{\text{asy}}
	t^{-1}(C+D)e^{-t^{-2}C^{2} }e^{-t^{-2}D^{2} }
	\end{align*}
\end{lem}
\begin{proof}
    From \cref{lem: B2=C2+D2+N}, we conclude that \[
    e^{-t^{-2}B^2}=e^{-{t^{-2}}(C^2+D^2)}e^{-t^{-2}N},
    \] and therefore, \cref{prop: mehler} implies that 
    \[
        e^{-t^{-2}B^2} =
	e^{-\frac{1}{2}\tau_1C^{2}}e^{-\tau_2D^{2}}e^{-\frac{1}{2}\tau_1C^{2}}e^{-t^{-2}N}
    ,\]
    where $\tau_1$ and $\tau_2$ are as in \cref{kkk}. It follows from \cref{kkk} that 
    \[
        e^{-t^{-2}B^2} \sim e^{-\frac{1}{2}t^{-2}C^2}e^{-t^{-2}D^2}e^{-\frac{1}{2}t^{-2}C^2}e^{-t^{-2}N},
    \] whence \cref{lem: HG lem 1.12} shows that 
    \[
         e^{-t^{-2}B^2} \sim  e^{-t^{-2}C^2} e^{-t^{-2}D^2},
    \] since the operators $ e^{-t^{-2}N}$ converge in norm to the identity operator.

    The other proof is similar.
\end{proof}
\begin{rem}
    The statements of the last lemma still hold good if we swap $C$ and $D$.
\end{rem}
We now prove \cref{thm: HG Thm 1.17}; it is also covered in \cite{SLB}, and given here for convenience.
\begin{proof} [Proof of \cref{thm: HG Thm 1.17}]
     Since $\s$ is generated by $u$ and $v$, it suffices to check that
	 \[
	\alpha(\operatorname{id} \hat{\otimes} \beta)(\Delta(f))
	\sim_{\text{asy}} \gamma_t(f),
	\]
	for $f = u$ and $f = v$.

	For $f = u$,
	\begin{align*}
		\alpha(\operatorname{id} \hat{\otimes}  \beta)(\Delta(u))
		&= \alpha(\operatorname{id} \hat{\otimes} \beta)(u
		\hat{\otimes} u)\\
		&= \alpha(u \hat{\otimes} \beta(u))\\
		&= \alpha(u \hat{\otimes} u(C))\\
		&= u(t^{-1}D)M_{u(C)_t}\\
		&= u(t^{-1}D)u(t^{-1}C)\\
		&=e^{-t^{-2}D^{2} }e^{-t^{-2}C^{2}}
	\end{align*}
	and
	\[
	\gamma_t(u) = u(t^{-1}B) = e^{-t^{-2}B^{2} },
	\]
	and these are both asymptotically equivalent by \cref{kkkkkkkk}.

	For $f = v$,
	\begin{align*}
		\alpha(\operatorname{id} \hat{\otimes}  \beta)(\Delta(u))
		&= \alpha(\operatorname{id} \hat{\otimes} \beta)(u
		\hat{\otimes} v + v \hat{\otimes} u)\\
		&= \alpha(u \hat{\otimes} \beta(v) + v \hat{\otimes}
		\beta(u))\\
		&= \alpha(u \hat{\otimes} v(C) + v \hat{\otimes} u(C))\\
		&\sim_\text{asy} \alpha(u \hat{\otimes} v(C)) + \alpha(v
		\hat{\otimes} u(C))\\
		&= u(t^{-1}D)M_{v(C)_t} + v(t^{-1}D)M_{u(C)_t}\\
		&= u(t^{-1}D)v(t^{-1}C) + v(t^{-1}D)u(t^{-1}C)\\
		&= e^{-t^{-2}D^{2} }t^{-1}Ce^{-t^{-2}C^{2} } +
		t^{-1}De^{-t^{-2}D^{2} }e^{-t^{-2}C^{2} }\\
		&= t^{-1}(C+D)e^{-t^{-2}C^{2} }e^{-t^{-2}D^{2} }
	\end{align*}
	and
	\[
	\gamma_t(v) = v(t^{-1}B) = t^{-1}Be^{-t^{-2}B^{2} }
	\]
	Then by \cref{kkkkkkkk} these are asymptotically equivalent.
	Hence the composition of $\alpha$ and $\beta$ is asymptotically
	equivalent to $\gamma_t$.
\end{proof}
We finally are able to prove \cref{thm: BP}. 

\begin{cor}
	The composition $\alpha_{*}\beta_{*}$ of the induced homomorphisms
	\[
	\beta_{*} : K_0(\mathbf{C}) \to K_0(C_0(V,
	\operatorname{Cliff}(V)))
	\]
	and
	\[
	\alpha_{*} : K_0(C_0(V, \operatorname{Cliff}(V))) \to
	K_0(\mathbf{C})
	\]
	is the identity homomorphism.
\end{cor}
\begin{proof}
	By \cref{thm: HG Thm 1.17}, the composition of $\alpha_{*}\beta_{*}$ is
	equivalent to the asymptotic morphism $\gamma$ given by
	$\gamma_t(f) = f(t^{-1}B)$. Since $f$ is in $\mathcal{S}$, each
	$\gamma_t$ is a $*$-homomorphism, and the asymptotic morphism $\gamma$ is homotopic to the
	$*$-homomorphism mapping $f$ to $f(B)$. (\cref{prop: legend functoriality} (iii)). Now, in light of \cref{cor: K-theory ungraded}, it suffices to define
	a homotopy between the functional calculus of $B$, $\Phi_B: \s \rightarrow \com(\h (V))$ and the map $\theta: \mathcal{S} \to
	\mathcal{K}(\mathcal{H}(V))$ defined by $\theta(f) = f(0)p$ where $p$
	is a rank 1 projection (by Corollary 3.2.15) onto the kernel of
	$B$. But this we have already done in the proof of \cref{lem: unbounded}.
\end{proof}
\subsection{The Periodicity Results}
Finally, we show the well known $2$-fold and $8$-fold periodicity of complex and Real $K$-theory. The advantage of working with Clifford algebras now comes handy; these periodicity results follow from well-known periodicity of Clifford algebras (see Theorem 1.2.6 of \cite{SCH}.)
Before, proceeding, we set up some notation.
\begin{defi}
    Let $\R^{p,q}$ denote the pseudo-Euclidean vector space $\R^{p+q}$, with the involution $\kappa_{p,q}$ given by $\kappa_{p,q}(x_1,x_2,...,x_{p+q})=(-x_1,-x_2,...,-x_p,x_{p+1},x_{p+2},...,x_{p+q})$. Set $Cl_{p,q}:=\operatorname{Cliff}(\R_{p,q},\kappa_{p,q})$, and $\mathbf{C}l_{p,q}:=\operatorname{Cliff}^{\C}(\R_{p,q},\kappa_{p,q})$.  
\end{defi}
\begin{prop}
    Let $A$ be a $\Z/2\Z$-graded $C^*$-algebra. Then, for all $n \in \N$, there is a natural isomorphism between $K_0(A \hat\otimes \mathbf{C}l_{n,0})$ and $K_n(A)$.
\end{prop}
\begin{proof}
    We have the following series of isomorphisms, with each one being natural:
    \begin{align*}
        K_0(A \hat\otimes \mathbf{C}l_{n,0}) &\cong K_0(A \hat\otimes \mathbf{C}l_{n,0} \hat\otimes \mathcal{C}(\mathbf{C}l_{0,n})) ~\text{( \cref{thm: BP}})\\ &\cong 
        K_0(A \hat\otimes \mathbf{C}l_{n,0} \hat\otimes C_0(\R^n) \hat\otimes \mathbf{C}l_{0,n}) ~\text{(since $\kappa_{0,n}=1$)}\\
        &\cong  K_0(A\hat\otimes C_0(\R^n) \hat\otimes \mathbf{C}l_{n,0}  \hat\otimes \mathbf{C}l_{0,n}) ~\text{(flip isomorphism)} \\
        &\cong K_0(A\hat\otimes C_0(\R^n) \hat\otimes \mathbf{C}l_{n,n}) ~\text{(\cref{lem: tensor product Cliff})}\\
        &\cong K_0(A\hat\otimes C_0(\R^n) \hat\otimes M_{2^n}(\mathbf{C})) \\
        &\cong K_0(A\hat\otimes C_0(\R^n))~\text{(K-theory stability)} \\
        &\cong K_n(A) ~\text{(\cref{higher K-group})}.
    \end{align*}
\end{proof}
\begin{cor}
    Let $A$ be a $\Z/2\Z$-graded $C^*$-algebra. If $A$ is complex, then there are natural isomorphisms $K_{n}(A) \cong K_{n+2}(A)$, and if $A$ is Real, then there are natural isomorphisms $K_{n}(A) \cong K_{n+8}(A)$.
\end{cor}
\begin{proof}
    This follows from the periodicity properties of Real and complex Clifford algebras, and K-theory stability. Indeed, we have $\mathbf{C}l_{8,0} \cong M_{16}(\mathbf{C})$ as Real $C^*$-algebras. In the complex case, we have $\mathbf{C}l_{2,0} \cong M_{2}(\mathbf{C})$ as complex $C^*$-algebras.
\end{proof}
\bibliographystyle{abbrv}
\bibliography{ref}
\end{document}